\documentclass[11pt]{article}
\usepackage[a4paper,margin=26mm]{geometry}
\usepackage[T1]{fontenc}
\usepackage{lmodern,amsmath,amssymb,amsthm,mathtools,microtype,enumitem,xcolor}
\definecolor{revisiongreen}{RGB}{0,110,65}
\usepackage[hidelinks]{hyperref}
\hypersetup{pdftitle={Logarithmic upper bounds and axial front selection for spatially periodic lattice Fisher--KPP equations},
 pdfauthor={Yuanyang Hu and Zhenkun Wang},
 pdfsubject={Periodic logarithmic bounds, axial front selection, arrival times, and a positive Dirichlet quotient comparison}}
\numberwithin{equation}{section}
\newcommand{\Z}{\mathbb Z}
\newcommand{\R}{\mathbb R}
\newcommand{\ee}{\mathrm e}
\newcommand{\dd}{\mathrm d}
\newcommand{\ii}{\mathrm i}
\newcommand{\1}{\mathbf 1}
\newcommand{\norm}[1]{\lVert #1\rVert}
\newcommand{\pheat}{\mathsf p}
\DeclareMathOperator{\supp}{supp}
\newtheorem{theorem}{Theorem}[section]
\newtheorem{lemma}[theorem]{Lemma}
\newtheorem{proposition}[theorem]{Proposition}
\newtheorem{corollary}[theorem]{Corollary}
\newtheorem{conjecture}[theorem]{Conjecture}
\theoremstyle{remark}
\newtheorem{remark}[theorem]{Remark}
\title{Logarithmic upper bounds and axial front selection\\
for spatially periodic lattice Fisher--KPP equations}
\author{Yuanyang Hu\textsuperscript{1}
\quad Zhenkun Wang\textsuperscript{2}\thanks{Corresponding author. Email:
[\texttt{zhenkunwang@ustc.edu.cn}]}\\[0.6em]
\small \textsuperscript{1}\,School of Mathematics and Statistics, Henan University\\
\small 475004 Kaifeng, China\\[0.3em]
\small \textsuperscript{2}\,School of Mathematical Sciences\\
\small University of Science and Technology of China\\
\small Hefei, Anhui 230026, China}
\date{}
\begin{document}
\maketitle
\begin{abstract}
We study the Cauchy problem $u_t=\Delta_{\Z^d}u+F(x,u)$ on $\Z^d$ ($d\geq2$)
with a spatially periodic Fisher--KPP reaction $F$ and a nonzero, finitely
supported initial datum $u_0:\Z^d\to[0,1]$. We establish a logarithmic upper
bound on the front position that holds uniformly in the observation direction,
with a logarithmic delay coefficient governed by the outward normal to the
asymptotic spreading body. When the linearized growth rate $F_s(x,0)$ is
independent of the longitudinal coordinate, we prove the matching axial
Bramson correction and show that the solution converges uniformly on
fixed-width half-tubes to a critical pulsating front up to a bounded, slowly
varying phase shift. For critical pulsating profiles $\Phi(s,x)$ satisfying
$\Phi(s,x)\leq C(1+s)\ee^{-\lambda s}$ for $s\geq0$ and $x\in\Z^d$, where
$\lambda>0$ is the critical spatial decay rate, we derive the precise leading-edge
asymptotic expansion and prove uniqueness up to translation. As a consequence,
in the axial regime we obtain sharp asymptotics for the arrival times
$T_\theta(x)=\inf\{t\geq0:u(t,x)\geq\theta\}$ ($0<\theta<1$), including the
asymptotic time lag between distinct level sets at a remote site and the
period-to-period passage time along the axis. For general periodic linearized
growth rates, we introduce a nonlinear subsolution built from a positive
quotient of linear Dirichlet evolutions on a retreating domain and formulate
sufficient kernel estimates for a matching directional lower bound.
\end{abstract}

\section{Introduction and main results}\label{sec:introduction}
Fix $d\geq2$, and let $u_0:\Z^d\to[0,1]$ be nonzero and finitely supported.
This paper investigates the logarithmic front delay and asymptotic profile
selection for spatially periodic Fisher--KPP equations on $\Z^d$ evolving from
finitely supported initial data. We establish a directional logarithmic upper
bound in every observation direction $e\in\mathbb S^{d-1}$. When the linearized
growth rate at zero is independent of the longitudinal coordinate, we prove the
matching axial Bramson correction and show that the solution converges on
fixed-width half-tubes to a critical pulsating front up to a bounded phase shift;
this profile convergence in turn yields sharp asymptotics for level-set arrival
times.

The classical reaction--diffusion model was introduced in 1937 by Fisher
\cite{Fisher1937} and by Kolmogorov, Petrovskii and Piskunov \cite{KPP1937}.
Multidimensional spreading in homogeneous media was established by Aronson and
Weinberger \cite{AW1978}, and order-preserving theories for discrete-time
migration and periodic habitats were subsequently developed by Weinberger
\cite{Weinberger1982,Weinberger2002}; see
\cite{FangYuZhao,FangZhao2014,LiangYiZhao,LiangZhao2007,LiangZhao} for
extensions to monotone evolution systems. In continuous periodic media, the
Freidlin--G\"artner formula \cite{FreidlinGartner1979} characterizes the
asymptotic spreading speed along a ray in terms of planar front speeds across
supporting directions. The theory of pulsating fronts and the principal-eigenvalue
characterization of minimal KPP wave speeds were developed by Berestycki, Hamel
and Nadirashvili \cite{BH2002,BHN2005}, and Rossi \cite{Rossi2017} later gave a
geometric proof extending the Freidlin--G\"artner relation to broader reaction
classes. As shown below, the outward unit normal that minimizes the
Freidlin--G\"artner variational formula also governs the directional logarithmic
delay on $\Z^d$.

Through the connection between KPP equations and branching Brownian motion due to
McKean \cite{McKean1975}, Bramson \cite{Bramson1978,Bramson} discovered the
universal logarithmic front delay: for $u_t=u_{xx}+u-u^2$ on $\R$ with a
nonzero initial datum $U_0:\R\to[0,1]$ supported in $[-R,R]$, the rightmost
$\theta$-level set lies at $2t-\frac32\log t+O(1)$ for each $\theta\in(0,1)$.
Hamel, Nolen, Roquejoffre and Ryzhik \cite{HNRRShort} gave a direct PDE proof
by comparing the solution with a linear Dirichlet problem on a domain with a
moving boundary, obtaining profile convergence along level sets. For step-like
initial data $V_0:\R\to[0,1]$ satisfying $V_0(x)=1$ for $x\leq x_1$ and
$V_0(x)=0$ for $x\geq x_2$ ($x_1<x_2$), Nolen, Roquejoffre and Ryzhik
\cite{NRRSingle} proved convergence to a single translate of the critical wave,
and later justified the Ebert--van Saarloos $-3\sqrt{\pi}/\sqrt t$ shift
asymptotics \cite{EbertVanSaarloos,NRRRefined}; higher-order expansions of the
front position and profile were obtained by Graham \cite{GrahamPrecise}.

In higher dimensions, G\"artner \cite{Gartner1982} established logarithmic
front localization for homogeneous KPP equations. For $u_t=\Delta u+u-u^2$ on
$\R^d$ with smooth initial data $U_0$ satisfying
$\1_{B_{R_1}}\leq U_0\leq\1_{B_{R_2}}$ ($0<R_1<R_2$), where
$B_R=\{x\in\R^d:|x|<R\}$, Roquejoffre, Rossi and Roussier-Michon \cite{RRR2019}
proved convergence to an expanding critical wave profile with a Lipschitz
direction-dependent limiting shift; here the logarithmic delay combines the
one-dimensional critical coefficient \textcolor{revisiongreen}{$3/2$} with the  {$(d-1)/2$}
contribution from transverse diffusion.

For one-dimensional spatially periodic KPP equations with positive periodic
equilibrium $\pi_{\mathrm{G}}$ and critical decay rate $\lambda_{\mathrm{G}}>0$,
Giletti \cite{Giletti} proved that  {the solution with differentiable initial datum}
$0\leq U_0\leq\pi_{\mathrm{G}}$ ($U_0\not\equiv0$) satisfying
$|U_0(x)|+|U_0'(x)|\leq A\ee^{-\lambda_{\mathrm{G}}x}$  {converges} on $[0,\infty)$
to a minimal-speed pulsating front with a shift $m(t)=o(t)$. For compactly
supported nonnegative initial data $U_0\not\equiv0$ under concave periodic
reactions, Hamel, Nolen, Roquejoffre and Ryzhik
\cite[Theorems~1.1--1.2]{HNRRPeriodic} established the $3/(2\lambda^*)$
logarithmic delay and convergence to the family of minimal-speed pulsating
fronts up to a bounded shift. In the multidimensional periodic setting on $\R^d$
with a periodic growth rate multiplying a KPP nonlinearity, Shabani \cite[Theorems~1.2--1.3]{Shabani} proved that each
observation direction $e$ admits a unique minimizing normal $e'$ and derived the
directional logarithmic coefficient $(d/2+1)/(\lambda^*(e')e\cdot e')$, where
$\lambda^*(e')$ denotes the critical decay rate in direction $e'$. From the
probabilistic viewpoint, Lubetzky, Thornett and Zeitouni
\cite{LubetzkyThornettZeitouni} established phase-dependent limiting
distributions for the maximum of periodic branching Brownian motion along
subsequences of fixed spatial phase, identified the corresponding KPP profile
for Heaviside initial data $\1_{\{x\geq0\}}$, and analyzed a periodic branching
random walk. For initial data on $\R$ with exponential decay and algebraic prefactors,
Alfaro, Giletti and Xiao \cite{AlfaroGilettiXiao} showed that front-like data
satisfying their Assumption~1.2 can induce either a logarithmic delay or a
logarithmic advance, while Zhang \cite{ZhangSharp} obtained sharp level-set
positions under conditions~\textup{(H1)} or~\textup{(H2)} and single-wave
convergence under the hypotheses of \cite[Theorems~1.3--1.5]{ZhangSharp}.

Asymptotic profile selection relies crucially on the leading-edge tail behavior
of traveling fronts. Hamel \cite{HamelDecay} analyzed exponential decay and
time monotonicity of monostable pulsating fronts, including the $(1+s)\ee^{-\lambda s}$
behavior at the critical speed, and Hamel and Roques
\cite[Theorems~1.1 and~1.5]{HamelRoques} established uniqueness and stability in
continuous periodic media. On homogeneous lattices, traveling-wave existence,
uniqueness and exact tail asymptotics were developed by Chen and Guo
\cite{ChenGuo2003}, Chen, Fu and Guo \cite{ChenFuGuo2006}, and Guo and C.-H. Wu
\cite{GuoWu2008}. On periodic lattices, Guo and Hamel \cite{GuoHamel} constructed
minimal-speed pulsating fronts in one dimension, and Guo and C.-H. Wu
\cite{GuoWu2010} proved existence of the minimal wave speed and strict profile
monotonicity in two dimensions. Uniqueness and stability of supercritical
periodic lattice waves with prescribed exponential tails were investigated in
\cite{ChengLiLin2015,GuoWu,WuChenHsu,ZhaoWuXue2024}. For the critical pulsating
fronts considered here, the characteristic dispersion relation has a double root;
we establish the critical tail expansion, including its linear prefactor
and periodic cell correctors, and use it to prove profile uniqueness by a
sliding argument.

On general heterogeneous and almost-periodic lattices, spreading speeds and
traveling fronts were studied via generalized principal eigenvalues by Liang and
Zhou \cite{LiangZhou2020} and Liang, Wang, Zhou and Zhou
\cite{LiangWangZhouZhou2024}. Regarding the logarithmic Bramson correction on
lattices, Besse, Faye, Roquejoffre and Zhang \cite{BFRZ} proved logarithmic
level-set localization and convergence to a boundedly shifted family of critical
waves on $\Z$ for nonzero initial data $U_0:\Z\to[0,1]$ vanishing on $[J,\infty)$.
Recently, Hu \cite[Theorems~1.1--1.3]{HuAxial} established the multidimensional
axial Bramson correction, convergence on fixed-width axial half-tubes, and
directional logarithmic upper bounds for homogeneous Fisher--KPP equations on
$\Z^d$. Analogous logarithmic corrections have also been proved for nonlocal
dispersal equations in one dimension by Graham \cite{GrahamNonlocal} and in
higher dimensions for compactly supported symmetric kernels $\mathcal K\in C_c^\infty(\R^d)$
and data $U_0\in C_c^\infty(\R^d)$ by Roquejoffre and Zhang
\cite{RoquejoffreZhang}, where the logarithmic coefficient reflects the
anisotropic spreading geometry.

Let $L_1,\ldots,L_d$ be positive integers, and put
$\Gamma=L_1\Z\times\cdots\times L_d\Z$ and
$\mathcal C=\prod_{i=1}^d\{0,\ldots,L_i-1\}$. For $x\in\Z^d$,
$x\bmod\Gamma$ denotes the unique representative of the coset
$x+\Gamma$ that belongs to $\mathcal C$. Consider
\begin{equation}\label{eq:main}
 u_t(t,x)=\Delta_{\Z^d}u(t,x)+F(x,u(t,x))\text{ for }t>0,\ x\in\Z^d,
\end{equation}
with $u(0,x)=u_0(x)$. For $m\geq1$, we write
$\Delta_{\Z^m}h(x)=\sum_{z\in\Z^m,\,|z|_1=1}(h(x+z)-h(x))$
for the nearest-neighbor discrete Laplacian on $\Z^m$.
Here $|\cdot|_1$ denotes the $\ell^1$ norm.
Throughout the paper, we assume that $F(x+\ell,s)=F(x,s)$ for all $x\in\Z^d$, $\ell\in\Gamma$, and $s\in[0,1]$, and that
\begin{equation}\label{eq:periodic-kpp}
\begin{gathered}
 F(x,\cdot)\in C^2([0,1]),\quad F(x,0)=F(x,1)=0,\\
 a(x):=F_s(x,0)>0,\quad F_s(x,1)<0,\\
 0<F(x,s)\leq a(x)s\text{ for }0<s<1.
\end{gathered}
\end{equation}
Since $F(x,\cdot)\in C^2([0,1])$ and $\mathcal C$ is finite, define
\begin{equation}\label{eq:KF}
 K_F:=\frac12\max_{x\in\mathcal C}\|F_{ss}(x,\cdot)\|_{L^\infty(0,1)}.
\end{equation}
Then $K_F<\infty$, and Taylor's formula and \eqref{eq:periodic-kpp} give
\begin{equation}\label{eq:quadratic-kpp-loss}
 0\le a(x)s-F(x,s)\le K_Fs^2 \text{ for }
x\in\Z^d,\ 0\le s\le1.
\end{equation}

Indeed, Taylor's formula gives
\[
 F(x,s)-a(x)s=\int_0^s(s-r)F_{ss}(x,r)\,\dd r.
\]
Hence $|F(x,s)-a(x)s|\leq\|F_{ss}(x,\cdot)\|_\infty s^2/2\leq K_Fs^2$.
The inequality $F(x,s)\leq a(x)s$ in \eqref{eq:periodic-kpp} gives
$a(x)s-F(x,s)\geq0$ for $0<s<1$.
Periodicity extends both bounds to every site, and continuity includes the endpoints $s=0,1$.

Throughout the paper, the initial datum $u_0:\Z^d\to[0,1]$ satisfies
\begin{equation}\label{eq:u0}
 0\leq u_0\leq1,\qquad u_0\not\equiv0,\qquad \supp u_0\text{ is finite}.
\end{equation}
For $u_0$ satisfying \eqref{eq:u0}, Section~\ref{sec:spectral} establishes
a unique global solution $u$ of \eqref{eq:main} in $\ell^\infty(\Z^d)$
with $u(0,\cdot)=u_0$ and $0\leq u(t,x)\leq1$ for all $t\geq0$.
Constants may depend on $d$, the periods, $F$ and $u_0$, unless a more
specific dependence is stated. We use $\sup\varnothing=-\infty$ for
level positions. A period cell consisting of one vertex gives the spatially homogeneous reaction--diffusion equation.

We write $e_i$ for the $i$th coordinate unit vector and
$\mathbb S^{d-1}=\{e\in\R^d:|e|=1\}$ for the unit sphere, where
$|\cdot|$ denotes the Euclidean norm.
For a set $E$ of lattice sites, $|E|$ denotes its cardinality, with $|E|=\infty$ if $E$ is infinite. Finite-cell vector norms are Euclidean, and matrix norms are their induced
operator norms. For $p\in\R^d$, define $B_p$ on $\Gamma$-periodic functions,
identified with vectors indexed by $\mathcal C$, by
\begin{equation}\label{eq:cell}
 (B_p\phi)(x)=\sum_{\substack{z\in\Z^d\\|z|_1=1}}
       \bigl(\ee^{-p\cdot z}\phi(x+z)-\phi(x)\bigr)+a(x)\phi(x),
 \qquad \phi(x+\ell)=\phi(x)\text{ for } \ell\in\Gamma.
\end{equation}
Let $H(p)$ be the principal eigenvalue of $B_p$: it is the simple real
eigenvalue equal to $\max\{\operatorname{Re}z:z\in\sigma(B_p)\}$.
Here $\sigma(B_p)$ denotes the spectrum of $B_p$, and
$\operatorname{Re}z$ and $\operatorname{Im}z$ denote the real and imaginary
parts of $z$, respectively. The
finite matrix $B_p$ is irreducible and has nonnegative off-diagonal entries,
so Perron--Frobenius theory \cite{Seneta} gives
$\phi_p,\phi_p^*>0$ (right and left eigenvectors of $B_p$), normalized by
$\sum_{x\in\mathcal C}\phi_p(x)=1$ and
$\sum_{x\in\mathcal C}\phi_p(x)\phi_p^*(x)=1$.
We extend $\phi_p$ and $\phi_p^*$ $\Gamma$-periodically to $\Z^d$.
Define
\begin{equation}\label{eq:geometry}
 \mathcal I(v)=\sup_{p\in\R^d}\{p\cdot v-H(p)\},\text{ for }v\in\R^d,\qquad
 \mathcal W=\{v\in\R^d:\mathcal I(v)\leq0\}.
\end{equation}
For $e\in\mathbb S^{d-1}$ there is a unique $w(e)>0$ such that
$\mathcal I(w(e)e)=0$. Set
\begin{equation}\label{eq:pe}
 v_e=w(e)e,\qquad p_e=\nabla\mathcal I(v_e),\qquad
 \alpha_e=p_e\cdot e,\qquad \kappa_d=d/2+1.
\end{equation}
By Lemma~\ref{lem:geometry}, $w(e)$, $v_e$ and $p_e$ are well defined and
$\alpha_e>0$, with $\nabla H(p_e)=v_e$ and $H(p_e)=p_e\cdot v_e$.
For $n\in\mathbb S^{d-1}$, define
$c(n)=\min_{\lambda>0}H(\lambda n)/\lambda$.
Put $n_e=p_e/|p_e|$ and $\lambda_e=|p_e|$. Lemma~\ref{lem:geometry} identifies
$n_e$ as the unique minimizer of $n\mapsto c(n)/(n\cdot e)$ over
$n\in\mathbb S^{d-1}$ with $n\cdot e>0$, and $\lambda_e$ as the unique
minimizer defining $c(n_e)$.
In view of \eqref{eq:geometry}, $v_e$ is the boundary point of
$\mathcal W$ on the ray in direction $e$, with outward unit normal
$n_e$.

\begin{theorem}\label{thm:upper}
Under \eqref{eq:periodic-kpp}--\eqref{eq:u0}, there is $t_0\geq1$, independent of $e$, such that the following estimates hold uniformly in $e\in\mathbb S^{d-1}$.
\begin{enumerate}[label=\textup{(\roman*)}]
\item There is $C\geq1$ such that, with
$z_e(t,x)=p_e\cdot x-tH(p_e)+\kappa_d\log t$,
for $t\geq t_0$ and $x\in\Z^d$ one has
\begin{equation}\label{eq:upper-main}
 u(t,x)\leq C(1+z_e(t,x))\ee^{-z_e(t,x)}
 \qquad\text{whenever }z_e(t,x)\geq0.
\end{equation}
\item For each $\theta\in(0,1)$ and $\rho\geq0$, there is
$C_{\theta,\rho}$ such that $u(t,x)<\theta$ whenever
\begin{equation}\label{eq:ray-upper}
 t\geq t_0,\quad |x-re|\leq\rho,\quad r\geq0,\quad
 r\geq w(e)t-\frac{d+2}{2\alpha_e}\log t+C_{\theta,\rho}.
\end{equation}
\item {There is $C\geq1$, independent of $e$, such that for every $t\geq0$,
\begin{equation}\label{eq:mass-main}
 \sum_{x\in\Z^d}\phi_{p_e}^*(x)
       \ee^{p_e\cdot x-tH(p_e)}u(t,x)\leq C(1+t)^{-1/2}.
\end{equation}}
\end{enumerate}
\end{theorem}

Since $\alpha_e=\lambda_e n_e\cdot e$, the logarithmic coefficient in
\eqref{eq:ray-upper} is $(d+2)/(2\lambda_e n_e\cdot e)$ and depends on
both the critical exponential decay rate and the angle between the
outward normal and the observation direction. By \eqref{eq:upper-main},
for every $\theta\in(0,1)$ there is $C_\theta>0$ such that, for all $t\geq t_0$, with $t_0$ from Theorem~\ref{thm:upper},
\begin{equation}\label{eq:outer-body}
 \{x\in\Z^d:u(t,x)\geq\theta\}\subset
 \bigcap_{e\in\mathbb S^{d-1}}
 \{x\in\Z^d:p_e\cdot x\leq tH(p_e)-\kappa_d\log t+C_\theta\}.
\end{equation}

For the matching axial estimate, write $x=(j,y)\in\Z\times\Z^{d-1}$.
For $B\ge0$, define
\[
 \mathcal T_B^+:=\{(j,y)\in\Z\times\Z^{d-1}:j\ge0,\ |y|\le B\}.
\]
We assume
\begin{equation}\label{eq:transverse-linearization}
 F_s((j,y),0)=a_\perp(y)>0\text{ for } j\in\Z,\ y\in\Z^{d-1}.
\end{equation}
Let $\mathcal C_\perp=\prod_{i=2}^d\{0,\ldots,L_i-1\}$ and
$\mathcal A_\perp=\Delta_{\Z^{d-1}}+a_\perp(y)$ act with periodic
boundary conditions on $\mathcal C_\perp$. Define
$a_{\mathrm{eff}}$ and $\psi_\perp>0$ by
\begin{equation}\label{eq:transverse-eigenpair}
 \mathcal A_\perp\psi_\perp=a_{\mathrm{eff}}\psi_\perp,
 \qquad \sum_{y\in\mathcal C_\perp}\psi_\perp(y)^2=1.
\end{equation}
We extend $\psi_\perp$ periodically to $\Z^{d-1}$.
In particular, $\min_{y\in\mathcal C_\perp}a_\perp(y)\leq a_{\mathrm{eff}}\leq\max_{y\in\mathcal C_\perp}a_\perp(y)$.
In the special case
\begin{equation}\label{eq:common-linearization}
 F_s(x,0)=a>0\text{ for } x\in\Z^d\textcolor{blue}{,}
\end{equation}
one has $a_{\mathrm{eff}}=a$.

\begin{theorem}\label{thm:axial}
Assume \eqref{eq:periodic-kpp}, \eqref{eq:u0} and
\eqref{eq:transverse-linearization}. Let $\lambda_*>0$ minimize
$[2(\cosh\lambda-1)+a_{\mathrm{eff}}]/\lambda$ over $\lambda>0$,
put $c_*=2\sinh\lambda_*$,
and set
\begin{equation}\label{eq:md}
 m_d(t)=c_*t-\frac{d+2}{2\lambda_*}\log t\text{ for }t\geq2.
\end{equation}
\begin{enumerate}[label=\textup{(\roman*)}]
\item For each $\theta\in(0,1)$ there are $C_\theta>0$ and $t_\theta^{\mathrm{ax}}\geq2$ such that
\begin{equation}\label{eq:axial-position}
 \left|\sup\{j\in\Z:u(t,j,0)\geq\theta\}-m_d(t)\right|
 \leq C_\theta\text{ for } t\geq t_\theta^{\mathrm{ax}}.
\end{equation}
Moreover, $u(t,j,0)\geq\theta$ whenever $t\geq t_\theta^{\mathrm{ax}}$ and
$j\in\Z$ satisfies $0\leq j\leq m_d(t)-C_\theta$.

\item For each $\theta\in(0,1)$ and $A\geq0$, there is $C_{\theta,A}>0$
such that, for all large $t$, some $j_t\in\Z$ satisfies
\begin{equation}\label{eq:diffusive-section}
 |j_t-m_d(t)|\leq C_{\theta,A},\qquad
 \inf_{\substack{y\in\Z^{d-1}\\|y|\leq A\sqrt t}}u(t,j_t,y)\geq\theta.
\end{equation}
\item For every $A\geq0$ and $\sigma>0$ there are
$b_{A,\sigma},B_{A,\sigma}>0$ and $T_{A,\sigma}\geq2$ such that,
with $\xi=j-m_d(t)$,
\begin{equation}\label{eq:axial-leading-edge}
 b_{A,\sigma}(1+\xi)\ee^{-\lambda_*\xi}
 \leq u(t,j,y)\leq
 B_{A,\sigma}(1+\xi)\ee^{-\lambda_*\xi}
\end{equation}
whenever $t\geq T_{A,\sigma}$, $1\leq\xi\leq\sigma\sqrt t$ and
$|y|\leq A\sqrt t$.
\end{enumerate}
\end{theorem}

 For the negative $j$-axis, set $\widetilde j=-j$,
$\widetilde F((\widetilde j,y),s)=F((-\widetilde j,y),s)$ and
$\widetilde u(t,\widetilde j,y)=u(t,-\widetilde j,y)$.
Then $\widetilde u$ solves \eqref{eq:main} with reaction $\widetilde F$
and initial datum $\widetilde u(0,\widetilde j,y)=u_0(-\widetilde j,y)$.
Since $\widetilde F$ shares the same periods and satisfies
$\partial_s\widetilde F((\widetilde j,y),0)=a_\perp(y)$,
Theorem~\ref{thm:axial} yields \eqref{eq:axial-position}--\eqref{eq:axial-leading-edge}
with $j$ replaced by $-j$. More generally, whenever $F_s(x,0)$ is independent of
a coordinate $x_i$, permuting coordinates so that $x_i$ serves as the longitudinal
variable gives the corresponding results along both signed $x_i$-axes; in particular,
under \eqref{eq:common-linearization} Theorem~\ref{thm:axial} applies along every
coordinate axis.

Theorem~\ref{thm:axial} applies, for example, to $F((j,y),s)=a_\perp(y)s(1-s)$ even when $a_\perp$ is nonconstant, and
to $F((j,y),s)=a_\perp(y)s(1-s)(1+\varepsilon(j,y)s)$ when
$0\leq\varepsilon\leq\varepsilon_0<1$ is periodic. For these reactions,
the lower comparison in Section~\ref{sec:lower} uses $\psi_\perp$
from \eqref{eq:transverse-eigenpair}.

An entire solution of \eqref{eq:main} is defined for every $t\in\R$.
Following \cite[Section~1.1]{BHN2005} and \cite[Section~1]{GuoHamel}, let
$V:\R\times\Z^d\to(0,1)$ be an entire solution of \eqref{eq:main}.
For $n\in\mathbb S^{d-1}$ and $c>0$, let
$\Phi:\R\times\Z^d\to(0,1)$. We call $V$ a pulsating front in
direction $n$ with speed $c$ and end states $1,0$ if
\[
 V(t,x)=\Phi(x\cdot n-ct,x)\text{ for }(t,x)\in\R\times\Z^d,
\]
and $\Phi$ satisfies
$\Phi(s,x+\ell)=\Phi(s,x)$ for every $s\in\R$, $x\in\Z^d$ and
$\ell\in\Gamma$, and
\[
 \lim_{s\to-\infty}\sup_{x\in\mathcal C}|\Phi(s,x)-1|=0,
 \qquad
 \lim_{s\to+\infty}\sup_{x\in\mathcal C}\Phi(s,x)=0.
\]
The end-state limits of $\Phi$ are uniform over $x\in\Z^d$. We call $V$ critical if $c=c(n)$, where
$c(n)=\min_{\lambda>0}H(\lambda n)/\lambda$.
Set $r=x\cdot n-ct$. For $R>0$, the regions $r\le-R$ and $r\ge R$
lie behind and ahead of the front, respectively.

For $r_0\in\mathcal C$, define
\[
 F^{[r_0]}(z,s):=F(r_0+z,s)\text{ for }z\in\Z^d,\ 0\le s\le1.
\]
 {The shifted reaction} $F^{[r_0]}$ has period lattice $\Gamma$ and satisfies
\eqref{eq:periodic-kpp}. Centering $u$ at a site with residue
$r_0$ replaces $F$ in \eqref{eq:main} by $F^{[r_0]}$, with $z$ as the
spatial variable.

The next theorem gives the asymptotic expansion of $\Phi$ as $s\to+\infty$ and uniqueness up to translation.
\begin{theorem}\label{thm:critical-uniqueness}
Assume \eqref{eq:periodic-kpp}. Fix $n\in\mathbb S^{d-1}$, let
$\lambda$ be the minimizer of $\eta\mapsto H(\eta n)/\eta$ over
$\eta>0$, and put
$c=H(\lambda n)/\lambda$, $p=\lambda n$,
$D_n=\tfrac12n^{\mathsf T}\nabla^2H(p)n$ and
$\dot\phi_p=\left.\partial_\eta\phi_{\eta n}\right|_{\eta=\lambda}$,
with the normalization specified after \eqref{eq:cell}.
Let $\Phi$ be a $\Gamma$-periodic pulsating profile in direction $n$
with speed $c$ and end states $1,0$. Suppose that, for some $C>0$,
\begin{equation}\label{eq:front-critical-upper}
 0<\Phi(s,x)\leq C(1+s)\ee^{-\lambda s}
             \text{ for } s\geq0,\ x\in\mathcal C.
\end{equation}
\begin{enumerate}[label=\textup{(\roman*)}]
\item There are $A_\Phi>0$, $B_\Phi\in\R$ and $\varepsilon>0$ such that,
uniformly in the cell as $s\to+\infty$,
\begin{equation}\label{eq:critical-tail-expansion}
 \Phi(s,x)=\ee^{-\lambda s}
 \left[(A_\Phi s+B_\Phi)\phi_p(x)-A_\Phi\dot\phi_p(x)
                                  +O(\ee^{-\varepsilon s})\right].
\end{equation}
Here $A_\Phi$ is given by
\begin{equation}\label{eq:critical-amplitude}
 A_\Phi=\frac1{D_n}\int_{\R}\ee^{\lambda s}
       \sum_{x\in\mathcal C}\phi_p^*(x)
           \bigl(a(x)\Phi(s,x)-F(x,\Phi(s,x))\bigr)\,\dd s.
\end{equation}
\item Each such $\Phi$ is strictly decreasing in $s$. Any two such profiles differ by a translation in $s$.
\end{enumerate}
\end{theorem}

By \eqref{eq:axial-leading-edge}, the limits obtained in
Proposition~\ref{prop:profile-limits} satisfy
\eqref{eq:front-critical-upper}. If $V(t,z)$ is a limit for
$F^{[r_0]}$, set $\widetilde V(t,x)=V(t,x-r_0)$. Then $\widetilde V$ satisfies
\eqref{eq:main}, and the pulsating profile of $\widetilde V$ satisfies
\eqref{eq:front-critical-upper} after adjusting the tail constant.
Theorem~\ref{thm:critical-uniqueness} identifies the profiles of $\widetilde V$ up to translation.
For a profile $\Phi(s,x)$, write $\partial_1\Phi(s,x)=\partial_s\Phi(s,x)$
for differentiation with respect to the first, real variable.

\begin{theorem}\label{thm:selection}
Under the assumptions of Theorem~\ref{thm:axial}, there is a unique
critical pulsating profile $\Phi_*$ in direction $e_1$ with speed $c_*$
in the class \eqref{eq:front-critical-upper} with $\lambda=\lambda_*$,
normalized by $\Phi_*(0,0)=1/2$.
There is a bounded phase $\zeta\in C^\infty([2,\infty))$, with
$\zeta'(t)\to0$ as $t\to\infty$, such that, for every fixed $B\geq0$,
\begin{enumerate}[label=\textup{(\roman*)}]
\item
\begin{equation}\label{eq:profile-selection}
 \sup_{\substack{j\in\Z,\ j\geq0\\y\in\Z^{d-1},\ |y|\leq B}}
 |u(t,j,y)-\Phi_*(j-m_d(t)+\zeta(t),(j,y))|\to0\qquad\text{as }t\to\infty.
\end{equation}
\item
\begin{equation}\label{eq:profile-time-derivative}
 \sup_{\substack{j\in\Z,\ j\geq0\\y\in\Z^{d-1},\ |y|\leq B}}
 \left|u_t(t,j,y)+c_*\partial_1\Phi_*
       (j-m_d(t)+\zeta(t),(j,y))\right|\to0\qquad\text{as }t\to\infty.
\end{equation}
\end{enumerate}
\end{theorem}

 Since $\zeta'(t)\to0$, we have $\sup_{|s|\leq S}|\zeta(t+s)-\zeta(t)|\to0$ as $t\to\infty$ for each fixed $S\geq0$.
Moreover, if $t_k\to\infty$ and $x_k=(j_k,y_k)$ satisfy
$|j_k-m_d(t_k)|\leq C$, $|y_k|\leq A\sqrt{t_k}$ and
$x_k\bmod\Gamma=r_0$ for fixed $C,A\geq0$ and $r_0\in\mathcal C$, then
Proposition~\ref{prop:profile-limits} and Theorem~\ref{thm:critical-uniqueness}
imply that whenever $u(t_k,x_k)\to\theta\in(0,1)$,
\begin{equation}\label{eq:residue-limit}
 u(t_k+s,x_k+z)\to
 \Phi_*(z_1-c_*s+\xi_{\theta,r_0},z+r_0),\qquad
 \Phi_*(\xi_{\theta,r_0},r_0)=\theta,
\end{equation}
uniformly for $s$ in compact time intervals and $z$ in finite subsets of $\Z^d$. For sequences
with $j_k\geq0$ and bounded $y_k$, the localization bound
$|j_k-m_d(t_k)|\leq C$ is a direct consequence of
$u(t_k,j_k,y_k)\to\theta\in(0,1)$, so \eqref{eq:residue-limit}
holds along any such level sequence with fixed cell residue $r_0$.

Put $U_*(s,x)=\Phi_*(x_1-c_*s,x)$. By
\eqref{eq:profile-selection}, $u(t,x)$ is approximated uniformly on each
$\mathcal T_B^+$, for fixed $B\geq0$, by
\begin{equation}\label{eq:temporal-phase-convention}
 U_*\!\left(t-\frac{d+2}{2c_*\lambda_*}\log t
                       -\frac{\zeta(t)}{c_*},x\right).
\end{equation}
Thus the spatial shift $\zeta$ corresponds to the bounded temporal phase
$-\zeta/c_*$.

To describe arrival times, for $x=(j,y)$ and $\theta\in(0,1)$ define
\begin{equation}\label{eq:arrival-definition}
 T_\theta(j,y)=\inf\{t\geq0:u(t,j,y)\geq\theta\},
 \qquad \inf\varnothing=+\infty.
\end{equation}
For a residue $r\in\mathcal C$, let $\xi_{\theta,r}$ be the unique real
number satisfying $\Phi_*(\xi_{\theta,r},r)=\theta$.

\begin{corollary}\label{cor:arrival-times}
Assume the hypotheses of Theorem~\ref{thm:axial}. Fix $B\geq0$ and
$0<\theta_-<\theta_+<1$. There are $\delta>0$ and $t_0$ such that
\begin{equation}\label{eq:temporal-transversality}
 u_t(t,j,y)\geq\delta
 \text{ for }t\geq t_0,\ j\geq0,\ |y|\leq B,
                  \ \theta_-\leq u(t,j,y)\leq\theta_+.
\end{equation}
For some $J\geq0$, the equation $u(t,j,y)=\theta$ has a unique solution
$t\geq0$ for every $(j,y)\in\mathcal T_B^+$ with $j\geq J$ and
$\theta\in[\theta_-,\theta_+]$. Moreover,
\begin{equation}\label{eq:arrival-logarithm}
 T_\theta(j,y)=\frac{j}{c_*}
       +\frac{d+2}{2\lambda_*c_*}\log j+O(1).
\end{equation}
Writing $r(j,y)=(j,y)\bmod\Gamma$, one has, uniformly over these sites
and levels as $j\to\infty$,
\begin{align}
 T_{\theta_2}(j,y)-T_{\theta_1}(j,y)
  &=\frac{\xi_{\theta_1,r(j,y)}-\xi_{\theta_2,r(j,y)}}{c_*}+o(1),
       \label{eq:arrival-level-gap}\\
 T_\theta(j+L_1,y)-T_\theta(j,y)&=\frac{L_1}{c_*}+o(1).
       \label{eq:arrival-cell-gap}
\end{align}
Here $\theta_-\leq\theta_1\leq\theta_2\leq\theta_+$.
\end{corollary}
\paragraph{Homogeneous comparison.}
In the special case \eqref{eq:common-linearization} where $F_s(x,0)=a>0$ is constant on $\Z^d$, we have $a_{\mathrm{eff}}=a$, and the axial level-set estimate \eqref{eq:axial-position} can also be deduced directly from its homogeneous counterpart \cite[Theorem~1.1]{HuAxial} by sandwiching $F(x,\cdot)$ between homogeneous $C^2$ KPP envelopes $f_\pm$ with $f_\pm'(0)=a$ (see Remark~\ref{rem:homogeneous-comparison} below).  {When $d=1$ and $F$ is homogeneous on $\Z$, Besse et al.\ \cite[Theorems~1--2]{BFRZ} established the $3\log t/(2\lambda_{\mathrm h})$ delay and convergence to a boundedly shifted critical wave, where $\lambda_{\mathrm h}>0$ denotes the critical decay rate for that homogeneous equation. For the continuous periodic equation $u_t=\Delta u+\mu(x)f(u)$ on $\R^d$, Shabani \cite[Theorems~1.2--1.3]{Shabani} obtained the directional logarithmic coefficient $(d/2+1)/(\lambda^*(e')e'\cdot e)$. Theorem~\ref{thm:upper} gives the corresponding directional upper bound on the lattice.}

The proofs of the directional upper bound and the axial lower bound are
contained in Sections~\ref{sec:upper} and~\ref{sec:lower}, respectively.
Sections~\ref{sec:profiles} and~\ref{sec:critical-tails} establish profile
selection and arrival-time asymptotics. Section~\ref{sec:extension}
formulates the remaining general-periodic lower-bound problem and proves
a sufficient criterion for it.

\section{The periodic principal eigenvalue and transition estimates}\label{sec:spectral}

We establish comparison and compactness for \eqref{eq:main}, then study
$H(p)$, the geometry of $\mathcal W$, and transition probabilities.
We use Perron--Frobenius theory and the Collatz--Wielandt characterization
\cite{Seneta}, together with analytic perturbation
\cite[Chapters~II--III]{Kato}.
For the kernel estimates, we use the finite-cell Fourier representation
developed below; see \cite{Kuchment,NeyNummelin} for the spectral methods.

To construct $u$ on $\ell^\infty(\Z^d)$, define
\begin{equation}\label{eq:reaction-extension}
 \widetilde F(x,s)=
 \begin{cases}
   a(x)s, & s<0,\\
   F(x,s), & 0\le s\le1,\\
   F_s(x,1)(s-1), & s>1.
 \end{cases}
\end{equation}
Since $\mathcal C$ is finite and $F(x,\cdot)\in C^2([0,1])$,
$\widetilde F_s$ is uniformly bounded in
\eqref{eq:reaction-extension}. The adjacent pieces agree in value and
first derivative at both endpoints, so $\widetilde F(x,\cdot)\in C^1(\R)$ and
\[
 L_F:=\max_{x\in\mathcal C}\max_{0\le s\le1}|F_s(x,s)|
       =\sup_{x\in\Z^d,\,s\in\R}|\widetilde F_s(x,s)|<\infty.
\]
{ Moreover, the extension $\widetilde F$ satisfies $\widetilde F(x,s)\le a(x)s$ for all $s\ge0$:
this holds for $0\leq s\leq1$ by \eqref{eq:periodic-kpp} and for $s>1$
since $F_s(x,1)<0<a(x)$.}

{ Since $\mathcal F(U)=\Delta_{\Z^d}U+\widetilde F(\cdot,U)$ is globally Lipschitz on the Banach space $\ell^\infty(\Z^d)$ with constant at most $4d+L_F$, for any $u_0\in\ell^\infty(\Z^d)$ standard Picard--Lindel\"of iteration yields a unique global solution $U\in C^1([0,\infty);\ell^\infty(\Z^d))$ of $U_t=\mathcal F(U)$ with $U(0)=u_0$.}

The following comparison lemma applies to the lattice Laplacian, the
tilted generators introduced below, and time-dependent reactions.
Fix a finite interval $[t_0,t_1]$ and a set $D\subset\Z^m$, where $m\ge1$.
For $x,z\in\Z^m$ with $|z|_1=1$, let
$c(\cdot,x,z):[t_0,t_1]\to[0,\infty)$ be measurable. Define
\[
 q(t,x):=\sum_{\substack{z\in\Z^m\\|z|_1=1}}c(t,x,z),
 \qquad
 q_0:=\sup_{\substack{t\in[t_0,t_1]\\x\in\Z^m}}q(t,x)<\infty,
\]
and set
\[
 \mathcal L_t f(x)=\sum_{\substack{z\in\Z^m\\|z|_1=1}}
                 c(t,x,z)(f(x+z)-f(x)).
\]
Let $G(t,x,s)$ be measurable in $t$ and Lipschitz in $s$ on an interval
containing the ranges of the functions under comparison. Its Lipschitz
constant is independent of $t\in[t_0,t_1]$ and $x\in D$.
We consider the equation
\begin{equation}\label{eq:comparison-evolution}
 \partial_t v(t,x)=\mathcal L_t v(t,x)+G(t,x,v(t,x)),
 \text{ for }t_0<t<t_1,\ x\in D.
\end{equation}
A subsolution of \eqref{eq:comparison-evolution} on $[t_0,t_1]\times D$
is a bounded function $v:[t_0,t_1]\times\Z^m\to\R$ such that
$v(\cdot,x)$ is absolutely continuous on $[t_0,t_1]$ for every $x\in D$
and
\begin{equation}\label{eq:subsolution-definition}
 \partial_t v(t,x)\le \mathcal L_t v(t,x)+G(t,x,v(t,x))
\end{equation}
for every $x\in D$ and almost every $t\in(t_0,t_1)$.
For $x\in\Z^m\setminus D$, $v(\cdot,x)$ is measurable on $[t_0,t_1]$.
A supersolution of \eqref{eq:comparison-evolution} has the same regularity
and satisfies the reverse inequality in \eqref{eq:subsolution-definition}.
On an unbounded time interval, we impose these requirements on every
compact subinterval. For a real number $r$, write
$r_+=\max\{r,0\}$; for a real-valued function $g$, $g_+$ denotes its
pointwise positive part.

\begin{lemma}\label{lem:fixed-domain-comparison}
Let $v$ and $w$ be a subsolution and a supersolution of
\eqref{eq:comparison-evolution} on $[t_0,t_1]\times D$.
If $v(t_0,x)\le w(t_0,x)$ for $x\in D$ and $v(t,x)\le w(t,x)$
for $t\in[t_0,t_1]$ and $x\in\Z^m\setminus D$ such that
$|x-y|_1=1$ for some $y\in D$, then
$v(t,x)\le w(t,x)$ on $[t_0,t_1]\times D$.
\end{lemma}
\begin{proof}
If $D=\varnothing$, there is nothing to prove. Let $D\ne\varnothing$
and $L\ge0$ be a Lipschitz constant for $G$. Define
$h=(v-w)_+$ in $D$ and $h=0$ on $\Z^m\setminus D$.
The positive-part map is Lipschitz, so $h(\cdot,x)$ is absolutely
continuous for each $x\in D$.
For each $x\in D$, choose a null set $N_x\subset(t_0,t_1)$ containing
the times when $v(\cdot,x)$, $w(\cdot,x)$ or $h(\cdot,x)$ is not
differentiable, or when either differential inequality fails.
The countable union $N=\bigcup_{x\in D}N_x$ is null.

Fix $t\in(t_0,t_1)\setminus N$ and $x\in D$.
If $h(t,x)>0$, then $h(t,x)=v(t,x)-w(t,x)$ and
$h_t(t,x)=v_t(t,x)-w_t(t,x)$.
The definition of $h$ in $D$ and the assumed order outside $D$ give
$(v-w)(t,x+z)\le h(t,x+z)$ for every $z\in\Z^m$ with $|z|_1=1$.
Since $c(t,x,z)\ge0$ and
$G(t,x,v(t,x))-G(t,x,w(t,x))\le Lh(t,x)$, subtraction of the
subsolution and supersolution inequalities yields
\[
 h_t(t,x)\le\sum_{\substack{z\in\Z^m\\|z|_1=1}}c(t,x,z)h(t,x+z)-q(t,x)h(t,x)+Lh(t,x).
\]
If $h(t,x)=0$, differentiability and $h(\cdot,x)\ge0$ give
$h_t(t,x)=0$. Moreover,
\[
 \mathcal L_th(t,x)=\sum_{\substack{z\in\Z^m\\|z|_1=1}}c(t,x,z)h(t,x+z)\ge0.
\]
Thus $h_t\le\mathcal L_t h+Lh$ for every $x\in D$ and
$t\in(t_0,t_1)\setminus N$.

Set $E(t)=\sup_{x\in D}h(t,x)$. Countability of $D$ and continuity
of each $h(\cdot,x)$ make $E$ measurable; boundedness of $v,w$ makes
$E$ bounded. Dropping $-q(t,x)h(t,x)\le0$ and integrating at each
$x\in D$ gives
\[
 h(t,x)\le h(t_0,x)+(q_0+L)\int_{t_0}^tE(s)\,\dd s.
\]
Since $h(t_0,x)=0$, taking the supremum over $x\in D$ yields
$E(t)\le(q_0+L)\int_{t_0}^tE(s)\,\dd s$ for every $t\in[t_0,t_1]$.
Put $A(t)=\int_{t_0}^tE(s)\,\dd s$. Then $A(t_0)=0$, $A\ge0$ and
$A'\le(q_0+L)A$ almost everywhere. Gronwall's inequality gives $A=0$.
The inequality $0\le E(t)\le(q_0+L)A(t)$ now gives $h(t,x)=0$
on $[t_0,t_1]\times D$.
\end{proof}
When $c(t,x,z)=c(x,z)$, let $\mathcal L_D$ denote the restriction of
$\mathcal L_t$ to $D$, with functions extended by zero on
$\Z^m\setminus D$. The operator $\mathcal L_D$ is bounded on $\ell^\infty(D)$, with norm at most $2q_0$. Since $\mathcal L_D+q_0I$ preserves nonnegative functions, its
exponential series shows that
$\ee^{t\mathcal L_D}=\ee^{-q_0t}\ee^{t(\mathcal L_D+q_0I)}$ is positive.
Comparison with the constant functions $0$ and $1$ gives
$\|\ee^{t\mathcal L_D}\|_{\ell^\infty(D)\to\ell^\infty(D)}\leq1$ for $t\geq0$.
Adding a bounded real-valued multiplication operator preserves
positivity of the generated semigroup.

The constants $0$ and $1$ solve
$u_t=\Delta_{\Z^d}u+\widetilde F(x,u)$. Apply
Lemma~\ref{lem:fixed-domain-comparison} on each finite time interval
with $D=\Z^d$, $c(t,x,z)=1$ and $G=\widetilde F$.
Since $0\le u_0\le1$, we obtain $0\le u(t,x)\le1$ for all
$t\ge0$ and $x\in\Z^d$. On $[0,1]$, $\widetilde F=F$, so $u$
satisfies \eqref{eq:main}.
For $v_t=\Delta_{\Z^d}v$ with $v(0,\cdot)=u_0$, $v$ takes values
in $[0,1]$, where $\widetilde F=F\geq0$. For
$v_t=(\Delta_{\Z^d}+a(\cdot))v$ with $v(0,\cdot)=u_0$, $v$ is nonnegative,
and $\widetilde F(x,s)\leq a(x)s$ for every $s\geq0$ by
\eqref{eq:reaction-extension}. Thus $\ee^{t\Delta_{\Z^d}}u_0$ and
$\ee^{t(\Delta_{\Z^d}+a(\cdot))}u_0$ are, respectively, a subsolution and
a supersolution of $u_t=\Delta_{\Z^d}u+\widetilde F(x,u)$.
Lemma~\ref{lem:fixed-domain-comparison} gives
\begin{equation}\label{eq:heat-comparison}
 \ee^{t\Delta_{\Z^d}}u_0\leq u(t,\cdot)\leq
 \ee^{t(\Delta_{\Z^d}+a(\cdot))}u_0.
\end{equation}
For any $x,y\in\Z^d$, choose a nearest-neighbor path from $y$ to $x$.
The event that the walk follows this path before time $t$ and makes no further jump
has positive probability. Hence the transition kernel of $\ee^{t\Delta_{\Z^d}}$
is strictly positive. Thus the left inequality in
\eqref{eq:heat-comparison} gives $u(t,x)>0$ for $t>0$.
For the strict upper bound put $v=1-u$. Since
$F(x,1)=0$ and $F(x,u)\le L_F(1-u)$, one has
$v_t\ge\Delta_{\Z^d}v-L_Fv$. Lemma~\ref{lem:fixed-domain-comparison}
gives
$1-u(t,\cdot)\ge \ee^{-L_Ft}\ee^{t\Delta_{\Z^d}}(1-u_0)>0$,
since $\supp u_0$ is finite by \eqref{eq:u0} and hence $1-u_0\not\equiv0$.

To obtain an exponential bound for $u$, set
$a_{\max}=\max_{x\in\mathcal C}a(x)$. For $p\in\R^d$ define
\begin{equation}\label{eq:exponential-supersolution}
 U_p(t,x)=A_p\exp\left(\left[2\sum_{i=1}^d(\cosh p_i-1)+a_{\max}\right]t-p\cdot x\right),
 \qquad
 A_p=\max_{x\in\supp u_0}u_0(x)\ee^{p\cdot x}.
\end{equation}
A direct substitution gives
$(U_p)_t-\Delta_{\Z^d}U_p-a(x)U_p=(a_{\max}-a(x))U_p\ge0$.
To apply comparison, set
$\overline U_p(t,x)=\min\{1,U_p(t,x)\}$. This function is bounded and
absolutely continuous on compact time intervals at every vertex.
Where $U_p<1$, truncation preserves its value and time derivative and
decreases the neighboring values. The bound $F(x,U_p)\leq a(x)U_p$
for $0\leq U_p<1$ therefore gives
\[
 (\overline U_p)_t-\Delta_{\Z^d}\overline U_p
             -F(x,\overline U_p)
 \ge (U_p)_t-\Delta_{\Z^d}U_p-a(x)U_p\ge0.
\]
If $U_p(t,x)\ge1$, then $\overline U_p(t,x)=1$. At every
differentiability time, $\overline U_p(\cdot,x)\le1$ gives
$(\overline U_p)_t(t,x)=0$, while
$\Delta_{\Z^d}\overline U_p(t,x)\le0$. Since $F(x,1)=0$, the
supersolution inequality holds in this case as well.
Since $\overline U_p(0,\cdot)\ge u_0$, Lemma~\ref{lem:fixed-domain-comparison}
gives $u\le\overline U_p\le U_p$. Consequently, for every $\varepsilon>0$ there are
constants $C_\varepsilon,c_\varepsilon>0$ such that
\begin{equation}\label{eq:global-exponential-bound}
 u(t,x)\le C_\varepsilon\exp(c_\varepsilon t-\varepsilon|x|_1),
 \text{ for }t\ge0,\ x\in\Z^d,
\end{equation}
by choosing $p=\varepsilon\sigma$ with
$\sigma_i=1$ when $x_i\geq0$ and $\sigma_i=-1$ when $x_i<0$.

To specify the constants in \eqref{eq:global-exponential-bound},
fix $\varepsilon>0$, use the $2^d$ vectors
$p=\varepsilon\sigma$, $\sigma\in\{-1,1\}^d$, and set
\[
 C_\varepsilon=\max_{\sigma\in\{-1,1\}^d}\max_{x\in\supp u_0}
 u_0(x)\ee^{\varepsilon\sigma\cdot x},
 \qquad c_\varepsilon=2d(\cosh\varepsilon-1)+a_{\max}.
\]
By \eqref{eq:u0}, $u_0$ is nonzero and $\supp u_0$ is finite, so $0<C_\varepsilon<\infty$.
Choosing $\sigma_i x_i=|x_i|$ gives
\eqref{eq:global-exponential-bound} with these constants, uniformly in $t\ge0$ and $x\in\Z^d$.

Finally, put
$F_0=\max_{x\in\mathcal C,\,0\le s\le1}|F(x,s)|$. Since
$0\le u\le1$, one has $|\Delta_{\Z^d}u|\le2d$, hence
$|u_t|\le2d+F_0$. The map $U\mapsto\widetilde F(\cdot,U)$ is continuously differentiable on
$\ell^\infty$. Indeed, $\widetilde F_s(x,\cdot)$ is uniformly continuous on each bounded interval.
This continuity is uniform in $x$ because $\mathcal C$ is finite.
The derivative at $U$ is multiplication by $\widetilde F_s(x,U(x))$.
Consequently $u$ is $C^2$, and differentiating the lattice ODE gives
$u_{tt}=\Delta_{\Z^d}u_t+F_s(x,u)u_t$ and
$|u_{tt}|\le(4d+L_F)(2d+F_0)$. These bounds are uniform in $t>0$ and
$x\in\Z^d$. { \begingroup\begin{remark}\label{rem:diagonal-compactness} 
For any sequence of space--time translates $u_k(s,z)=u(t_k+s,x_k+z)$ with $t_k\to\infty$ (or translates $V_k(s,z)=V(t_k+s,x_k+z)$ of an entire solution $V:\R\times\Z^d\to[0,1]$) and fixed cell residue $x_k\bmod\Gamma=r_0\in\mathcal C$, the uniform bounds $0\le u_k\le 1$, $|(u_k)_s|\le 2d+F_0$ and $|(u_k)_{ss}|\le(4d+L_F)(2d+F_0)$ imply uniform boundedness and equicontinuity of $u_k$ and $(u_k)_s$ on each finite cylinder $[-m,m]\times\{z\in\Z^d:|z|_1\le m\}$. By the Arzel\`a--Ascoli theorem and a standard diagonal extraction over $m\ge1$, there is a subsequence converging in $C^1([-m,m])$ at every site $z\in\Z^d$ to an entire solution $U$ of \eqref{eq:main} with reaction $F^{[r_0]}(z,s)=F(r_0+z,s)$. The same diagonal compactness principle applies to any sequence of functions $Z_k$ on $\R\times\Z^d$ that is locally uniformly bounded in $C^2$ in time on finite subsets of $\Z^d$.
\end{remark}\endgroup}

\begin{lemma}\label{lem:homogeneous-envelopes}
Assume \eqref{eq:periodic-kpp} and $F_s(x,0)=a>0$ for every $x\in\Z^d$.
There are functions $f_-,f_+\in C^2([0,1])$ such that, for each
$\sigma\in\{-,+\}$,
\[
 f_\sigma(0)=f_\sigma(1)=0,\qquad
 f_\sigma'(0)=a,\qquad f_\sigma'(1)<0,\qquad
 0<f_\sigma(s)\le as\text{ for }0<s<1.
\]
Moreover, $f_-(s)\le F(x,s)\le f_+(s)$ for every $x\in\Z^d$ and
$0\le s\le1$.
\end{lemma}
\begin{proof}
Enumerate the cell as $x_1,\ldots,x_N$, put $F_i(s)=F(x_i,s)$, and choose
\[
 C\geq\max\left\{0,
 \sup_{\substack{1\leq i\leq N\\0<s<1}}
 \frac{a s(1-s)/F_i(s)-1}{s}\right\}.
\]
To see that this choice is possible,  {since $F_i\in C^2([0,1])$, Taylor's formula gives}
$F_i(s)=as+b_i s^2+o(s^2)$ as $s\downarrow0$  {with $b_i=\tfrac12 F_i''(0)=\tfrac12\partial_s^2 F(x_i,0)$}. Then
\[
 \frac{as(1-s)/F_i(s)-1}{s}
   =-1-b_i/a+o(1).
\]
At $s=1$, put $\gamma_i=-F_i'(1)>0$. Since
$F_i(s)=\gamma_i(1-s)+O((1-s)^2)$, the same quotient tends to
$a/\gamma_i-1$ as $s\uparrow1$.
It is continuous on $(0,1)$ because $F_i>0$ there.
There are only finitely many $i$, so the supremum defining $C$ is finite.
Set
\[
 f_-(s)=\frac{a s(1-s)}{1+Cs},\qquad
 f_+(s)=a s\left[1-\prod_{i=1}^N
       \left(1-\frac{F_i(s)}{a s}\right)\right]\text{ for } 0<s\leq1,
 \qquad f_+(0)=0.
\]
The choice of $C$ gives $f_-\leq F_i$.
For $0<s<1$, each factor $1-F_i(s)/(as)$ belongs to $[0,1)$, so
$F_i\leq f_+\leq as$. Both $f_-$ and $f_+$ are positive on $(0,1)$
and vanish at $0$ and $1$.
The formula for $f_-$ gives $f_-\in C^2([0,1])$,
$f_-'(0)=a$ and $f_-'(1)=-a/(1+C)<0$.
For $N=1$, $f_+=F_1\in C^2([0,1])$.
For $N\geq2$, put $g_i(s)=as-F_i(s)$.
Then $g_i=O(s^2)$, $g_i'=O(s)$ and $g_i''=O(1)$ as $s\downarrow0$.
The formula for $f_+$ already gives $C^2$ regularity near $1$;
we check the extension at $0$.

For $N\ge2$, let $G(s)=\prod_{i=1}^N g_i(s)$. The product rule gives
$G=O(s^{2N})$, $G'=O(s^{2N-1})$ and
$G''=O(s^{2N-2})$. Applying the product rule twice to
$a^{1-N}s^{1-N}G(s)$ gives $O(s^{N+1})$, $O(s^N)$ and $O(s^{N-1})$
for the function and its first two derivatives, respectively, as $s\downarrow0$.
{ Since $N\ge2$, we have $N-1\ge1$, so the function $a^{1-N}s^{1-N}G(s)=as-f_+(s)$ and its first two derivatives all vanish as $s\downarrow0$.} The fundamental theorem of calculus makes
these extensions successive derivatives, so $f_+$ extends to a $C^2$
function at zero, with $f_+'(0)=a$.
At $s=1$, each factor $1-F_i(s)/(as)$ equals one. Differentiating
$as[1-\prod_{i=1}^N(1-F_i(s)/(as))]$ gives
$f_+'(1)=\sum_{i=1}^N F_i'(1)<0$.

\end{proof}

{ \begingroup\begin{remark}\label{rem:homogeneous-comparison} 
Assume \eqref{eq:common-linearization}, namely $F_s(x,0)=a>0$ for every $x\in\Z^d$. Let $\lambda_0>0$ minimize $[2(\cosh\lambda-1)+a]/\lambda$ over $\lambda>0$, and set $c_0=2\sinh\lambda_0$ and $m_0(t)=c_0t-(d+2)\log t/(2\lambda_0)$. Since $a_{\mathrm{eff}}=a$, the parameters in Theorem~\ref{thm:axial} coincide with $(\lambda_0,c_0,m_0)$. Let $f_-,f_+\in C^2([0,1])$ be the homogeneous KPP envelopes provided by Lemma~\ref{lem:homogeneous-envelopes}, and let $u_\pm$ solve $(u_\pm)_t=\Delta_{\Z^d}u_\pm+f_\pm(u_\pm)$ on $[0,\infty)\times\Z^d$ with $u_\pm(0,\cdot)=u_0$. Then Lemma~\ref{lem:fixed-domain-comparison} gives $u_-\le u\le u_+$. Because $f_\pm$ and $u_0$ satisfy the homogeneous hypotheses \cite[(1.3)--(1.4)]{HuAxial} with identical linearization $f_-'(0)=f_+'(0)=a$, \cite[Theorem~1.1]{HuAxial} yields \eqref{eq:axial-position} for both $u_-$ and $u_+$ with $m_d=m_0$. The level-set inclusions
\[
 \{j\in\Z:u_-(t,j,0)\ge\theta\}\subseteq\{j\in\Z:u(t,j,0)\ge\theta\}\subseteq\{j\in\Z:u_+(t,j,0)\ge\theta\}
\]
therefore imply \eqref{eq:axial-position} for $u$ whenever \eqref{eq:common-linearization} holds.
\end{remark}\endgroup}

\begin{lemma}\label{lem:spectral}
The function $H$ is real analytic, even, superlinear, and has positive
definite Hessian everywhere. If
$H_0(p)=2\sum_{i=1}^d(\cosh p_i-1)$, then
\begin{equation}\label{eq:H-comparison}
 H_0(p)+\min_{x\in\mathcal C}a(x)\leq H(p)
       \leq H_0(p)+\max_{x\in\mathcal C}a(x).
\end{equation}
Let $\pi_p(x)=\phi_p(x)\phi_p^*(x)$, periodically extended to $\Z^d$.
The operators
\begin{align}
 Q_ph(x)&=\sum_{\substack{z\in\Z^d\\|z|_1=1}}q_p(x,z)(h(x+z)-h(x)),
 &q_p(x,z)&=\ee^{-p\cdot z}\frac{\phi_p(x+z)}{\phi_p(x)},
 \label{eq:Q}\\
 \widehat Q_ph(x)&=\sum_{\substack{z\in\Z^d\\|z|_1=1}}\widehat q_p(x,z)(h(x+z)-h(x)),
 &\widehat q_p(x,z)&=\ee^{p\cdot z}\frac{\phi_p^*(x+z)}{\phi_p^*(x)}
 \label{eq:Qreverse}
\end{align}
are Markov generators with invariant probability distribution $\pi_p$
on the cell. They are adjoint on $\Z^d$ for the periodic measure assigning mass
$\pi_p(x)$ to each site $x$. When the cell residue has distribution
$\pi_p$, their mean displacements per unit time are
$-\nabla H(p)$ and $\nabla H(p)$, respectively. Writing $X_t$ for either walk, one has
$\lim_{t\to\infty}t^{-1}\operatorname{Cov}(X_t-X_0)=\Sigma_p=\nabla^2H(p)$.
\end{lemma}
\begin{proof}
We use Perron--Frobenius theory \cite[Chapters~1--2]{Seneta} and analytic
perturbation \cite[Chapter~II]{Kato} for $B_p$, together with the
Collatz--Wielandt inequalities \cite{Seneta}.

Retaining all physical nearest-neighbor edges, for $r,s\in\mathcal C$ we have
\[
 (B_p)_{rs}=\sum_{\substack{z\in\Z^d,\ |z|_1=1\\r+z\bmod\Gamma=s}}
 \ee^{-p\cdot z}+(a(r)-2d)\1_{\{r=s\}}.
\]
Thus a period-one coordinate contributes both of its edges to the
diagonal, and a period-two coordinate contributes two distinct terms
to the same off-diagonal entry. Both directed edges are included. Every
pair of residues is connected by a path of such edges. For $N=|\mathcal C|>1$,
choose $c$ so that $B_p+cI$ has positive diagonal; its directed graph
is connected. Any two residues can be joined by a path of at most $N-1$ edges.
Positive diagonal entries extend that path to $N-1$ steps. Hence $(B_p+cI)^{N-1}$ has strictly positive entries.
For $N=1$, $B_p$ is a scalar and all statements about its principal
eigenvalue are immediate. Perron--Frobenius applied to the positive
power proves both simplicity and the strict real-part separation of
$H(p)$ from every other eigenvalue. On a compact set of real $p$, the
same $c$ works and all edge weights have a common positive lower
bound.

The comparison with $H_0$ follows directly from the left eigenvector equation. {Since $(B_p\1)(r)=H_0(p)+a(r)$ for every $r\in\mathcal C$, taking the inner product with the positive left eigenvector $\phi_p^*$}
gives
\[
 H(p)=H_0(p)+
 \frac{\sum_{r\in\mathcal C}\phi_p^*(r)a(r)}{\sum_{r\in\mathcal C}\phi_p^*(r)},
\]
so \eqref{eq:H-comparison} follows from $\phi_p^*>0$. Also, if $|p|\to\infty$,
some coordinate has modulus at least $|p|/\sqrt d$, so
$H(p)/|p|\to\infty$. Reversing each physical edge gives
$(B_p)_{rs}=(B_{-p})_{sr}$ even in the period-one and period-two
cases, proving evenness without an assumption about distinct cell
neighbors.

For later derivative calculations, let $B_i=\partial_{p_i}B_p$
and $B_{ij}=\partial_{p_i}\partial_{p_j}B_p$. Their actions are
\[
 (B_i f)(r)=-\sum_{\substack{z\in\Z^d\\|z|_1=1}}z_i \ee^{-p\cdot z}f(r+z),\qquad
 (B_{ij}f)(r)=\sum_{\substack{z\in\Z^d\\|z|_1=1}}z_i z_j \ee^{-p\cdot z}f(r+z).
\]
In formulas for cell vectors, function arguments are reduced modulo
$\Gamma$, while each physical displacement $z$ is retained in the edge weights. The characteristic
polynomial has a simple root at $H(p)$, and analytic perturbation
provides a locally holomorphic eigenvalue and eigenvectors for
complex parameters near each real $p$. The positive normalization
$\sum_{r\in\mathcal C}\phi_p(r)=1$ never vanishes near that real point. The further
normalization $(\phi_p^*)^{\mathsf T}\phi_p=1$ likewise has a
nonzero denominator. The normalized local branches agree at real
points by uniqueness, hence give the real-analytic functions used
in the manuscript.

Write $H_i=\partial_{p_i}H(p)$ and
$\phi_i=\partial_{p_i}\phi_p$. Differentiating the right
eigenvalue equation and its normalization yields the finite system
\[
 (B_p-HI)\phi_i=-(B_i-H_iI)\phi_p,
 \qquad \sum_{r\in\mathcal C}\phi_i(r)=0,
 \qquad H_i=(\phi_p^*)^{\mathsf T}B_i\phi_p.
\]
Multiplying the equation for $\phi_i$ by $(\phi_p^*)^{\mathsf T}$
gives the formula for $H_i$, since
$(\phi_p^*)^{\mathsf T}(B_p-HI)=0$ and
$(\phi_p^*)^{\mathsf T}\phi_p=1$.
Substituting the rates from \eqref{eq:Q} gives
\[
 H_i=-\sum_{r\in\mathcal C}\pi_p(r)\sum_{\substack{z\in\Z^d\\|z|_1=1}}q_p(r,z)z_i,
\]
so the mean displacement rate under $\pi_p$ is $-\nabla H(p)$.
The formula for $H_i$ also gives compatibility of the equation for $\phi_i$.
Its solution is unique: the difference of two solutions is a
multiple of $\phi_p$, and the zero-sum condition sets that multiple
to zero. For $\phi_i^*=\partial_{p_i}\phi_p^*$ the corresponding
equations are
\[
 (B_p^{\mathsf T}-HI)\phi_i^*=-(B_i^{\mathsf T}-H_iI)\phi_p^*,
 \qquad (\phi_i^*)^{\mathsf T}\phi_p
             +(\phi_p^*)^{\mathsf T}\phi_i=0.
\]
Write $H_{ij}=\partial_{p_i}\partial_{p_j}H(p)$.
Differentiating twice and multiplying by $\phi_p^*$ gives the
second-derivative identity
\[
 H_{ij}=(\phi_p^*)^{\mathsf T}
 \bigl[B_{ij}\phi_p+(B_i-H_iI)\phi_j+(B_j-H_jI)\phi_i\bigr].
\]
The differentiated eigenvector equations determine the parameter derivatives with the stated normalizations. The bordered systems are
invertible at every real parameter. Their inverses and all the
displayed coefficients are continuous; compactness consequently
bounds them uniformly on every compact real parameter set.

The eigenvector equation for \eqref{eq:cell} identifies the generator in
\eqref{eq:Q} as
$Q_p=\phi_p^{-1}(B_p-H(p))\phi_p$ on periodic functions,
where $\phi_p$ and $\phi_p^{-1}$ act by pointwise multiplication.
The right and left eigenvector equations for \eqref{eq:cell}, together
with \eqref{eq:Q}--\eqref{eq:Qreverse}, give
\[
 \sum_{\substack{z\in\Z^d\\|z|_1=1}}q_p(x,z)
 =\sum_{\substack{z\in\Z^d\\|z|_1=1}}\widehat q_p(x,z)
 =H(p)-a(x)+2d,
\]
so the forward and reversed total jump rates agree. The individual rates
in \eqref{eq:Q}--\eqref{eq:Qreverse} satisfy, for every oriented
nearest-neighbor edge $(x,x+z)$,
\[
 \pi_p(x)q_p(x,z)
 =\phi_p^*(x)\ee^{-p\cdot z}\phi_p(x+z)
 =\pi_p(x+z)\widehat q_p(x+z,-z).
\]
Hence
\begin{equation}\label{eq:reverse-edge}
 \pi_p(x)q_p(x,z)=\pi_p(x+z)\widehat q_p(x+z,-z).
\end{equation}
Summing \eqref{eq:reverse-edge} over incoming edges and using the equality
of the total jump rates proves that $\pi_p$ is invariant on the periodicity
cell. It also identifies $\widehat Q_p$ as the stationary time reversal
of $Q_p$; see \cite[Chapter~3]{Norris} for the continuous-time
formalism. { In particular, multiplying \eqref{eq:reverse-edge} by $z$ and summing over $x\in\mathcal C$ and $|z|_1=1$ directly gives the mean displacement rate of $\widehat Q_p$ under $\pi_p$:
\[
 \sum_{x\in\mathcal C}\pi_p(x)\sum_{\substack{z\in\Z^d\\|z|_1=1}}\widehat q_p(x,z)z
 =-\sum_{y\in\mathcal C}\pi_p(y)\sum_{\substack{w\in\Z^d\\|w|_1=1}}q_p(y,w)w
 =\nabla H(p).
\]}

Fix a compact parameter set and choose $\Lambda$ larger than
$\sum_{\substack{z\in\Z^d\\|z|_1=1}}q_p(r,z)=\sum_{\substack{z\in\Z^d\\|z|_1=1}}\widehat q_p(r,z)$, uniformly for $r\in\mathcal C$
and $p$ in that set. These sums count physical jumps, including jumps
with the same endpoint residue. We construct both walks by a rate-$\Lambda$
Poisson clock, as in \cite{Norris}. Write $\mathbb P_x,\mathbb E_x$ for
the law and expectation when $X_0=x$. For either generator $\mathcal Q\in\{Q_p,\widehat Q_p\}$, this construction
gives the martingale
\[
 f(X_t)-f(X_0)-\int_0^t(\mathcal Qf)(X_s)\,\dd s
\]
when $f(x)=\eta\cdot x+\chi(x)$, with $\eta\in\R^d$ and $\chi$ a
bounded periodic function. To verify the martingale and its square integrability, use a Poisson random measure $N$ on
$[0,\infty)\times[0,1]$ with intensity $\Lambda\,\dd s\,\dd v$.
Order the $2d$ displacement vectors once and, at a site $x$, divide
$[0,1]$ into intervals of lengths $q_p(x,z)/\Lambda$, followed by
an interval for no displacement. Denote the selected displacement
by $J(x,v)$, including $J=0$ on the last interval. Iterating over the
finitely many points of $N$ in each bounded time interval constructs
a right-continuous path with left limits. Distinct physical edges
remain distinct marks even when they have the same endpoint residue.
Take the initial state independent of $N$ and use the completed,
right-continuous natural filtration of $N$ and that state.
Future Poisson points are independent of
this past, and their law is unchanged by a deterministic time shift.
Conditional on $X_s$, the same construction after $s$ therefore
gives the transition law starting from $X_s$.

For $f(x)=\eta\cdot x+\chi(x)$, put
$D_f(x,v)=f(x+J(x,v))-f(x)$.
Then $|D_f|\le |\eta|+2\|\chi\|_\infty$ and
\[
 \Lambda\int_0^1D_f(x,v)\,\dd v=\mathcal Qf(x),
 \qquad
 \Lambda\int_0^1D_f(x,v)^2\,\dd v
   =\sum_{\substack{z\in\Z^d\\|z|_1=1}}q_p(x,z)(f(x+z)-f(x))^2,
\]
with $q_p$ replaced by $\widehat q_p$ for the reversed walk.
Subtracting the first compensator from the sum over jumps represents
the asserted martingale as
\[
 \int_{(0,t]\times[0,1]}D_f(X_{s-},v)
       \bigl(N(\dd s,\dd v)-\Lambda\,\dd s\,\dd v\bigr).
\]
For a bounded predictable step integrand, zero conditional mean and
the second-moment identity follow directly by independence of
Poisson increments on disjoint time intervals. Approximation in
$L^2(\Lambda\,\dd s\,\dd v\,\dd\mathbb P)$ extends both identities
to this bounded predictable integrand. In particular, the martingale's second
moment equals the expected time integral of the squared jump increments weighted by the jump rates.
The uniform jump bound and the Poisson second moment also imply
\[
 \mathbb E_x\sup_{0\le s\le T}|f(X_s)-f(X_0)|^2
 \le (|\eta|+2\|\chi\|_\infty)^2
                  \bigl(\Lambda T+(\Lambda T)^2\bigr).
\]
Adding the bounded time integral of $\mathcal Qf$ shows that the martingale also has finite expected squared path supremum. These bounds justify the integrability used in the optional-stopping arguments in Section~\ref{sec:upper}.

For the walk generated by $Q_p$ and $k\in\R^d$, weighting each jump term
$q_p(x,z)f(x+z)$ by $\ee^{k\cdot z}$, while retaining the diagonal term,
gives the following matrix on periodic functions:
\[
 L_k
 =\phi_p^{-1}\bigl(B_{p-k}-H(p)I\bigr)\phi_p.
\]
Its principal eigenvalue is therefore $H(p-k)-H(p)$. This is the
finite-state continuous-time instance of the transform-kernel construction
for Markov additive processes in \cite{NeyNummelin}. For a periodic vector $g$, define
\[
 T_k(t)g(r)=\mathbb E_r[\ee^{k\cdot(X_t-r)}g(X_t\bmod\Gamma)].
\]
Conditioning on the clock during $[0,h]$, the probabilities of zero,
one, and at least two events are
$1-\Lambda h+O(h^2)$, $\Lambda h+O(h^2)$, and $O(h^2)$.
On a bounded complex $k$-set, the contribution of two or more events is bounded
by a constant times $h^2$. Indeed, the number of events is Poisson,
and each event changes the exponential weight by a bounded factor. The one-event calculation therefore gives the generator
\[
 L_k g(r)=\sum_{\substack{z\in\Z^d\\|z|_1=1}}q_p(r,z)\ee^{k\cdot z}g(r+z)
                    -\sum_{\substack{z\in\Z^d\\|z|_1=1}}q_p(r,z)g(r).
\]
The conditional construction also gives $T_k(t+h)=T_k(t)T_k(h)$.
The finite-dimensional differential equation and its exponential
series now give $T_k(t)=\ee^{tL_k}$. Substitution of the rates and the
eigenvector equation proves
$L_k=\phi_p^{-1}(B_{p-k}-H(p)I)\phi_p$ term by term,
including its diagonal.

At $k=0$, $L_0$ is a finite irreducible Markov generator.
Its zero eigenvalue is simple and all others have negative real parts.
For $k$ near zero, let $P_k$ be the spectral projection onto
$\lambda(k)=H(p-k)-H(p)$, and put
$a_r(k)=[P_k\1](r)$ and
$R_r(t,k)=[\ee^{tL_k}(I-P_k)\1](r)$.
The identity $T_k(t)=\ee^{tL_k}$ gives
\[
 \mathbb E_r \ee^{k\cdot(X_t-r)}
 =\ee^{t\lambda(k)}a_r(k)+R_r(t,k),\qquad
 \lambda(k)=H(p-k)-H(p),\quad a_r(0)=1,
\]
where $a_r$ is analytic and $R_r$, together with its first two
$k$-derivatives on a smaller neighborhood, is exponentially small
relative to $\ee^{t\lambda(k)}$. On a sufficiently small closed complex polydisc about zero,
choose a fixed positively oriented
rectangular contour $\mathcal G$ that encloses all the other
eigenvalues, excludes $\lambda(k)$, and lies in
$\{\operatorname{Re}z\le-\gamma\}$ for some $\gamma>0$.
The product of $\mathcal G$ with this polydisc is compact,
and $zI-L_k$ is invertible on this product. Continuity therefore
bounds the resolvent uniformly on that product.
The finite-matrix contour identity gives
\[
 \begin{aligned}
 \ee^{tL_k}(I-P_k)&=\frac1{2\pi\ii}
 \int_{\mathcal G}\ee^{tz}(zI-L_k)^{-1}(I-P_k)\,\dd z,\\
 \|\ee^{tL_k}(I-P_k)\|&\le C {\ee}^{-\gamma t}.
 \end{aligned}
\] Choosing this neighborhood so that
$|\lambda(k)|<\gamma/4$ gives
$|\ee^{-t\lambda(k)}R_r(t,k)|\leq C {\ee}^{-3\gamma t/4}$. To bound derivatives, integrate the
convergent power series of an analytic entry $f$ term by term on
positively oriented coordinate circles of a fixed radius $\rho$ lying in the larger
polydisc. For a multi-index $\alpha$ this gives
\[
 \begin{aligned}
 \partial^\alpha f(k)&=\frac{\alpha!}{(2\pi\ii)^d}
 \int_{|z_1|=\rho}\!\cdots\!\int_{|z_d|=\rho}
       \frac{f(k+z)}{z_1^{\alpha_1+1}\cdots z_d^{\alpha_d+1}}
                   \,\dd z_1\cdots\dd z_d,\\
 |\partial^\alpha f(k)|&\le\alpha!\rho^{-|\alpha|}
                  \sup_{\substack{z\in\mathbb C^d\\|z_j|=\rho,\ 1\leq j\leq d}}|f(k+z)|.
 \end{aligned}
\]
Apply the displayed Cauchy estimate to $f(k)=\ee^{-t\lambda(k)}R_r(t,k)$.
Its first two $k$-derivatives are bounded by $C {\ee}^{-\gamma t/2}$ on a smaller polydisc. Since $a_r(0)=1$,
the analytic logarithm of $a_r$ exists near zero.
Thus
\[
 \left.\partial_{k_i}\log\mathbb E_r \ee^{k\cdot(X_t-r)}\right|_{k=0}
       =t\lambda_i(0)+O(1),\qquad
 \left.\partial_{k_i}\partial_{k_j}
       \log\mathbb E_r \ee^{k\cdot(X_t-r)}\right|_{k=0}
       =t\lambda_{ij}(0)+O(1).
\]
Direct differentiation of the exponential moment identifies the
left-hand sides as mean displacement and covariance. These
derivatives are legitimate because displacement is bounded by the
Poisson count, which has finite exponential moments of every fixed
order. The formulas for the first two derivatives of
$\log\mathbb E_r \ee^{k\cdot(X_t-r)}$ at $k=0$ identify the limiting
drift and covariance for every initial residue. When $X_0\bmod\Gamma$ has law $\pi_p$,
$\sum_{r\in\mathcal C}\pi_p(r)\sum_{\substack{z\in\Z^d\\|z|_1=1}}q_p(r,z)z$ gives the mean displacement rate at every time;
comparison with the limit identifies this expression as $-\nabla H(p)$.
For the reversed walk, the exponential-moment calculation uses
$B_{-p-k}$ and $\phi_p^*$, whose eigenvalue is
$H(-p-k)-H(p)=H(p+k)-H(p)$.

To prove that $\nabla^2H(p)$ is positive definite, fix $\eta\in\R^d$ and write
\[
 b(x)=\sum_{\substack{z\in\Z^d\\|z|_1=1}}q_p(x,z)z,
 \qquad \bar b=\sum_{x\in\mathcal C}\pi_p(x)b(x).
\]
Here $\bar b=-\nabla H(p)$ by the first derivative of the exponential moment. The equation
\begin{equation}\label{eq:corrector-vector}
 Q_p\chi_\eta(x)=\eta\cdot(\bar b-b(x)),\qquad
 \sum_{x\in\mathcal C}\pi_p(x)\chi_\eta(x)=0
\end{equation}
has a unique solution on $\mathcal C$, as follows.

If $Q_pf=0$, the real and imaginary parts of $f$ attain their maxima
on the cell. Positivity of the rates forces each part to be constant at
neighboring residues, hence on the connected cell.
Thus the complex kernel consists of constants. Since $\pi_p^{\mathsf T}Q_p=0$, the range is
contained in the hyperplane of $\pi_p$-mean zero, and equality of
dimensions makes it exactly that hyperplane. Let
$\Pi_p f=\1\sum_{r\in\mathcal C}\pi_p(r)f(r)$. The matrix $Q_p-\Pi_p$ is
invertible: if $(Q_p-\Pi_p)f=0$, multiplication by $\pi_p$ gives
$\Pi_pf=0$, and then $Q_pf=0$ gives $f=0$.
For a mean-zero right-hand side $g$, the solution
$f=(Q_p-\Pi_p)^{-1}g$ consequently has mean zero and satisfies
$Q_pf=g$. The representation $f=(Q_p-\Pi_p)^{-1}g$ gives uniqueness.
Since the inverse matrix depends continuously on $p$, so does $f$ whenever $g$ does.
It also bounds $f$ uniformly on compact parameter sets when $g$ is uniformly bounded. When
the cell has one vertex, the mean-zero space is $\{0\}$ and $\chi_\eta=0$; $(Q_p-\Pi_p)^{-1}$ remains well defined.

Extend $\chi_\eta$ $\Gamma$-periodically to $\Z^d$. For $f(x)=\eta\cdot x+\chi_\eta(x)$, the martingale formula together with
\eqref{eq:corrector-vector} shows that
$\eta\cdot(X_t-X_0-\bar b t)+\chi_\eta(X_t)-\chi_\eta(X_0)$
is a martingale. Its jumps equal $\eta\cdot z+\chi_\eta(x+z)-\chi_\eta(x)$, so under
the stationary cell distribution its quadratic variation has mean $t$ times
\begin{equation}\label{eq:variance-formula}
 \sum_{x\in\mathcal C}\pi_p(x)
 \sum_{\substack{z\in\Z^d\\|z|_1=1}}q_p(x,z)
 \bigl(\eta\cdot z+\chi_\eta(x+z)-\chi_\eta(x)\bigr)^2.
\end{equation}
Denote the expression in \eqref{eq:variance-formula} by $V_p(\eta)$.
For the covariance limit, define
$M_t=\eta\cdot(X_t-X_0-\bar b t)+\chi_\eta(X_t)-\chi_\eta(X_0)$, put
$C_t=\chi_\eta(X_t)-\chi_\eta(X_0)$, and start the residue chain
with law $\pi_p$. The second-moment identity for the compensated
Poisson integral gives $\mathbb E M_t^2=tV_p(\eta)$.
Since $\eta\cdot(X_t-X_0-\bar b t)=M_t-C_t$ and
$|C_t|\le2\|\chi_\eta\|_\infty$, Cauchy--Schwarz gives
\[
 \left|\operatorname{Var}(M_t-C_t)-\operatorname{Var}(M_t)\right|
 \le4\|\chi_\eta\|_\infty\sqrt{tV_p(\eta)}
                   +4\|\chi_\eta\|_\infty^2.
\]
Dividing by $t$ and using
$t^{-1}\operatorname{Cov}(X_t-X_0)\to\nabla^2H(p)$ from the
exponential-moment calculation proves
$V_p(\eta)=\eta^{\mathsf T}\nabla^2H(p)\eta$.
If $V_p(\eta)=0$, every summand in \eqref{eq:variance-formula}
vanishes because its weight is positive. Along $L_i$ successive physical $e_i$ edges, the
corrector differences telescope to zero by periodicity, leaving
$L_i\eta_i=0$. This applies also to a period-one loop and to the
two distinct edges in a period-two coordinate. Thus $\eta=0$.
For uniform positivity on a compact parameter set $K$, the
continuous function
$(p,\eta)\mapsto\eta^{\mathsf T}\nabla^2H(p)\eta$ is strictly
positive on $K\times\{\eta:|\eta|=1\}$. Consequently, $\inf_{p\in K,\,|\eta|=1}\eta^{\mathsf T}\nabla^2H(p)\eta>0$. Uniform upper and lower
bounds for the positive eigenvector entries follow from continuity on $K$ and the finite number of residues.

\end{proof}

\begin{lemma}\label{lem:geometry}
The set $\mathcal W$ is a compact strictly convex body with smooth
boundary and with zero in its interior. The quantities in
\eqref{eq:pe} depend continuously on $e$, and
$\inf_{e\in\mathbb S^{d-1}}\alpha_e>0$. For $n\in\mathbb S^{d-1}$, define $c(n)$ by
\begin{equation}\label{eq:normal-speed}
 c(n)=\min_{\lambda>0}\frac{H(\lambda n)}{\lambda}.
\end{equation}
The minimizer $\lambda(n)>0$ is unique and is called the critical exponential decay rate. Moreover,
\begin{equation}\label{eq:normal-ray}
 c(n)=\max_{v\in\mathcal W}n\cdot v,\qquad
 w(e)=\min_{\substack{n\in\mathbb S^{d-1}\\n\cdot e>0}}\frac{c(n)}{n\cdot e}.
\end{equation}
The minimizer in the second formula is uniquely $n_e=p_e/|p_e|$, with
$\lambda(n_e)=|p_e|$.
\end{lemma}
\begin{proof}
Strict convexity and superlinearity of $H$, proved in
Lemma~\ref{lem:spectral}, make $\nabla H:\R^d\to\R^d$ bijective:
for each $v$, the function $H(p)-p\cdot v$ has a unique minimum.
The positive definite Hessian and the inverse function theorem give
a smooth inverse $p(v)=(\nabla H)^{-1}(v)$. Differentiating
$\mathcal I(v)=p(v)\cdot v-H(p(v))$ yields
$\nabla\mathcal I(v)=p(v)$ and
$\nabla^2\mathcal I(v)=(\nabla^2H(p(v)))^{-1}$.
By Lemma~\ref{lem:spectral} and the definition \eqref{eq:geometry},
$\mathcal I$ is even and $\mathcal I(0)=-H(0)<0$. By \eqref{eq:geometry}, $\mathcal I$ is coercive: for any $R>0$,
$\mathcal I(v)\geq R|v|-\max_{|p|=R}H(p)$.
Its gradient vanishes only at zero. The positive definite Hessian of $\mathcal I$ makes it strictly convex.
Coercivity makes $\mathcal W=\{\mathcal I\leq0\}$ compact, and
$\mathcal I(0)<0$ places zero in its interior. Since $\nabla\mathcal I$ does not
vanish on $\partial\mathcal W$, that boundary is smooth; strict convexity
then makes each positive ray meet $\partial\mathcal W$ at exactly one
point. At $v_e=w(e)e$, the vector $p_e=\nabla\mathcal I(v_e)$ is an outward normal,
so $\alpha_e=p_e\cdot e>0$. The radial root argument below proves continuity in $e$ and
$\inf_{e\in\mathbb S^{d-1}}\alpha_e>0$.

For \eqref{eq:normal-speed}, $H(\lambda n)/\lambda\to\infty$
as $\lambda\downarrow0$, since $H(0)>0$, and as $\lambda\to\infty$,
by superlinearity. Thus a minimum is attained in $(0,\infty)$.
The derivative of $H(\lambda n)/\lambda$ has numerator
$g_n(\lambda)=\lambda n\cdot\nabla H(\lambda n)-H(\lambda n)$.
It satisfies $g_n(0)=-H(0)<0$ and
$g_n'(\lambda)=\lambda n^{\mathsf T}\nabla^2H(\lambda n)n>0$
by Lemma~\ref{lem:spectral}. At the interior minimum, $g_n=0$;
strict increase of $g_n$ makes this zero unique. Fenchel's inequality applied to \eqref{eq:geometry} gives
$n\cdot v\leq H(\lambda n)/\lambda$ on $\mathcal W$. At the minimizing
$\lambda$, the identity $g_n(\lambda)=0$ gives equality at
$v=\nabla H(\lambda n)$, for which $\mathcal I(v)=0$. Hence
$c(n)=\max_{v\in\mathcal W}n\cdot v$. Intersecting the resulting
half-spaces $\{v:n\cdot v\leq c(n)\}$ with the ray $\{re:r\geq0\}$
proves the second formula in \eqref{eq:normal-ray}; its minimizing normal
is unique because $\partial\mathcal W$ has a unique outward normal at
$v_e$. Finally, $H(p_e)=p_e\cdot v_e$ is the equality case of Fenchel's inequality at
$\mathcal I(v_e)=0$.

Choose $r_*>0$ so that the closed ball of radius $r_*$ is inside
$\mathcal W$; this is possible since $\mathcal I(0)<0$.
For a boundary point $v$ and $p=\nabla\mathcal I(v)$, convexity
gives $p\cdot(y-v)\le0$ for all $y\in\mathcal W$.
Taking $y=r_*p/|p|$ yields $p\cdot v\ge r_*|p|>0$.
{For fixed $e\in\mathbb S^{d-1}$, the function $r\mapsto\mathcal I(re)$ is strictly convex on $[0,\infty)$,  {has derivative zero at $r=0$} with $\mathcal I(0)<0$, and tends to $+\infty$ as $r\to\infty$. It therefore crosses zero at a unique $w(e)>0$, where $\partial_r[\mathcal I(re)]_{r=w(e)}=p_e\cdot e>0$. Since $\mathcal I$ is smooth, the implicit function theorem implies that $e\mapsto w(e)$ is smooth on $\mathbb S^{d-1}$ with differential
\[
 Dw(e)[h]=-\frac{w(e)p_e\cdot h}{p_e\cdot e}\qquad\text{for }h\in T_e\mathbb S^{d-1}.
\]
In particular, $w(e)$, $v_e$, $p_e$ and $\alpha_e$ depend smoothly on $e$. Since $|p_e|$ is continuous and nonvanishing on $\mathbb S^{d-1}$ and $w(e)$ is bounded above, the inequality}
\[
 \alpha_e=\frac{p_e\cdot v_e}{w(e)}
       \ge\frac{r_*\min_{e'\in\mathbb S^{d-1}}|p_{e'}|}{\max_{e'\in\mathbb S^{d-1}}w(e')}>0
\]
proves uniform positivity.

The equalities $\nabla H(p_e)=v_e$ and $H(p_e)=p_e\cdot v_e$ show that
$|p_e|$ is the unique minimizing rate in that normal direction.

\end{proof}

The following result quantifies the relation between the observation
direction and the outward normal to $\mathcal W$.
For $n\in\mathbb S^{d-1}$, use $\lambda(n)$ from
Lemma~\ref{lem:geometry} and set $p(n)=\lambda(n)n$ and
$v(n)=\nabla H(p(n))$.

\begin{proposition}\label{thm:quantitative-geometry}
The map $E(n)=v(n)/|v(n)|$ is a smooth diffeomorphism of the sphere,
with inverse $e\mapsto n_e=p_e/|p_e|$. There are $b_0,b_1>0$ such that
\begin{equation}\label{eq:quadratic-support-gap}
 b_0|n-n_e|^2\leq H(p(n))-p(n)\cdot v_e
                \leq b_1|n-n_e|^2
 \text{ for }e,n\in\mathbb S^{d-1}.
\end{equation}
For fixed $C\in\R$, define the convex set
\begin{equation}\label{eq:logarithmic-support-body}
 \mathcal O_t(C)=\bigcap_{n\in\mathbb S^{d-1}}
 \{x\in\R^d:p(n)\cdot x\leq tH(p(n))-\kappa_d\log t+C\}.
\end{equation}
For all sufficiently large $t$, it contains zero in its interior. Its
radial function $r_t(e)=\sup\{r\geq0:re\in\mathcal O_t(C)\}$ satisfies
\begin{equation}\label{eq:radial-support-expansion}
 r_t(e)=w(e)t-\frac{\kappa_d}{\alpha_e}\log t
        +\frac C{\alpha_e}+O_C\!\left(\frac{(\log t)^2}{t}\right),
\end{equation}
uniformly in $e$. In the representation
\[
 r_t(e)=\min_{\substack{n\in\mathbb S^{d-1}\\p(n)\cdot e>0}}
 \frac{tH(p(n))-\kappa_d\log t+C}{p(n)\cdot e},
\]
every minimizing normal differs from $n_e$ by $O_C(\log t/t)$.
\end{proposition}

Theorem~\ref{thm:upper} and \eqref{eq:outer-body} show that the superlevel set
$\{x\in\Z^d:u(t,x)\geq\theta\}$ lies in $\mathcal O_t(C_\theta)$ for each fixed
$\theta\in(0,1)$.

\begin{proof}[Proof of Proposition~\ref{thm:quantitative-geometry}]
 {The critical decay rate} $\lambda(n)$ from Lemma~\ref{lem:geometry} is  {the unique positive root} of
$G(\lambda,n)=\lambda n\cdot\nabla H(\lambda n)-H(\lambda n)$.
Writing $\lambda=\lambda(n)$ and
$\Sigma=\nabla^2H(\lambda n)$, we have
$\partial_\lambda G(\lambda,n)=\lambda n^{\mathsf T}\Sigma n>0$
by Lemma~\ref{lem:spectral}.
The implicit function theorem thus makes $\lambda(n)$ smooth.
Along a spherical curve with $n'(0)=\tau\perp n$, write
$p'=(D\lambda(n)[\tau])n+\lambda\tau$.
Differentiating $p\cdot\nabla H(p)-H(p)=0$ cancels
$p'\cdot\nabla H(p)$ and leaves
$\lambda n^{\mathsf T}\Sigma p'=0$.
Solving for $D\lambda(n)[\tau]$ and using $Dv(n)[\tau]=\Sigma p'$ gives
\begin{align}
 D\lambda(n)[\tau]
   &=-\lambda\frac{n^{\mathsf T}\Sigma\tau}{n^{\mathsf T}\Sigma n},
       \label{eq:normal-lambda-derivative}\\
 Dv(n)[\tau]
   &=\lambda\left(\Sigma\tau-
       \frac{n^{\mathsf T}\Sigma\tau}{n^{\mathsf T}\Sigma n}\Sigma n\right).
       \label{eq:normal-velocity-derivative}
\end{align}
Formula \eqref{eq:normal-velocity-derivative} gives
$n\cdot Dv(n)[\tau]=0$, so $Dv(n)$ maps $T_n\mathbb S^{d-1}=n^\perp$ into
$T_{v(n)}\partial\mathcal W=n^\perp$. Its quadratic form on the tangent space is
\[
 \lambda\left(\tau^{\mathsf T}\Sigma\tau
       -\frac{(n^{\mathsf T}\Sigma\tau)^2}{n^{\mathsf T}\Sigma n}\right)>0
       \text{ for }\tau\ne0.
\]
Strict Cauchy--Schwarz in the $\Sigma$ inner product proves this
positivity, since equality would make $\tau$ parallel to $n$. Consider
the compact set
$\mathcal K=\{(n,\tau):n\in\mathbb S^{d-1},\ \tau\perp n,\ |\tau|=1\}$.
The quadratic form is continuous and strictly positive on $\mathcal K$,
because $\lambda(n)$ and $\nabla^2H(\lambda(n)n)$ are continuous.
Its minimum and maximum on $\mathcal K$ are therefore numbers
$0<c_0\le C_0<\infty$, giving uniform tangent-space bounds for
\eqref{eq:normal-velocity-derivative}. Differentiating
$c(n)=n\cdot v(n)$ then gives
\begin{equation}\label{eq:support-gradient}
 v(n)=c(n)n+\nabla_{\mathbb S^{d-1}}c(n),\qquad
 \nabla^2_{\mathbb S^{d-1}}c+cI=Dv|_{n^\perp}.
\end{equation}
Here $\nabla_{\mathbb S^{d-1}}$ and $\nabla^2_{\mathbb S^{d-1}}$ denote
the Riemannian gradient and Hessian on the unit sphere, and $I$ is the identity on
$n^\perp$. The first identity follows from $n\cdot Dv=0$; differentiating it
and projecting onto $n^\perp$ gives the second.

Lemma~\ref{lem:geometry} gives a bijection between the radial and normal
parametrizations of $\partial\mathcal W$. The positive tangent-space
bounds in \eqref{eq:normal-velocity-derivative} make the normal
parametrization a local diffeomorphism. Radial projection has derivative
\[
 D(v/|v|)[h]=|v|^{-1}
          \bigl(h-(v\cdot h)v/|v|^2\bigr).
\]
If this derivative vanishes for $h=Dv(n)[\tau]\in n^\perp$, then
$h$ is parallel to $v(n)$. Since $n\cdot v(n)=c(n)>0$, we get
$h=0$, and the positive tangent quadratic form gives $\tau=0$.
Thus $E$ is a bijective local diffeomorphism, hence a smooth
diffeomorphism. Its inverse is $e\mapsto n_e$ from
Lemma~\ref{lem:geometry}; consequently $p_e$, $\alpha_e$ and $w(e)$
are smooth.

Fix $e$ and write $n_0=n_e$. By Lemma~\ref{lem:geometry},
$f_e(n)=c(n)-n\cdot v_e$ is nonnegative and vanishes only at
$n_0$. At $n_0$, its first derivative vanishes and its spherical Hessian is
$\nabla^2_{\mathbb S^{d-1}}c(n_0)+c(n_0)I$.
Identity \eqref{eq:support-gradient} and the tangent-space bounds above
therefore put the eigenvalues of the spherical Hessian of $f_e$ at
$n_e$ in the fixed interval $[c_0,C_0]$, independently of $e$.
For a unit-speed shortest spherical geodesic $\gamma$ from $n_e$ to
$n$, uniform continuity of the spherical Hessian gives $\theta_0>0$,
independent of $e$, such that
\[
 c_0/2\le (f_e\circ\gamma)''(s)\le 2C_0
 \text{ for }0\le s\le\theta\le\theta_0,
\]
where $\theta$ is the geodesic length. Since $f_e(n_e)=0$ and its
first derivative vanishes there, twice integrating gives
$c_0\theta^2/4\le f_e(n)\le C_0\theta^2$.
The identity $|n-n_e|=2\sin(\theta/2)$ gives the corresponding
chordal-distance bounds. On the compact set of pairs $(e,n)$ with
$|n-n_e|\geq2\sin(\theta_0/2)$, $f_e(n)$ has a positive minimum
and a finite maximum. Since $|n-n_e|\leq2$, the quadratic bounds
extend to all $e,n$. Finally,
$H(p(n))-p(n)\cdot v_e=\lambda(n)f_e(n)$,
and $\lambda(n)$ has positive minimum and finite maximum on the sphere.
These facts prove
\eqref{eq:quadratic-support-gap}.

Write $H_-=\min_{n\in\mathbb S^{d-1}}H(p(n))>0$ and $P_+=\max_{n\in\mathbb S^{d-1}}|p(n)|$.
When $tH_- -\kappa_d\log t+C>0$, the ball of radius
$(tH_- -\kappa_d\log t+C)/(2P_+)$ lies in $\mathcal O_t(C)$.
Thus zero is in its interior for large $t$.
The defining half-spaces give the radial formula
\begin{equation}\label{eq:radial-minimum}
 r_t(e)=\min_{\substack{n\in\mathbb S^{d-1}\\n\cdot e>0}}
 \left\{t\frac{c(n)}{n\cdot e}
       -\frac{\kappa_d\log t-C}{\lambda(n)n\cdot e}\right\}.
\end{equation}
Put $c_-=\min_{n\in\mathbb S^{d-1}}c(n)>0$. For large $t$, the numerator
$t c(n)-(\kappa_d\log t-C)/\lambda(n)$ is at least $tc_-/2$,
uniformly in $n$. The objective in
\eqref{eq:radial-minimum}, evaluated at $n_e$, is at most $C_2t$ uniformly in $e$.
Hence a minimizing sequence can be restricted to
$n\cdot e\ge c_-/(2C_2)$, so the minimum is attained on a compact set.
Shrink this lower bound to $\delta_0>0$ if necessary so that
$n_e\cdot e\ge\delta_0$ for all $e$.
For $n\cdot e\geq\delta_0$, put $L_t=\kappa_d\log t-C$ and
$d_e(n)=1/(\lambda(n)n\cdot e)$.
Subtracting the value at $n_e$ from the objective in
\eqref{eq:radial-minimum} gives
\[
 t\left(\frac{c(n)}{n\cdot e}-w(e)\right)
       -L_t\bigl(d_e(n)-d_e(n_e)\bigr).
\]
By \eqref{eq:quadratic-support-gap},
$c(n)/(n\cdot e)-w(e)
=[H(p(n))-p(n)\cdot v_e]/[p(n)\cdot e]
\geq(b_0/P_+)|n-n_e|^2$.
Also, $|d_e(n)-d_e(n_e)|\leq C|n-n_e|$, since smoothness of
$\lambda$ and the positive lower bounds for $\lambda(n)$ and $n\cdot e$
make $d_e$ uniformly Lipschitz on this set. At a minimizing normal in
\eqref{eq:radial-minimum}, the difference is nonpositive, so these two
bounds imply $|n-n_e|\leq C_C\log t/t$.
For each minimizing normal, the bound $|n-n_e|\leq C_C\log t/t$,
the quadratic upper bound in \eqref{eq:quadratic-support-gap}, and the
Lipschitz bound for $d_e$ show that the objective in \eqref{eq:radial-minimum}
differs from its value at $n_e$ by at most $C_C(\log t)^2/t$.
Evaluating that objective at $n_e$ gives the three explicit terms in
\eqref{eq:radial-support-expansion}.

\end{proof}

For the remaining estimates, put
$\mathcal P=\{p_e:e\in\mathbb S^{d-1}\}$. This set is compact and bounded
away from zero. On this set write $v_p=\nabla H(p)$ and $h_p=H(p)=p\cdot v_p$.
Let $\chi_p^-$ and $\chi_p^+$ be periodic functions satisfying
\begin{equation}\label{eq:scalar-correctors}
 Q_p(p\cdot x+\chi_p^-(x))=-h_p,\qquad
 \widehat Q_p(p\cdot x+\chi_p^+(x))=h_p.
\end{equation}
Equation~\eqref{eq:corrector-vector} and its reversed counterpart give
$\chi_p^-$ and $\chi_p^+$, normalized to have zero mean under $\pi_p$.
Write
$\langle f\rangle_{\pi_p}=\sum_{x\in\mathcal C}\pi_p(x)f(x)$ and
for a differentiable function $f$ of the parameter, define  {the radial Euler derivative}
\[
  {\mathcal D_p f(p)}:=\left.\frac{\dd}{\dd\varepsilon}
                         f(p+\varepsilon p)\right|_{\varepsilon=0}
  {=(p\cdot\nabla_p)f(p)},
\]
with $p$ fixed as the
direction of differentiation.  {The correctors} $\chi_p^-$ and $\chi_p^+$ are given  {explicitly} by
\begin{equation}\label{eq:spectral-corrector-formulas}
 \begin{split}
 \chi_p^-&=- {\mathcal D_p}\log\phi_p
             +\langle {\mathcal D_p}\log\phi_p\rangle_{\pi_p},\\
 \chi_p^+&= {\mathcal D_p}\log\phi_p^*
             -\langle {\mathcal D_p}\log\phi_p^*\rangle_{\pi_p}.
 \end{split}
\end{equation}
Indeed, differentiating $B_p\phi_p=H(p)\phi_p$ gives
\[
 Q_p( {\mathcal D_p}\log\phi_p)(x)
   =h_p+\sum_{\substack{z\in\Z^d\\|z|_1=1}}q_p(x,z)\,p\cdot z.
\]
Differentiating the transposed eigenvalue equation gives
\[
 \widehat Q_p( {\mathcal D_p}\log\phi_p^*)(x)
   =h_p-\sum_{\substack{z\in\Z^d\\|z|_1=1}}\widehat q_p(x,z)\,p\cdot z.
\]
Substitution into \eqref{eq:spectral-corrector-formulas} gives
\eqref{eq:scalar-correctors}; the added constants impose zero
$\pi_p$-mean without changing these equations. Analytic perturbation makes
$ {\mathcal D_p}\log\phi_p$ and $ {\mathcal D_p}\log\phi_p^*$ continuous in $p$.
Compactness of $\mathcal P$ and finiteness of $\mathcal C$ therefore
give uniform bounds for $\chi_p^-$ and $\chi_p^+$ in
\eqref{eq:spectral-corrector-formulas}.

If $\chi$ and $\widetilde\chi$ solve the first equation in
\eqref{eq:scalar-correctors}, then $Q_p(\chi-\widetilde\chi)=0$.
For two solutions of the second equation,
$\widehat Q_p(\chi-\widetilde\chi)=0$.
In each case the finite-cell maximum principle makes
$\chi-\widetilde\chi$ constant. If both functions have zero
$\pi_p$-mean, this constant is zero.
Here $p\cdot\nabla H(p)=H(p)$ is used only on $\mathcal P$;
for general $p$, differentiation gives $p\cdot\nabla H(p)$ in place of $h_p$.

\begin{lemma}\label{lem:Markov-estimates}
There is a constant $C$ such that, for either walk generated by $Q_p$
or $\widehat Q_p$,
\[
 \mathbb P_x(X_t=y)\leq C(1+t)^{-d/2}
 \text{ for }p\in\mathcal P,\ t\geq0,\ x,y\in\Z^d.
\]
For either walk, start at $x$ and define $Y_t$ by
$Y_t=p\cdot(X_t-x)+\chi_p^-(X_t)-\chi_p^-(x)+h_pt$ for $Q_p$,
and $Y_t=p\cdot(X_t-x)+\chi_p^+(X_t)-\chi_p^+(x)-h_pt$ for
$\widehat Q_p$. For each fixed $D\geq0$, both estimates below hold for all sufficiently large $t$,
uniformly in $p$ and the initial cell residue:
\begin{equation}\label{eq:crossing}
 \mathbb P_x(Y_t\leq-D-r)\geq\frac25-
                  \frac{C_D(1+r_+)}{\sqrt t}\text{ for }r\in\R.
\end{equation}
The corresponding upper-tail estimate is
\[
 \mathbb P_x(Y_t\ge D+r)\ge\frac25-
                  \frac{C_D(1+r_+)}{\sqrt t}\text{ for }r\in\R,
\]
with the same uniformity in $p$ and the initial cell residue.
\end{lemma}
\begin{proof}
We first consider $Q_p$ and use finite-cell Fourier inversion
as in \cite{Kuchment}.
Let
$\mathcal B=\prod_{i=1}^d[-\pi/L_i,\pi/L_i]$, with opposite faces
identified through diagonal gauge conjugation. The spectra and the full
kernel integrand below respect this identification. For $k\in\mathcal B$ and a vector $f$
indexed by $\mathcal C$, define
\[
 (A_p(k)f)(r)=\sum_{\substack{z\in\Z^d\\|z|_1=1}}q_p(r,z)
       \bigl(\ee^{\ii k\cdot z}f(r+z)-f(r)\bigr),
\]
where the sum retains every physical nearest-neighbor displacement and
only the argument of $f$ is reduced modulo $\Gamma$. Write
$P_t^{Q_p}(x,y)=\mathbb P_x(X_t=y)$ for the transition probabilities
of the $Q_p$-walk, and
$P_t^{Q_p}f(x)=\sum_{y\in\Z^d}P_t^{Q_p}(x,y)f(y)$ for its action on
bounded functions. The transition kernel is given by the finite-cell
inverse Fourier formula
\begin{equation}\label{eq:bloch-kernel-formula}
 P_t^{Q_p}(x,y)=\frac{|\mathcal C|}{(2\pi)^d}
   \int_{\mathcal B} \ee^{\ii k\cdot(x-y)}
       [\ee^{tA_p(k)}]_{r_x,r_y}\,\dd k,
 \qquad r_x=x\bmod\Gamma,\quad r_y=y\bmod\Gamma.
\end{equation}

For a function $f$ on $\Z^d$ with finite support, define
\[
 \widehat f(k,r)=\sum_{\ell\in\Gamma}\ee^{-\ii k\cdot(r+\ell)}
                  f(r+\ell)\text{ for }r\in\mathcal C.
\]
This is a finite sum. For $\ell=(L_i m_i)_i\in\Gamma$, direct
integration of the exponentials gives
\[
 \frac{|\mathcal C|}{(2\pi)^d}\int_{\mathcal B}
       \ee^{\ii k\cdot\ell}\,\dd k
       =\prod_{i=1}^d\1_{\{m_i=0\}}.
\]
Thus inversion is
$f(x)=|\mathcal C|(2\pi)^{-d}\int_{\mathcal B}
\ee^{\ii k\cdot x}\widehat f(k,r_x)\,\dd k$.
In the transform of a neighbor term, substitute
$y=r+\ell+z$; its phase contributes $\ee^{\ii k\cdot z}$ and its
new residue is $r+z\bmod\Gamma$. Periodicity of the rates therefore
gives $\widehat{Q_pf}=A_p(k)\widehat f$.
Starting from $f=\1_{\{y\}}$, inversion of
$\ee^{tA_p(k)}\widehat f$ yields \eqref{eq:bloch-kernel-formula}. The integral and its time
derivative are bounded uniformly in $x$ on each compact time
interval, since the matrices and the integration domain are
bounded. Writing $v(t,x)$ for this integral, we have
$\partial_t v=Q_pv$ and $v(0,\cdot)=\1_{\{y\}}$,
and comparison uniqueness identifies it with the transition kernel
constructed by the Poisson clock. Finite linear combinations give
the formula for every finitely supported $f$. For $f\in\ell^\infty(\Z^d)$, the Markov expectation defines $P_t^{Q_p}f$.
The probability law gives positivity and $\sum_{y\in\Z^d}P_t^{Q_p}(x,y)=1$,
which imply absolute convergence of $\sum_{y\in\Z^d}P_t^{Q_p}(x,y)f(y)$.

For $m\in\Z^d$ and $g=(2\pi m_i/L_i)_i$, let $D_g$ be the
diagonal matrix with entries $D_g(r,r)=\ee^{\ii g\cdot r}$ for
$r\in\mathcal C$. Since
$\ee^{\ii g\cdot\ell}=1$ for $\ell\in\Gamma$, the edge formula
gives $A_p(k+g)=D_g^{-1}A_p(k)D_g$.
In inversion, $\ee^{\ii g\cdot(x-y)}$ cancels
$\ee^{-\ii g\cdot r_x}\ee^{\ii g\cdot r_y}$.
Thus the spectra agree on opposite faces, and the corresponding phase factors cancel in the inverse Bloch integral.

Since $0$ is a simple eigenvalue of $A_p(0)$,
analytic perturbation extends it to $k$ near zero. The analytic
continuation of $H$ near $p$ identifies this eigenvalue as
\[
 H(p-\ii k)-H(p)=-\ii v_p\cdot k-\tfrac12 k^{\mathsf T}\Sigma_p k
                         +O(|k|^3).
\]
The principal eigenprojection is analytic near zero.
To exclude other eigenvalues with zero real part, consider each physical edge
separately. If $A_p(k)f=\nu f$ and $|f(r)|=\max_{s\in\mathcal C}|f(s)|>0$,
then
\[
 \operatorname{Re}\nu
  =\sum_{\substack{z\in\Z^d\\|z|_1=1}}q_p(r,z)
    \left(\operatorname{Re}
       \frac{\ee^{\ii k\cdot z}f(r+z)}{f(r)}-1\right)\le0.
\]
If equality holds, every quotient has modulus at most one and real
part one, so each equals one. All neighboring components have the
same maximal modulus. At each such neighbor, the real parts of the summands in the divided eigenvalue equation are again nonpositive and sum to zero.
Thus every adjacent phase quotient equals one. A finite nearest-neighbor path reaches each cell residue. The eigenvalue equation then gives $\nu=0$, and a physical
period loop gives $\ee^{\ii k_iL_i}=1$. This also covers edges whose endpoints have the same residue for $L_i=1$
and distinct edges with the same pair of endpoint residues for $L_i=2$.

The matrices $A_p(k)$ have a common norm bound for
$(p,k)\in\mathcal P\times\mathcal B$.
If their spectral real parts approached zero on a compact subset
of the Bloch torus separated from $k=0$, choose corresponding
eigenvalues and convergent subsequences of both $p$ and $k$. The eigenvalues have a bounded
subsequence because their modulus is bounded by the matrix norm.
Continuity of the determinant makes its limit an eigenvalue of
the limiting matrix, contradicting the equality analysis above.
For the complementary spectrum near $k=0$, a sequence of eigenvalues
whose real parts tend to zero would, by compactness, converge to the
zero eigenvalue of $A_{p_*}(0)$ for a limiting $p_*\in\mathcal P$.
That eigenvalue is simple and isolated. Its continuous rank-one
projection already contains the unique eigenvalue near zero,
contradicting membership in the complementary spectrum. Thus the
remaining eigenvalues have a common negative real-part bound. For a one-vertex cell the complementary projection
is zero, so its estimates are immediate. Finitely many local
analytic neighborhoods covering
$\mathcal P$ provide one radius and uniform bounds for the
principal eigenvalue, its projection, and their first three derivatives.
The expansion of $H(p-\ii k)-H(p)$ above therefore has a uniform
remainder $C|k|^3$.
If $m>0$ is the uniform lower eigenvalue bound for $\Sigma_p$,
choose $\delta$ so small that $C\delta\le m/4$; then
\[
 \operatorname{Re}[H(p-\ii k)-H(p)]
       \le -m|k|^2/4\text{ for }|k|\le\delta.
\]
The spectral gap and the bound on $\operatorname{Re}[H(p-\ii k)-H(p)]$
are uniform in $p\in\mathcal P$.

Fix a sufficiently small $\delta>0$. On $|k|\le\delta$, the principal
eigenvalue is uniformly separated in real part from the complementary
spectrum. Choose a rectangular contour enclosing that complementary
spectrum and leaving the principal eigenvalue to its right.
On the compact region $|k|\ge\delta/2$ of the Bloch torus, the full
spectrum has a uniform negative real-part bound. Choose a second
rectangular contour enclosing the full spectrum on this region.
Both contours lie in the left half-plane, and their other three edges
lie beyond a common bound for the matrix norms. Let $\gamma>0$ be small
enough that both right edges lie in $\operatorname{Re}z\le-\gamma/2$.
In the near-zero representation, $P^\perp$ is the complementary spectral
projection; in the second representation it is the identity. Write
$\mathcal G$ for the contour appropriate to each representation.

{ By the Dunford--Taylor functional calculus for finite matrices,
\[
 \ee^{tA_p(k)}P^\perp(k)=\frac1{2\pi\ii}
 \int_{\mathcal G} \ee^{tz}(zI-A_p(k))^{-1}P^\perp(k)\,\dd z.
\]
If $R$ is a common upper bound for $\|A_p(k)\|$ over $(p,k)\in\mathcal P\times\mathcal B$, we may fix the horizontal edges of $\mathcal G$ at $\operatorname{Im}z=\pm(R+1)$, the left vertical edge at $\operatorname{Re}z=-(R+1)$, and the right vertical edge inside $\{\operatorname{Re}z\leq-\gamma/2\}$ strictly between the separated spectral sets. By continuity of $(z,p,k)\mapsto(zI-A_p(k))^{-1}$ on the compact product of $\mathcal G$ with the corresponding parameter domain, the resolvent norm is uniformly bounded on $\mathcal G$. Since $\|P^\perp(k)\|\leq 1+\|P(k)\|$ is likewise uniformly bounded for $|k|\leq\delta$ and $P^\perp(k)=I$ for $|k|\geq\delta/2$, the contour integral yields $\|\ee^{tA_p(k)}P^\perp(k)\|\le C\ee^{-\gamma t/2}$ uniformly,}
because the resolvent is bounded on both contours, uniformly over the
corresponding compact parameter sets. For the principal
mode near $k=0$, the uniform positive definiteness of $\Sigma_p$ from
Lemma~\ref{lem:spectral} bounds the real part of the exponent by $-c|k|^2$.
Thus $\int_{\{k\in\R^d:|k|\le\delta\}}\ee^{-ct|k|^2}\,\dd k\le Ct^{-d/2}$, and inserting
the principal and complementary estimates into
\eqref{eq:bloch-kernel-formula} proves the kernel bound for $Q_p$ when
$t\geq1$; for $0\leq t\leq1$, the bound follows from
$\mathbb P_x(X_t=y)\leq1$ after increasing $C$.
Since $B_{-p}=B_p^{\mathsf T}$, the positive eigenvector of $B_{-p}$ is
proportional to $\phi_p^*$, and hence $\widehat Q_p=Q_{-p}$ by
\eqref{eq:Q}--\eqref{eq:Qreverse}. Evenness of $H$ and uniqueness in
Lemma~\ref{lem:geometry} give $-p\in\mathcal P$.
The kernel bound therefore holds for $\widehat Q_p$.

For $Y_t$, use $A_p(\theta p)$ for $Q_p$ and
$A_{-p}(\theta p)$ for $\widehat Q_p$; their jump factors are
$\ee^{\ii\theta p\cdot z}$. Compactness of $\mathcal P$ gives a common
scalar neighborhood in which $\theta p$ remains in the Bloch neighborhood.
The endpoint factors $\ee^{\ii\theta(\chi_p^\pm(X_t)-\chi_p^\pm(x))}$
include the correctors, and the deterministic terms in $Y_t$ cancel
the principal linear phases. The quadratic expansion gives uniform
constants $\delta,c,C>0$ such that
\begin{equation}\label{eq:char-small}
 |\mathbb E_x\ee^{\ii\theta Y_t}|
       \leq C\ee^{-ct\theta^2}+C\ee^{-ct}
       \text{ for }|\theta|\leq\delta.
\end{equation}

For $Q_p$, put $v_\theta(r)=\ee^{\ii\theta\chi_p^-(r)}$.
The finite-cell exponential-moment formula gives the exact scalar identity
\[
 \mathbb E_x \ee^{\ii\theta Y_t}
 =\ee^{\ii\theta h_pt}\ee^{-\ii\theta\chi_p^-(x)}
                [\ee^{tA_p(\theta p)}v_\theta](r_x).
\]
If $P_p(k)$ is the principal projection, the coefficient of its mode
in this expectation is
$a_{p,r}(\theta)=\ee^{-\ii\theta\chi_p^-(r)}
[P_p(\theta p)v_\theta](r)$.
At zero, $P_p(0)=\1\pi_p^{\mathsf T}$ and $v_0=\1$, so $a_{p,r}(0)=1$. Bounded derivatives of the projection and the
uniformly bounded correctors give
$|a_{p,r}(\theta)-1|\le C|\theta|$ on one common neighborhood.
The critical relation $p\cdot\nabla H(p)=h_p$ cancels the
deterministic linear phase, leaving
\[
 t[H(p-\ii\theta p)-H(p)+\ii\theta h_p]
   =-\tfrac12t\theta^2\sigma_p^2+O(t|\theta|^3).
\]
Here $\sigma_p^2=p^{\mathsf T}\Sigma_pp$ has bounds
$0<\sigma_-^2\le\sigma_p^2\le\sigma_+^2<\infty$:
the Hessians are uniformly positive and bounded, and $\mathcal P$
is compact and avoids zero. Reduce the scalar radius until the
remainder has real part at most
$\sigma_-^2t\theta^2/4$. The complementary projection contributes
at most $C {\ee}^{-ct}$, since the endpoint factors have modulus one.
The bound $\operatorname{Re}(H(p-\ii\theta p)-H(p)+\ii\theta h_p)
\leq-\sigma_-^2\theta^2/4$ for $|\theta|\leq\delta$, together with the
contribution $C {\ee}^{-ct}$ from the complementary projection, gives
\eqref{eq:char-small} with uniform constants.
For $\widehat Q_p$, use $Q_{-p}$ and
$\ee^{\ii\theta p\cdot z}$. Evenness changes its principal exponent
to $H(p+\ii\theta p)-H(p)$;
$\ee^{-\ii\theta h_pt}$ cancels its linear term. The exponent is
$-t\theta^2\sigma_p^2/2+O(t|\theta|^3)$, the coefficient of the
principal projection is $1+O(|\theta|)$, and the complementary term
is $O(\ee^{-ct})$, uniformly in $p$ and the residue. Including $\ee^{\ii\theta(\chi_p^+(X_t)-\chi_p^+(x))}$ in the finite-cell formula gives
$|\mathbb E_x {\ee}^{\ii\theta Y_t}|\leq C {\ee}^{-ct\theta^2}+C {\ee}^{-ct}$
for $|\theta|\leq\delta$, as required.

Fix $\theta$ in a bounded set. Use $A_p(\theta p/\sqrt t)$ for $Q_p$ and
$A_{-p}(\theta p/\sqrt t)$ for $\widehat Q_p$. The jump terms contain
$\ee^{\ii\theta p\cdot z/\sqrt t}$. After centering, the exponent of the principal term in
$\mathbb E_x\ee^{\ii\theta Y_t/\sqrt t}$ is
$-\sigma_p^2\theta^2/2+O(|\theta|^3/\sqrt t)$.
The contribution of the other eigenvalues is $O(\ee^{-ct})$.
Both estimates are uniform in $p\in\mathcal P$ and the initial cell residue.
Here $\sigma_p^2=p^{\mathsf T}\Sigma_pp$ is bounded above and away from zero. At $\theta=0$, the coefficient of the principal term equals one,
since $P_p(0)\1=\1$ and both endpoint factors equal one. Analyticity and the bounded endpoint
functions $\chi_p^\pm$ in \eqref{eq:spectral-corrector-formulas} make the coefficient
tend uniformly to one under the stated scaling. Consequently,
$\mathbb E_x \ee^{\ii\theta Y_t/\sqrt t}\to \ee^{-\sigma_p^2\theta^2/2}$
as $t\to\infty$, uniformly in $p\in\mathcal P$, the initial
residue, and bounded $\theta$-sets.
L\'evy's continuity theorem \cite[Section~26]{Billingsley} identifies
the Gaussian limit for each fixed parameter. We prove uniform convergence
of distribution functions by smoothing. Write $Z_t=Y_t/\sqrt t$ for $t\ge1$.
For the $Q_p$-walk, $Y_t$ is a martingale starting at zero, and its
jump at $X_s\to X_s+z$ is
$p\cdot z+\chi_p^-(X_s+z)-\chi_p^-(X_s)$.
The compensated-jump second-moment identity therefore reads
\[
 \begin{aligned}
 \mathbb E_xY_t^2
 &=\mathbb E_x\int_0^t
   \sum_{\substack{z\in\Z^d\\|z|_1=1}}q_p(X_s,z)\\
 &\hspace{2em}\times
   \bigl(p\cdot z+\chi_p^-(X_s+z)-\chi_p^-(X_s)\bigr)^2\,\dd s
 \le Ct.
 \end{aligned}
\]
For the $\widehat Q_p$-walk, the same identity holds with
$q_p,\chi_p^-$ replaced by $\widehat q_p,\chi_p^+$.
Compactness of $\mathcal P$ and the finite-cell rate and corrector
bounds make $C$ uniform in $p$ and the initial residue. Thus
$\mathbb E_xZ_t^2\le C$ for both walks.
By periodicity, the law of $Y_t$ depends on $x$ only through
$r=x\bmod\Gamma$. Let $G$ be an independent standard normal variable.
For fixed $\eta>0$, write $f_{t,\eta}$ for the density of
$Z_t+\eta G$, suppressing its dependence on $p$ and $r$. Then
\[
 f_{t,\eta}(z)=\frac1{2\pi}\int_{\mathbb R}
  \ee^{-\ii\theta z}\ee^{-\eta^2\theta^2/2}
                   \mathbb E_x {\ee}^{\ii\theta Z_t}\,\dd\theta.
\]
Conditioning on $Z_t$ and using Gaussian Fourier inversion proves
this formula by Fubini. The target density is that of $N(0,\sigma_p^2+\eta^2)$.
Uniform characteristic-function convergence on bounded intervals,
together with the integrable Gaussian tail, gives
\[
 \sup_{\substack{p\in\mathcal P,\ r\in\mathcal C\\z\in\R}}|f_{t,\eta}(z)
       -f_{N(0,\sigma_p^2+\eta^2)}(z)|\to0
          \text{ as }t\to\infty\text{ for each fixed }\eta>0.
\]
For any $R>0$, the second-moment bound gives
$\mathbb P(|Z_t+\eta G|>R)\le C_\eta/R^2$;
the Gaussian target satisfies the same bound uniformly in $p$.
For thresholds $z\in[-R,R]$, integrate the density difference
from $-R$ to $z$ and add the two left-tail bounds. Thresholds
outside $[-R,R]$ are controlled directly by these tail bounds.
First choose $R$ large, then let $t\to\infty$.
The distribution functions of $Z_t+\eta G$ therefore converge to those of $N(0,\sigma_p^2+\eta^2)$,
uniformly in the threshold, $p$ and the initial residue.

To remove the smoothing, set $\varepsilon=\sqrt\eta$.
For every $z\in\R$, event inclusions give
\[
 \mathbb P(Z_t+\eta G\le z-\varepsilon)
              -\mathbb P(|\eta G|>\varepsilon)
 \le\mathbb P(Z_t\le z)
 \le\mathbb P(Z_t+\eta G\le z+\varepsilon)
              +\mathbb P(|\eta G|>\varepsilon).
\]
The error probability is at most $\eta^2/\varepsilon^2=\eta$.
All centered Gaussian densities with variance at least
$\sigma_-^2$ are bounded by $C_-=1/(\sqrt{2\pi}\sigma_-)$.
Changing a threshold by $\varepsilon$ therefore costs at most
$C_-\varepsilon$. Moreover, if $Z$ is an
independent $N(0,\sigma_p^2)$ variable, conditioning on $G$ shows
\[
 \sup_{z\in\R}|\mathbb P(Z+\eta G\le z)-\mathbb P(Z\le z)|
          \le C_-\eta\,\mathbb E|G|.
\]
Combining these inequalities with the smoothed convergence gives
\[
 \limsup_{t\to\infty}\sup_{\substack{p\in\mathcal P,\ r\in\mathcal C\\z\in\R}}
 \left|\mathbb P_x(Y_t/\sqrt t\le z)
       -\mathbb P(N(0,\sigma_p^2)\le z)\right|
 \le C_-(\sqrt\eta+\eta\mathbb E|G|)+\eta.
\]
Letting $\eta\downarrow0$ in the preceding display gives
$\sup_{\substack{p\in\mathcal P,\ r\in\mathcal C\\z\in\R}}|\mathbb P_x(Y_t/\sqrt t\leq z)
-\mathbb P(N(0,\sigma_p^2)\leq z)|\to0$ as $t\to\infty$.

Esseen's smoothing argument \cite[p.~211, (2.7)]{Esseen1966},
applied to \eqref{eq:char-small}, gives
\begin{equation}\label{eq:concentration}
 \sup_{\xi\in\R}\mathbb P_x(Y_t\in[\xi,\xi+\ell])
       \leq C(1+\ell)t^{-1/2}\text{ for }t\geq1,\ \ell\geq0.
\end{equation}
For the explicit calculation, define
$g(s)=(\sin s/s)^2$, with $g(0)=1$. Multiplying two elementary
integrals $\sin s/s=\tfrac12\int_{-1}^1\ee^{\ii vs}\,\dd v$
gives
\[
 g(s)=\frac14\int_{-2}^2(2-|v|)\ee^{\ii vs}\,\dd v.
\]
On $|s|\le1$, $g(s)\ge(\sin1)^2>0$. For an interval $J$ of
length $\ell$ and midpoint $m$, set
$L=\max\{1,\ell,2/\delta\}$, where $\delta$ is the radius in
\eqref{eq:char-small}. Then $Y_t\in J$ implies
$|(Y_t-m)/L|\le1$. Positivity of $g$, its integral representation,
and Fubini therefore give
\[
 \begin{split}
 \mathbb P_x(Y_t\in J)
 &\le (\sin1)^{-2}\mathbb E_xg((Y_t-m)/L)\\
 &\le C L\int_{-2/L}^{2/L}
                  |\mathbb E_x {\ee}^{\ii\theta Y_t}|\,\dd\theta\\
 &\le C L\int_{-\delta}^{\delta}
               (\ee^{-ct\theta^2}+\ee^{-ct})\,\dd\theta
 \le C_\delta(1+\ell)t^{-1/2}\text{ for }t\ge1.
 \end{split}
\]
The last step uses the change of variables $v=\sqrt t\theta$
in the Gaussian integral and the boundedness of
$\sqrt t \ee^{-ct}$ for $t\ge1$. The calculation also applies to closed intervals with $\ell=0$.

To prove \eqref{eq:crossing}, choose $t_D$ so large that the uniform
distribution-function error is at most $1/40$ and
$C_-D/\sqrt t\leq1/40$ for $t\geq t_D$.
The centered Gaussian probability at $-D/\sqrt t$ is then at
least $19/40$, so $\mathbb P_x(Y_t\leq-D)\geq9/20$.
For $r>0$, \eqref{eq:concentration} on $[-D-r,-D]$ gives
\[
 \mathbb P_x(Y_t\leq-D-r)
 \geq\frac9{20}-\frac{C(1+r)}{\sqrt t}
 \geq\frac25-\frac{C(1+r)}{\sqrt t}.
\]
For $r\leq0$, the inclusion
$\{Y_t\leq-D\}\subset\{Y_t\leq-D-r\}$ gives
$\mathbb P_x(Y_t\leq-D-r)\geq9/20$, which implies
\eqref{eq:crossing} with $r_+=0$.
To obtain the upper-tail estimate, replace $\theta$ by $-\theta$
in the characteristic-function bounds. The Gaussian term is
unchanged, and the modulus of the remainder has the identical
bound, so the distribution-function and concentration estimates
hold for $-Y_t$ with common constants. Using $\mathbb P_x(-Y_t\leq-D)\geq9/20$, \eqref{eq:concentration} for $-Y_t$,
and $\{-Y_t\leq-D\}\subset\{-Y_t\leq-D-r\}$ for $r\leq0$ gives the upper-tail counterpart of \eqref{eq:crossing}.
\end{proof}

\section{The logarithmic upper bound}\label{sec:upper}
The proof of Theorem~\ref{thm:upper} combines estimates for $w$ and $W$
defined below with stopping at a moving boundary. We use the method of
\cite[Sections~3.2--3.3]{HNRRPeriodic} for linear Dirichlet problems and extend
the weighted estimates in \cite[Lemmas~5.3--5.4 and the proof of Theorem~1.3]{HuAxial}
to periodic coefficients. All constants are uniform for $p\in\mathcal P$.
Write $h=H(p)=p\cdot v_p$, $\phi=\phi_p$, $\phi^*=\phi_p^*$,
$\pi=\pi_p$ and $\chi^\pm=\chi_p^\pm$, and put
$a_{\max}=\max_{x\in\mathcal C}a(x)$. Set
$\mathcal N(t,x)=a(x)u(t,x)-F(x,u(t,x))$. By
\eqref{eq:periodic-kpp}, $0\leq\mathcal N(t,x)\leq a_{\max}u(t,x)$.
Define $w$ and $W$ by
\begin{equation}\label{eq:weighted-variables}
 w(t,x)=\frac{\ee^{p\cdot x-ht}}{\phi(x)}u(t,x),\qquad
 W(t,x)=\pi(x)w(t,x)=\phi^*(x)\ee^{p\cdot x-ht}u(t,x).
\end{equation}
Direct substitution and \eqref{eq:reverse-edge} give
\begin{equation}\label{eq:weighted-evolutions}
 w_t=Q_pw-\frac{\ee^{p\cdot x-ht}}{\phi(x)}\mathcal N,\qquad
 W_t=\widehat Q_p^{\mathsf T}W-\mathcal R,\qquad
 \mathcal R=\phi^*(x)\ee^{p\cdot x-ht}\mathcal N.
\end{equation}
The transpose is taken with respect to counting measure, and
$0\leq\mathcal R\leq a_{\max}W$.

Set $E(t,x)= {\ee^{p\cdot x-ht}}/\phi(x)$.
Substitution in \eqref{eq:weighted-variables} gives
\[
 E(t,x)u(t,x+z)=q_p(x,z)w(t,x+z),\qquad
 \sum_{\substack{z\in\Z^d\\|z|_1=1}}q_p(x,z)=h+2d-a(x).
\]
Consequently,
$w_t(t,x)=\sum_{\substack{z\in\Z^d\\|z|_1=1}}q_p(x,z)
(w(t,x+z)-w(t,x))-E(t,x)\mathcal N(t,x)$.
The identity $\pi(x)q_p(x,z)=\pi(x+z)\widehat q_p(x+z,-z)$
converts each incoming term into the forward equation for $W$;
the diagonal terms agree because $\sum_{\substack{z\in\Z^d\\|z|_1=1}}q_p(x,z)=\sum_{\substack{z\in\Z^d\\|z|_1=1}}\widehat q_p(x,z)$.
Each $z\in\Z^d$ with $|z|_1=1$ labels a directed nearest-neighbor
increment; both increments are retained when they have the same cell residue.

\begin{lemma}\label{lem:mass-decay}
 {The weighted mass} $M(t)=\sum_{x\in\Z^d}W(t,x)$ satisfies
$M(t)\leq C(1+t)^{-1/2}$ for every $t\geq0$.
\end{lemma}
\begin{proof}
To justify the sums involving $W$ on bounded time intervals, choose
$\varepsilon>1+\max_{p\in\mathcal P}\max_{1\leq i\leq d}|p_i|$ in \eqref{eq:global-exponential-bound}. The same estimate controls the weighted sums of $u_t$, since
\[
 |u_t(t,x)|\leq\sum_{\substack{z\in\Z^d\\|z|_1=1}}u(t,x+z)
                    +(2d+L_F)u(t,x).
\]
The bounded exit rates in Lemma~\ref{lem:spectral} make
$\widehat Q_p^{\mathsf T}$ bounded on $\ell^1(\Z^d)$, with norm at most
twice the maximal exit rate.

Fix $S>0$ and an integer $m\geq0$. Since
$\varepsilon-\max_{\substack{p\in\mathcal P\\1\leq i\leq d}}|p_i|>1$,
\eqref{eq:global-exponential-bound} and the displayed bound for $u_t$ give
\[
 (1+|x|_1)^m\bigl(|W(t,x)|+|W_t(t,x)|\bigr)
 \leq g_{S,m}(x):=C_{S,m}(1+|x|_1)^m \ee^{-|x|_1}
\]
for $0\leq t\leq S$ and $p\in\mathcal P$.
Here $\sum_{x\in\Z^d}g_{S,m}(x)<\infty$ and $C_{S,m}$ is independent of $p$.
The finitely many nearest-neighbor shifts satisfy this bound after
increasing $C_{S,m}$.
For a difference quotient use
$W(t+h,x)-W(t,x)=\int_t^{t+h}W_s(s,x)\,\dd s$.
The summability of $g_{S,m}$ permits interchanging the weighted sum
and time integral by Fubini's theorem. Dominated convergence passes $h\to0$ in this identity and
identifies the weighted $\ell^1$ derivative as $W_t$.
The bound by $g_{S,m}$ and pointwise continuity of $W_t$ also give
continuity of $t\mapsto W_t(t,\cdot)$ in the norm
$f\mapsto\sum_{x\in\Z^d}(1+|x|_1)^m|f(x)|$.
Thus $t\mapsto W(t,\cdot)$ is continuously differentiable in that norm,
and the weighted sums may be differentiated term by term.
Summing the forward generator gives zero:
each directed incoming term is paired with its outgoing term by
$x\mapsto x+z$, an admissible reindexing of an absolutely convergent sum.

 The variation-of-constants formula for $W$ in \eqref{eq:weighted-evolutions} therefore
holds in $\ell^1(\Z^d)$. Summing the equation for $W$ gives
\begin{equation}\label{eq:M-derivative}
 M'(t)=-\sum_{x\in\Z^d}\mathcal R(t,x)\leq0.
\end{equation}
We also need
\begin{equation}\label{eq:unweighted-polynomial}
 \sum_{x\in\Z^d}u(t,x)\leq C(1+t)^d.
\end{equation}
Fix $\varepsilon>0$. The bound \eqref{eq:global-exponential-bound} gives
$u(t,x)\le C_\varepsilon\exp(c_\varepsilon t-\varepsilon|x|_1)$.
Choose $R>2c_\varepsilon/\varepsilon+1$ and split the sum at
$|x|_1\leq R(1+t)$. The interior contains
$O((1+t)^d)$ lattice points and contributes at most this order because
$u\le1$. For $|x|_1>R(1+t)$,
$c_\varepsilon t-\varepsilon|x|_1\leq-\varepsilon|x|_1/2$, and the sum of $\ee^{-\varepsilon|x|_1/2}$ over $x\in\Z^d$
is finite and independent of $t$.
Adding the interior and exterior contributions gives \eqref{eq:unweighted-polynomial}.

Put $L=d+3$ and
\[
\begin{gathered}
 z(t,x)=p\cdot x-ht+\chi^+(x)+L\log(1+t),\\
 J(t)=\sum_{x\in\Z^d}z(t,x)W(t,x),\qquad J_+(t)=\sum_{x\in\Z^d}z(t,x)_+W(t,x).
\end{gathered}
\]
We next prove
\begin{equation}\label{eq:negative-moment}
 \sum_{\substack{x\in\Z^d\\z(t,x)<0}}(1+|z(t,x)|)W(t,x)\leq C(1+t)^{-3}.
\end{equation}
By the definition of $z$ and \eqref{eq:weighted-variables},
$W=\phi^*\ee^{z-\chi^+}(1+t)^{-L}u$.
For $z<0$, $(1+|z|)\ee^z$ is bounded. The uniform bounds for
$\phi_p^*$ and $\chi_p^+$, together with \eqref{eq:unweighted-polynomial},
therefore give \eqref{eq:negative-moment}.
Equation \eqref{eq:scalar-correctors} gives
$(\partial_t+\widehat Q_p)z=L/(1+t)$, and differentiation of
$J$ using \eqref{eq:weighted-evolutions} yields
\begin{equation}\label{eq:J-derivative}
 J'(t)=\frac{L}{1+t}M(t)-\sum_{x\in\Z^d}z(t,x)\mathcal R(t,x)
       \leq\frac{L}{1+t}M(t)+C(1+t)^{-3}.
\end{equation}
In \eqref{eq:J-derivative}, the contribution from $z\geq0$ is
nonpositive, while the contribution from $z<0$ is controlled by
\eqref{eq:negative-moment} and $\mathcal R\leq a_{\max}W$.
Equation \eqref{eq:M-derivative} bounds $M(t)$ by $M(0)$. Since
$\supp u_0$ is finite by \eqref{eq:u0}, both $M(0)$ and $|J(0)|$
are bounded uniformly for $p\in\mathcal P$. Integrating
\eqref{eq:J-derivative} gives $J(t)\leq C(1+\log(1+t))$.
By \eqref{eq:negative-moment}, $0\leq J_+(t)-J(t)\leq C(1+t)^{-3}$,
so $J_+(t)\leq C(1+\log(1+t))$.

Let $P_p^+(s)=\ee^{s\widehat Q_p}$ act on $\ell^\infty(\Z^d)$;
its transpose acts on $\ell^1(\Z^d)$. For the $\widehat Q_p$-walk with
$X_0=x$, the definition of $Y_t$ in Lemma~\ref{lem:Markov-estimates} gives
\[
 z(2t,X_t)=z(t,x)+Y_t+L\log\frac{1+2t}{1+t}.
\]
Since $L\log((1+2t)/(1+t))\leq L\log2$, the event
$\{Y_t\leq-(L\log2+1)-z(t,x)\}$ implies
$z(2t,X_t)\leq-1$. Applying \eqref{eq:crossing} with
$D=L\log2+1$ and $r=z(t,x)$, then summing with weights $W(t,x)$,
gives the following bound for all sufficiently large $t$.
The term $r_+$ contributes
$\sum_{x\in\Z^d}z(t,x)_+W(t,x)=J_+(t)$:
\[
 \sum_{\substack{y\in\Z^d\\z(2t,y)<0}}
     \bigl(P_p^+(t)^{\mathsf T}W(t,\cdot)\bigr)(y)
 \geq\frac25M(t)-\frac{C}{\sqrt t}\bigl(M(t)+J_+(t)\bigr).
\]
For $W$ in \eqref{eq:weighted-evolutions}, variation of constants gives,
in $\ell^1$,
\begin{equation}\label{eq:mass-transport-deficit}
 \begin{split}
 W(2t,\cdot)&=P_p^+(t)^{\mathsf T}W(t,\cdot)-\mathcal D_t,\\
 \mathcal D_t&=\int_t^{2t}P_p^+(2t-s)^{\mathsf T}\mathcal R(s,\cdot)\,\dd s
                \ge0.
 \end{split}
\end{equation}
Since $P_p^+(s)\1=\1$, Tonelli's theorem and
\eqref{eq:M-derivative} give
\[
 \sum_{x\in\Z^d}\mathcal D_t(x)
 =\int_t^{2t}\sum_{x\in\Z^d}\mathcal R(s,x)\,\dd s
 =M(t)-M(2t).
\]
For $A_t=\{x\in\Z^d:z(2t,x)<0\}$, combining the displayed identity with
\eqref{eq:mass-transport-deficit} gives
\[
 \sum_{x\in A_t}\bigl(P_p^+(t)^{\mathsf T}W(t,\cdot)\bigr)(x)
 \le \sum_{x\in A_t}W(2t,x)+M(t)-M(2t).
\]
We combine this upper bound with the lower bound obtained by integrating
\eqref{eq:crossing} against $W(t,x)$ with $D=L\log2+1$ and
$r=z(t,x)$. Applying \eqref{eq:negative-moment} at time $2t$ then gives
\begin{equation}\label{eq:mass-dyadic}
 M(2t)\leq\frac35M(t)+\frac{C}{\sqrt t}
                       \bigl(M(t)+J_+(t)\bigr)+Ct^{-3}.
\end{equation}
Choose $t_0$ large and set $t_n=2^nt_0$ and
$b_n=\sqrt{t_n}M(t_n)$ for $n=0,1,\ldots$. Multiplying
\eqref{eq:mass-dyadic} at $t=t_n$ by $\sqrt{2t_n}$ and using the
bound $J_+(t)\leq C(1+\log(1+t))$ gives
\[
 b_{n+1}\leq\left(\frac{3\sqrt2}{5}+\frac{C}{\sqrt{t_n}}\right)b_n
                 +C(1+n).
\]
Set $q_*=3\sqrt2/5<1$ and choose $q\in(q_*,1)$.
Taking $t_0$ so that $C/\sqrt{t_0}\leq q-q_*$ gives
$b_{n+1}\leq qb_n+C(1+n)$.
Induction gives
\[
 b_n\leq q^nb_0+C\sum_{k=0}^{n-1}q^{n-1-k}(1+k)
       \leq b_0+\frac{C(1+n)}{1-q}.
\]
For $t_n\leq t<t_{n+1}$ one has
$\sqrt t\,M(t)\leq\sqrt2\,b_n$ and
$n\leq\log(t/t_0)/\log2$. Hence $M(t)\leq C\log(2+t)/\sqrt{1+t}$ outside
$[0,t_0]$, while $M(t)\leq M(0)$ covers that interval.
Moreover
$\int_1^\infty (1+\log t)t^{-3/2}\,\dd t<\infty$,
as follows by integration by parts for the logarithmic term. All choices are uniform in $p$ because the constants in $b_{n+1}\leq qb_n+C(1+n)$,
together with $M(0)$, $|J(0)|$ and $J_+(0)$, are bounded uniformly on $\mathcal P$.

The bound $M(t)\leq C\log(2+t)/\sqrt{1+t}$ implies
$\int_0^\infty M(s)/(1+s)\,\dd s<\infty$, uniformly in $p$.
Integrating \eqref{eq:J-derivative} and using
$0\leq J_+(t)-J(t)\leq C(1+t)^{-3}$ from
\eqref{eq:negative-moment}, we obtain
$J_+(t)\leq |J(0)|+L\int_0^t M(s)/(1+s)\,\dd s+C\leq C$.
Substituting $J_+(t)\leq C$ into \eqref{eq:mass-dyadic} at $t=t_n$ and
multiplying by $\sqrt{2t_n}$ gives
\[
 b_{n+1}\le \left(\frac{3\sqrt2}{5}+\frac{C}{\sqrt{t_n}}\right)b_n+C.
\]
Increase $t_0$ if necessary so that the coefficient in parentheses is
at most a fixed $q<1$ for every $n$. The recurrence becomes
$b_{n+1}\le qb_n+C$, whose iteration gives $\sup_{n\geq0}b_n<\infty$.
Monotonicity of $M$ in \eqref{eq:M-derivative}, together with $\sup_{n\geq0}b_n<\infty$,
gives $M(t)\le C(1+t)^{-1/2}$ for $t\geq t_0$;
$M(t)\leq M(0)$ covers $0\leq t<t_0$.
\end{proof}

\begin{lemma}\label{lem:survival}
For either $Y$ defined in
Lemma~\ref{lem:Markov-estimates}, with $Y_0=0$ and $X_0=x\in\Z^d$,
fix $r>0$ and set
$\tau=\inf\{s\geq0:r+Y_s\leq0\}$. Then
\begin{equation}\label{eq:survival}
 \mathbb P_x(\tau>t)\leq
       \min\{1,C(1+r)/\sqrt{1+t}\}\text{ for }t\geq0.
\end{equation}
\end{lemma}
\begin{proof}
Lemma~\ref{lem:spectral} and \eqref{eq:spectral-corrector-formulas} bound
the jump rates and $\chi_p^\pm$ uniformly.
Thus the jumps of $Y$ are uniformly bounded, and its derivative between jumps is $h_p$ for $Q_p$ and $-h_p$ for $\widehat Q_p$. At an exit between jumps,
$r+Y_\tau=0$, while the absolute overshoot at a jump exit is bounded
by the maximum absolute jump of $Y$;
thus $r+Y_\tau\in[-B,0]$ on $\{\tau<\infty\}$ for a uniform $B$.
The exit time $\tau$ is a stopping time for the completed,
right-continuous clock filtration. On each bounded interval the path of $r+Y$ has
finitely many affine segments. Its first continuous crossing or jump
into $(-\infty,0]$ is determined by the path observed up to that time.
Choose a common clock rate
\[
 \Lambda>
 \max_{\substack{p\in\mathcal P\\x\in\mathcal C}}
 \max\left\{\sum_{\substack{z\in\Z^d\\|z|_1=1}}q_p(x,z),
             \sum_{\substack{z\in\Z^d\\|z|_1=1}}\widehat q_p(x,z)\right\}.
\]
Let $N_S$ count the rings of this rate-$\Lambda$ Poisson clock on
$[0,S]$. With $B$ as above and
$H_*=\max_{p\in\mathcal P}H(p)$, we have
\[
 \sup_{0\leq v\leq S}|Y_v|\leq BN_S+H_*S,\qquad
 \mathbb E N_S^2=\Lambda S+(\Lambda S)^2.
\]
In particular, for each $s\geq0$,
$\mathbb E_x\sup_{0\leq v\leq s+1}|Y_v|^2<\infty$.
Set $\sigma_n=2^{-n}\lceil2^n(\tau\wedge s)\rceil$.
These times are bounded by $s+1$ and decrease to $\tau\wedge s$.
They are grid stopping times because
$\{\sigma_n\leq k2^{-n}\}=\{\tau\wedge s\leq k2^{-n}\}$.
Write $Y_{\sigma_n}$ as a finite sum of increments
$Y_{(j+1)2^{-n}}-Y_{j2^{-n}}$, each multiplied by
$\1_{\{\sigma_n>j2^{-n}\}}$.
The event $\{\sigma_n>j2^{-n}\}$ is measurable at time $j2^{-n}$,
and the increment has conditional mean zero at that time.
Hence $\mathbb E_xY_{\sigma_n}=0$.
Right-continuity and domination by the integrable supremum give
$\mathbb E_xY_{\tau\wedge s}=0$.

Splitting $\mathbb E_x[r+Y_{\tau\wedge s}]=r$ over
$\{\tau>s\}$ and $\{\tau\le s\}$, and using
$r+Y_\tau\ge-B$ on $\{\tau\le s\}$, yields
\[
 \mathbb E_x[(r+Y_s)\1_{\{\tau>s\}}]\leq r+B.
\]
On $\{\tau>s\}$, the height $r+Y_s$ is positive.
Conditional on the path up to time $s$, the increments
$Y_{s+\rho}-Y_s$ have the law of $Y_\rho$ for the same generator
started at $X_s$.
Apply \eqref{eq:crossing} with $D=0$, its parameter $r$ replaced by
$r+Y_s$, and duration $s$.
The conditional probability of $r+Y_{2s}>0$ is at most
$3/5+C(1+r+Y_s)/\sqrt s$.
Since $\{\tau>2s\}\subset\{\tau>s,\ r+Y_{2s}>0\}$,
this also bounds the conditional probability of surviving to $2s$.
Write $h(s)=\mathbb P_x(\tau>s)$.
Integrating $3/5+C(1+r+Y_s)/\sqrt s$ over $\{\tau>s\}$ and using
$\mathbb E_x[(r+Y_s)\1_{\{\tau>s\}}]\leq r+B$ gives
\[
 h(2s)\leq\frac35h(s)+\frac{C(1+r+B)}{\sqrt s}.
\]
Choose $s_0\geq1$, independent of $p$, the initial residue and $r$,
so that the preceding recurrence holds for $s\geq s_0$, and set
$s_n=2^ns_0$ and $b_n=\sqrt{s_n}\,h(s_n)$ for $n=0,1,\ldots$. Multiplying
$h(2s_n)\leq(3/5)h(s_n)+C(1+r+B)/\sqrt{s_n}$ by
$\sqrt{2s_n}$ yields
\[
 b_{n+1}\le \frac{3\sqrt2}{5}b_n+C(1+r+B).
\]
Since $b_0\leq\sqrt{s_0}$ and $3\sqrt2/5<1$, iteration gives
$b_n\leq C(1+r+B)$ for every $n$.
For $s_n\leq t<s_{n+1}$, monotonicity gives
$\sqrt t\,h(t)\leq\sqrt2\,b_n$.
Absorb the uniform $B$ into $C$ and use $h(t)\leq1$ for $0\leq t<s_0$.
These bounds give \eqref{eq:survival}.
\end{proof}

\begin{lemma}\label{lem:killed-kernel-upper}
Fix $p\in\mathcal P$, put $h=H(p)$, and fix $b\in\R$. Define
\begin{equation}\label{eq:normal-coordinate}
 d_{p,b}(r,z):=p\cdot z-hr+b.
\end{equation}
For $0\le s<t$, let $(X_\rho)_{\rho\ge0}$ be the continuous-time Markov
chain with generator $Q_p$, started at $x\in\Z^d$, and put $T=t-s$.
Define
\[
 \tau_{s,t}^{x}:=\inf\{0\le\rho\le T:
        d_{p,b}(t-\rho,X_\rho)\le0\},\qquad\inf\varnothing=+\infty
\]
and, for $y\in\Z^d$, put
\[
 \mathsf K_p(s,t;x,y)
 :=\mathbb P_x\{X_T=y,\ \tau_{s,t}^{x}>T\}.
\]
For $f\in\ell^\infty(\Z^d)$, the kernel construction below gives a bounded
function $v$. For each $z$, $v(\cdot,z)$ is locally absolutely continuous
while $d_{p,b}(r,z)>0$ and satisfies $\partial_r v=Q_pv$ for almost every
such $r$. It satisfies $v(s,z)=f(z)$ when $d_{p,b}(s,z)>0$ and
$v(r,z)=0$ when $d_{p,b}(r,z)\le0$, for $s\le r\le t$.
Here $x$ indexes $v(t,\cdot)$ and $y$ indexes $v(s,\cdot)$, and
\[
 v(t,x)=\sum_{\substack{y\in\Z^d\\d_{p,b}(s,y)>0}}
              \mathsf K_p(s,t;x,y)f(y).
\] There is $C\ge1$, uniform for
$p\in\mathcal P$ and $b\in\R$, such that
\begin{equation}\label{eq:killed-kernel-upper}
 \mathsf K_p(s,t;x,y)
 \le C\frac{(1+d_{p,b}(t,x))_+(1+d_{p,b}(s,y))_+}
              {(1+t-s)^{d/2+1}}
\end{equation}
whenever $d_{p,b}(t,x)>0$ and $d_{p,b}(s,y)>0$. For a $\Gamma$-periodic function $\chi:\Z^d\to\R$ with
$\|\chi\|_\infty\leq C_0$,
put $d^\chi_{p,b}(r,z)=d_{p,b}(r,z)+\chi(z)$ and define
$\mathsf K_{p,\chi}$ by killing when $d^\chi_{p,b}\le0$.
Whenever $d^\chi_{p,b}(t,x)>0$ and $d^\chi_{p,b}(s,y)>0$,
\begin{equation}\label{eq:corrected-killed-kernel-upper}
 \mathsf K_{p,\chi}(s,t;x,y)
 \le C(C_0)\frac{(1+d^\chi_{p,b}(t,x))(1+d^\chi_{p,b}(s,y))}
 {(1+t-s)^{d/2+1}}.
\end{equation}
\end{lemma}
\begin{proof}
We prove \eqref{eq:corrected-killed-kernel-upper} and then set $\chi=0$.
Throughout the proof, $C$ may depend on $C_0$ but is uniform in
$p\in\mathcal P$ and $b\in\R$.
Assume $T\ge3$ and split backward time into three intervals of length
$T/3$; the case $0<T<3$ is treated below.

The representation and composition law follow first in a finite box,
where the domain has finitely many vertex deletion times.
Between successive deletion times, restrict $Q_p$ to the current vertices,
retaining its original diagonal entries, so a jump outside the current
vertices kills the path. The finite matrix exponential has a clock expansion. Conditioning on the number of rings and their
ordered times retains precisely the paths that stay in the domain.
Multiplication by the deletion projections gives the probability of reaching
$y$ from $x$ along a path confined to the box and satisfying
$d^\chi_{p,b}(t-\rho,X_\rho)>0$ for $0\leq\rho\leq t-s$.
As the box increases, the retained paths form an increasing family.
Every finite path lies in some box, since the clock has finitely many rings
on $[0,T]$. Monotone convergence gives $\mathsf K_{p,\chi}(s,t;x,y)$.
For real $f\in\ell^\infty(\Z^d)$, sum separately against $f_+$ and $(-f)_+$;
both sums are bounded by the probability mass times $\|f\|_\infty$.
Splitting paths at deterministic times gives the composition law by
conditional independence of the future clock rings and marks.
In each finite box the kernel acts on $f$ by solving the equation between
deletion times and projecting to zero at deletion. Fix a vertex $z$ and
a compact time interval on which $d^\chi_{p,b}(r,z)>0$.
The bound $|v|\le\|f\|_\infty$ and the finite number of terms in the
$z$th row of $Q_p$ permit dominated convergence in its integral equation
as the box increases. Thus $v(\cdot,z)$ is absolutely continuous on that
interval and $\partial_rv=Q_pv$ holds there almost everywhere.
At a deterministic $r$ with $d^\chi_{p,b}(r,z)=0$, the value at $z$ is zero.

For the interval $[t-T/3,t]$, start the $Q_p$-walk at $x$ and associate
elapsed time $\rho$ with physical time $t-\rho$.
Write $a_x=d^\chi_{p,b}(t,x)>0$ and
$C_-=\|\chi-\chi_p^-\|_\infty$. By \eqref{eq:scalar-correctors}, along
backward time $\rho$,
\[
 d^\chi_{p,b}(t-\rho,X_\rho)
 =a_x+Y_\rho+
    (\chi-\chi_p^-)(X_\rho)-(\chi-\chi_p^-)(x).
\]
Since $|(\chi-\chi_p^-)(X_\rho)-(\chi-\chi_p^-)(x)|\le2C_-$,
survival throughout this piece implies $a_x+2C_-+Y_\rho>0$
for every $0\leq\rho\leq T/3$.

Apply Lemma~\ref{lem:survival} with starting height $a_x+2C_-$
and duration $T/3$. Since
$C_-\le C_0+\sup_{p\in\mathcal P}\|\chi_p^-\|_\infty$, this gives
\[
 \sum_{z\in\Z^d}\mathsf K_{p,\chi}(t-T/3,t;x,z)
 \le \frac{C(1+a_x+2C_-)}{\sqrt{1+T/3}}
 \le C(1+a_x)T^{-1/2}.
\]

For the interval $[s,s+T/3]$, reverse the path and start the
$\widehat Q_p$-walk at $y$ at physical time $s$.
The total physical jump rates agree at each vertex:
\begin{equation}\label{eq:equal-exit-rates}
 \Lambda_p(z):=\sum_{\substack{\zeta\in\Z^d\\|\zeta|_1=1}}q_p(z,\zeta)
       =h+2d-a(z)=\sum_{\substack{\zeta\in\Z^d\\|\zeta|_1=1}}\widehat q_p(z,\zeta).
\end{equation}
The right and left eigenvector equations in \eqref{eq:cell} give
these equalities. A $Q_p$-path through $z_0,\ldots,z_m\in\Z^d$ with
ordered jump times $0=\rho_0<\rho_1<\cdots<\rho_m<\rho_{m+1}=T$
has density, with respect to $\dd\rho_1\cdots\dd\rho_m$,
\[
 \exp\!\left[-\sum_{i=0}^m
        \Lambda_p(z_i)(\rho_{i+1}-\rho_i)\right]
       \prod_{i=1}^m q_p(z_{i-1},z_i-z_{i-1}).
\]
Reversal preserves the exponential in this density by \eqref{eq:equal-exit-rates};
by \eqref{eq:reverse-edge}, the product of the $q_p$ jump rates
divided by the product of the reversed $\widehat q_p$ jump rates is
$\pi_p(z_m)/\pi_p(z_0)$. Integrating over the ordered jump times and
summing over finite nearest-neighbor paths, including the path with no
jump, gives
\[
 \pi_p(x)\mathsf K_{p,\chi}(s,t;x,y)
   =\pi_p(y)\widehat{\mathsf K}^{\,\to}_{p,\chi}(s,t;y,x),
\]
where $\widehat{\mathsf K}^{\,\to}_{p,\chi}$ is the kernel of the
$\widehat Q_p$-walk run from physical time $s$ to $t$, killed when
$d^\chi_{p,b}(s+\rho,X_\rho)\le0$.
For fixed $p$, $b$ and $\chi$, a site $z$ lies on the boundary at
physical time $(p\cdot z+b+\chi(z))/h$. These times form a countable
deterministic set. With probability one, no jump occurs at any of these times.
On that event, reversal preserves survival both during the holding intervals and at the jumps.
Bounded rates give only finitely many jumps on each bounded interval. Thus the path-reversal identity applies to the killed
kernels as stated.

At the earlier endpoint put $a_y=d^\chi_{p,b}(s,y)>0$ and
$C_+=\|\chi-\chi_p^+\|_\infty$. By \eqref{eq:scalar-correctors},
$d^\chi_{p,b}(s+\rho,X_\rho)$ equals
$a_y+Y_\rho+(\chi-\chi_p^+)(X_\rho)-(\chi-\chi_p^+)(y)$, where
\[
 Y_\rho=p\cdot(X_\rho-y)+\chi_p^+(X_\rho)-\chi_p^+(y)-h\rho
\]
is a martingale for the $\widehat Q_p$-walk started at $y$. The difference
$(\chi-\chi_p^+)(X_\rho)-(\chi-\chi_p^+)(y)$ has absolute value
at most $2C_+$, so if $d^\chi_{p,b}(s+\rho,X_\rho)>0$ for every $0\le\rho\le T/3$, then
$a_y+2C_++Y_\rho>0$ for $0\le\rho\le T/3$.
For this interval, the reversal identity reads
$\mathsf K_{p,\chi}(s,s+T/3;z,y)
=[\pi_p(y)/\pi_p(z)]\widehat{\mathsf K}^{\,\to}_{p,\chi}(s,s+T/3;y,z)$.
The finite-cell bounds in Lemma~\ref{lem:spectral} bound
$\pi_p(y)/\pi_p(z)$ uniformly in $p,y,z$.
Summing over $z$ and applying Lemma~\ref{lem:survival} to the
$\widehat Q_p$-walk with height $a_y+2C_+$ and duration $T/3$ gives
\[
 \sum_{z\in\Z^d}\mathsf K_{p,\chi}(s,s+T/3;z,y)
 \le C(1+a_y)T^{-1/2}.
\]
Set $r_1=s+T/3$ and $r_2=s+2T/3$.
The composition law, positivity and the whole-space kernel bound
in Lemma~\ref{lem:Markov-estimates} give
\[
 \begin{aligned}
 \mathsf K_{p,\chi}(s,t;x,y)
 &=\sum_{z_2\in\Z^d}\sum_{z_1\in\Z^d}\mathsf K_{p,\chi}(r_2,t;x,z_2)
       \mathsf K_{p,\chi}(r_1,r_2;z_2,z_1)
       \mathsf K_{p,\chi}(s,r_1;z_1,y)\\
 &\leq CT^{-d/2}
   \left(\sum_{z_2\in\Z^d}\mathsf K_{p,\chi}(r_2,t;x,z_2)\right)
   \left(\sum_{z_1\in\Z^d}\mathsf K_{p,\chi}(s,r_1;z_1,y)\right).
 \end{aligned}
\]
The row sum is at most $C(1+a_x)T^{-1/2}$ and the column sum at most
$C(1+a_y)T^{-1/2}$, as proved with starting heights $a_x+2C_-$ and
$a_y+2C_+$. Their product gives
$\mathsf K_{p,\chi}(s,t;x,y)\leq C(1+a_x)(1+a_y)T^{-d/2-1}$.
Since $T\geq3$, $1+T\leq4T/3$, yielding
\eqref{eq:corrected-killed-kernel-upper}; setting $\chi=0$ gives
\eqref{eq:killed-kernel-upper}.
All sums converge because $\sum_{z_2\in\Z^d}\mathsf K_{p,\chi}(r_2,t;x,z_2)\le1$ and
$\sum_{z_1\in\Z^d}\mathsf K_{p,\chi}(s,r_1;z_1,y)$ is bounded by $\max_{r\in\mathcal C}\pi_p(r)/\min_{r\in\mathcal C}\pi_p(r)$ under reversal.
For $0<T<3$, $\mathsf K_{p,\chi}\leq1$,
$(1+a_x)(1+a_y)\geq1$, and $(1+T)^{d/2+1}\leq4^{d/2+1}$.
Thus \eqref{eq:corrected-killed-kernel-upper} holds with $C\geq4^{d/2+1}$.
The parameter $b$ does not enter the bounds for the rates or $\chi_p^\pm$, so $C$ is independent of $b$.

The dependence on $\chi$ enters through
$C_\pm=\|\chi-\chi_p^\pm\|_\infty
\le C_0+\sup_{p\in\mathcal P}\|\chi_p^\pm\|_\infty$.
The last two suprema are finite by
\eqref{eq:spectral-corrector-formulas} and compactness of $\mathcal P$,
so the constants depend on $\chi$ only through $C_0$.
\end{proof}

\begin{remark}
In \eqref{eq:corrected-killed-kernel-upper}, the numerator is
$(1+d^\chi_{p,b}(t,x))(1+d^\chi_{p,b}(s,y))$, with killing when $d^\chi_{p,b}\le0$;
$d^\chi_{p,b}$ may be positive when $d_{p,b}<-1$.
\end{remark}

\begin{proof}[Proof of Theorem~\ref{thm:upper}]
Estimate \eqref{eq:mass-main} is
Lemma~\ref{lem:mass-decay} written using \eqref{eq:weighted-variables}.
For \eqref{eq:upper-main}, choose $t_0\geq4$ sufficiently large,
independently of $p$. Fix $p\in\mathcal P$, $t\geq t_0$ and $x$ with
$z_0=p\cdot x-ht+\kappa_d\log t\geq0$, and put
$C_0=1+\max_{p\in\mathcal P,\,x\in\mathcal C}|\chi_p^-(x)|$.
Keep the terminal time $t$, and hence $\kappa_d\log t$, fixed along
the backward $Q_p$-walk from $(t,x)$. Let $\tau$ be its first exit time from
\begin{equation}\label{eq:stopping-domain}
 p\cdot X_s-h(t-s)+\chi^-(X_s)+\kappa_d\log t+C_0>0.
\end{equation}
Stop this walk at $\tau\wedge(t/2)$.
For this $Q_p$-walk, write
\[
 Y_s=p\cdot(X_s-x)+\chi^-(X_s)-\chi^-(x)+hs,\qquad
 r=z_0+\chi^-(x)+C_0\in[1,z_0+2C_0].
\]
Equation \eqref{eq:scalar-correctors} makes $Y_s$
a martingale with $Y_0=0$, and the left side of
\eqref{eq:stopping-domain} is exactly $r+Y_s$.

The bound \eqref{eq:global-exponential-bound} places $w$ in
$\ell^\infty(\Z^d)$ on each bounded time interval. Since
\eqref{eq:weighted-evolutions} gives $w_t\leq Q_pw$,
Lemma~\ref{lem:fixed-domain-comparison} yields
$w(s,\cdot)\le P_s^{Q_p}w(0,\cdot)$, hence
$\|w(s,\cdot)\|_\infty\le\|w(0,\cdot)\|_\infty$.
For $f(\rho,z)=w(t-\rho,z)$,
\[
 (\partial_\rho+Q_p)f(\rho,z)
       =-w_t(t-\rho,z)+Q_pw(t-\rho,z)\ge0.
\]
For each nearest-neighbor increment $b$, let $N_b(\rho)$ count
the jumps of $X$ with that increment up to time $\rho$.
The compensated process
$N_b(\rho)-\int_0^\rho q_p(X_v,b)\,\dd v$ is a martingale.
The rates and $f$ are bounded on $[0,t/2]\times\Z^d$.
The identity $w_t=Q_pw-E\mathcal N$ and
$0\leq E\mathcal N\leq a_{\max}w$ also bound $w_t(s,z)$ for
$s\in[t/2,t]$. Thus $\partial_\rho f(\rho,z)$ is bounded for
$0\leq\rho\leq t/2$. Evaluating each jump increment at its pre-jump
state, summing the compensated increments and integrating the time
derivative gives the martingale
\[
 f(\rho,X_\rho)-f(0,x)
 -\int_0^\rho(\partial_v+Q_p)f(v,X_v)\,\dd v.
\]
The nonnegative integrand makes $w(t-\rho,X_\rho)$ a submartingale,
bounded by $\|w(0,\cdot)\|_\infty$ for $0\leq\rho\leq t/2$.
The time $\sigma=\tau\wedge(t/2)$ is a bounded stopping time for the
completed, right-continuous clock filtration, so optional stopping gives
\[
 w(t,x)\leq\mathbb E_x[w(t-\sigma,X_\sigma)].
\] To estimate $w(t-\tau,X_\tau)$ on $\{\tau\le t/2\}$, define
\[
 B=\max_{\substack{p\in\mathcal P,\ y\in\mathcal C\\
                         z\in\Z^d,\ |z|_1=1}}
   \left|p\cdot z+\chi_p^-(y+z)-\chi_p^-(y)\right|<\infty.
\]
Finiteness follows from compactness of $\mathcal P$ and the uniform
bounds for $\chi_p^-$. At a continuous exit, $r+Y_s$ is zero;
at a jump exit its decrease is at most $B$. Thus, on
$\{\tau\leq t/2\}$, putting $s=\tau$ gives $r+Y_s\in[-B,0]$.
At this stopping point,
$p\cdot X_s-h(t-s)
=r+Y_s-\chi^-(X_s)-\kappa_d\log t-C_0
\leq-\kappa_d\log t+C$.
Here $C$ is uniform because $r+Y_s\leq0$ and $\chi^-$ is uniformly bounded.
Using $u\leq1$ and the uniform positive lower bound for $\phi_p$ in
\eqref{eq:weighted-variables} gives $w(t-s,X_s)\leq Ct^{-\kappa_d}$.
For the contribution from $\{\tau>t/2\}$, we need a bound on
$w(t/2,\cdot)$ at every site.
By \eqref{eq:weighted-evolutions} and comparison,
$w(t/2,\cdot)\leq P_{t/4}^{Q_p}w(t/4,\cdot)$.
Lemma~\ref{lem:Markov-estimates} bounds
$P_{t/4}^{Q_p}(x,y)\leq Ct^{-d/2}$.
Also, \eqref{eq:weighted-variables} gives
$\sum_{y\in\Z^d}w(t/4,y)\leq
(\min_{y\in\mathcal C}\pi_p(y))^{-1}M(t/4)$,
with a uniform constant. Hence Lemma~\ref{lem:mass-decay} yields
\[
 \|w(t/2,\cdot)\|_\infty
 \leq Ct^{-d/2}\sum_{y\in\Z^d}w(t/4,y)
 \leq Ct^{-d/2}M(t/4)\leq Ct^{-(d+1)/2}.
\]
Split the stopped expectation over $\{\tau\leq t/2\}$ and
$\{\tau>t/2\}$. Applying Lemma~\ref{lem:survival} with height $r$
and duration $t/2$, and using $1+r\leq C(1+z_0)$, gives
\[
 w(t,x)\leq Ct^{-\kappa_d}
       +Ct^{-(d+1)/2}\min\{1,C(1+z_0)t^{-1/2}\}
       \leq C(1+z_0)t^{-\kappa_d}.
\]
Since $\ee^{-p\cdot x+ht}=t^{\kappa_d}\ee^{-z_0}$, multiplication
by $\phi(x)\ee^{-p\cdot x+ht}$ in \eqref{eq:weighted-variables}
gives \eqref{eq:upper-main} for $t\geq t_0$, using the uniform upper bound for $\phi$.
The constants and $t_0$ are uniform for $p\in\mathcal P$, hence for $e\in\mathbb S^{d-1}$.

To prove \eqref{eq:ray-upper}, take $p=p_e$.
Lemma~\ref{lem:geometry} gives $H(p_e)=p_e\cdot v_e$; using
$v_e=w(e)e$ and $\alpha_e=p_e\cdot e$ from \eqref{eq:pe}, we obtain
\[
 z_e(t,x)=\alpha_e(r-w(e)t)+p_e\cdot(x-re)+\kappa_d\log t.
\]
Write $\alpha_*:=\min_{e\in\mathbb S^{d-1}}\alpha_e>0$ and
$P_*:=\max_{e\in\mathbb S^{d-1}}|p_e|<\infty$, as provided by Lemma~\ref{lem:geometry},
and let $C_U$ be the constant in \eqref{eq:upper-main}.
Given $\theta\in(0,1)$ and $\rho\ge0$, choose $Z\ge1$ so that
$C_U(1+Z)\ee^{-Z}<\theta$ and set
$C_{\theta,\rho}=(Z+P_*\rho)/\alpha_*$.
If $t\geq t_0$, $r\ge0$, $|x-re|\le\rho$, and
$r\ge w(e)t-\kappa_d\log t/\alpha_e+C_{\theta,\rho}$, then
\[
 z_e(t,x)\ge\alpha_e C_{\theta,\rho}-P_*\rho\ge Z.
\]
Since $(1+z)\ee^{-z}$ is decreasing for $z\ge0$,
\eqref{eq:upper-main} gives
$u(t,x)\le C_U(1+Z)\ee^{-Z}<\theta$.

For \eqref{eq:outer-body}, choose $C_\theta\geq1$ so that
$C_U(1+C_\theta)\ee^{-C_\theta}<\theta$.
Fix $t\geq t_0$ and $x\in\Z^d$ with $u(t,x)\geq\theta$.
For each $e$, if
$p_e\cdot x-tH(p_e)+\kappa_d\log t>C_\theta$,
then \eqref{eq:upper-main} would give $u(t,x)<\theta$.
Hence
$p_e\cdot x\le tH(p_e)-\kappa_d\log t+C_\theta$ for every $e$; taking the
intersection over $e\in\mathbb S^{d-1}$ gives \eqref{eq:outer-body}.
\end{proof}

\section{The lower bound for transverse periodic growth}\label{sec:lower}
\subsection{ {Transverse kernel and scalar subsolutions}}\label{sec:scalar-lower}
Throughout this section, assume \eqref{eq:transverse-linearization} and
write $n=d-1$, $\beta=n/2$ and $a=a_{\mathrm{eff}}$  {(we use the shorthand $a=a_{\mathrm{eff}}$ in the scalar one-dimensional equations below, while writing $a_{\mathrm{eff}}$ or $F_s(x,0)$ whenever comparing with the site-dependent growth rate)}.
Choose
\[
 a_{\max}=\max_{y\in\mathcal C_\perp}a_\perp(y),\qquad
 K=\max\{a_{\max},K_F\}.
\]
By \eqref{eq:transverse-eigenpair}, $a\leq a_{\max}$,
so the choice of $K$ gives $a/K\leq1$. In view of
\eqref{eq:quadratic-kpp-loss}, we also have
$F((j,y),s)\geq a_\perp(y)s-Ks^2$ for $(j,y)\in\Z\times\Z^n$
and $s\in[0,1]$.
Following \cite[Lemma~6.1]{HuAxial}, we multiply a solution of
\eqref{eq:scalar-logistic} by a transverse factor satisfying
\eqref{eq:transverse-factor-equation}.
Define $(\lambda_*,c_*)$, $D_*$ and $\delta_*$ by
\begin{equation}\label{eq:critical-system}
 c_*\lambda_*=2(\cosh\lambda_*-1)+a,\qquad
 c_*=2\sinh\lambda_*,\qquad D_*=\cosh\lambda_*,\quad
 \delta_*=\frac{c_*}{4D_*}\in(0,1/2).
\end{equation}

\begin{lemma}\label{lem:longitudinal}
For $m\geq1$, let $\pheat_m$ be the heat kernel of $\Delta_{\Z^m}$,
with $\pheat_m(0,y)=\1_{\{0\}}(y)$.
For $t>0$ and $y\in\Z^n$, $\pheat_n(t,y)$ is positive and bounded by
$C(1+t)^{-n/2}$. For every fixed $A\geq0$, as $t\to\infty$,
\begin{equation}\label{eq:heat-local}
 \pheat_n(t,y)=(4\pi t)^{-n/2}\ee^{-|y|^2/(4t)}
                 +O_A(t^{-(n+1)/2})\text{ for }|y|\leq A\sqrt t.
\end{equation}
Let $G(t,j)$ be the probability mass at $j$ of the difference of independent
Poisson variables with parameters $\ee^{\lambda_*}t$ and
$\ee^{-\lambda_*}t$. Put $w_1(t,j)=G(t,j-1)-G(t,j+1)$.
For every fixed $A\geq0$, as $t\to\infty$, uniformly for
$j\in\Z$ with $s=j-c_*t$ and $|s|\leq A\sqrt t$,
\begin{equation}\label{eq:signed-kernel}
 w_1(t,j)=\frac{(s+\delta_*)\ee^{-s^2/(4D_*t)}}
                  {\sqrt{4\pi}D_*^{3/2}t^{3/2}}
              +O_A((1+|s|)t^{-2}).
\end{equation}
The function $Z(t,j)=\ee^{-\lambda_*(j-c_*t)}w_1(t,j)$ solves
$Z_t=\Delta_{\Z}Z+aZ$.
\end{lemma}
\begin{proof}

For the estimates below, see also \cite[Lemmas~3.3--3.4]{HuAxial}.
We use the Fourier convention
$\widehat f(\xi)=\sum_{y\in\Z^n}f(y)\ee^{-\ii y\cdot\xi}$.
The independent rate-one Poisson processes in each positive and
negative coordinate direction give
\[
 \pheat_n(t,y)=\frac1{(2\pi)^n}\int_{[-\pi,\pi]^n}
       \ee^{\ii y\cdot\xi+ta_n(\xi)}\,\dd\xi,
 \qquad
 a_n(\xi)=-4\sum_{i=1}^n\sin^2(\xi_i/2).
\]
The probability masses have absolute sum one. Termwise integration
of their Fourier series therefore gives the inversion formula. Since
$|\sin(\xi_i/2)|\geq|\xi_i|/\pi$ on $[-\pi,\pi]$, the bound
$a_n(\xi)\leq-4|\xi|^2/\pi^2$ holds on the whole integration cube.
In particular, all the Fourier integrals below have integrable Gaussian
majorants for $t\geq1$. For $m\geq0$ and $c>0$, the substitution
$\eta=\sqrt t\,\xi$ gives
\[
 \int_{\R^n}|\xi|^m \ee^{-ct|\xi|^2}\,\dd\xi
 =t^{-(n+m)/2}\int_{\R^n}|\eta|^m \ee^{-c|\eta|^2}\,\dd\eta<\infty.
\]
The last integral is finite by polar coordinates.
The Fourier symbol of $\Delta_{\Z^n}$ is
$a_n(\xi)=2\sum_{i=1}^n(\cos\xi_i-1)$ for $\xi\in[-\pi,\pi]^n$.
Near zero it has expansion $-|\xi|^2+O(|\xi|^4)$, and it is strictly
negative elsewhere. Choose $\delta>0$ so that the symbol is at most
$-c|\xi|^2$ on $|\xi|\le\delta$ and at most $-c_\delta<0$ on
$\{\xi\in[-\pi,\pi]^n:|\xi|>\delta\}$.
The inverse Fourier integral over $|\xi|\le\delta$ is therefore bounded by
$C\int_{\R^n}\ee^{-ct|\xi|^2} {\,\dd\xi}=Ct^{-n/2}$. The complementary part is $O(\ee^{-c_\delta t})$.
Together with $\pheat_n(t,y)\leq1$ for $0<t\leq1$, these bounds give
$\pheat_n(t,y)\leq C(1+t)^{-n/2}$.
For \eqref{eq:heat-local}, set $\xi=t^{-1/2}\eta$ on $|\xi|\leq\delta$. The rescaled symbol satisfies
$2\sum_{i=1}^n(\cos(t^{-1/2}\eta_i)-1)=-t^{-1}|\eta|^2+O(t^{-2}|\eta|^4)$.
Decrease $\delta$ so that
both $a_n(\xi)$ and $-|\xi|^2$ are at most $-c|\xi|^2$. By these quadratic bounds and the mean value theorem applied to
the exponential function, we obtain
\[
 |\ee^{ta_n(\xi)}-\ee^{-t|\xi|^2}|
       \le Ct|\xi|^4\ee^{-ct|\xi|^2}\text{ for }|\xi|\le\delta.
\]
Integrating the exponential difference gives $O(t^{-(n+2)/2})$.
The Gaussian Fourier integral over $\R^n$ equals
$(4\pi t)^{-n/2}\ee^{-|y|^2/(4t)}$. The Gaussian tail and the integral over
$\{\xi\in[-\pi,\pi]^n:|\xi|>\delta\}$ are exponentially small, which
proves \eqref{eq:heat-local} uniformly in $y$.
For positivity, represent each coordinate of the walk as the difference
of two independent Poisson processes. Every prescribed pair of
nonnegative counts has positive probability at positive time, so every
lattice site is reached with positive probability.

For the probability masses $G(t,j)$, define the Fourier symbol
\[
 h(\xi)=\ee^{\lambda_*}(\ee^{-\ii\xi}-1)
       +\ee^{-\lambda_*}(\ee^{\ii\xi}-1).
\]
By \eqref{eq:critical-system} and Taylor's formula,
\[
 h(\xi)+\ii c_*\xi=-D_*\xi^2+\ii c_*\xi^3/6+R_4(\xi),
 \qquad |R_4(\xi)|\leq C|\xi|^4
\]
for $\xi$ near zero.
The difference defining $w_1$ has multiplier $-2\ii\sin\xi$.
Put $g_t(s)=(4\pi D_*t)^{-1/2}\ee^{-s^2/(4D_*t)}$.
The terms $-2\ii\xi$ and
$(-2\ii\xi)(\ii c_*t\xi^3/6)$ give
$-2\partial_s g_t(s)+(c_*t/3)\partial_s^4g_t(s)$ in Fourier inversion.
Choose $\xi_0>0$ small enough that
$\operatorname{Re}h(\xi)\leq-c\xi^2$ for $|\xi|\leq\xi_0$.
We estimate the error before evaluating these derivatives.
On $|\xi|\leq t^{-2/5}$, put
$v=\ii c_*t\xi^3/6+tR_4(\xi)$. Then $|v|\leq Ct^{-1/5}$,
and $|\ee^v-1-v|\leq\frac12|v|^2\ee^{|v|}$ gives
\[
 \left|\ee^{t(h(\xi)+\ii c_*\xi)}
 -\ee^{-D_*t\xi^2}(1+\ii c_*t\xi^3/6)\right|
 \leq C {\ee}^{-ct\xi^2}(t|\xi|^4+t^2|\xi|^6).
\]
Multiplication by $-2\ii\sin\xi=-2\ii\xi+O(|\xi|^3)$
leaves an error bounded by
$C {\ee}^{-ct\xi^2}(|\xi|^3+t|\xi|^5+t^2|\xi|^7)$.
The Gaussian moment formula gives an integral bounded by $Ct^{-2}$.
On $t^{-2/5}\leq|\xi|\leq\xi_0$, the exact exponential and the
approximating Gaussian terms have integrals bounded by a polynomial in
$t$ times $\ee^{-ct^{1/5}}$. This bound is at most $Ct^{-2}$ for
$t\geq1$. The integral over $\xi_0\leq|\xi|\leq\pi$ is bounded by
$C {\ee}^{-c_0t}$ since $\operatorname{Re}h=2D_*(\cos\xi-1)$.

Differentiating the Gaussian Fourier integral is justified by
$|\xi|^m \ee^{-D_*t\xi^2}\in L^1(\R)$ for every fixed $m$.
Explicitly,
\[
 g_t'=-\frac{s}{2D_*t}g_t,\qquad
 g_t^{(4)}=\left(\frac{s^4}{16D_*^4t^4}
       -\frac{3s^2}{4D_*^3t^3}+\frac3{4D_*^2t^2}\right)g_t.
\]
Consequently, $-2g_t'(s)+(c_*t/3)g_t^{(4)}(s)$ equals
\[
 \frac{g_t(s)}{D_*t}
 \left(s+\frac{c_*}{4D_*}
       -\frac{c_*s^2}{4D_*^2t}
       +\frac{c_*s^4}{48D_*^3t^2}\right).
\]
For $|s|\leq A\sqrt t$, the terms
$-c_*s^2g_t(s)/(4D_*^3t^2)$ and
$c_*s^4g_t(s)/(48D_*^4t^3)$ in
$-2g_t'+(c_*t/3)g_t^{(4)}$ have absolute values at most
$C_A s^2t^{-5/2}$ and $C_A s^4t^{-7/2}$, respectively.
Since $s^2\leq A|s|\sqrt t$ and $s^4\leq A^3|s|t^{3/2}$,
both are bounded by $C_A(1+|s|)t^{-2}$, including at $s=0$.
Combining the Fourier remainder bounds on $|\xi|\leq t^{-2/5}$,
$t^{-2/5}\leq|\xi|\leq\xi_0$ and $\xi_0\leq|\xi|\leq\pi$
with the Gaussian derivatives proves \eqref{eq:signed-kernel}.
The difference $w_1(t,j)=G(t,j-1)-G(t,j+1)$ satisfies
\[
 (w_1)_t=\ee^{\lambda_*}w_1(t,j-1)
          +\ee^{-\lambda_*}w_1(t,j+1)-2D_*w_1(t,j).
\]
For $Z(t,j)=\ee^{-\lambda_*(j-c_*t)}w_1(t,j)$, the neighbor
coefficients become one and the diagonal coefficient becomes
$c_*\lambda_*-2D_*=a-2$. Hence $Z_t=\Delta_{\Z}Z+aZ$.
To express $Z$ in terms of $\pheat_1$, expand the two Poisson
probabilities defining $G$. For $j\in\Z$, each pair of counts with difference $j$ contributes
$\ee^{\lambda_*j}$ relative to the rate-one pair, giving
\[
 G(t,j)=\ee^{\lambda_*j-2(D_*-1)t}\pheat_1(t,j).
\]
Using \eqref{eq:critical-system} consequently gives
\[
 Z(t,j)=\ee^{at}\left(\ee^{-\lambda_*}\pheat_1(t,j-1)
                      -\ee^{\lambda_*}\pheat_1(t,j+1)\right).
\]
Since $0\leq\pheat_1\leq1$, $Z$ is bounded on every finite time interval.
At $t=0$, the heat-kernel identity gives
$Z(0,j)=\ee^{-\lambda_*}\1_{\{1\}}(j)-\ee^{\lambda_*}\1_{\{-1\}}(j)$ for $j\in\Z$.

\end{proof}

\begin{lemma}\label{lem:moving-halfline-comparison}
Let $B\in C^1([t_0,t_1])$ be strictly increasing on a finite interval
and $\mathcal L_t$ be the operator defined before
\eqref{eq:subsolution-definition}, with $m=1$ and measurable
nonnegative rates satisfying
$\sup_{t\in[t_0,t_1],\,j\in\Z}\sum_{z\in\{-1,1\}}c(t,j,z)<\infty$.
Suppose $q,\psi:[t_0,t_1]\times\Z\to\R$ are bounded, locally
absolutely continuous in time at each site, and satisfy
\[
 q_t=\mathcal L_tq+G(t,j,q),\qquad
 \psi_t\leq\mathcal L_t\psi+G(t,j,\psi)
\]
for almost every $t$ and each $j\leq B(t)$.
Assume that $G$ has the measurability specified before
\eqref{eq:comparison-evolution} and is uniformly Lipschitz in its last
variable on a bounded interval containing the ranges of $q$ and $\psi$.
If $\psi(t_0,j)\leq q(t_0,j)$ for $j\leq B(t_0)$ and
$\psi(t,j)\leq q(t,j)$ for $B(t)<j\leq B(t)+1$ and
$t\in[t_0,t_1]$, then $\psi(t,j)\leq q(t,j)$ for $j\leq B(t)$
and $t\in[t_0,t_1]$.
\end{lemma}
\begin{proof}
The set $\{j\in\Z:j\leq B(t)\}$ changes at most
$1+\lfloor B(t_1)\rfloor-\lfloor B(t_0)\rfloor$ times.
Fix a time interval between consecutive changes and denote its fixed
interior by $D\subset\Z$. The only exterior neighbor of $D$ is the first integer to the right of $B(t)$.
At each $j\in D$, set $h(t,j)=(\psi(t,j)-q(t,j))_+$, and set $h=0$ outside $D$.
Since $\psi(t,j)\leq q(t,j)$ for $B(t)<j\leq B(t)+1$,
the scalar positive-part chain rule gives
$h_t\leq\mathcal L_th+Lh$ almost everywhere, where $L$ is a
Lipschitz constant for $G$ on the stated common bounded range.
If $C_r$ bounds the total jump rate, integration at each site gives
\[
 h(t,j)\leq h(s,j)+(C_r+L)\int_s^t\sup_{k\in D}h(v,k)\,\dd v.
\]
The supremum is bounded and measurable because $D$ is countable.
Taking the supremum over $j\in D$ and applying Gronwall gives
$\psi(t,j)\leq q(t,j)$ for $j\in D$ up to the next crossing time.
If $B(t_*)=j$ and $j$ is newly admitted at $t_*$, then
$B(t)<j\leq B(t)+1$ for $t<t_*$ sufficiently close to $t_*$.
The assumed ordering on this exterior neighbor gives
$\psi(t,j)\leq q(t,j)$. Sitewise continuity at $t_*$ gives
$\psi(t,j)\leq q(t,j)$ at the crossing, both at that integer and on the old interior.
The ordering at the crossing times allows the finite
iteration, including a crossing at the terminal endpoint.

\end{proof}

\begin{lemma}\label{lem:transverse-kernel}
For $q\in\R^n$, define $B_{\perp,q}$ on functions on
$\mathcal C_\perp$, extended to $\Z^n$ with periods $L_2,\ldots,L_d$, by
\[
 (B_{\perp,q}f)(y)
 =\sum_{\substack{b\in\Z^n\\|b|_1=1}}
   \bigl(\ee^{-q\cdot b}f(y+b)-f(y)\bigr)+a_\perp(y)f(y),
 \text{ for }y\in\mathcal C_\perp.
\]
Let $H_\perp(q)$ be its principal eigenvalue and set
$\Sigma_\perp=\nabla^2H_\perp(0)$. Then $\Sigma_\perp$ is positive
definite and, for $p=(p_1,p')\in\R\times\R^n$,
\begin{equation}\label{eq:transverse-factorization}
 H(p_1,p')=2(\cosh p_1-1)+H_\perp(p'),\qquad H_\perp(0)=a.
\end{equation}
Let $K_\perp(s;y,z)$ be the kernel of
$\ee^{s\mathcal A_\perp}$ on $\ell^\infty(\Z^n)$, where
$\mathcal A_\perp=\Delta_{\Z^n}+a_\perp(y)$. The restriction of $\mathcal A_\perp$ to functions with periods
$L_2,\ldots,L_d$ is the cell operator in \eqref{eq:transverse-eigenpair}. For $f\in\ell^\infty(\Z^n)$,
$(\ee^{s\mathcal A_\perp}f)(y)=\sum_{z\in\Z^n}K_\perp(s;y,z)f(z)$.
 {The kernel $K_\perp(s;y,z)$ is strictly positive for $s>0$ and satisfies}
\begin{equation}\label{eq:transverse-kernel-upper}
 \ee^{-as}K_\perp(s;y,z)\leq C(1+s)^{-n/2}.
\end{equation}
For every fixed $A\geq0$, uniformly for $s\geq1$ and
$|y-z|\leq A\sqrt s$,
\begin{equation}\label{eq:transverse-kernel-local}
\begin{split}
 \ee^{-as}K_\perp(s;y,z)
 &=\frac{|\mathcal C_\perp|\psi_\perp(y)\psi_\perp(z)}
         {(2\pi s)^{n/2}\sqrt{\det\Sigma_\perp}}\,
   \exp\!\left(-\frac{(y-z)^{\mathsf T}\Sigma_\perp^{-1}(y-z)}{2s}\right)\\
 &\quad+O_A(s^{-(n+1)/2}).
\end{split}
\end{equation}
The constants in
\eqref{eq:transverse-kernel-upper}--\eqref{eq:transverse-kernel-local}
are uniform in $y,z\in\Z^n$.
\end{lemma}
\begin{proof}
A positive eigenvector of $B_{\perp,p'}$, extended independently of
the longitudinal coordinate, is a positive eigenvector of $B_p$ with
eigenvalue $2(\cosh p_1-1)+H_\perp(p')$. Perron--Frobenius theory gives
\eqref{eq:transverse-factorization}. Transposing the directed edge
terms gives $B_{\perp,q}^{\mathsf T}=B_{\perp,-q}$, hence
$H_\perp(q)=H_\perp(-q)$. For $\eta\in\R^n\setminus\{0\}$,
differentiating the factorization in the transverse variables gives
\[
 \eta^{\mathsf T}\nabla^2H_\perp(q)\eta
 =(0,\eta)^{\mathsf T}\nabla^2H(0,q)(0,\eta)>0.
\]
Lemma~\ref{lem:spectral} gives strict positivity for the nonzero vector
$(0,\eta)$, including when $n=1$. Taking $q=0$ proves positive
definiteness of $\Sigma_\perp$.

Define
\[
 q_\perp(r,b)=\frac{\psi_\perp(r+b)}{\psi_\perp(r)},\qquad
 Q_\perp f(r)=\sum_{\substack{b\in\Z^n\\|b|_1=1}}q_\perp(r,b)(f(r+b)-f(r)),\qquad
 \pi(r)=\psi_\perp(r)^2.
\]
All functions of the residue $r$ are extended periodically.
The eigenvector normalization gives $\sum_{r\in\mathcal C_\perp}\pi(r)=1$, and
$\pi(r)q_\perp(r,b)=\psi_\perp(r)\psi_\perp(r+b)
=\pi(r+b)q_\perp(r+b,-b)$. Thus $\pi$ is invariant and
$\sum_{r\in\mathcal C_\perp}\pi(r)\sum_{\substack{b\in\Z^n\\|b|_1=1}}q_\perp(r,b)b=0$, by pairing each directed
edge with its reverse. The rates are positive and bounded on the finite
cell. A Poisson clock with rate greater than their total exit rates
therefore constructs a nonexplosive walk with generator $Q_\perp$.
Write $P_\perp(s;y,z)$ for its transition probability from $y$ to $z$.
Conjugation of the generators gives
\[
 \ee^{-as}K_\perp(s;y,z)
 =\frac{\psi_\perp(y)}{\psi_\perp(z)}P_\perp(s;y,z).
\]
For $s>0$, choose a nearest-neighbor path of length $m$ from $y$ to
$z$. The event that the prescribed $m$ jumps occur in order before $s$, with no other jump,
has positive probability. On $0<t_1<\cdots<t_m<s$, its density is the positive product
of the $m$ rates and the exponential holding factors. For $m=0$, the probability of
no jump is positive. Hence $P_\perp(s;y,z)>0$ and $K_\perp(s;y,z)>0$.

Put $N=|\mathcal C_\perp|$, $r_y=y\bmod\prod_{i=2}^dL_i\Z$ and
$\mathcal B_\perp=\prod_{i=2}^d[-\pi/L_i,\pi/L_i]$. For real
$k\in\mathcal B_\perp$, define the finite matrix
\[
 (A(k)f)(r)=\sum_{\substack{b\in\Z^n\\|b|_1=1}}q_\perp(r,b)
                  \bigl(\ee^{\ii k\cdot b}f(r+b)-f(r)\bigr).
\]
The argument $r+b$ of $f$ is reduced to the cell; the displacement
$b$ is retained in the exponential, including for periods one and two.
With $r_z$ defined by the same residue map as $r_y$, the inverse Bloch formula is
\[
 P_\perp(s;y,z)=\frac{N}{(2\pi)^n}\int_{\mathcal B_\perp}
 \ee^{\ii k\cdot(y-z)}[\ee^{sA(k)}]_{r_y,r_z}\,\dd k.
\]
At $s=0$ the matrix entry is zero unless $r_y=r_z$; in that case
$y-z$ is a period vector, and integration of its phase gives
$ {\1}_{\{y=z\}}$. Differentiating the integral in $s$ is valid
because the matrices and their derivatives are bounded on the compact
integration domain on each finite time interval. Applying $Q_\perp$
in $y$ gives the same derivative. The integral is bounded uniformly
in $y$, so uniqueness of the bounded linear Cauchy problem proves the
formula. For the finite-cell Floquet representation, see also \cite{Kuchment}.

Let $D_\psi=\operatorname{diag}(\psi_\perp(r))$. Then
$A(k)=D_\psi^{-1}(B_{\perp,-\ii k}-aI)D_\psi$.
For real $k$, the matrix $B_{\perp,-\ii k}$ is Hermitian, since
reversing an edge conjugates its phase.
With $\langle f,g\rangle_\pi=\sum_{r\in\mathcal C_\perp}\pi(r)\overline{f(r)}g(r)$, the
identity $\pi(r)q_\perp(r,b)=\pi(r+b)q_\perp(r+b,-b)$ gives
\[
 -\langle f,A(k)f\rangle_\pi
 =\frac12\sum_{r\in\mathcal C_\perp}\sum_{\substack{b\in\Z^n\\|b|_1=1}}\pi(r)q_\perp(r,b)
           |\ee^{\ii k\cdot b}f(r+b)-f(r)|^2\geq0.
\]
Equality for a nonzero vector forces
$\ee^{\ii k\cdot b}f(r+b)=f(r)$ on every edge. Connectedness makes
all components nonzero; multiplication along the $L_i$ edges of a
period loop gives $\ee^{\ii k_iL_i}=1$ for each transverse coordinate.
Thus only $k=0$ on the Bloch torus admits a zero eigenvalue. At zero
this eigenvalue is simple. Continuity of the finite spectra and compactness give $\gamma>0$ such
that all other eigenvalues near zero, and all eigenvalues outside a fixed
neighborhood of zero, are at most $-\gamma$.

Analytic perturbation of the simple root gives
\[
 \nu(k)=H_\perp(-\ii k)-a
       =-\tfrac12 k^{\mathsf T}\Sigma_\perp k+O(|k|^3),\qquad
 \Pi(0)_{r,v}=\pi(v),\qquad \|\Pi(k)-\Pi(0)\|\leq C|k|,
\]
where $\Pi(k)$ is the spectral projection for $\nu(k)$ and
$r,v\in\mathcal C_\perp$ index its matrix entries.
Evenness of $H_\perp$ gives the vanishing linear term. Positive
definiteness of $\Sigma_\perp$ gives $\nu(k)\leq-c|k|^2$ for small
real $k$. The norm induced by $\langle\cdot,\cdot\rangle_\pi$ and the Euclidean
norm are equivalent, with constants depending only on the positive
cell values of $\pi$. Diagonalization with respect to
$\langle\cdot,\cdot\rangle_\pi$ therefore gives, for
$s\geq1$, the bounds
\[
 \begin{aligned}
 \|\ee^{sA(k)}\|&\leq C(\ee^{-cs|k|^2}+\ee^{-\gamma s})
                                      &&\text{near }k=0,\\
 \|\ee^{sA(k)}\|&\leq C {\ee}^{-\gamma s}
                                      &&\text{away from }k=0.
 \end{aligned}
\]
Consequently the integral in the inverse Bloch formula for $P_\perp$ is bounded by
$C\int_{\R^n}\ee^{-cs|k|^2}\,\dd k+C {\ee}^{-\gamma s}\leq Cs^{-n/2}$.
Conjugation proves \eqref{eq:transverse-kernel-upper} for $s\geq1$.
For $0\leq s\leq1$, use $P_\perp(s;y,z)\leq1$ and the finite maximum
of $\psi_\perp(y)/\psi_\perp(z)$.

To obtain \eqref{eq:transverse-kernel-local}, subtract
$\ee^{-s k^{\mathsf T}\Sigma_\perp k/2}\Pi(0)$ from $\ee^{sA(k)}$ near
zero. The exponential difference formula and the Taylor remainder give
\[
 \bigl\|\ee^{sA(k)}-\ee^{-s k^{\mathsf T}\Sigma_\perp k/2}\Pi(0)\bigr\|
 \leq C(|k|+s|k|^3)\ee^{-cs|k|^2}+C {\ee}^{-\gamma s}.
\]
Substituting $k=s^{-1/2}\xi$ in
$(|k|+s|k|^3)\ee^{-cs|k|^2}$ gives an integral of order $s^{-(n+1)/2}$;
the term $\ee^{-\gamma s}$ is smaller.
Extending the Gaussian integral from the neighborhood of zero to
$\R^n$ adds an exponentially small error, and
\[
 \frac1{(2\pi)^n}\int_{\R^n}
 \ee^{\ii k\cdot(y-z)-s k^{\mathsf T}\Sigma_\perp k/2}\,\dd k
 =\frac{\ee^{-(y-z)^{\mathsf T}\Sigma_\perp^{-1}(y-z)/(2s)}}
 {(2\pi s)^{n/2}\sqrt{\det\Sigma_\perp}}.
\]
The factor $\ee^{\ii k\cdot(y-z)}$ has modulus one, so the error is uniform in $y,z$.
Multiplication by $N\pi(r_z)\psi_\perp(y)/\psi_\perp(z)
=N\psi_\perp(y)\psi_\perp(z)$ gives
\eqref{eq:transverse-kernel-local}. If $a_\perp$ is constant, then
$\psi_\perp=N^{-1/2}$ and $\Sigma_\perp=2I$, so this formula reduces
to \eqref{eq:heat-local}.
\end{proof}

\begin{lemma}\label{lem:product-subsolution}
There is $\beta_\perp>0$ such that, for $T\geq1$ and a fixed
$y_0\in\Z^n$,
\begin{equation}\label{eq:transverse-factor}
 H_T(t,y;y_0)=\beta_\perp(t+T)^{n/2}\ee^{-a(t+T)}
                 K_\perp(t+T;y,y_0)
\end{equation}
satisfies $0<H_T\leq1$ for $t\geq0$. Let $q$ solve
\begin{equation}\label{eq:scalar-logistic}
 q_t=\Delta_{\Z}q+\left(a-\frac n{2(t+T)}\right)q-Kq^2,
 \qquad 0\leq q(0,\cdot)\leq a/K.
\end{equation}
Then $q(t,j)H_T(t,y;y_0)$ is a subsolution of \eqref{eq:main}
on $[0,\infty)\times\Z^d$ in the sense of
\eqref{eq:subsolution-definition}, with $\mathcal L_t=\Delta_{\Z^d}$
and $G(t,x,s)=F(x,s)$.
If $(j_0,y_0)\in\Z^d$ satisfies $u_0(j_0,y_0)>0$, some
$T\geq1$ and $\varepsilon>0$ make $q$, with
$q(0,j)=\varepsilon\1_{\{j_0\}}(j)$, satisfy
\begin{equation}\label{eq:product-comparison}
 u(T+t,j,y)\geq q(t,j)H_T(t,y;y_0)\text{ for }t\geq0.
\end{equation}
\end{lemma}
\begin{proof}
Choose $\beta_\perp=\min\{1,C^{-1}\}$, with $C$ from
\eqref{eq:transverse-kernel-upper}. Since
$(t+T)^{n/2}(1+t+T)^{-n/2}\leq1$, \eqref{eq:transverse-factor}
and positivity of $K_\perp$ give $0<H_T\leq1$ uniformly for $T\geq1$.
The operator $\Delta_{\Z^n}+a_\perp$ is bounded on $\ell^\infty(\Z^n)$, so its
exponential series may be differentiated uniformly on compact time intervals.
Differentiating \eqref{eq:transverse-factor} gives
\begin{equation}\label{eq:transverse-factor-equation}
 (H_T)_t=\Delta_{\Z^n} H_T+
       \left(a_\perp(y)-a+\frac n{2(t+T)}\right)H_T.
\end{equation}
For \eqref{eq:scalar-logistic}, the constants $0$ and $a/K$ are a
subsolution and a supersolution, respectively: at
$q=a/K$, the right-hand side equals $-an/[2K(t+T)]\le0$. Since $0\le q(0,\cdot)\le a/K$, Lemma~\ref{lem:fixed-domain-comparison} gives
$0\leq q(t,j)\leq a/K\leq1$. The reaction in \eqref{eq:scalar-logistic}
is uniformly Lipschitz on $[0,a/K]$; local Picard iteration and these
bounds extend $q$ to all $t\geq0$. Set $H=H_T(t,y;y_0)$.
The coordinate Laplacians separate:
\[
 \Delta_{\Z^d}(q(t,j)H_T(t,y;y_0))
 =H_T\Delta_{\Z}q+q\Delta_{\Z^n}H_T.
\]
Combining \eqref{eq:transverse-factor-equation} with
\eqref{eq:scalar-logistic} therefore gives
\begin{equation}\label{eq:product-subsolution-residual}
 (\partial_t-\Delta_{\Z^d})(qH)
 =a_\perp(y)qH-Kq^2H
 \leq a_\perp(y)qH-K(qH)^2\leq F((j,y),qH).
\end{equation}
Here the terms involving $a$ and $n/[2(t+T)]$ cancel, while
$0<H\leq1$ gives $q^2H\geq q^2H^2$, hence
$-Kq^2H\leq-K(qH)^2$. Also $0\leq qH\leq1$ and $K\geq K_F$, so
\eqref{eq:quadratic-kpp-loss} gives
$a_\perp(y)qH-K(qH)^2\leq F((j,y),qH)$.

To establish \eqref{eq:product-comparison}, choose $(j_0,y_0)$ with
$u_0(j_0,y_0)=\eta_0>0$ and fix $T\geq1$. By \eqref{eq:heat-comparison}, positivity of
the heat semigroup and $u_0\geq\eta_0\1_{\{(j_0,y_0)\}}$ give
$u(T,j,y)\geq\eta_0\pheat_1(T,j-j_0)\pheat_n(T,y-y_0)$.
Here the $\Z^d$ heat kernel is the product of its longitudinal and transverse kernels.
To compare $H_T(0,y;y_0)$ with $\pheat_n(T,y-y_0)$, use the
constant $a_{\max}$ fixed at the start of the section.
The function $\ee^{a_{\max}t}\pheat_n(t,\cdot-y_0)$ is a supersolution of
$z_t=\Delta_{\Z^n}z+a_\perp(y)z$ with
$z(0,y)=\1_{\{y_0\}}(y)$ for $y\in\Z^n$. Lemma~\ref{lem:fixed-domain-comparison} then yields
$K_\perp(T;y,y_0)\leq\ee^{a_{\max}T}\pheat_n(T,y-y_0)$.
Take $q(0,j)=\varepsilon\1_{\{j_0\}}(j)$, with
\[
 0<\varepsilon\leq\frac12\min\left\{\frac aK,
 \frac{\eta_0\pheat_1(T,0)}
 {\beta_\perp T^{n/2}\ee^{(a_{\max}-a)T}}\right\}.
\]
Both entries in this minimum are positive, and the choice is independent of $y$.
At $j=j_0$, \eqref{eq:transverse-factor} and the bound for $K_\perp$ give
\[
 \begin{aligned}
 q(0,j_0)H_T(0,y;y_0)
 &\leq\varepsilon\beta_\perp T^{n/2}
       \ee^{(a_{\max}-a)T}\pheat_n(T,y-y_0)\\
 &\leq\eta_0\pheat_1(T,0)\pheat_n(T,y-y_0)
 \leq u(T,j_0,y).
 \end{aligned}
\]
The choice of $\varepsilon$ ensures the second inequality for every $y$. At every other longitudinal site, $q(0,j)=0$, so this
ordering holds for all $(j,y)$. Applying
Lemma~\ref{lem:fixed-domain-comparison} with
\eqref{eq:product-subsolution-residual} now proves
\eqref{eq:product-comparison}.

\end{proof}

\begin{lemma}\label{lem:logconcavity}
 {For every $t>0$, the distribution $G(t,j)$ from Lemma~\ref{lem:longitudinal} is strictly positive and
log-concave in $j$:}
\[
    G(t,j)^2\geq G(t,j-1)G(t,j+1)\text{ for }j\in\Z.
\]
Consequently, $G(t,j+1)/G(t,j-1)$ is nonincreasing in $j$.
\end{lemma}

\begin{proof}
Let $N_+,N_-$ be independent, with
$N_\pm\sim\operatorname{Poisson}(\ee^{\pm\lambda_*}t)$.
By Lemma~\ref{lem:longitudinal}, $G(t,j)=\mathbb P(N_+-N_-=j)$.
For $\rho>0$ and $k\in\Z$, set $\pi_\rho(k)=\ee^{-\rho}\rho^k/k!$ for $k\geq0$
and $\pi_\rho(k)=0$ for $k<0$.
The ratio $\pi_\rho(k+1)/\pi_\rho(k)=\rho/(k+1)$ is nonincreasing for
$k\geq0$, so $\pi_\rho$ is log-concave.
Reflection preserves log-concavity. The convolution of
$\pi_{\ee^{\lambda_*}t}$ with $k\mapsto\pi_{\ee^{-\lambda_*}t}(-k)$
equals $G(t,\cdot)$; both factors have interval supports.
Thus \cite[Theorem~4.1(a)]{SaumardWellner2014} gives log-concavity of $G(t,\cdot)$.
For $j\geq0$, the event $(N_+,N_-)=(j,0)$ has positive probability;
for $j<0$, use $(N_+,N_-)=(0,-j)$.
Thus $G(t,j)>0$ for every $j\in\Z$.
Set $R_j=G(t,j)/G(t,j-1)>0$.
Log-concavity gives $R_{j+1}\leq R_j$.
The identity $G(t,j+1)/G(t,j-1)=R_{j+1}R_j$ therefore proves that
$j\mapsto G(t,j+1)/G(t,j-1)$ is nonincreasing.
\end{proof}

For each fixed $A>0$, the next lemma gives $L_0\geq1$ and proves
the lower estimate in \cite[Lemma~6.3]{HuAxial} on
$L_0\leq j-c_*t\leq A\sqrt t$.
\begin{lemma}\label{lem:autonomous-leading-lower}
Let $K>0$ and $Q$ solve
\[
    Q_t=\Delta_{\Z}Q+aQ-KQ^2,
    \qquad 0\leq Q(0,\cdot)\leq a/K,
\]
with $Q(0,\cdot)\not\equiv0$ and $|\supp Q(0,\cdot)|<\infty$.
For every $A>0$ there
are $c_0>0$, $L_0\geq1$ and $t_0\geq1$, which may depend on $A$,
such that
\begin{equation}\label{eq:autonomous-leading-lower}
    Q(t,j)\geq c_0t^{-3/2}(j-c_*t)
                         \ee^{-\lambda_*(j-c_*t)}
\end{equation}
whenever $t\geq t_0$ and $L_0\leq j-c_*t\leq A\sqrt t$.
\end{lemma}

\begin{proof}
Use $w_1$ and $Z$ from Lemma~\ref{lem:longitudinal}, and define
\[
    W(t,j)=Z(t,j)_+,\qquad
    Z(t,j)=\ee^{-\lambda_*(j-c_*t)}w_1(t,j).
\]
By that lemma, $Z_t=\Delta_{\Z}Z+aZ$ and
$Z(0,j)=\ee^{-\lambda_*}\1_{\{1\}}(j)-\ee^{\lambda_*}\1_{\{-1\}}(j)$
for $j\in\Z$.
Since $Z(\cdot,j)$ is continuously differentiable and the positive-part
map is Lipschitz, $W(\cdot,j)$ is locally absolutely continuous. The
positive part $W=Z_+$ satisfies
\begin{equation}\label{eq:positive-part-inequality}
    W_t\leq\Delta_{\Z}W+aW
\end{equation}
for almost every time at each site. To verify this, let $N_j\subset(0,\infty)$
be the null set where $W(\cdot,j)$ is not differentiable, and put
$N=\bigcup_{j\in\Z}N_j$. This union is null. Fix $t\notin N$.
If $Z(t,j)>0$, then $W(t,j)=Z(t,j)$ and
\[
 \begin{aligned}
 W_t(t,j)&=Z_t(t,j)\\
 &=Z(t,j-1)+Z(t,j+1)+(a-2)W(t,j)\\
 &\leq W(t,j-1)+W(t,j+1)+(a-2)W(t,j).
 \end{aligned}
\]
If $Z(t,j)\leq0$, then $W(t,j)=0$; since $W(\cdot,j)\geq0$,
its derivative at this minimum is zero. Thus $W_t(t,j)=0$ and
$\Delta_{\Z}W(t,j)+aW(t,j)=W(t,j-1)+W(t,j+1)\geq0$.
These two cases prove \eqref{eq:positive-part-inequality} for every
$j\in\Z$ and $t\notin N$.

We first bound $W$ uniformly in the lattice site:
\begin{equation}\label{eq:integrable-positive-mode}
    \sup_{j\in\Z}W(t,j)\leq C_0(1+t)^{-3/2} {,\qquad t\geq0,}
\end{equation}
for a constant $C_0 {>0}$. Put $j_{\mathrm{cut}}(t)=\lfloor c_*t\rfloor-2$.
Then $s=j_{\mathrm{cut}}(t)-c_*t\in[-3,-2]$ and
$s+\delta_*\leq-3/2$ by \eqref{eq:critical-system}.
For all sufficiently large $t\geq9$, \eqref{eq:signed-kernel} with $A=1$ gives
$w_1(t,j_{\mathrm{cut}}(t))\leq-ct^{-3/2}+Ct^{-2}<0$
once also $t\geq(2C/c)^2$.
Lemma~\ref{lem:logconcavity} gives $w_1(t,j)\leq0$ for
$j\leq j_{\mathrm{cut}}(t)$, because $G(t,j+1)/G(t,j-1)$ is at least its value at
$j_{\mathrm{cut}}(t)$, which exceeds one. Hence $W(t,j)>0$ can occur only when
$j-c_*t\geq-3$. On the remaining interval
$-3\leq s:=j-c_*t\leq\sqrt t$, use \eqref{eq:signed-kernel} in
Lemma~\ref{lem:longitudinal} with $A=1$ and sufficiently large $t$ to obtain
\[
    W(t,j)\leq Ct^{-3/2}(1+|s|)\ee^{-\lambda_*s}
             \leq C't^{-3/2}.
\]
For $s>\sqrt t$, the bound $|w_1(t,j)|\leq2$ gives
$W(t,j)\leq2\ee^{-\lambda_*\sqrt t}\leq C't^{-3/2}$.
On a bounded time interval, $Z$ satisfies
\[
 Z(t,j)=\ee^{at}\bigl(\ee^{-\lambda_*}\pheat_1(t,j-1)
                   -\ee^{\lambda_*}\pheat_1(t,j+1)\bigr).
\]
Since $0\leq\pheat_1\leq1$, this bounds $\|Z(t,\cdot)\|_\infty$
by $\ee^{at}(\ee^{-\lambda_*}+\ee^{\lambda_*})$. Combining this finite-time
bound with the large-time bounds on $-3\leq s\leq\sqrt t$ and
$s>\sqrt t$ proves
\eqref{eq:integrable-positive-mode} after increasing $C_0$.

Choose a site $j_0$ with $Q(0,j_0)>0$ and let
$\widetilde W(t,j)=W(t,j-j_0+1)$. Define
\[
    \chi(t)=\left(\chi_0^{-1}
                 +KC_0\int_0^t(1+s)^{-3/2} {\,\dd s}\right)^{-1},
\]
where $0<\chi_0\leq\frac12\min\{\ee^{\lambda_*}Q(0,j_0),a/(KC_0)\}$.
The definition gives $\chi'=-KC_0(1+t)^{-3/2}\chi^2$ and
\[
 \int_0^t(1+s)^{-3/2}\,\dd s=2-2(1+t)^{-1/2}\leq2.
\]
Consequently
\[
    0<(\chi_0^{-1}+2KC_0)^{-1}\leq\chi(t)\leq\chi_0.
\]
The choice of $\chi_0$ and \eqref{eq:integrable-positive-mode} give
$0\leq\chi\widetilde W\leq a/K$.
Using \eqref{eq:positive-part-inequality} and
\eqref{eq:integrable-positive-mode}, we obtain
\[
\begin{split}
    &(\partial_t-\Delta_{\Z})(\chi\widetilde W)
               -a\chi\widetilde W+K\chi^2\widetilde W^2\\
    &\qquad\leq
       \chi^2\widetilde W\bigl[-KC_0(1+t)^{-3/2}
                                      +K\widetilde W\bigr]\leq0.
\end{split}
\]
Since $\widetilde W(0,j)=\ee^{-\lambda_*}\1_{\{j_0\}}(j)$,
the choice of $\chi_0$ gives
$\chi(0)\widetilde W(0,j)\leq Q(0,j)$ for every $j\in\Z$.
The function $\chi$ is continuously differentiable and
$\widetilde W$ is sitewise locally absolutely continuous.
Their product therefore has the time regularity required in
\eqref{eq:subsolution-definition}.
Thus $\chi\widetilde W$ is a subsolution of
$Q_t=\Delta_{\Z}Q+aQ-KQ^2$ on $[0,\infty)\times\Z$, in the sense of
\eqref{eq:subsolution-definition}. Lemma~\ref{lem:fixed-domain-comparison} thus yields
$Q\geq\chi\widetilde W$.

Fix $A>0$, put $h=j_0-1$ and choose $L_0\geq\max\{2,2|h|+2\}$.
For $s=j-c_*t\geq L_0$, set $s'=s-h$. Then $s'\geq s/2\geq1$. If also $s\leq A\sqrt t$ and $t\geq h^2$,
then $s'\leq(A+1)\sqrt t$. In \eqref{eq:signed-kernel} at $s'$,
$(s'+\delta_*)\ee^{-(s')^2/(4D_*t)}/(\sqrt{4\pi}D_*^{3/2}t^{3/2})$
is at least
\[
 \frac{\ee^{-(A+1)^2/(4D_*)}}{\sqrt{4\pi}D_*^{3/2}}s't^{-3/2}.
\]
The remainder in \eqref{eq:signed-kernel} is at most
$C_{A+1}(1+s')t^{-2}\leq2C_{A+1}s't^{-2}$.
Increase $t_0$ so that $2C_{A+1}t^{-1/2}$ is at most half
of $\ee^{-(A+1)^2/(4D_*)}/(\sqrt{4\pi}D_*^{3/2})$.
Multiplying by $\ee^{-\lambda_*s'}$ and using the lower bound for
$\chi$ gives a positive constant times
$s t^{-3/2}\ee^{-\lambda_*s}$. Absorbing $\ee^{\lambda_*h}$ into this
constant proves \eqref{eq:autonomous-leading-lower}.

\end{proof}

The definition of $\chi$ in Lemma~\ref{lem:autonomous-leading-lower}
follows the construction in
\cite[Section~2.1.2, (2.5)--(2.7)]{HNRRPeriodic}.
The following construction is from \cite[Lemma~6.4]{HuAxial}.
\begin{lemma}\label{lem:delayed-pulse}
Fix $\beta\geq0$, $K>0$ and $T\geq1$, and set
\[
    b_\beta:=\frac{\beta+3/2}{\lambda_*}.
\]
Choose $0<\mu<\lambda_*$ and define
\begin{equation}\label{eq:pulse-parameters}
\begin{gathered}
    \omega_\mu:=\frac1{2(\ee^\mu-1)},\qquad
    P_\mu(r):=r-\omega_\mu+\omega_\mu\ee^{-\mu r},\\
    \sigma_\mu:=2\bigl(\cosh(\lambda_*+\mu)-1\bigr)+a
                                      -c_*(\lambda_*+\mu)>0.
\end{gathered}
\end{equation}
For $\varepsilon>0$, let
\begin{equation}\label{eq:critical-pulse}
    \Phi_\varepsilon(r)=
    \begin{cases}
      \varepsilon\ee^{-\lambda_*r}P_\mu(r),&r>0,\\
      0,&r\leq0.
    \end{cases}
\end{equation}
There are $\varepsilon_*>0$ and $\tau_*>0$, depending only on
$a,K,\beta,\mu$, such that, for every
$0<\varepsilon\leq\varepsilon_*$, the function
\[
    \Psi(t,j):=\Phi_\varepsilon
          \bigl(j-c_*t+b_\beta\log(t+T)\bigr)
\]
has values in $[0,a/K]$ and satisfies
\begin{equation}\label{eq:delayed-pulse-subsolution}
    \Psi_t\leq\Delta_{\Z}\Psi+
               \left(a-\frac\beta{t+T}\right)\Psi-K\Psi^2
\end{equation}
for $t\geq0$ with $t+T\geq\tau_*$, in the almost-everywhere sense
at each site.
\end{lemma}

\begin{proof}
The relations in \eqref{eq:critical-system} identify $\lambda_*,c_*$
with those in \cite[Lemma~6.4]{HuAxial}, for the scalar growth rate $a$.
They also give
$\sigma_\mu=2\int_0^\mu(\mu-v)\cosh(\lambda_*+v)\,\dd v>0$.
Define $F_\varepsilon(r)=\varepsilon\ee^{-\lambda_*r}P_\mu(r)$ for $r\in\R$.
For $f:\R\to\R$, set $\Delta_1f(r)=f(r-1)-2f(r)+f(r+1)$. Then
\begin{equation}\label{eq:pulse-linear-residual}
    -c_*F_\varepsilon'(r)-\Delta_1F_\varepsilon(r)
                        -aF_\varepsilon(r)
       =-\varepsilon\omega_\mu\sigma_\mu
                                     \ee^{-(\lambda_*+\mu)r}.
\end{equation}
Since $\mu\omega_\mu<1/2$,
$1/2<1-\mu\omega_\mu\leq P_\mu'(r)\leq1$ for $r\geq0$.
Integration from $P_\mu(0)=0$ gives $r/2\leq P_\mu(r)\leq r$.
On $[-1,0]$, $P_\mu$ is convex, with
$P_\mu(-1)=-1/2$ and $P_\mu(0)=0$, so $P_\mu\leq0$ there.
For $r>0$, truncation can change only the neighbor $r-1$.
When $0<r<1$, it replaces $F_\varepsilon(r-1)\leq0$ by zero.
Thus $\Delta_1\Phi_\varepsilon(r)\geq\Delta_1F_\varepsilon(r)$ for every $r>0$.
Set $M_\mu:=\sup_{r\geq0}r^2\ee^{-(\lambda_*-\mu)r}
=4/(\ee^2(\lambda_*-\mu)^2)$ and choose
\[
    0<\varepsilon_*\leq
      \min\left\{\frac{a\ee\lambda_*}{K},
                  \frac{\omega_\mu\sigma_\mu}{2KM_\mu}\right\}.
\]
The first restriction gives $0\leq\Phi_\varepsilon\leq a/K$.
The second gives
$K\Phi_\varepsilon(r)^2\leq\varepsilon\omega_\mu\sigma_\mu
\ee^{-(\lambda_*+\mu)r}/2$ for $r>0$. Hence
\begin{equation}\label{eq:pulse-nonlinear-residual}
    -c_*\Phi_\varepsilon'(r)-\Delta_1\Phi_\varepsilon(r)
                     -a\Phi_\varepsilon(r)+K\Phi_\varepsilon(r)^2
       \leq-\frac{\varepsilon\omega_\mu\sigma_\mu}{2}
                                  \ee^{-(\lambda_*+\mu)r}
       \text{ for }r>0.
\end{equation}
Put $\tau=t+T$ and $r=j-c_*t+b_\beta\log\tau$.
For $r>0$, subtracting the left-hand side of \eqref{eq:pulse-nonlinear-residual}
from $\Psi_t-\Delta_{\Z}\Psi-(a-\beta/(t+T))\Psi+K\Psi^2$ gives
\begin{equation}\label{eq:pulse-temporal-residual}
\begin{split}
    \frac1\tau\bigl(b_\beta\Phi_\varepsilon'(r)
                                      +\beta\Phi_\varepsilon(r)\bigr)
      &=\frac{\varepsilon\ee^{-\lambda_*r}}\tau
          \left[b_\beta P_\mu'(r)-\frac32P_\mu(r)\right]\\
      &\leq\frac{\varepsilon\ee^{-\lambda_*r}}\tau
                       \left(b_\beta-\frac34r\right).
\end{split}
\end{equation}
Set $R_0:=1+4b_\beta/3$ and
\[
    \tau_*:=\max\left\{1,
          \frac{2b_\beta\ee^{\mu R_0}}{\omega_\mu\sigma_\mu}\right\}.
\]
For $r\geq R_0$, the right-hand side of
\eqref{eq:pulse-temporal-residual} is nonpositive.
For $0<r<R_0$ and $\tau\geq\tau_*$, the temporal term is at most
$\varepsilon b_\beta\tau^{-1}\ee^{-\lambda_*r}
\leq\varepsilon\omega_\mu\sigma_\mu\ee^{-(\lambda_*+\mu)r}/2$,
so it is dominated by the negative term in \eqref{eq:pulse-nonlinear-residual}.
For $r<0$, $\Psi=\Psi_t=0$ and its neighboring values are nonnegative.
For $\tau\geq\tau_*$, \eqref{eq:delayed-pulse-subsolution} therefore
holds away from $r=0$.
The function $\Phi_\varepsilon$ is globally Lipschitz.
At each fixed $j$, $r(t,j)$ is strictly concave and has at most two zeros.
Their union over $j\in\Z$ is countable, which gives the asserted
almost-everywhere inequality. The constants are independent of $T$,
since $T$ enters only through $\tau=t+T$.
\end{proof}

The next proposition extends the range $r\leq(t+T)^{1/4}$ in
\cite[Proposition~6.5]{HuAxial} to $r\leq\sigma\sqrt{t+T}$ for each fixed $\sigma>0$.
\begin{proposition}\label{prop:scalar-logarithmic-lower}
Fix $\beta\geq0$, $K>0$ and $T\geq1$. Let $q$ solve
\begin{equation}\label{eq:scalar-growth-loss}
    q_t=\Delta_{\Z}q+\left(a-\frac\beta{t+T}\right)q-Kq^2,
    \qquad q(0,\cdot)=q_0,
\end{equation}
where $q_0:\Z\to[0,a/K]$ satisfies $q_0\not\equiv0$ and
$|\supp q_0|<\infty$.
Fix $0<\mu<\lambda_*$ and $\sigma>0$, and let $\Phi_\varepsilon$
be as in Lemma~\ref{lem:delayed-pulse}.
There exist $\varepsilon>0$ and $t_0\geq1$, depending on the fixed parameters
and $\sigma$, such that, setting
\[
    r(t,j):=j-c_*t+b_\beta\log(t+T),\qquad
    b_\beta=\frac{\beta+3/2}{\lambda_*},
\]
one has
\begin{equation}\label{eq:scalar-profile-lower}
    q(t,j)\geq\Phi_\varepsilon(r(t,j))
    \text{ for }t\geq t_0,\ r(t,j)\leq\sigma\sqrt{t+T}.
\end{equation}
In particular, for some $\eta>0$,
\begin{equation}\label{eq:scalar-positive-seed}
    q\left(t,
       \left\lceil c_*t-b_\beta\log(t+T)+1\right\rceil\right)
       \geq\eta\text{ for }t\geq t_0,
\end{equation}
after increasing $t_0$ if necessary.
\end{proposition}

\begin{proof}
The constants $0$ and $a/K$ are respectively a subsolution and a
supersolution of \eqref{eq:scalar-growth-loss}, since its right-hand
side at $a/K$ is $-a\beta/[K(t+T)]\leq0$. Comparison gives
$0\leq q\leq a/K$. The right-hand side of \eqref{eq:scalar-growth-loss} is locally Lipschitz on
$\ell^\infty(\Z)$ and bounded on $[0,a/K]$, so $q$
exists for all $t\geq0$. Moreover,
$q_t\geq\Delta_{\Z}q-(\beta/T)q$, and hence
\[
 q(t,j)\geq \ee^{-\beta t/T}(\ee^{t\Delta_{\Z}}q_0)(j)>0
 \quad\text{for }t>0.
\]
Indeed, if $q_0(j_0)>0$, the term at $j_0$ is positive because
$\pheat_1(t,j-j_0)\geq \ee^{-2t}t^{|j-j_0|}/|j-j_0|!>0$.
This lower bound is the probability of the prescribed
$|j-j_0|$ nearest-neighbor jumps and no others.

Let $Q$ solve
\[
 Q_t=\Delta_{\Z}Q+aQ-KQ^2,\qquad Q(0,\cdot)=q_0,
\]
where $a=a_{\mathrm{eff}}$. Put $\rho(t)=(T/(t+T))^\beta$.
Then $\rho'=-(\beta/(t+T))\rho$, and
\[
 \begin{split}
 &(\partial_t-\Delta_{\Z})(\rho Q)
       -\left(a-\frac{\beta}{t+T}\right)\rho Q+K\rho^2Q^2\\
 &\hspace{4em}=-K\rho(1-\rho)Q^2\leq0.
 \end{split}
\]
Since $0<\rho\leq1$ and $\rho(0)=1$,
Lemma~\ref{lem:fixed-domain-comparison} gives $q\geq\rho Q$.
Since $Q(0,\cdot)=q_0$, we have $0\leq Q(0,\cdot)\leq a/K$,
$Q(0,\cdot)\not\equiv0$ and $|\supp Q(0,\cdot)|<\infty$.
Lemma~\ref{lem:autonomous-leading-lower}, with $A=2\sigma+1$,
therefore gives $c_0>0$, $L_0\geq1$ and $t_1\geq1$ such that
\[
 Q(t,j)\geq c_0t^{-3/2}(j-c_*t)\ee^{-\lambda_*(j-c_*t)}
 \quad\text{for }t\geq t_1,\quad
 L_0\leq j-c_*t\leq(2\sigma+1)\sqrt t.
\]
Multiplying by $\rho(t)=T^\beta(t+T)^{-\beta}$ gives
\begin{equation}\label{eq:nonautonomous-leading-lower}
    q(t,j)\geq c_0T^\beta(t+T)^{-\beta}t^{-3/2}
                     (j-c_*t)\ee^{-\lambda_*(j-c_*t)}
\end{equation}
for $t\geq t_1$ and $L_0\leq j-c_*t\leq(2\sigma+1)\sqrt t$.

For the fixed $0<\mu<\lambda_*$ in the statement, let
$\varepsilon_*>0$ and $\tau_*>0$ be the constants in
Lemma~\ref{lem:delayed-pulse}. Write $\tau=t+T$ and set
\[
 B(t)=c_*t-b_\beta\log\tau+\sigma\sqrt\tau.
\]
Choose $t_0\geq\max\{t_1,T\}$ so that, for every $t\geq t_0$,
\[
 \begin{gathered}
 b_\beta\log\tau\leq\tfrac12\sigma\sqrt\tau,\qquad
 \tfrac12\sigma\sqrt\tau\geq L_0,\qquad
 \sigma\sqrt\tau+1\leq(2\sigma+1)\sqrt t,\\
 b_\beta/\tau\leq c_*/2,\qquad
 \tau\geq\tau_*,\qquad \sigma\sqrt\tau\geq2.
 \end{gathered}
\]
These choices are possible because
$\log\tau/\sqrt\tau\to0$ and $\tau/t\to1$.
The choice of $t_0$ depends only on $a,K,\beta,T,q_0,\mu,\sigma$. In particular,
$B'(t)=c_*-b_\beta/\tau+\sigma/(2\sqrt\tau)\geq c_*/2>0$.

If $B(t)<j\leq B(t)+1$, then
$\sigma\sqrt\tau<r:=r(t,j)\leq\sigma\sqrt\tau+1$.
Set $s=j-c_*t=r-b_\beta\log\tau$. The chosen thresholds give
$L_0\leq s\leq(2\sigma+1)\sqrt t$ and $s\geq r/2$.
Since $\lambda_*b_\beta=\beta+3/2$,
\[
 \ee^{-\lambda_*s}=\ee^{-\lambda_*r}\tau^{\beta+3/2}.
\]
Thus \eqref{eq:nonautonomous-leading-lower} yields
\[
 q(t,j)\geq c_0T^\beta(\tau/t)^{3/2}s \ee^{-\lambda_*r}
          \geq c_1r \ee^{-\lambda_*r},
 \qquad c_1=c_0T^\beta/2>0.
\]
Since $P_\mu(r)\leq r$ for $r\geq0$, the definition
\eqref{eq:critical-pulse} gives, for every $t\geq t_0$,
$B(t)<j\leq B(t)+1$ and $0<\varepsilon\leq c_1$,
\[
 \Phi_\varepsilon(r(t,j))
 \leq\varepsilon r(t,j)\ee^{-\lambda_*r(t,j)}
 \leq c_1r(t,j)\ee^{-\lambda_*r(t,j)}
 \leq q(t,j).
\]

To prove $\Phi_\varepsilon(r(t_0,j))\leq q(t_0,j)$ for $j\leq B(t_0)$, let
\[
 E=\{j\in\Z:0<r(t_0,j)\leq\sigma\sqrt{t_0+T}\}.
\]
{ Since $\sigma\sqrt{t_0+T}\geq2$, the set $E$ is finite and nonempty. Writing $\Phi_1(r)=\ee^{-\lambda_*r}P_\mu(r)\1_{(0,\infty)}(r)$, we have $q(t_0,j)>0$ and $\Phi_1(r(t_0,j))\geq\frac12 r(t_0,j)\ee^{-\lambda_* r(t_0,j)}>0$ for each $j\in E$, so we may choose
\[
 0<\varepsilon\leq\frac12\min\left\{\varepsilon_*,c_1,
       \min_{j\in E}\frac{q(t_0,j)}{\Phi_1(r(t_0,j))}\right\}.
\]
Since $\Phi_\varepsilon=\varepsilon\Phi_1$ vanishes for $r\leq0$ and $q(t_0,\cdot)\geq0$, this choice ensures that $\Phi_\varepsilon(r(t_0,j))\leq q(t_0,j)$ for all $j\leq B(t_0)$.}

Lemma~\ref{lem:delayed-pulse} gives
\eqref{eq:delayed-pulse-subsolution} for
$(t,j)\mapsto\Phi_\varepsilon(r(t,j))$ when $t\geq t_0$.
Both $\Phi_\varepsilon(r(t,j))$ and $q$ take values in $[0,a/K]$,
where $s\mapsto(a-\beta/(t+T))s-Ks^2$ is uniformly Lipschitz
with constant at most $3a+\beta/T$. Both $q(\cdot,j)$ and $\Phi_\varepsilon(r(\cdot,j))$ are locally
absolutely continuous, and $B$ is increasing. We have
$\Phi_\varepsilon(r(t_0,j))\leq q(t_0,j)$ for $j\leq B(t_0)$,
and $\Phi_\varepsilon(r(t,j))\leq q(t,j)$ for $t\geq t_0$ and
$B(t)<j\leq B(t)+1$. Applying Lemma~\ref{lem:moving-halfline-comparison} on $[t_0,S]$
for arbitrary $S>t_0$ gives $q(t,j)\geq\Phi_\varepsilon(r(t,j))$ for
$t_0\leq t\leq S$ and $j\leq B(t)$.
The condition $j\leq B(t)$ is equivalent to
$r(t,j)\leq\sigma\sqrt{t+T}$, so letting $S$ be arbitrary proves
\eqref{eq:scalar-profile-lower}.

At
$j=\lceil c_*t-b_\beta\log(t+T)+1\rceil$ one has
$1\leq r(t,j)<2\leq\sigma\sqrt{t+T}$.
Since $P_\mu(r)\geq r/2\geq1/2$ there,
\[
 q(t,j)\geq\Phi_\varepsilon(r(t,j))
          \geq\varepsilon \ee^{-2\lambda_*}/2.
\]
Taking $\eta=\varepsilon \ee^{-2\lambda_*}/2>0$ proves
\eqref{eq:scalar-positive-seed}.
\end{proof}

\subsection{Lower bounds on axial level sets}\label{sec:amplification}
We adapt the finite-box construction in \cite[Lemmas~7.1--7.3]{HuAxial}
to the finitely many periodic translates of $F$.

\begin{lemma}\label{lem:amplification}
Assume \eqref{eq:periodic-kpp}. For each $\eta,\theta\in(0,1)$ there is
$S=S(\eta,\theta,F)>0$ such that every $[0,1]$-valued solution of
\eqref{eq:main} on $[t_0,\infty)\times\Z^d$ with
$u(t_0,z)\geq\eta$ at a site $z\in\Z^d$ satisfies
$u(t_0+s,z)\geq\theta$ for all $s\geq S$.
The constant $S$ is independent of $z$, $t_0$ and $u$.
\end{lemma}
\begin{proof}
For an integer $L\geq1$, let $Q_L=\{-L,\ldots,L\}^d$ and set
\[
 \varphi_L(x)=\prod_{i=1}^d\cos\frac{\pi x_i}{2L+2}
 \text{ for }x\in Q_L,\qquad \varphi_L=0\quad\text{outside }Q_L.
\]
Put $\vartheta=\pi/(2L+2)$. The cosine addition formula gives
\[
 \cos(\vartheta(x_i+1))+\cos(\vartheta(x_i-1))
       =2\cos\vartheta\cos(\vartheta x_i).
\]
At $x_i=\pm L$, the missing neighbor has cosine zero. Multiplying by
the other factors and summing over $i$ therefore gives
$-\Delta_{\Z^d}\varphi_L=\mu_L\varphi_L$ in $Q_L$, where
$\mu_L=2d(1-\cos(\pi/(2L+2)))\to0$.
Write $a_{\min}=\min_{x\in\mathcal C}a(x)>0$ and choose $L$ with $\mu_L<a_{\min}$.
By \eqref{eq:quadratic-kpp-loss}, $F(x,s)/s\geq a_{\min}-K_Fs$.
Choose $\varepsilon_L\in(0,1)$ with
$K_F\varepsilon_L<a_{\min}-\mu_L$; if $K_F=0$, any
$\varepsilon_L\in(0,1)$ works. Then
$F(x,s)>\mu_Ls$ for $x\in\Z^d$ and $0<s\leq\varepsilon_L$.
Thus $\varepsilon\varphi_L$, with $0<\varepsilon\leq\varepsilon_L$,
is a strict stationary subsolution on $Q_L$, equal to zero outside.
The function $\1_{Q_L}$ is a supersolution of \eqref{eq:main} in $Q_L$ with zero values outside $Q_L$.

Let $v_L$ solve \eqref{eq:main} in $Q_L$, with $v_L=0$ outside and
$v_L(0,x)=\varepsilon\varphi_L(x)$. The finite-dimensional locally
Lipschitz vector field gives a local solution. Lemma~\ref{lem:fixed-domain-comparison}
gives $\varepsilon\varphi_L\leq v_L\leq1$, so $v_L$ is global.
Comparison first gives $v_L(h,\cdot)\geq v_L(0,\cdot)$ and then, by
time translation, $v_L(t+h,\cdot)\geq v_L(t,\cdot)$ for $t,h\geq0$.
Hence $q_L(x)=\lim_{t\to\infty}v_L(t,x)>0$ exists in $Q_L$; set $q_L=0$
outside. Continuity of the vector field gives
$\partial_tv_L(t,x)\to\Delta_{\Z^d}q_L(x)+F(x,q_L(x))$.
A nonzero limit would keep the derivative bounded away from zero,
contradicting convergence of $v_L(t,x)$. Thus
$\Delta_{\Z^d}q_L+F(x,q_L)=0$ in $Q_L$.

We show convergence to one on finite sets for every choice of
$[0,1]$-valued equilibria $q_L$ that are positive in $Q_L$ and zero outside.
Fix an integer $K_0\geq1$ with
$\mu_{K_0}<\min_{x\in\mathcal C}a(x)$ and choose $b\in(0,1)$ such that
$F(x,s)>\mu_{K_0}s$ for all $x\in\Z^d$ and $0<s\leq b$.
For an arbitrary positive equilibrium $q_L$ and
$z+Q_{K_0}\subset Q_L$, set
\[
 \sigma_*=\min_{x\in z+Q_{K_0}}
                 \frac{q_L(x)}{\varphi_{K_0}(x-z)}>0.
\]
The minimum is positive and attained because the box is finite and
$\varphi_{K_0}>0$ there. If $\sigma_*<b$, then
$g_{\sigma_*}=q_L-\sigma_*\varphi_{K_0}(\cdot-z)\geq0$
vanishes at a minimizing site $x$, where
$0<\sigma_*\varphi_{K_0}(x-z)<b$. The equations give
\[
 0\leq\Delta_{\Z^d}g_{\sigma_*}(x)
 =\mu_{K_0}\sigma_*\varphi_{K_0}(x-z)
   -F\bigl(x,\sigma_*\varphi_{K_0}(x-z)\bigr)<0.
\]
Thus $\sigma_*\geq b$, independently of $L,z$ and the choice of $q_L$.
Since $\varphi_{K_0}(0)=1$, $q_L(z)\geq b$ whenever
$z+Q_{K_0}\subset Q_L$.

Enumerate $\Z^d$ as $x^{(1)},x^{(2)},\ldots$. From any $L_k\to\infty$,
successive extraction at these sites gives a diagonal subsequence
converging to $q$ at every site. Each equilibrium equation has finitely
many neighbors, so continuity of $F$ passes it to the limit.
The eventual bound $q_L(x)\geq b$ gives $b\leq q\leq1$.
Put $m=\inf_{x\in\Z^d}q(x)$ and choose $x_k$ with $m\leq q(x_k)<m+1/k$.
The equation implies
\[
 0\leq F(x_k,q(x_k))\leq2d(q(x_k)-m)\leq2d/k.
\]
On a subsequence with $x_k\bmod\Gamma=r_0$, continuity gives
$F(r_0,m)=0$. Since $b\leq m\leq1$ and $F(r_0,s)>0$ for $0<s<1$,
we have $m=1$ and $q\equiv1$.
If convergence to one failed on a finite set $E$, there would be
$\delta>0$, a sequence $L_k\to\infty$, and a site $x\in E$ such that
$q_{L_k}(x)\leq1-\delta$.
A diagonal subsequence would have limit $q\equiv1$, a contradiction.
Thus $q_L\to1$ uniformly on each finite set, independently of the
choice of positive equilibria.

Let $U^z$ solve \eqref{eq:main} with $U^z(0,\cdot)=\eta\1_{\{z\}}$.
By \eqref{eq:heat-comparison}, $U^z(1,x)\geq\eta\pheat_d(1,x-z)>0$.
Fix a residue of $z$. Let $q_L^z:\Z^d\to[0,1]$ satisfy
$\Delta_{\Z^d}q_L^z+F(x,q_L^z)=0$ on $z+Q_L$, with $q_L^z>0$ there
and $q_L^z=0$ outside. Applying the construction on $Q_L$ to
$(\xi,s)\mapsto F(z+\xi,s)$, then translating back, gives existence
and $q_L^z(z)\to1$ for every choice of such equilibria.
Choose $L$ so that $q_L^z(z)>(1+\theta)/2$ for every such positive equilibrium.
The finite positive minimum
\[
 \min_{x\in Q_L}\frac{\eta\pheat_d(1,x)}{\varphi_L(x)}>0
\]
permits $0<\varepsilon\leq\varepsilon_L$ with
$\varepsilon\varphi_L\leq\eta\pheat_d(1,\cdot)$ on $Q_L$.
Let $v_L^z$ solve \eqref{eq:main} on $z+Q_L$, with $v_L^z=0$ outside
and $v_L^z(0,x)=\varepsilon\varphi_L(x-z)$.
Lemma~\ref{lem:fixed-domain-comparison} gives
$U^z(1+t,x)\geq v_L^z(t,x)$ on $z+Q_L$.
Since $v_L^z(t,z)$ increases to a limit greater than $\theta$,
$v_L^z(t,z)>\theta$ for every sufficiently large $t$.

The reaction $F(z+x,s)$ depends on $z$ only through its residue in
$\mathcal C$. For each residue, choose $L$, $\varepsilon$ and a time
after which $v_L^z(t,z)>\theta$. One plus the maximum of these finitely many times gives
$S$ with $U^z(s,z)\geq\theta$ for all $z\in\Z^d$ and $s\geq S$.
By the hypothesis on $u$,
$u(t_0,\cdot)\geq\eta\1_{\{z\}}=U^z(0,\cdot)$, so comparison gives
$u(t_0+s,z)\geq U^z(s,z)\geq\theta$ for every $s\geq S$.
Time independence of $F$ makes $S$ independent of $t_0$.
\end{proof}

\begin{proof}[Proof of Theorem~\ref{thm:axial}]
By \eqref{eq:transverse-factorization}, the evenness of $H_\perp$ in
Lemma~\ref{lem:transverse-kernel} and \eqref{eq:critical-system}, we have
$\nabla H(\lambda_*e_1)=c_*e_1$ and
$H(\lambda_*e_1)=c_*\lambda_*$. Lemma~\ref{lem:geometry} therefore identifies
$p_{e_1}=\lambda_*e_1$ and $w(e_1)=c_*$. Applying
Theorem~\ref{thm:upper} with $e=e_1$ yields
$\sup\{j\in\Z:u(t,j,0)\geq\theta\}\leq m_d(t)+C_\theta$.
Taking $e=-e_1$ gives the corresponding upper bound for $-j$.
By \eqref{eq:upper-main}, the positive-axis bound holds with $0$
replaced by any $y\in\Z^{d-1}$, with $C_\theta$ independent of $y$.

Choose $(j_0,y_0)$ with $u_0(j_0,y_0)=\eta_0>0$ and fix
$T\geq1$ and $\varepsilon_0>0$ from
Lemma~\ref{lem:product-subsolution}. Let $q$ solve
\eqref{eq:scalar-logistic} with
$q(0,j)=\varepsilon_0\1_{\{j_0\}}(j)$. In the original time $s$,
\eqref{eq:product-comparison} reads
\[
 u(s,j,y)\geq q(s-T,j)H_T(s-T,y;y_0)\quad\text{for }s\geq T.
\]
Set $a=a_{\mathrm{eff}}$ and fix $0<\mu<\lambda_*$.
Apply Proposition~\ref{prop:scalar-logarithmic-lower} with
$\beta=(d-1)/2$, $\sigma=1$ and
$K=\max\{\max_{y\in\mathcal C_\perp}a_\perp(y),K_F\}$, as fixed at the start of Section~\ref{sec:lower}. The definition of $b_\beta$ in that
proposition gives $b_\beta=(d+2)/(2\lambda_*)$.
Write $b=b_\beta=(d+2)/(2\lambda_*)$. With $s=T+t$,
\[
 j-c_*t+b\log(t+T)=j-m_d(s)+c_*T.
\]
Thus, at
$j_{\mathrm{seed}}(s)=\lceil m_d(s)-c_*T+1\rceil$,
the argument of $\Phi_\varepsilon$ belongs to $[1,2]$.
Equation~\eqref{eq:scalar-positive-seed} therefore gives
$q(s-T,j_{\mathrm{seed}}(s))\geq\eta>0$ for all large $s$.

Fix $A\geq0$ and recall that $n=d-1$.
If $s\geq|y_0|^2$ and $|y|\leq A\sqrt s$,
then $|y-y_0|\leq(A+1)\sqrt s$.
Put $\psi_{\min}=\min_{\mathcal C_\perp}\psi_\perp>0$,
and let $\nu_{\max}$ be the largest eigenvalue of
$\Sigma_\perp^{-1}$. In \eqref{eq:transverse-kernel-local}, the main term after multiplication
by $s^{n/2}$ is at least
\[
 L_A:=\frac{|\mathcal C_\perp|\psi_{\min}^2}
 {(2\pi)^{n/2}\sqrt{\det\Sigma_\perp}}
 \ee^{-\nu_{\max}(A+1)^2/2}>0.
\]
For $s^{n/2}\ee^{-as}K_\perp(s;y,y_0)$, the remainder is at most
$C_{A+1}s^{-1/2}$ in absolute value.
For $s\geq\max\{1,|y_0|^2,(2C_{A+1}/L_A)^2\}$,
\eqref{eq:transverse-factor} consequently gives
\[
 H_T(s-T,y;y_0)\geq\beta_\perp L_A/2=:h_A>0
 \quad\text{for }|y|\leq A\sqrt s,
\]
provided also $s\geq T$. Choose
$\eta_A=\min\{1/2,h_A\eta\}>0$. Combining \eqref{eq:product-comparison} with the bounds for
$q(s-T,j_{\mathrm{seed}}(s))$ and $H_T(s-T,y;y_0)$ gives
\begin{equation}\label{eq:section-seed}
 \inf_{\substack{y\in\Z^{d-1}\\|y|\leq A\sqrt s}}u(s,j_{\mathrm{seed}}(s),y)\geq\eta_A>0.
\end{equation}

For the axis, let $\eta_{\mathrm{ax}}$ be the constant in
\eqref{eq:section-seed} with $A=0$, and choose $S$ from
Lemma~\ref{lem:amplification} for $\eta_{\mathrm{ax}}$ and $\theta$.
Define $a_0(s)=m_d(s)-c_*T+1$. Choose
$s_0\geq\max\{2,2b/c_*\}$ so that
\eqref{eq:section-seed} with $A=0$ holds for all $s\geq s_0$.
Then $a_0'(s)=c_*-b/s\geq c_*/2$ on $[s_0,\infty)$.
For every integer $j\geq J_0:=\max\{1,\lceil a_0(s_0)\rceil\}$,
there is a unique $s_j\geq s_0$ such that $a_0(s_j)=j$.
At that time $j_{\mathrm{seed}}(s_j)=j$, so
\eqref{eq:section-seed} and Lemma~\ref{lem:amplification} give
$u(t,j,0)\geq\theta$ for $t\geq s_j+S$.
For $t-S\geq s_0$,
\[
 m_d(t)-m_d(t-S)
   =c_*S-b\log\frac{t}{t-S}\leq c_*S,
 \qquad
 a_0(t-S)\geq m_d(t)-c_*(S+T)+1.
\]
Choose $C_\theta\geq c_*(S+T)$. If
$J_0\leq j\leq m_d(t)-C_\theta$, then
$j\leq a_0(t-S)$ and $s_j\leq t-S$, which proves the lower bound
at these integers, including the right endpoint.

For the remaining integers $0\leq j<J_0$, heat comparison gives
\[
 u(1,j,0)\geq\eta_0\pheat_d(1,(j,0)-(j_0,y_0))>0.
\]
Indeed, a prescribed nearest-neighbor path from $(j_0,y_0)$ to
$(j,0)$ has positive probability. The minimum
$\eta_*:=\min_{0\leq j<J_0}u(1,j,0)$ is therefore positive.
Lemma~\ref{lem:amplification}, with $\eta=\eta_*$ and the fixed
$\theta$, gives $u(t,j,0)\geq\theta$ for all sufficiently large $t$,
uniformly over $0\leq j<J_0$.
{ Choosing $t_\theta^{\mathrm{ax}}\geq2$ large enough that $m_d(t)\geq C_\theta+1$ for all $t\geq t_\theta^{\mathrm{ax}}$, we conclude that $u(t,j,0)\geq\theta$ for every integer
$0\leq j\leq m_d(t)-C_\theta$.
Taking $j=\lfloor m_d(t)-C_\theta\rfloor\geq0$ yields}
\[
 \sup\{j\in\Z:u(t,j,0)\geq\theta\}\geq m_d(t)-C_\theta-1.
\]
{ Combining this lower bound with
$\sup\{j\in\Z:u(t,j,0)\geq\theta\}\leq m_d(t)+C_\theta$
establishes \eqref{eq:axial-position} after replacing $C_\theta$ by $C_\theta+1$.}

For \eqref{eq:diffusive-section}, choose $S_A$ from
Lemma~\ref{lem:amplification} for $\eta_{2A+1}$ and $\theta$.
Since $(2A+1)\sqrt{1-S_A/t}\to2A+1>A$, for all  {sufficiently} large $t$ the
bound \eqref{eq:section-seed} at time $t-S_A$ gives
\[
 u(t-S_A,j_{\mathrm{seed}}(t-S_A),y)\geq\eta_{2A+1}
 \quad\text{for }|y|\leq A\sqrt t.
\]
 {Because $S_A$ is independent of $y$, setting $j_t=j_{\mathrm{seed}}(t-S_A)$ gives}
$\inf_{\substack{y\in\Z^{d-1}\\|y|\leq A\sqrt t}}u(t,j_t,y)\geq\theta$.
For $t\geq2S_A$, the definition of $j_{\mathrm{seed}}$ and
$|\log(1-S_A/t)|\leq\log2$ give
\[
 |j_t-m_d(t)|\leq c_*(S_A+T)+2+b\log2.
\]
The lower bound for $u(t,j_t,y)$ on $|y|\leq A\sqrt t$ and the bound on
$|j_t-m_d(t)|$ prove \eqref{eq:diffusive-section}, including $A=0$.

To prove \eqref{eq:axial-leading-edge}, fix $A\geq0$ and
$\sigma>0$. Apply Proposition~\ref{prop:scalar-logarithmic-lower}
to $q$ with $\sigma+1$ in place of the width parameter; denote
the resulting constants by $\varepsilon$ and $t_0$.
Here $T$ and $\mu$ remain fixed. In the original time $s=T+t$,
put $\xi=j-m_d(s)$ and $r=\xi+c_*T$. If
$1\leq\xi\leq\sigma\sqrt s$ and $s\geq(c_*T)^2$, then
$0<r\leq(\sigma+1)\sqrt s$. The definition \eqref{eq:pulse-parameters} and the bound proved in Lemma~\ref{lem:delayed-pulse} give
$P_\mu(r)\geq r/2\geq(1+\xi)/4$. Thus
\eqref{eq:scalar-profile-lower} and \eqref{eq:product-comparison} give
\[
 \begin{aligned}
 u(s,j,y)&\geq q(s-T,j)H_T(s-T,y;y_0)\\
 &\geq\varepsilon h_A \ee^{-\lambda_*(\xi+c_*T)}P_\mu(\xi+c_*T)\\
 &\geq\frac{\varepsilon h_A {\ee}^{-\lambda_*c_*T}}4
                  (1+\xi)\ee^{-\lambda_*\xi}
 \end{aligned}
\]
for $|y|\leq A\sqrt s$, once
$s\geq\max\{2,T+t_0,|y_0|^2,(2C_{A+1}/L_A)^2,(c_*T)^2\}$.
We may therefore take
$b_{A,\sigma}=\varepsilon h_A {\ee}^{-\lambda_*c_*T}/4$.
For the upper bound, $p_{e_1}=\lambda_*e_1$ and
$H(p_{e_1})=c_*\lambda_*$ imply
$z_{e_1}(s,j,y)=\lambda_*\xi$. Equation~\eqref{eq:upper-main} gives
$u(s,j,y)\leq C(1+\lambda_*\xi)\ee^{-\lambda_*\xi}$, so
$B_{A,\sigma}=C\max\{1,\lambda_*\}$ suffices.
Choose $T_{A,\sigma}$ larger than
$\max\{2,T+t_0,|y_0|^2,(2C_{A+1}/L_A)^2,(c_*T)^2\}$ and
the time threshold in Theorem~\ref{thm:upper}. Both bounds then
hold on the longitudinal and transverse windows in
\eqref{eq:axial-leading-edge}.

The positive first-axis conclusions established above hold for every
reaction and initial value satisfying \eqref{eq:periodic-kpp},
\eqref{eq:u0} and \eqref{eq:transverse-linearization}.
We now transport these conclusions to the other admissible signed axes.
Fix $i$ such that  {the linearized growth rate $x\mapsto F_s(x,0)$ satisfies $F_s(x+me_i,0)=F_s(x,0)$} for all $x\in\Z^d$ and $m\in\Z$,
and fix $\varepsilon_1\in\{-1,1\}$.
Choose a signed permutation matrix $S$ with $Se_1=\varepsilon_1e_i$.
Write $Se_r=\varepsilon_re_{\vartheta(r)}$, where $\vartheta$ is a
permutation and $\varepsilon_r\in\{-1,1\}$, and define
\[
 u^S(t,z)=u(t,Sz),\qquad F^S(z,s)=F(Sz,s),\qquad
 u_0^S(z)=u_0(Sz),\qquad \Gamma^S=S^{-1}\Gamma.
\]
Since $S$ permutes the physical nearest-neighbor increments,
\[
 u_t^S=\Delta_{\Z^d}u^S+F^S(z,u^S),\qquad
 u^S(0,z)=u_0^S(z).
\]
The periods are $L_r^S=L_{\vartheta(r)}$.
Moreover, $F^S$ satisfies \eqref{eq:periodic-kpp} with the same $K_F$,
and $0\leq u_0^S\leq1$, $u_0^S\not\equiv0$ and $\supp u_0^S$ is finite.
For $z=(j,y)$, independence of  {$F_s(\cdot,0)$} along $e_i$ gives
\[
 \partial_sF^S((j,y),0)= {F_s(S(j,y),0)=F_s(S(0,y),0)}=:a_\perp^S(y).
\]
Let $a_{\mathrm{eff}}^S$ be the principal eigenvalue of
$\Delta_{\Z^{d-1}}+a_\perp^S$ on
$\mathcal C_\perp^S=\prod_{r=2}^d\{0,\ldots,L_r^S-1\}$.
Let $H^S(p)$ be the principal eigenvalue in \eqref{eq:cell} for $F^S$.
The change $x=Sz$ in that operator gives $H^S(p)=H(Sp)$.
Lifting a positive transverse eigenvector independently of $j$ then gives
\[
 a_{\mathrm{eff}}^S=H^S(0)=H(0)= {a_{\mathrm{eff}}}.
\]
Thus the critical parameters $\lambda_*,c_*$ and the function $m_d$
are unchanged. Also $\max_{y\in\mathcal C_\perp^S}a_\perp^S(y)=a_{\max}$,
so the scalar comparison may use the same $K$.
All hypotheses of the positive first-axis result are therefore satisfied
by $u^S,F^S,u_0^S$.
For each $\theta\in(0,1)$, let $C_\theta^S>0$ be the constant in its
conclusion \eqref{eq:axial-position}.
For each $\theta\in(0,1)$ and $A\geq0$, let $C_{\theta,A}^S>0$ and
$j_t^S\in\Z$ be supplied by \eqref{eq:diffusive-section}.
For each $A\geq0$ and $\sigma>0$, let $b_{A,\sigma}^S,B_{A,\sigma}^S>0$
be the constants in \eqref{eq:axial-leading-edge}.
Put $\xi=j-m_d(t)$. For fixed $\theta,A,\sigma$, these conclusions read
in the original coordinates, for all sufficiently large $t$,
\[
 \begin{gathered}
 \left|\sup\{j\in\Z:u(t,\varepsilon_1je_i)\geq\theta\}-m_d(t)\right|
       \leq C_\theta^S,\\
 u(t,\varepsilon_1je_i)\geq\theta
       \text{ for }0\leq j\leq m_d(t)-C_\theta^S,\\
 |j_t^S-m_d(t)|\leq C_{\theta,A}^S,\qquad
 \inf_{\substack{y\in\Z^{d-1}\\|y|\leq A\sqrt t}}
       u(t,S(j_t^S,y))\geq\theta,\\
 b_{A,\sigma}^S(1+\xi)\ee^{-\lambda_*\xi}
 \leq u(t,S(j,y))\leq
 B_{A,\sigma}^S(1+\xi)\ee^{-\lambda_*\xi}
 \text{ for }1\leq\xi\leq\sigma\sqrt t,\quad |y|\leq A\sqrt t.
 \end{gathered}
\]
Here $x=S(j,y)$ satisfies $x_i=\varepsilon_1j$ and
$|x-x_ie_i|=|y|$. These identities give the required estimates along
the signed $x_i$-axis, with leading coordinate $\varepsilon_1x_i-m_d(t)$.
In particular, $S=\operatorname{diag}(-1,1,\ldots,1)$ gives the negative
first axis. If  {$F_s(x,0)$ is constant on $\Z^d$}, every coordinate satisfies the
required independence condition, so the conclusions hold along all
coordinate axes.
\end{proof}

\begin{remark}\label{rem:fixed-transverse-rear}
For every $B\geq0$ and $\theta\in(0,1)$ there are $C,T_B>0$ such
that $u(t,j,y)\geq\theta$ whenever $t\geq T_B$,
$(j,y)\in\Z\times\Z^{d-1}$, $0\leq j\leq m_d(t)-C$ and $|y|\leq B$.
Fix the time shift $T$ used in \eqref{eq:product-comparison}, and
let $\eta_1>0$ be the constant $\eta_A$ in \eqref{eq:section-seed}
for $A=1$. Set
\[
 b=\frac{d+2}{2\lambda_*},\qquad
 a_0(s)=m_d(s)-c_*T+1,\qquad
 j_{\mathrm{seed}}(s)=\lceil a_0(s)\rceil.
\]
Choose $s_0$ so large that $a_0'(s)=c_*-b/s\geq c_*/2$ and
\eqref{eq:section-seed} with $A=1$ holds for $s\geq s_0$.
For every integer $j\geq\lceil a_0(s_0)\rceil$, define $s_j$ by
\[
 a_0(s_j)=j\text{ with }s_j\geq s_0.
\]
Strict monotonicity and $a_0(s)\to\infty$ give existence and
uniqueness, and $s_j\to\infty$. Choose an integer
$J_B\geq\max\{1,\lceil a_0(s_0)\rceil\}$
so that $s_j\geq B^2$ whenever $j\geq J_B$. At these times
$j_{\mathrm{seed}}(s_j)=j$, and $|y|\leq B$ implies $|y|\leq\sqrt{s_j}$.
Consequently \eqref{eq:section-seed} gives
\[
 u(s_j,j,y)\geq\eta_1>0
 \quad\text{for }j\geq J_B,\quad y\in\Z^{d-1},\quad |y|\leq B.
\]
Lemma~\ref{lem:amplification} gives a time $S_B$ independent of
$j,y$ such that $u(t,j,y)\geq\theta$ for $t\geq s_j+S_B$ at
these sites. Choose $C\geq c_*(S_B+T)$. For $t-S_B\geq s_0$,
\[
 \begin{aligned}
 m_d(t)-m_d(t-S_B)
   &=c_*S_B-b\log\frac{t}{t-S_B}\leq c_*S_B,\\
 a_0(t-S_B)&\geq m_d(t)-c_*(S_B+T)+1\geq m_d(t)-C.
 \end{aligned}
\]
Thus $J_B\leq j\leq m_d(t)-C$ implies $s_j\leq t-S_B$,
which proves the lower bound at every such $j$ and $|y|\leq B$.
The set $E_B=\{(j,y):0\leq j<J_B,\ y\in\Z^{d-1},\ |y|\leq B\}$
is finite. Heat comparison gives
$\eta_B:=\min_{(j,y)\in E_B}u(1,j,y)>0$.
Applying Lemma~\ref{lem:amplification} with $\eta_B$ and $\theta$
gives a uniform later time for this finite set. Taking $T_B$ larger
than that time and $s_0+S_B$ proves the assertion.

\end{remark}

\section{Pulsating fronts and local limits}\label{sec:profiles}

We prove Proposition~\ref{thm:rigidity} by comparison with time shifts and
spatial period translates under \eqref{eq:rigidity-hyp}.
The method follows \cite[Lemma~6.2]{HNRRPeriodic}; for the homogeneous
lattice arguments, see \cite[Lemmas~8.1--8.3]{HuAxial}.
Proposition~\ref{prop:profile-limits} applies the comparison to axial limits of $u$.

\begin{proposition}\label{thm:rigidity}
Assume \eqref{eq:periodic-kpp}.
Let $V:\R\times\Z^d\to(0,1)$ be an entire solution of
\eqref{eq:main}, and set $r=x\cdot n-ct$, where $|n|=1$ and $c>0$.
Suppose that, for some $\lambda,b,B,R_0>0$, the following hold
uniformly over $t\in\R$ and $x\in\Z^d$ in the indicated regions:
\begin{equation}\label{eq:rigidity-hyp}
 \lim_{R\to\infty}\inf_{r\leq-R}V(t,x)=1,\qquad
 b(1+r)\ee^{-\lambda r}\leq V(t,x)\leq B(1+r)\ee^{-\lambda r}
 \text{ for }r\geq R_0.
\end{equation}
Then $V$ is a pulsating front, strictly increasing in time, and
\begin{equation}\label{eq:covariance}
 V(t+n\cdot\ell/c,x+\ell)=V(t,x)\text{ for }\ell\in\Gamma.
\end{equation}
\end{proposition}

\begin{proposition}\label{prop:profile-limits}
Under the assumptions of Theorem~\ref{thm:axial}, let
$t_k\to\infty$ and $x_k=(j_k,y_k)$ satisfy
$|j_k-m_d(t_k)|\leq C$ and $|y_k|\leq A\sqrt{t_k}$ for fixed
$C,A\geq0$. Along a subsequence, $x_k\bmod\Gamma=r_0\in\mathcal C$ is
constant and $u(t_k+s,x_k+z)$ converges locally uniformly on $\R\times\Z^d$ to a critical
pulsating front of \eqref{eq:main} with reaction $F^{[r_0]}$,
direction $e_1$, and speed $c_*$.
\end{proposition}

\begin{lemma}\label{lem:entire-comparison}
Let $V,W:\R\times\Z^d\to[0,1]$ be entire solutions of
\eqref{eq:main} for the same periodic reaction $F$. Choose $\delta\in(0,1)$ and $\nu>0$ such that
$F_s(x,s)\leq-\nu$ for all $x\in\Z^d$ and $s\in[1-\delta,1]$.
If $V\geq1-\delta$ on a set $E\subset\R\times\Z^d$ and $W\leq V$
on $(\R\times\Z^d)\setminus E$, then $W\leq V$ everywhere.
Moreover, if $W\geq V$ and $W(t_0,x_0)=V(t_0,x_0)$ at some
$(t_0,x_0)\in\R\times\Z^d$, then $W\equiv V$.
\end{lemma}
\begin{proof}
For the first assertion, set $D=(W-V)_+$. The positive-part map is
Lipschitz, so $D$ is bounded and sitewise locally absolutely continuous.
For each $x$, let $N_x\subset\R$ be the null set where $D(\cdot,x)$
is not differentiable, and put $N=\bigcup_{x\in\Z^d}N_x$.
The countable union $N$ is null. Fix $t\notin N$.
Where $D(t,x)>0$, the point $(t,x)$ belongs to $E$, so
$1-\delta\leq V(t,x)<W(t,x)\leq1$ and
\[
 F(x,W)-F(x,V)
 =(W-V)\int_0^1F_s(x,V+\theta(W-V))\,\dd\theta
 \leq-\nu D.
\]
Here $D_t=(W-V)_t$ and $(W-V)(t,x)=D(t,x)$, whereas
$(W-V)(t,x+z)\leq D(t,x+z)$ for every neighbor. Hence
\[
 \begin{aligned}
 D_t(t,x)
 &=\sum_{\substack{z\in\Z^d\\|z|_1=1}}\bigl((W-V)(t,x+z)-D(t,x)\bigr)
       +F(x,W(t,x))-F(x,V(t,x))\\
 &\leq\sum_{\substack{z\in\Z^d\\|z|_1=1}}\bigl(D(t,x+z)-D(t,x)\bigr)-\nu D(t,x).
 \end{aligned}
\]
Where $D(t,x)=0$, nonnegativity makes this a minimum of the
differentiable function $D(\cdot,x)$, so $D_t(t,x)=0$, while
$\Delta_{\Z^d}D(t,x)=\sum_{\substack{z\in\Z^d\\|z|_1=1}}D(t,x+z)\geq0$.
Thus $D_t\leq\Delta_{\Z^d}D-\nu D$ at every site outside $N$.

Fix $h>0$ and define, for $s\in[t-h,t]$,
$H_h(s,\cdot)=\ee^{-\nu(s-t+h)}\ee^{(s-t+h)\Delta_{\Z^d}}D(t-h,\cdot)$.
Apply Lemma~\ref{lem:fixed-domain-comparison} on this finite interval
with domain $\Z^d$, rates $c(s,x,z)=1$ and reaction $G(s,x,r)=-\nu r$.
Both $D$ and $H_h$ are bounded and sitewise absolutely continuous,
their data agree at $t-h$, and the domain has no exterior neighbors.
Positivity of the heat semigroup and its preservation of constants give
$0\leq H_h\leq1$. We obtain
\[
 0\leq D(t,x)\leq\ee^{-\nu h}
   \bigl(\ee^{h\Delta_{\Z^d}}D(t-h,\cdot)\bigr)(x)
 \leq\ee^{-\nu h}\text{ for }x\in\Z^d.
\]
Letting $h\to\infty$ gives $D=0$, hence $W\leq V$.

For the second assertion, set $D=W-V\geq0$ and define
\[
 b(t,x)=\int_0^1 F_s\bigl(x,V(t,x)+\alpha D(t,x)\bigr)\,\dd\alpha.
\]
Subtracting \eqref{eq:main} for $V$ and $W$ gives
$D_t=\Delta_{\Z^d}D+b(t,x)D$, with $|b(t,x)|\leq L_F$.
We prove the contact conclusion in a form that also permits
exponential growth, as needed in Proposition~\ref{thm:rigidity}.
The bounded function $D=W-V$ is included by taking any $\rho>0$.
Consider $D\geq0$ satisfying
$D_t=\Delta_{\Z^d}D+b(t,x)D$, with bounded measurable
$b:\R\times\Z^d\to\R$. Suppose that $D(\cdot,x)$ is locally absolutely
continuous for each $x$ and that, for every compact interval $I\subset\R$,
there are $\rho,C_I>0$ such that
$|D(t,x)|\leq C_I\ee^{\rho|x|_1}$ for $(t,x)\in I\times\Z^d$.
We prove that $D(t_0,x_0)=0$ at any point implies $D\equiv0$.
For such a $\rho$, define
\[
 X_\rho=\bigl\{f:\Z^d\to\R:
       \|f\|_\rho:=\sup_{x\in\Z^d}\ee^{-\rho|x|_1}|f(x)|<\infty\bigr\}.
\]
Each neighbor shift has norm at most $\ee^\rho$, so
$\|\Delta_{\Z^d}\|_{X_\rho\to X_\rho}\le2d(\ee^\rho+1)$. Multiplication by
$b$ has norm at most $\|b\|_\infty$. For $s,t\in I$, the integral equation at each site gives
\[
 \|D(t,\cdot)-D(s,\cdot)\|_\rho
 \leq\bigl[2d(\ee^\rho+1)+\|b\|_\infty\bigr]|t-s|
       \sup_{v\in I}\|D(v,\cdot)\|_\rho.
\]
Thus $D$ is continuous with values in $X_\rho$.
If $D_1,D_2$ solve the same linear equation and
$D_1(s,\cdot)=D_2(s,\cdot)$, their coordinate integral equations give
\[
 \|D_1(t,\cdot)-D_2(t,\cdot)\|_\rho
 \leq[2d(\ee^\rho+1)+\|b\|_\infty]
       \int_s^t\|D_1(v,\cdot)-D_2(v,\cdot)\|_\rho\,\dd v
 \quad\text{for }s\leq t\text{ in }I.
\]
The integrand is measurable, being a supremum over the countable
lattice, and bounded on the compact interval under consideration.
Gronwall's inequality gives uniqueness forward from $s$.

The heat-kernel series acts on $X_\rho$: if $N_h$ is Poisson with
mean $2dh$, then
\[
 \sum_{y\in\Z^d}\pheat_d(h,x-y)\ee^{\rho|y|_1}
 \leq \ee^{\rho|x|_1}\mathbb E \ee^{\rho N_h}
 =\ee^{\rho|x|_1+2dh(\ee^\rho-1)}.
\]
Thus the series converges absolutely for every function in $X_\rho$.
Since $\Delta_{\Z^d}$ is bounded on $X_\rho$, its exponential is
differentiable there and is given by this heat-kernel series.
For $s\leq t$ in $I$, choose
$K\geq\sup_{\substack{v\in I\\x\in\Z^d}}|b(v,x)|$ and define
$H(t,x)=\ee^{-K(t-s)}(\ee^{(t-s)\Delta_{\Z^d}}D(s,\cdot))(x)$ and
$P(t,x)=(H(t,x)-D(t,x))_+$. On $\{P>0\}$,
\[
 (H-D)_t=\Delta_{\Z^d}(H-D)-K(H-D)-(K+b)D
             \leq\Delta_{\Z^d}(H-D)-K(H-D).
\]
Replacing neighboring differences by their positive parts increases
the Laplacian. At zeros of $P$, its time derivative is zero almost
everywhere. Thus $P_t\leq\Delta_{\Z^d}P-KP$ almost everywhere.
Dropping the nonpositive diagonal term and using $P(s,\cdot)=0$ gives
\[
 \|P(t,\cdot)\|_\rho
 \leq2d {\ee}^\rho\int_s^t\|P(v,\cdot)\|_\rho\,\dd v.
\]
Gronwall's inequality yields $P=0$, so $H\leq D$.
If $D(t_0,x_0)=0$, apply $H\leq D$ on $[t_0-h,t_0]$ with
$s=t_0-h$. For every $h>0$, this gives
\[
 0\geq \ee^{-Kh}\sum_{y\in\Z^d}\pheat_d(h,x_0-y)D(t_0-h,y)\geq0.
\]
For $z\in\Z^d$, choose a nearest-neighbor path of length $m=|z|_1$
from zero to $z$. The event that exactly these directed jumps occur in
their prescribed order before time $h$ has probability
$\ee^{-2dh}(2dh)^m(2d)^{-m}/m!=\ee^{-2dh}h^m/m!>0$.
Thus $\pheat_d(h,z)\geq \ee^{-2dh}h^{|z|_1}/|z|_1!>0$.
Every nonnegative summand in the preceding zero sum must vanish,
so $D(t_0-h,y)=0$ for all $y\in\Z^d$ and $h>0$.
Uniqueness in $X_\rho$ on each compact time interval then gives
$D\equiv0$ on $\R\times\Z^d$.

\end{proof}

\begin{proof}[Proof of Proposition~\ref{thm:rigidity}]
Write $r=x\cdot n-ct$.
Fix $\ell\in\Gamma$. For $\tau\in\R$, set
$W_\tau(t,x)=V(t+\tau,x+\ell)$ and
$\sigma(\tau)=n\cdot\ell-c\tau$.
Since $\ell\in\Gamma$ and $F$ is independent of time, spatial
periodicity makes $W_\tau$ a solution of \eqref{eq:main}. We claim that
$W_\tau\leq V$ for all sufficiently negative $\tau$. By \eqref{eq:periodic-kpp} and the finiteness of $\mathcal C$,
choose $\delta\in(0,1)$ and $\nu>0$ as in
Lemma~\ref{lem:entire-comparison}, and use the first limit in
\eqref{eq:rigidity-hyp} to choose $A$ with $V\geq1-\delta$ on $r<-A$.
Fix $R\geq\max\{R_0,1\}$, put $M=\max\{A,R\}$, and choose
$h>0$ with $ch-M\geq R_0$.
On $-A\leq r\leq R$, the coordinate at time $t-h$ belongs to
$[ch-M,ch+M]$. Since \eqref{eq:main} and $F(x,V)\geq0$ give
$V_t\geq-2dV$, integration and \eqref{eq:rigidity-hyp} yield
\[
 m_{A,R}:=\inf_{-A\leq r\leq R}V(t,x)
 \geq b \ee^{-2dh}\min_{v\in[ch-M,ch+M]}(1+v)\ee^{-\lambda v}>0.
\]
For $\sigma=\sigma(\tau)\geq A+R_0$,
\[
 \sup_{-A\leq r\leq R}W_\tau(t,x)
 \leq B(1+R+\sigma)\ee^{-\lambda(\sigma-A)}.
\]
For $r\geq R$ and $\sigma\geq0$, \eqref{eq:rigidity-hyp} gives
$W_\tau/V\leq(B/b)(1+r+\sigma)(1+r)^{-1}\ee^{-\lambda\sigma}
\leq(B/b)(1+\sigma)\ee^{-\lambda\sigma}$.
Choose $\sigma_0\geq A+R_0$ so that
$B(1+R+\sigma)\ee^{-\lambda(\sigma-A)}\leq m_{A,R}$ and
$(B/b)(1+\sigma)\ee^{-\lambda\sigma}\leq1$ for every $\sigma\geq\sigma_0$.
Both left sides tend to zero as $\sigma\to\infty$.
Thus $W_\tau\leq V$ on $r\geq-A$; Lemma~\ref{lem:entire-comparison},
with $E=\{(t,x):r<-A\}$, extends the ordering to all of $\R\times\Z^d$.
It holds for every $\tau\leq(n\cdot\ell-\sigma_0)/c$.

Define
\[
 \tau_*:=\sup\bigl\{\tau\in\R:
 W_{\tau'}(t,x)\leq V(t,x)\text{ for all }\tau'\leq\tau
 \text{ and }(t,x)\in\R\times\Z^d\bigr\}.
\]
The bound $W_\tau\leq V$ holds for every
$\tau\leq(n\cdot\ell-\sigma_0)/c$, so the set defining $\tau_*$ is nonempty. By the first limit in
\eqref{eq:rigidity-hyp}, $W_\tau(t,x)\to1>V(t,x)$ at each fixed
$(t,x)$ as $\tau\to\infty$, so $\tau_*$ is finite. Continuity in
$\tau$ gives $W_{\tau_*}\le V$ by passage to the limit from
$\tau<\tau_*$. We claim that
$\sigma_*:=n\cdot\ell-c\tau_*\leq0$.

Suppose $\sigma_*>0$. Put $\tau_0=\tau_*+1$ and choose recursively
$\tau_k\in(\tau_*,\min\{\tau_{k-1},\tau_*+1/k\})$ at which
$W_{\tau_k}\leq V$ fails. Such a shift exists in every right neighborhood
of $\tau_*$. Otherwise the ordering for all earlier shifts would extend
to a larger endpoint, contrary to the definition of $\tau_*$.
Lemma~\ref{lem:entire-comparison}, with $E=\{r<-A\}$, gives failure points
\[
 W_{\tau_k}(t_k,x_k)>V(t_k,x_k),\qquad
 r_k:=x_k\cdot n-ct_k\geq-A.
\]
Pass to a subsequence with $x_k\bmod\Gamma=r_0\in\mathcal C$ constant.
If $(r_k)$ has a bounded subsequence, pass to one for which
$r_k\to r_\infty$ and set $V_k(s,z)=V(t_k+s,x_k+z)$.
{ By the standard diagonal compactness argument in Remark~\ref{rem:diagonal-compactness}, we may pass to a subsequence along which $V_k$ converges locally in $C^1(\R)$ on finite subsets of $\Z^d$ to an entire solution $V_\infty$ of \eqref{eq:main} with reaction $F^{[r_0]}(z,v)=F(r_0+z,v)$ (with $|V_t|\le C_1:=2d+\max_{\mathcal C\times[0,1]}|F|$).}

Ordering at $\tau_*$ passes to the limit. Moreover,
\[
 |V_k(\tau_k,\ell)-V_k(\tau_*,\ell)|
 \leq C_1|\tau_k-\tau_*|\to0 \quad\text{as }k\to\infty,
\]
so $V_k(\tau_k,\ell)>V_k(0,0)$ implies
$V_\infty(\tau_*,\ell)\geq V_\infty(0,0)$.
The limiting order gives $V_\infty(\tau_*,\ell)\leq V_\infty(0,0)$,
so $V_\infty(\tau_*,\ell)=V_\infty(0,0)$.
Since $\ell\in\Gamma$, both $V_\infty(s,z)$ and
$V_\infty(s+\tau_*,z+\ell)$ solve the equation with reaction
$F^{[r_0]}$. Their nonnegative difference therefore vanishes
identically by Lemma~\ref{lem:entire-comparison}.

For each fixed $(s,z)$, put $v=r_\infty+n\cdot z-cs$.
The coordinate for $V_k(s,z)$ is $r_k+n\cdot z-cs\to v$.
Given $ {\varepsilon}>0$, choose $ {A_\varepsilon}$ so that
$V\geq1- {\varepsilon}$ for $r\leq- {A_\varepsilon}$.
For all sufficiently large $k$, $|r_k-r_\infty|<1$; hence
$V_\infty(s,z)\geq1- {\varepsilon}$ when $v\leq- {A_\varepsilon}-1$.
Passing the other inequalities in \eqref{eq:rigidity-hyp} to the limit gives
\[
 b(1+v)\ee^{-\lambda v}\leq V_\infty(s,z)
       \leq B(1+v)\ee^{-\lambda v}
 \quad\text{for }v\geq R_0+1.
\]
The constants and thresholds are independent of $(s,z)$.
In particular, $V_\infty$ is positive somewhere and is less than one
sufficiently far ahead. If it attained $0$, Lemma~\ref{lem:entire-comparison}
applied to $(0,V_\infty)$ would give $V_\infty\equiv0$; contact with $1$
would give $V_\infty\equiv1$. Both contradict the inherited bounds.
Thus $0<V_\infty<1$.

 Iterating $V_\infty(s,z)=V_\infty(s+\tau_*,z+\ell)$ gives
$V_\infty(0,0)=V_\infty(m\tau_*,m\ell)$ for every integer $m\geq1$.
Here $r_\infty+m\sigma_*$ tends to $+\infty$ as $m\to\infty$. The upper tail bound inherited from \eqref{eq:rigidity-hyp} therefore makes the right-hand side tend to zero, contradicting
$V_\infty(0,0)>0$.

If $(r_k)$ has no bounded subsequence, then $r_k\to+\infty$ because
$r_k\geq-A$. Set
$A_k=(1+r_k)\ee^{-\lambda r_k}$ and
$Z_k(s,z)=A_k^{-1}V(t_k+s,x_k+z)$. The tail bounds in
\eqref{eq:rigidity-hyp} bound $Z_k$ above and below by positive constants on each compact time interval and finite set of sites. By \eqref{eq:main},
\[
 (Z_k)_s(s,z)=\Delta_{\Z^d}Z_k(s,z)
       +A_k^{-1}F(r_0+z,A_kZ_k(s,z)).
\]
We next prove local $C^1$ convergence in time, on finite sets of sites,
along a subsequence to an entire solution of
\begin{equation}\label{eq:normalized-linear-limit}
 Z_s=\Delta_{\Z^d} Z+a(r_0+z)Z,
 \qquad b\ee^{-\lambda(n\cdot z-cs)}\leq Z(s,z)
                  \leq B\ee^{-\lambda(n\cdot z-cs)}.
\end{equation}
For the convergence of $Z_k$, fix a compact time interval $I$
and a finite set $E_0$ of sites. Put
 {$E=\{z\in\Z^d:|z-z_0|_1\leq2\text{ for some }z_0\in E_0\}$.} There
is $L<\infty$ with $|n\cdot z-cs|\leq L$ on $I\times E$.
For large $k$,
\[
 \frac{1+r_k+n\cdot z-cs}{1+r_k}\leq2,\qquad
 |Z_k(s,z)|\leq2B {\ee}^{\lambda L}.
\]
{ The equation for $Z_k$ and $0\leq F(x,r)\leq a(x)r$ bound
$(Z_k)_s$ uniformly on $I\times E_1$, where
$E_1=\{z\in\Z^d:|z-z_0|_1\leq1\text{ for some }z_0\in E_0\}$.
Differentiating the equation gives
\[
 (Z_k)_{ss}=\Delta_{\Z^d}(Z_k)_s
       +F_s(r_0+z,A_kZ_k)(Z_k)_s.
\]
Since $0\leq A_kZ_k\leq1$ and
$|F_s(r_0+z,A_kZ_k)|\leq L_F$, this identity bounds
$(Z_k)_{ss}$ uniformly on $I\times E_0$.} { Since Taylor's remainder gives
\[
 \left|A_k^{-1}F(r_0+z,A_kZ_k)-a(r_0+z)Z_k\right|
 \leq K_FA_kM^2\to0 \quad\text{as }k\to\infty
\]
with $M=2B\ee^{\lambda L}$, the diagonal compactness argument in Remark~\ref{rem:diagonal-compactness} yields a subsequence along which $Z_k$ converges locally in $C^1(\R)$ on finite subsets of $\Z^d$ to an entire solution $Z$ of $Z_s=\Delta_{\Z^d}Z+a(r_0+z)Z$.}
For every fixed $(s,z)$,
$(1+r_k+n\cdot z-cs)/(1+r_k)\to1$. Passing the tail bounds to the limit
therefore gives \eqref{eq:normalized-linear-limit} on all of
$\R\times\Z^d$, with the original constants $b,B$ and $Z>0$.

Since $V(t+\tau_*,x+\ell)\leq V(t,x)$, passage to the limit gives
$Z(s+\tau_*,z+\ell)\leq Z(s,z)$.
At the chosen points, $Z_k(\tau_k,\ell)>Z_k(0,0)$.
The local bound for $(Z_k)_s$ gives
$Z_k(\tau_k,\ell)-Z_k(\tau_*,\ell)\to0$.
Passing to the limit gives $Z(\tau_*,\ell)\geq Z(0,0)$, hence
$Z(\tau_*,\ell)=Z(0,0)$.
Define $D_Z(s,z)=Z(s,z)-Z(s+\tau_*,z+\ell)$.
The limiting order gives $D_Z\geq0$. Since $\ell\in\Gamma$, $D_Z$
solves the linear equation in \eqref{eq:normalized-linear-limit}. For $\rho>\lambda$,
recall $\|f\|_\rho=\sup_{z\in\Z^d}\ee^{-\rho|z|_1}|f(z)|$ and
$X_\rho=\{f:\|f\|_\rho<\infty\}$.
On a compact time interval $I$, the upper bound in
\eqref{eq:normalized-linear-limit} gives
\[
 |D_Z(s,z)|\leq
 B\max_{v\in I}\ee^{\lambda cv}
 \bigl(1+\ee^{\lambda c\tau_*-\lambda n\cdot\ell}\bigr)
 \ee^{\lambda|z|_1}\quad\text{for }s\in I.
\]
Thus $D_Z$ satisfies the exponential bound required by the contact
implication proved in Lemma~\ref{lem:entire-comparison}.
Indeed, $D_Z\geq0$, $D_Z(0,0)=0$, and
$(D_Z)_s=\Delta_{\Z^d}D_Z+a(r_0+z)D_Z$, with the bounded coefficient
$a(r_0+z)$. Consequently $D_Z\equiv0$, so
$Z(s,z)=Z(s+\tau_*,z+\ell)$ for all $(s,z)$.
Iterating $Z(s,z)=Z(s+\tau_*,z+\ell)$ and using
\eqref{eq:normalized-linear-limit} gives
$0<Z(0,0)\leq B\ee^{-m\lambda\sigma_*}$ for every integer
$m\geq1$. Letting $m\to\infty$ gives a contradiction.

Thus $\tau_*\geq n\cdot\ell/c$, so
$V(t+n\cdot\ell/c,x+\ell)\leq V(t,x)$.
Replacing $\ell$ by $-\ell$ gives
$V(t-n\cdot\ell/c,x-\ell)\le V(t,x)$; evaluation at
$(t+n\cdot\ell/c,x+\ell)$ then yields
$V(t,x)\le V(t+n\cdot\ell/c,x+\ell)$. In particular, for every
$\ell\in\Gamma$ and $(t,x)\in\R\times\Z^d$,
\[
 V(t+n\cdot\ell/c,x+\ell)\leq V(t,x)
            \leq V(t+n\cdot\ell/c,x+\ell).
\]
Hence $V(t+n\cdot\ell/c,x+\ell)=V(t,x)$, proving \eqref{eq:covariance}.

For time monotonicity, the comparison already proved for every
$\ell\in\Gamma$ states
\[
 V(t+\tau,x+\ell)\leq V(t,x)
 \quad\text{whenever }\tau\leq n\cdot\ell/c.
\]
Here the endpoint is included by continuity in $\tau$. Since
$0\in\Gamma$, setting $\ell=0$ and $\tau=-h$ gives
$V(t-h,x)\leq V(t,x)$ for every $h\geq0$. Thus
\[
 V_t(t,x)=\lim_{h\downarrow0}
               \frac{V(t,x)-V(t-h,x)}h\geq0.
\]
Define $D(t,x)=V_t(t,x)$ and
$F_0=\max_{\mathcal C\times[0,1]}|F|$,
$F_1=\max_{\mathcal C\times[0,1]}|F_s|$. Differentiating \eqref{eq:main} for $V$ gives
\[
 D_t=\Delta_{\Z^d}D+F_s(x,V)D,\qquad
 0\leq D\leq2d+F_0,\qquad |F_s(x,V)|\leq F_1.
\]
If $D(t_0,x_0)=0$, the implication proved in
Lemma~\ref{lem:entire-comparison} gives $D\equiv0$:
$D\geq0$ is bounded and satisfies
$D_t=\Delta_{\Z^d}D+F_s(x,V)D$, with $|F_s(x,V)|\leq F_1$.
Thus $V$ would be independent of time.
At a fixed vertex, however,
the bounds \eqref{eq:rigidity-hyp} give
$V(t,x)\to0$ as $t\to-\infty$ and $V(t,x)\to1$ as $t\to+\infty$.
Therefore $D(t,x)>0$, that is, $V_t(t,x)>0$ everywhere.

Define $\Phi(s,x)=V((n\cdot x-s)/c,x)$ for
$(s,x)\in\R\times\Z^d$. By \eqref{eq:covariance},
$\Phi(s,x+\ell)=\Phi(s,x)$ for every $\ell\in\Gamma$, and
$V_t>0$ gives
\[
 \partial_s\Phi(s,x)
 =-\frac1cV_t\!\left(\frac{n\cdot x-s}{c},x\right)<0.
\] The assumptions in \eqref{eq:rigidity-hyp} yield
$\Phi(s,x)\to1$ as $s\to-\infty$ and $\Phi(s,x)\to0$ as
$s\to+\infty$, uniformly in the cell variable $x$. Substituting $\Phi(s,x)=V((n\cdot x-s)/c,x)$ into
\eqref{eq:main} gives
\begin{equation}\label{eq:pulsating-profile-equation}
 -c\Phi_s(s,x)=\sum_{\substack{z\in\Z^d\\|z|_1=1}}
   \bigl(\Phi(s+n\cdot z,x+z)-\Phi(s,x)\bigr)+F(x,\Phi(s,x)).
\end{equation}
\end{proof}

\begin{proof}[Proof of Proposition~\ref{prop:profile-limits}]
By the bounds on $j_k-m_d(t_k)$ and $y_k/\sqrt{t_k}$ and the
finiteness of $\mathcal C$, choose a subsequence such that
$x_k\bmod\Gamma=r_0\in\mathcal C$ is constant,
$j_k-m_d(t_k)\to\delta\in\R$ and
$y_k/\sqrt{t_k}\to\eta\in\R^{d-1}$. { Setting $u_k(s,z)=u(t_k+s,x_k+z)$ for $s\geq-t_k$ and $z\in\Z^d$, the diagonal compactness argument in Remark~\ref{rem:diagonal-compactness} yields a subsequence converging locally in $C^1(\R)$ on finite subsets of $\Z^d$ to an entire solution $V$ of
\[
 V_s=\Delta_{\Z^d}V+F^{[r_0]}(z,V),\qquad
 F^{[r_0]}(z,v)=F(r_0+z,v).
\]}

Write $z=(z_1,z')\in\Z\times\Z^{d-1}$ and set $r=z_1-c_*s$.
For each fixed $(s,z)$, \eqref{eq:md} gives
$j_k+z_1-m_d(t_k+s)\to r+\delta$.
Since $p_{e_1}=\lambda_*e_1$, the variable in
\eqref{eq:upper-main} satisfies
\[
 z_{e_1}(t_k+s,x_k+z)
 =\lambda_*\bigl(j_k+z_1-m_d(t_k+s)\bigr)
 \to\lambda_*(r+\delta).
\]
For $r+\delta\geq1$, passage to the limit in \eqref{eq:upper-main}
therefore gives
$V(s,z)\leq C(1+\lambda_*(r+\delta))\ee^{-\lambda_*(r+\delta)}$.
Increasing the constant and the lower threshold for $r$, we obtain
$V(s,z)\leq B(1+r)\ee^{-\lambda_*r}$, with both independent of $(s,z)$.

Choose $(j_0,y_0)$ with $u_0(j_0,y_0)>0$.
Fix $T\geq1$ and $\varepsilon_0>0$ as in
Lemma~\ref{lem:product-subsolution}, and let $q$ solve
\eqref{eq:scalar-logistic} with
$q(0,j)=\varepsilon_0\1_{\{j_0\}}(j)$ for $j\in\Z$.
Choose $0<\mu<\lambda_*$ and apply
Proposition~\ref{prop:scalar-logarithmic-lower} to this $q$ with
$\beta=(d-1)/2$ and $\sigma=1$. Let $\varepsilon>0$ be the resulting constant in
\eqref{eq:scalar-profile-lower}.
For fixed $(s,z)$, the argument of $\Phi_\varepsilon$ in
\eqref{eq:scalar-profile-lower}, at time $t_k+s-T$ and site $j_k+z_1$, is
\[
 j_k+z_1-m_d(t_k+s)+c_*T\to r+\delta+c_*T.
\]
For every fixed $(s,z)$, all sufficiently large $k$ satisfy
$t_k+s-T\geq t_0$ and
$j_k+z_1-m_d(t_k+s)+c_*T\leq\sqrt{t_k+s}$, where $t_0$ is from
\eqref{eq:scalar-profile-lower}. Thus both its time threshold and
its upper bound on the argument of $\Phi_\varepsilon$ are satisfied.
Write $r_0=(r_{0,1},r_{0,\perp})$.
Since $(y_k+z'-y_0)/\sqrt{t_k+s}\to\eta$, equations
\eqref{eq:transverse-factor} and \eqref{eq:transverse-kernel-local} give
convergence of $H_T(t_k+s-T,y_k+z';y_0)$ to
\[
 \frac{\beta_\perp|\mathcal C_\perp|
        \psi_\perp(r_{0,\perp}+z')\psi_\perp(y_0)}
      {(2\pi)^{(d-1)/2}\sqrt{\det\Sigma_\perp}}
       \exp\!\left(-\tfrac12\eta^{\mathsf T}\Sigma_\perp^{-1}\eta\right).
\]
Taking the minimum over the periodic variable $z'\in\mathcal C_\perp$ gives
\[
 h_\eta:=\frac{\beta_\perp|\mathcal C_\perp|\psi_\perp(y_0)
                    \min_{y\in\mathcal C_\perp}\psi_\perp(y)}
                  {(2\pi)^{(d-1)/2}\sqrt{\det\Sigma_\perp}}
          \exp\!\left(-\tfrac12\eta^{\mathsf T}\Sigma_\perp^{-1}\eta\right)>0.
\]
By periodicity, $h_\eta$ bounds the limit of
$H_T(t_k+s-T,y_k+z';y_0)$ from below for every $z'\in\Z^{d-1}$,
independently of $s$.
In \eqref{eq:transverse-kernel-local}, the error after multiplication by
$(t_k+s)^{(d-1)/2}$ is $O((t_k+s)^{-1/2})$.
The threshold for $k$ may depend on $(s,z)$, but
$h_\eta$ and $\varepsilon$ do not.

For each fixed $(s,z)$ and all sufficiently large $k$,
\eqref{eq:product-comparison} and \eqref{eq:scalar-profile-lower}
give:
\[
 \begin{aligned}
 u_k(s,z)&\geq q(t_k+s-T,j_k+z_1)
                  H_T(t_k+s-T,y_k+z';y_0)\\
 &\geq\Phi_\varepsilon\bigl(j_k+z_1-m_d(t_k+s)+c_*T\bigr)
                  H_T(t_k+s-T,y_k+z';y_0).
 \end{aligned}
\]
The convergence of $H_T(t_k+s-T,y_k+z';y_0)$ and continuity of
$\Phi_\varepsilon$ therefore yield
\begin{equation}\label{eq:limit-pulse-bound}
 V(s,z)\geq h_\eta\Phi_\varepsilon(r+\delta+c_*T).
\end{equation}
Because the same $h_\eta$ and $\varepsilon$ apply at every fixed $(s,z)$,
\eqref{eq:limit-pulse-bound} holds on all of $\R\times\Z^d$.
Put $v=r+\delta+c_*T$. For
$r\geq\max\{1,1-2(\delta+c_*T)\}$, one has
$v\geq(1+r)/2>0$. Since $P_\mu(v)\geq v/2$ for $v\geq0$,
\eqref{eq:critical-pulse} gives
\[
 \Phi_\varepsilon(v)=\varepsilon \ee^{-\lambda_*v}P_\mu(v)
 \geq\frac{\varepsilon \ee^{-\lambda_*(\delta+c_*T)}}4
                         (1+r)\ee^{-\lambda_*r}.
\]
Set $b=h_\eta\varepsilon \ee^{-\lambda_*(\delta+c_*T)}/4>0$.
Then $b(1+r)\ee^{-\lambda_*r}\leq V(s,z)\leq B(1+r)\ee^{-\lambda_*r}$
for all sufficiently large $r$, with $b$, $B$ and the threshold
independent of $(s,z)$. These are the tail inequalities in
\eqref{eq:rigidity-hyp} for $n=e_1$, $c=c_*$ and $\lambda=\lambda_*$.

By \eqref{eq:limit-pulse-bound}, $V(s,z)>0$ at every point with
$r+\delta+c_*T=1$. The upper bound ahead of the front gives values below one for
all sufficiently large $r$. If $V(s_0,z_0)$ were $0$ at a finite point, the second assertion of
Lemma~\ref{lem:entire-comparison}, applied to $(0,V)$, would force $V\equiv0$.
If $V(s_0,z_0)$ were $1$, apply that assertion to $(V,1)$ to obtain $V\equiv1$.
Each conclusion contradicts the bounds just established. Hence $0<V<1$ everywhere.

For each $z\in\Z^d$, set $s_0=(z_1+\delta+c_*T-1)/c_*$.
By \eqref{eq:limit-pulse-bound},
$0<h_\eta\Phi_\varepsilon(1)\leq V(s_0,z)<1$.
For $\theta\in(0,1)$, Lemma~\ref{lem:amplification}, applied to $F^{[r_0]}$
with initial lower bound $h_\eta\Phi_\varepsilon(1)$, gives a time
$S_\theta$ independent of $z$ and $s_0$.
If $r\leq1-\delta-c_*T-c_*S_\theta$, then
$s-s_0=[1-(r+\delta+c_*T)]/c_*\geq S_\theta$ and $V(s,z)\geq\theta$.
The threshold is independent of $(s,z)$; taking the infimum in this
region and then $\theta\uparrow1$ gives
\[
 \lim_{R\to\infty}
 \inf_{\substack{(s,z)\in\R\times\Z^d\\z_1-c_*s\leq-R}}V(s,z)=1.
\]

For every $p\in\R^d$, the positive vector
$z\mapsto\phi_p(r_0+z)$ is an eigenvector of the cell operator in
\eqref{eq:cell} with $a(x)$ replaced by $a(r_0+x)$, with eigenvalue
$H(p)$. Thus $H(p)$ and $c_*$ are unchanged by translation.
Proposition~\ref{thm:rigidity}, applied to $F^{[r_0]}$ with $n=e_1$,
$c=c_*$ and $\lambda=\lambda_*$, shows that $V$ is a critical pulsating front.
To obtain a front for $F$, define $U(s,x)=V(s,x-r_0)$.
Substitution gives $U_s=\Delta_{\Z^d}U+F(x,U)$.
If $V(s,z)=\Phi(z_1-c_*s,z)$, then
\[
 U(s,x)=\Phi(x_1-c_*s-r_{0,1},x-r_0).
\]
Hence the profile of $U$ is $(\xi,x)\mapsto\Phi(\xi-r_{0,1},x-r_0)$.
It is $\Gamma$-periodic in $x$ and retains the end states and tail bounds,
after adjusting the constants and threshold. Therefore $U$ is a critical
pulsating front for $F$, with direction $e_1$ and speed $c_*$.
\end{proof}

\section{Critical tails and uniqueness}\label{sec:critical-tails}

We derive the critical tail from the bilateral Laplace transform of
\eqref{eq:pulsating-profile-equation}, then match amplitudes to prove
uniqueness by sliding. This identifies the limits in
Proposition~\ref{prop:profile-limits} and yields the selection and
arrival-time formulas. Throughout this section we assume
\eqref{eq:periodic-kpp} and \eqref{eq:front-critical-upper}, with $a(x)$
allowed to depend on every coordinate. Write
$a_{\max}=\max_{x\in\mathcal C}a(x)$.

Analytic continuation of the finite-cell eigenvalue, eigenvectors, and
spectral projection follows from the perturbation theory for a simple
isolated eigenvalue \cite[Chapter~II]{Kato}. { For background on bilateral Laplace and Fourier transforms, see \cite{Doetsch}.}

Fix $n\in\mathbb S^{d-1}$, let $\lambda=\lambda(n)$ and $c=c(n)$
be as in Lemma~\ref{lem:geometry}, and put $p=\lambda n$ and
$D_n=\tfrac12n^{\mathsf T}\nabla^2H(p)n>0$.
For fixed $n$, let $\dot\phi_p(x)=\left.\partial_\eta\phi_{\eta n}(x)\right|_{\eta=\lambda}$
denote the derivative of the right eigenvector normalized by
$\sum_{x\in\mathcal C}\phi_{\eta n}(x)=1$, as specified after \eqref{eq:cell}.
For a profile $\Phi$, define $v,g:\R\times\Z^d\to\R$ by
\begin{equation}\label{eq:tail-transform-data}
 v(s,x)=\ee^{\lambda s}\Phi(s,x),\qquad
 g(s,x)=\ee^{\lambda s}
               \bigl(a(x)\Phi(s,x)-F(x,\Phi(s,x))\bigr).
\end{equation}
For $h:\R\to\mathbb C^{\mathcal C}$, we write
\begin{equation}\label{eq:bilateral-transform-definition}
 \widehat h(z)=\int_{\R}\ee^{-zs}h(s)\,\dd s
\end{equation}
for its bilateral Laplace transform at each $z\in\mathbb C$ for which
the integral converges absolutely in every component. We apply this
notation to the vectors $(v(s,x))_{x\in\mathcal C}$ and
$(g(s,x))_{x\in\mathcal C}$.
For $q\in\mathbb C^d$, define $B_q$ by the finite-matrix formula
\eqref{eq:cell}. Near $p$, analytic continuation gives $\phi_q$ and
$\phi_q^*$ satisfying $B_q\phi_q=H(q)\phi_q$ and
$B_q^{\mathsf T}\phi_q^*=H(q)\phi_q^*$, with
$\sum_{x\in\mathcal C}\phi_q(x)=1$ and
$\sum_{x\in\mathcal C}\phi_q(x)\phi_q^*(x)=1$; the latter is the
bilinear pairing $(\phi_q^*)^{\mathsf T}\phi_q$.

\begin{lemma}\label{lem:tail-source}
Suppose that $\Phi$ is a pulsating profile solving
\eqref{eq:pulsating-profile-equation} with the fixed $n$ and $c$, with
end states $1,0$, and satisfying \eqref{eq:front-critical-upper}. Then $g\geq0$, $g\not\equiv0$,
and, for $k=0,1,2$,
\begin{equation}\label{eq:tail-source-derivatives}
 \max_{x\in\mathcal C}|\partial_s^kg(s,x)|\leq
 \begin{cases}
 C\ee^{\lambda s},&s\leq0,\\
 C(1+s)^2\ee^{-\lambda s},&s\geq0.
 \end{cases}
\end{equation}
Consequently, $\widehat g$ defined by
\eqref{eq:bilateral-transform-definition} is analytic in
$|\operatorname{Re}z|<\lambda$. For each $0<\delta<\lambda$,
\begin{equation}\label{eq:tail-source-transform}
 \sup_{|\sigma|\leq\delta}
 \norm{\widehat g(\sigma+\ii\tau)}
       \leq \frac{C_\delta}{1+\tau^2}\text{ for }\tau\in\R.
\end{equation}
 {Moreover,} $\widehat v(z)$ is analytic for
$0<\operatorname{Re}z<\lambda$ and satisfies
\begin{equation}\label{eq:tail-transform-equation}
 \mathcal M(z)\widehat v(z)=\widehat g(z),\qquad
 \mathcal M(z)=B_{p-zn}-c(\lambda-z)I.
\end{equation}

\end{lemma}
\begin{proof}
Let $C_\Phi$ be the constant in \eqref{eq:front-critical-upper}.
For $s\geq1$ and $|n\cdot b|\leq1$,
\[
 (1+s+n\cdot b)\ee^{-\lambda(s+n\cdot b)}
 \leq2\ee^\lambda(1+s)\ee^{-\lambda s}.
\]
The equation for $\Phi$ and $|F(x,r)|\leq a_{\max}r$ for $x\in\Z^d$ and $0\leq r\leq1$ therefore give
\[
 c|\Phi_s(s,x)|
 \leq [4d {\ee}^\lambda+2d+a_{\max}]
       C_\Phi(1+s)\ee^{-\lambda s}.
\]
Set $V(t,x)=\Phi(x\cdot n-ct,x)$. The lattice equation
$V_t=\Delta_{\Z^d}V+F(x,V)$ bounds $V_t$ because $0<V<1$.
Its right-hand side has the continuous derivative
$\Delta_{\Z^d}V_t+F_s(x,V)V_t$, so $V$ is $C^2$ in time and
$V_{tt}$ is bounded. These facts justify differentiating
\eqref{eq:pulsating-profile-equation}. For $s\geq2$, apply the
displayed bound for $\Phi_s$ at $s+n\cdot b\geq1$. With
$F_1=\max_{\mathcal C\times[0,1]}|F_s|$, this gives
\[
 c|\Phi_{ss}(s,x)|
 \leq [4d {\ee}^\lambda+2d+F_1]
       \frac{[4d {\ee}^\lambda+2d+a_{\max}]C_\Phi}{c}
       (1+s)\ee^{-\lambda s}.
\]
The identities
\[
 \Phi_s(s,x)=-\frac1c V_t\!\left(\frac{x\cdot n-s}{c},x\right),
 \qquad
 \Phi_{ss}(s,x)=\frac1{c^2}V_{tt}\!\left(\frac{x\cdot n-s}{c},x\right)
\]
therefore bound both profile derivatives on $\R\times\mathcal C$.
Increasing $C$ extends their decay bounds to all $s\geq0$.
Set $N_x(r)=a(x)r-F(x,r)$, so that
$N_x'(r)=a(x)-F_s(x,r)$ and $N_x''(r)=-F_{ss}(x,r)$.
By the $C^2$ regularity in \eqref{eq:periodic-kpp}, Taylor's formula gives
\[
 N_x(r)=-\int_0^r(r-v)F_{ss}(x,v)\,\dd v,\qquad
 N'_x(r)=-\int_0^rF_{ss}(x,v)\,\dd v.
\]
Thus $0\leq N_x(r)\leq K_Fr^2$, $|N'_x(r)|\leq2K_Fr$ and
$|N''_x(r)|\leq2K_F$ for $0\leq r\leq1$, uniformly in $x$;
the nonnegativity follows from $F(x,r)\leq a(x)r$.
Differentiating $g$ as defined in \eqref{eq:tail-transform-data} gives
\[
 \begin{split}
 g_s&=\ee^{\lambda s}\bigl(\lambda N_x(\Phi)+N_x'(\Phi)\Phi_s\bigr),\\
 g_{ss}&=\ee^{\lambda s}\bigl(\lambda^2N_x(\Phi)
          +2\lambda N_x'(\Phi)\Phi_s
          +N_x''(\Phi)\Phi_s^2+N_x'(\Phi)\Phi_{ss}\bigr).
 \end{split}
\]
By these Taylor inequalities and the decay estimates for $\Phi$ and
its derivatives, every bracketed term is
$O((1+s)^2\ee^{-2\lambda s})$ for $s\ge0$. For $s\le0$, the same
terms are bounded, which proves \eqref{eq:tail-source-derivatives}.
Since $g(s,x)=\ee^{\lambda s}N_x(\Phi(s,x))$ and $N_x\geq0$ on $[0,1]$,
we have $g\geq0$. Also,
$a(x)\Phi(s,x)-F(x,\Phi(s,x))\to a(x)>0$ as $s\to-\infty$.
By \eqref{eq:tail-transform-data}, $g(s,x)>0$ at every cell site
for all sufficiently negative $s$.

On a closed strip $|\operatorname{Re} z|\leq\delta<\lambda$, fix an integer $m\geq0$. After $m$ formal
differentiations with respect to $z$, the integrand defining $\widehat g$
has norm bounded uniformly in $z$ by a constant times
\[
 |s|^m(1+|s|)^2\ee^{-(\lambda-\delta)|s|},
\]
which is integrable for every integer $m\geq0$. Dominated
differentiation therefore makes $\widehat g$ analytic for $|\operatorname{Re} z|<\lambda$.
The bounds for $g,g_s$ also give
$\ee^{-zs}g(s)\to0$ and $\ee^{-zs}g_s(s)\to0$ at both ends,
uniformly on the closed strip. Integrating first on $[-R,R]$
and then letting $R\to\infty$ gives
$\int_\R \ee^{-zs}g_{ss}(s,\cdot)\,\dd s=z^2\widehat g(z)$.
By \eqref{eq:tail-source-derivatives}, this proves
\eqref{eq:tail-source-transform} for $|\operatorname{Im} z|\geq1$; the direct integral
bound covers $|\operatorname{Im} z|\leq1$.

Multiplying \eqref{eq:pulsating-profile-equation} by $\ee^{\lambda s}$ gives
\[
 c v_s(s,x)+\sum_{\substack{b\in\Z^d\\|b|_1=1}}\ee^{-p\cdot b}
                   v(s+n\cdot b,x+b)
       +(a(x)-2d-c\lambda)v(s,x)=g(s,x).
\]
Fix $0<\delta<\lambda/2$ and $\delta\leq\operatorname{Re} z\leq\lambda-\delta$.
The identity $v_s=\lambda v+\ee^{\lambda s}\Phi_s$ and the bounds for
$\Phi,\Phi_s$ give
\[
 \ee^{-\operatorname{Re} z\,s}\bigl(|v(s,x)|+|v_s(s,x)|\bigr)
 \leq C(1+|s|)\ee^{-\delta|s|}
 \quad\text{for }s\in\R,\ x\in\mathcal C.
\]
Multiplication by any fixed power of $|s|$ preserves integrability,
uniformly on this closed substrip. Dominated differentiation makes
$\widehat v$ analytic for $0<\operatorname{Re} z<\lambda$.
The same bound removes the endpoint terms in integration by parts,
so $\widehat{v_s}(z)=z\widehat v(z)$.
For a real shift $\alpha$,
the substitution $r=s+\alpha$ gives
\[
 \int_{\R}\ee^{-zs}v(s+\alpha,x)\,\dd s
 =\ee^{z\alpha}\int_{\R}\ee^{-zr}v(r,x)\,\dd r.
\]
Consequently the coefficient of $\widehat v(x+b)$ is
$\ee^{-p\cdot b+zn\cdot b}$ and the derivative term $cv_s$ contributes
$cz\widehat v(x)$, which gives exactly
$B_{p-zn}-c(\lambda-z)I$.

\end{proof}

\begin{lemma}\label{lem:critical-transform-gap}
There is $\delta_0\in(0,\lambda)$ such that $\mathcal M(z)$ is
invertible on $|\operatorname{Re}z|\leq\delta_0$, except at $z=0$.
For $|\operatorname{Im}z|$ sufficiently large in this strip,
\begin{equation}\label{eq:critical-resolvent-large}
 \norm{\mathcal M(z)^{-1}}\leq
                 \frac{C}{1+|\operatorname{Im}z|}.
\end{equation}
For nonzero $z$ near zero, $\mathcal M(z)^{-1}$ has the decomposition
\begin{equation}\label{eq:critical-resolvent-local}
 \mathcal M(z)^{-1}
 =\frac{\phi_{p-zn}(\phi^*_{p-zn})^{\mathsf T}}{k(z)}
       +\mathcal R(z),\qquad
 k(z)=H(p-zn)-c(\lambda-z)=D_nz^2+O(z^3),
\end{equation}
where $\mathcal R$ is analytic and the eigenvectors are analytically
continued from the real parameter $p$.
\end{lemma}
\begin{proof}
At zero, Lemma~\ref{lem:spectral} shows that $B_p-H(p)I$ has a simple zero eigenvalue and all its other
eigenvalues are different from zero. Analytic perturbation of this
simple isolated eigenvalue \cite[Chapter~II]{Kato} supplies its analytic
spectral projection near zero.

Let
$P(z)=\phi_{p-zn}(\phi^*_{p-zn})^{\mathsf T}$. Then
$P(z)^2=P(z)$ and
$\mathcal M(z)P(z)=P(z)\mathcal M(z)=k(z)P(z)$.
At zero, $\mathcal M(0)+P(0)$ acts as the identity on the
one-dimensional kernel and as the invertible restriction of
$\mathcal M(0)$ on its complementary spectral subspace. It
therefore remains invertible near zero, and
\[
 \mathcal R(z)=(\mathcal M(z)+P(z))^{-1}(I-P(z))
\]
is analytic there. Multiplication by $\mathcal M(z)$ verifies
\eqref{eq:critical-resolvent-local} whenever $k(z)\ne0$. Taylor
expansion along the line $p-zn$ gives
\[
 H(p-zn)=H(p)-z\,n\cdot\nabla H(p)
          +\tfrac12z^2n^{\mathsf T}\nabla^2H(p)n+O(z^3).
\]
By Lemma~\ref{lem:geometry}, $H(p)=c\lambda$ and
$n\cdot\nabla H(p)=c$, so $k(z)=D_nz^2+O(z^3)$. Shrinking
the disk until $|k(z)/z^2-D_n|\leq D_n/2$ gives
$|k(z)|\geq(D_n/2)|z|^2$ there.

To exclude further zeros on the imaginary axis, suppose
$\mathcal M(\ii\tau)\xi=0$ and $\xi\ne0$, and write
$\xi(x)=\phi_p(x)\eta(x)$. The rate formula \eqref{eq:Q} shows that conjugation by $\phi_p$ gives
\begin{equation}\label{eq:imaginary-axis-equation}
 \sum_{\substack{b\in\Z^d\\|b|_1=1}}q_p(x,b)
      \bigl(\ee^{\ii\tau n\cdot b}\eta(x+b)-\eta(x)\bigr)
                +\ii c\tau\eta(x)=0\text{ for }x\in\mathcal C.
\end{equation}
Choose $x_0\in\mathcal C$ such that
$|\eta(x_0)|=\max_{x\in\mathcal C}|\eta(x)|>0$.
Multiplying \eqref{eq:imaginary-axis-equation} at $x_0$ by
$\overline{\eta(x_0)}$ and taking real parts gives a sum of terms
\[
 q_p(x_0,b)\left(
 \operatorname{Re}\bigl(\ee^{\ii\tau n\cdot b}\eta(x_0+b)
                        \overline{\eta(x_0)}\bigr)-|\eta(x_0)|^2\right)
 \leq0\text{ for }b\in\Z^d,\ |b|_1=1.
\]
Their sum is zero, and each rate is positive by \eqref{eq:Q}.
Equality in the modulus and real-part estimates therefore gives
$\ee^{\ii\tau n\cdot b}\eta(x_0+b)=\eta(x_0)$ for every such $b$.
In particular, $|\eta|$ is maximal at each neighboring site.
By the irreducibility in Lemma~\ref{lem:spectral}, repeating this
argument gives $\ee^{\ii\tau n\cdot b}\eta(x+b)=\eta(x)$ for all
$x\in\mathcal C$ and $b\in\Z^d$ with $|b|_1=1$. The sum in
\eqref{eq:imaginary-axis-equation} consequently vanishes everywhere,
so $c\tau\eta=0$. Since the maximal modulus is positive and $c>0$,
we obtain $\tau=0$. 
An explicit large-frequency bound follows by setting
\[
 C_B=\sup_{|\operatorname{Re} z|\leq1}
       \|B_{p-zn}-c\lambda I\|<\infty.
\]
Every exponential entry has modulus
$\ee^{-p\cdot b+\operatorname{Re} z\,n\cdot b}$, so this supremum is finite.
For $|\operatorname{Re} z|\leq1$ and $|z|\geq2C_B/c$, the factorization
\[
 \mathcal M(z)=cz\left[I+(cz)^{-1}(B_{p-zn}-c\lambda I)\right]
\]
has inverse norm at most $2/(c|z|)$ by the geometric matrix series.
The bound $2/(c|z|)$ gives \eqref{eq:critical-resolvent-large}.
Choose $\varepsilon_0>0$ so that \eqref{eq:critical-resolvent-local} excludes zeros for
$0<|z|<2\varepsilon_0$, and choose $R$ so that \eqref{eq:critical-resolvent-large}
excludes zeros for $|\operatorname{Im}z|>R$. For $0<\delta\leq\min\{1,\lambda/2\}$, the set
$\{z\in\mathbb C:|\operatorname{Re}z|\leq\delta,\ \varepsilon_0\leq|z|,\
|\operatorname{Im}z|\leq R\}$ is compact.
If such sets contained zeros for arbitrarily small $\delta$, a
subsequence of those zeros would converge to a nonzero point of the
imaginary axis where $\det\mathcal M$ vanishes.
This contradicts the exclusion of nonzero imaginary roots.
Choosing $\delta_0$ sufficiently small
therefore proves the lemma for the fixed direction $n$.

\end{proof}

\begin{lemma}\label{lem:critical-tail-expansion}
Every $\Phi$ satisfying the hypotheses of Theorem~\ref{thm:critical-uniqueness} satisfies
\eqref{eq:critical-tail-expansion} and
\eqref{eq:critical-amplitude}. In particular, there exist $b,B>0$ and $s_0\geq1$ such that
$b\,s\ee^{-\lambda s}\phi_p(x)\leq\Phi(s,x)
\leq B\,s\ee^{-\lambda s}\phi_p(x)$ for all $s\geq s_0$
and $x\in\mathcal C$.
\end{lemma}
\begin{proof}
Choose $0<\delta<\delta_0$. On $\operatorname{Re}z=\delta$,
\eqref{eq:tail-transform-equation} gives
$\widehat v(z)=\mathcal M(z)^{-1}\widehat g(z)$.
By Lemmas~\ref{lem:tail-source} and~\ref{lem:critical-transform-gap},
the product $\mathcal M(z)^{-1}\widehat g(z)$ is
$O((1+|\operatorname{Im}z|)^{-3})$ on both vertical lines
$\operatorname{Re}z=\pm\delta$.
Set $f_\delta(s)=\ee^{-\delta s}v(s,\cdot)$.
The bounds in Lemma~\ref{lem:tail-source}, with
$0<\delta<\delta_0<\lambda$, give
\[
 \begin{aligned}
 f_\delta(s)&=O((1+s)\ee^{-\delta s})
                         &&\text{as }s\to+\infty,\\
 f_\delta(s)&=O(\ee^{(\lambda-\delta)s})
                         &&\text{as }s\to-\infty.
 \end{aligned}
\]
{ Thus $f_\delta\in L^1(\R;\mathbb C^{\mathcal C})\cap C_b(\R;\mathbb C^{\mathcal C})$. Moreover, \eqref{eq:tail-source-transform} and \eqref{eq:critical-resolvent-large}, together with the invertibility in Lemma~\ref{lem:critical-transform-gap}, show that its Fourier transform $\tau\mapsto\widehat v(\delta+\ii\tau)=\mathcal M(\delta+\ii\tau)^{-1}\widehat g(\delta+\ii\tau)$ belongs to $L^1(\R;\mathbb C^{\mathcal C})$. Classical $L^1$ Fourier inversion therefore yields, after multiplying by $\ee^{\delta s}$, for every $s\in\R$,}
\begin{equation}\label{eq:tail-inversion}
 v(s,\cdot)=\frac{1}{2\pi\ii}
       \int_{\delta-\ii\infty}^{\delta+\ii\infty}
       \ee^{zs}\mathcal M(z)^{-1}\widehat g(z)\,\dd z.
\end{equation}

By Lemmas~\ref{lem:tail-source} and~\ref{lem:critical-transform-gap},
$\mathcal M(z)^{-1}\widehat g(z)$ gives a meromorphic continuation of
$\widehat v$ to $|\operatorname{Re}z|\leq\delta$, with its only possible
pole at zero.
On $z=\sigma\pm\ii R$, $|\sigma|\le\delta$, those lemmas give
$\|\mathcal M(z)^{-1}\widehat g(z)\|\le CR^{-3}$.

{ Applying Cauchy's residue theorem to $\mathcal F_s(z)=\ee^{zs}\mathcal M(z)^{-1}\widehat g(z)$ on the positively oriented rectangle $\mathcal R_R=[-\delta,\delta]\times[-R,R]$ yields
\[
 \int_{\partial\mathcal R_R}\mathcal F_s(z)\,\dd z
       =2\pi\ii\,\operatorname*{Res}_{z=0}\mathcal F_s(z).
\]
Each horizontal side of $\mathcal R_R$ has length $2\delta$, and its contribution is bounded in norm by $2\delta C R^{-3}\ee^{\delta|s|}$ for all sufficiently large $R$. Letting $R\to\infty$, the horizontal integrals vanish and we obtain
\[
 v(s,\cdot)=\operatorname*{Res}_{z=0}
      \left(\ee^{zs}\mathcal M(z)^{-1}\widehat g(z)\right)
 +\frac{1}{2\pi\ii}
   \int_{-\delta-\ii\infty}^{-\delta+\ii\infty}
      \ee^{zs}\mathcal M(z)^{-1}\widehat g(z)\,\dd z.
\]
On the shifted vertical contour $\operatorname{Re}z=-\delta$, we have
\[
 \left\|\frac1{2\pi\ii}\int_{-\delta-\ii\infty}^{-\delta+\ii\infty}
     \ee^{zs}\mathcal M(z)^{-1}\widehat g(z)\,\dd z\right\|
 \leq\frac{\ee^{-\delta s}}{2\pi}
       \int_{\R}\|\mathcal M(-\delta+\ii\tau)^{-1}
                       \widehat g(-\delta+\ii\tau)\|\,\dd\tau.
\]}
The integral on the right is finite and independent of $s$, so the
integral on $\operatorname{Re} z=-\delta$ is $O(\ee^{-\delta s})$ uniformly on the cell
for $s\geq0$.

Set
$\gamma(z)=(\phi^*_{p-zn})^{\mathsf T}\widehat g(z)$.
The local decomposition \eqref{eq:critical-resolvent-local} gives
\[
 \mathcal M(z)^{-1}\widehat g(z)
      =\frac{\phi_{p-zn}\gamma(z)}{k(z)}
                   +\mathcal R(z)\widehat g(z).
\]
For the analytic eigenvectors normalized after \eqref{eq:cell},
$\phi_{p-zn}=\phi_p-z\dot\phi_p+O(z^2)$.
Writing $H_n'''=\left.\partial_\eta^3H(\eta n)\right|_{\eta=\lambda}$,
the critical identities in Lemma~\ref{lem:geometry} give
$k(z)=D_nz^2-\tfrac16H_n'''z^3+O(z^4)$.
Multiply the expansions
\[
 \begin{aligned}
 k(z)^{-1}
   &=D_n^{-1}z^{-2}+\frac{H_n'''}{6D_n^2}z^{-1}+O(1),\\
 \phi_{p-zn}\gamma(z)
   &=\phi_p\gamma(0)
      +z\bigl(\phi_p\gamma'(0)-\dot\phi_p\gamma(0)\bigr)
      +O(z^2).
 \end{aligned}
\]
The term $\mathcal R(z)\widehat g(z)$ is analytic at zero and contributes
no negative powers. With
$B_\Phi=\gamma'(0)/D_n+H_n'''\gamma(0)/(6D_n^2)$,
multiplication gives
\begin{equation}\label{eq:tail-Laurent}
 \mathcal M(z)^{-1}\widehat g(z)
  =\frac{A_\Phi\phi_p}{z^2}
      +\frac{B_\Phi\phi_p-A_\Phi\dot\phi_p}{z}+O(1),\qquad
 A_\Phi=\frac{\gamma(0)}{D_n}.
\end{equation}
All factors are real for real $z$ near zero, so their Taylor coefficients
and $B_\Phi$ are real.
Multiplying by $\ee^{zs}=1+zs+O(z^2)$ gives the residue
$(A_\Phi s+B_\Phi)\phi_p-A_\Phi\dot\phi_p$.
By the definition of $\gamma$,
\[
 \gamma(0)=\sum_{x\in\mathcal C}\phi_p^*(x)
                 \int_{\R}g(s,x)\,\dd s.
\]
Substitution of \eqref{eq:tail-transform-data} therefore gives
\eqref{eq:critical-amplitude}. Its integrand is nonnegative, and
Lemma~\ref{lem:tail-source} makes it positive for all sufficiently
negative $s$, since $\phi_p^*>0$.
Thus $A_\Phi>0$, and division of the shifted contour formula by
$\ee^{\lambda s}$ proves \eqref{eq:critical-tail-expansion},
uniformly over the finite cell. Since $A_\Phi>0$ and
$\min_{\mathcal C}\phi_p>0$, the expansion gives both asserted tail bounds
for large $s$.
\end{proof}

\begin{proof}[Proof of Theorem~\ref{thm:critical-uniqueness}]
Lemma~\ref{lem:critical-tail-expansion} proves both the expansion and the amplitude identity.
Since $s\leq1+s\leq2s$ for $s\geq1$ and $\phi_p$ has a
positive minimum on $\mathcal C$, the two-sided tail bounds in
Lemma~\ref{lem:critical-tail-expansion} imply the bounds in
\eqref{eq:rigidity-hyp} after adjusting the constants. Together with
the uniform limit $\Phi(s,x)\to1$ as $s\to-\infty$, they verify the
hypotheses of Proposition~\ref{thm:rigidity} for
$V(t,x)=\Phi(x\cdot n-ct,x)$, which solves \eqref{eq:main} for all time.
The proof of Proposition~\ref{thm:rigidity} establishes
$V_t(t,x)>0$ under these hypotheses. Hence
\[
 -c\Phi_s(s,x)=V_t((x\cdot n-s)/c,x)>0.
\]
Since $c>0$, this proves $\Phi_s(s,x)<0$ for every $(s,x)$.

Let $\Phi$ and $\Psi$ be two profiles in the theorem, and set
\begin{equation}\label{eq:amplitude-matching-shift}
 h_0=\frac1\lambda\log\frac{A_\Psi}{A_\Phi}.
\end{equation}
Expansion~\eqref{eq:critical-tail-expansion} gives, uniformly in the cell,
\begin{equation}\label{eq:tail-profile-ratio}
 \frac{\Psi(s+h,x)}{\Phi(s,x)}
       \to \frac{A_\Psi}{A_\Phi}\ee^{-\lambda h}
                    \text{ as }s\to+\infty.
\end{equation}
For uniformity in $h$, fix a compact interval $I$ of shifts. By
\eqref{eq:critical-tail-expansion} and $\min_{\mathcal C}\phi_p>0$,
\[
 \frac{\Psi(s+h,x)}{\Phi(s,x)}
 =\ee^{-\lambda h}\frac{A_\Psi(s+h)+O_I(1)}{A_\Phi s+O(1)},
\]
uniformly for $x\in\mathcal C$ and $h\in I$ as $s\to+\infty$.
For large $s$ the denominator is at least $A_\Phi s/2$.
Subtracting $(A_\Psi/A_\Phi)\ee^{-\lambda h}$ therefore bounds
the difference by $C_I/s$.

We first show that, for all sufficiently large positive $h$, one has
$\Psi(s+h,x)\leq\Phi(s,x)$ for all $(s,x)\in\R\times\mathcal C$.
Choose $\delta\in(0,1)$ and $\nu>0$ as in
Lemma~\ref{lem:entire-comparison}, so that
$F_s(x,r)\leq-\nu$ for every $x\in\Z^d$ and $r\in[1-\delta,1]$.
Take $R>0$ so that $\Phi(s,x)\geq1-\delta$ for $s<-R$ and
$x\in\mathcal C$. Choose $R_1>R$ so that the two-sided tail bounds
in Lemma~\ref{lem:critical-tail-expansion} hold for both $\Phi$ and
$\Psi$ when $s\geq R_1$.
These bounds give
\[
 \frac{\Psi(s+h,x)}{\Phi(s,x)}
 \leq C(1+h)\ee^{-\lambda h},
 \text{ for }s\geq R_1,\ h\geq0,\ x\in\mathcal C,
\]
which is less than one for sufficiently large $h$.
On $[-R,R_1]\times\mathcal C$, the positive minimum of $\Phi$ and
the uniform convergence $\Psi(s+h,x)\to0$ as $h\to\infty$
give the same ordering. To extend it to $s<-R$, apply
Lemma~\ref{lem:entire-comparison} to
$V(t,x)=\Phi(x\cdot n-ct,x)$ and
$W_h(t,x)=\Psi(x\cdot n-ct+h,x)$.
Both $V$ and $W_h$ solve \eqref{eq:main}. Applying the lemma with
$E=\{(t,x)\in\R\times\Z^d:x\cdot n-ct<-R\}$ gives
$W_h\leq V$ on $E$, and hence the ordering for all $s\in\R$.

By the strict decrease of $\Psi$ proved above, the set
$\{h\in\R:\Psi(s+h,x)\leq\Phi(s,x)\text{ for all }
(s,x)\in\R\times\mathcal C\}$ is an upper interval. It contains all sufficiently large $h$, and
$\Psi(s+h,x)\to1>\Phi(s,x)$ at each fixed point as $h\to-\infty$.
Its infimum $h_*$ is therefore finite.
Continuity in $h$ gives $\Psi(s+h_*,x)\le\Phi(s,x)$ at every
$(s,x)$ by passage to the limit from $h>h_*$. Taking $s\to+\infty$ and using
\eqref{eq:tail-profile-ratio} gives $h_*\geq h_0$.
Suppose that $h_*>h_0$. If the ordered profiles at $h_*$ touched,
the second assertion of Lemma~\ref{lem:entire-comparison}, applied to
$(W_{h_*},V)$, would give $W_{h_*}\equiv V$.
Equation~\eqref{eq:tail-profile-ratio} would then imply $h_*=h_0$.
Thus $\Phi(s,x)>\Psi(s+h_*,x)$ everywhere.

Take
$ {\varepsilon}=(h_*-h_0)/2$ and
$\vartheta=\ee^{-\lambda(h_*- {\varepsilon}-h_0)}<1$.
By \eqref{eq:tail-profile-ratio}, choose $R_1$ so that
$\Psi(s+h,x)/\Phi(s,x)$ is
at most $(1+\vartheta)/2$ for $s\geq R_1$ and
$h\in[h_*- {\varepsilon},h_*]$. On the compact strip let
\[
 m=\min_{\substack{-R\leq s\leq R_1\\x\in\mathcal C}}
              [\Phi(s,x)-\Psi(s+h_*,x)]>0,
 \qquad L=\sup_{\R\times\mathcal C}|\Psi_s|<\infty.
\]
Choose $0< {\varepsilon'}<\min\{ {\varepsilon},m/(2(L+1))\}$.
The mean-value estimate gives
$\Phi(s,x)-\Psi(s+h_*- {\varepsilon'},x)\geq m/2$ on the strip.
For $s\geq R_1$, $\Psi(s+h_*- {\varepsilon'},x)/\Phi(s,x)\leq(1+\vartheta)/2$
because $h_*- {\varepsilon'}$
lies in the chosen interval. Lemma~\ref{lem:entire-comparison}
then supplies the ordering for $s<-R$, yielding the required
contradiction to the definition of $h_*$.
Hence $h_*=h_0$ and $\Psi(\cdot+h_0,\cdot)\leq\Phi$.

The comparison $\Psi(\cdot+h_0,\cdot)\leq\Phi$ holds for any pair of profiles with the
stated end states and critical tails. For the pair $(\Psi,\Phi)$,
the translation in \eqref{eq:amplitude-matching-shift} is
$\lambda^{-1}\log(A_\Phi/A_\Psi)=-h_0$, so it gives
$\Phi(s-h_0,x)\leq\Psi(s,x)$ for every $(s,x)$.
Replacing $s$ by $s+h_0$ and writing the comparison for $(\Phi,\Psi)$
at the same point gives
\[
 \Phi(s,x)\leq\Psi(s+h_0,x)\leq\Phi(s,x).
\]
Thus $\Psi(s+h_0,x)=\Phi(s,x)$ everywhere, proving uniqueness.
\end{proof}

\begin{remark}\label{rem:constant-linearization-tail}
Under \eqref{eq:common-linearization}, $\phi_p\equiv|\mathcal C|^{-1}$
for every $p$, and $\dot\phi_p=0$.
In the axial direction the expansion becomes
\[
 \Phi_*(s,x)=\ee^{-\lambda_*s}
                 \bigl(A_*s+B_*+O(\ee^{-\delta s})\bigr),
                    \text{ as }s\to+\infty,
\]
for some $\delta>0$, where $A_*=A_{\Phi_*}/|\mathcal C|>0$ and
$B_*=B_{\Phi_*}/|\mathcal C|$ are independent of $x\in\mathcal C$.
For a general periodic linearization, $-A_\Phi\dot\phi_p(x)$ also appears
in the constant term of \eqref{eq:critical-tail-expansion}.
\end{remark}

Under \eqref{eq:transverse-linearization}, the eigenvalue equation
\eqref{eq:transverse-eigenpair} and the normalization after
\eqref{eq:cell} give
\[
 \phi_{\eta e_1}(j,y)
 =\frac{\psi_\perp(y)}{L_1\displaystyle\sum_{z\in\mathcal C_\perp}\psi_\perp(z)}
 \text{ for }(j,y)\in\mathcal C,\ \eta\in\R.
\]
For a profile in direction $e_1$, one has $\lambda=\lambda_*$.
Since $\phi_{\eta e_1}$ is independent of $\eta$,
$\dot\phi_{\lambda_*e_1}=0$, and both $A_\Phi\phi_{\lambda_*e_1}$ and
$B_\Phi\phi_{\lambda_*e_1}-A_\Phi\dot\phi_{\lambda_*e_1}$ are scalar
multiples of $\psi_\perp$.

\begin{proof}[Proof of Theorem~\ref{thm:selection}]
Choose an entire limit $V$ from Proposition~\ref{prop:profile-limits}
with residue $r_0\in\mathcal C$. Since $V$ solves
\eqref{eq:main} with reaction $F^{[r_0]}$, the function
$U(t,x)=V(t,x-r_0)$ solves \eqref{eq:main} with reaction $F$.
Let $\Phi$ be its pulsating profile, so that
$U(t,x)=\Phi(x_1-c_*t,x)$.
The limit $V$ in Proposition~\ref{prop:profile-limits} inherits
$V(t,z)\leq C(1+r)\ee^{-\lambda_*r}$ for all sufficiently large
$r=z_1-c_*t$. Passing from $V$ to $U(t,x)=V(t,x-r_0)$ replaces
$r$ by $x_1-c_*t-(r_0)_1$. This fixed shift can be absorbed into
the tail constant; using $0<\Phi<1$ on the remaining bounded
interval gives
\[
 \Phi(s,x)\leq C(1+s)\ee^{-\lambda_*s}
       \quad\text{for }s\geq0,\ x\in\mathcal C.
\]
Thus $\Phi$ satisfies \eqref{eq:front-critical-upper}.
Strict decrease and the end states give a unique $h\in\R$ with
$\Phi(h,0)=1/2$. Define $\Phi_*(s,x)=\Phi(s+h,x)$.
Theorem~\ref{thm:critical-uniqueness} makes $\Phi_*$ the unique profile
in the stated class with $\Phi_*(0,0)=1/2$. Set
$\ell(t)=\lfloor m_d(t)\rfloor$.
The values $u(t,\ell(t),0)$ stay in a compact subinterval of
$(0,1)$ for all sufficiently large $t$. Indeed, if
$u(t_k,\ell(t_k),0)$ tended to zero or one along a sequence
$t_k\to\infty$, Proposition~\ref{prop:profile-limits} would give a
subsequence along which $u(t_k+s,(\ell(t_k),0)+z)$ converges locally
to an entire front. Evaluating at $(s,z)=(0,0)$ gives a limit in
$(0,1)$, a contradiction.
Strict decrease and the end states give a continuous inverse
$\Phi_*(\cdot,x)^{-1}:(0,1)\to\R$ at each cell site. Define
\begin{equation}\label{eq:phase-definition}
 \zeta_0(t)=\bigl(\Phi_*(\cdot,(\ell(t),0))\bigr)^{-1}
                        (u(t,\ell(t),0))+m_d(t)-\ell(t).
\end{equation}
Choose $T_1\geq2$ and $\eta>0$ such that
$\eta\leq u(t,\ell(t),0)\leq1-\eta$ for $t\geq T_1$.
The finitely many inverse profiles map $[\eta,1-\eta]$ into a bounded
set, and $0\leq m_d(t)-\ell(t)<1$, so $\zeta_0$ is bounded there.
On each Borel set $\{t\geq T_1:\ell(t)=j\}$, its defining expression
is continuous. The countable union over $j\in\Z$ gives measurability.
Set $\zeta_0(t)=\zeta_0(T_1)$ for $t<T_1$; then $\zeta_0$ is bounded
and measurable on $\R$.

Put $F_0=\max_{\mathcal C\times[0,1]}|F|$ and
$F_1=\max_{\mathcal C\times[0,1]}|F_s|$. The lattice equation and
its time derivative give the uniform bounds
\[
 |u_t|\leq2d+F_0,\qquad
 |u_{tt}|\leq(4d+F_1)(2d+F_0).
\]
Fix $R,B>0$. If \eqref{eq:profile-selection} with $\zeta$ replaced by $\zeta_0$ failed uniformly on
$\{(j,y):|j-m_d(t)|\leq R,\ |y|\leq B\}$, there would be times
$t_k\to\infty$ and sites where the error exceeds some $\varepsilon>0$.
Their offsets from $(\ell(t_k),0)$ lie in a finite set. Fix an offset
and a residue $r_0=(\ell(t_k),0)\bmod\Gamma$ along a subsequence.
 {By Remark~\ref{rem:diagonal-compactness}, we may extract a further subsequence converging locally in $C^1(\R)$ on finite subsets of $\Z^d$ to an entire solution $V_\infty$.}
Since $|\ell(t_k)-m_d(t_k)|<1$,
Proposition~\ref{prop:profile-limits} makes $V_\infty$ a front for $F^{[r_0]}$.
Translation back to $F$ and Theorem~\ref{thm:critical-uniqueness} give
\[
 V_\infty(s,z)=\Phi_*(z_1-c_*s+\xi,z+r_0)
\]
for some $\xi\in\R$. Put $\delta_k=m_d(t_k)-\ell(t_k)$ and extract
$\delta_k\to\delta$ and $\zeta_0(t_k)\to\zeta_\infty$.
Equation~\eqref{eq:phase-definition} reads
\[
 u(t_k,\ell(t_k),0)
 =\Phi_*(\zeta_0(t_k)-\delta_k,(\ell(t_k),0)).
\]
Passing to the limit gives
$\Phi_*(\xi,r_0)=\Phi_*(\zeta_\infty-\delta,r_0)$.
Strict decrease implies $\xi=\zeta_\infty-\delta$. At the selected fixed lattice offset $z=(z_1,z')$, periodicity identifies the cell argument with $z+r_0$.
The approximating first argument is
$\ell(t_k)+z_1-m_d(t_k)+\zeta_0(t_k)
=z_1-\delta_k+\zeta_0(t_k)\to z_1+\xi$.
Local convergence gives
$u(t_k,(\ell(t_k),0)+z)\to\Phi_*(z_1+\xi,z+r_0)$.
The approximating profile has the same limit, contradicting the
selected positive error. Hence \eqref{eq:profile-selection} holds uniformly in every
such bounded transition window.
Fix $\eta>0$. By Theorem~\ref{thm:upper}, the uniform end states of
$\Phi_*$ and boundedness of $\zeta_0$, choose $R$ so that
$u(t,j,y)$ and $\Phi_*(j-m_d(t)+\zeta_0(t),(j,y))$ are below $\eta$
on $j-m_d(t)\geq R$, $|y|\leq B$, for all sufficiently large $t$.
Increase $R$ so that $\Phi_*$ is within $\eta$ of one
when $j-m_d(t)\leq-R$, and apply
Remark~\ref{rem:fixed-transverse-rear} with target $1-\eta$.
For all sufficiently large $t$, both functions are then
within $\eta$ of one on $0\leq j\leq m_d(t)-R$.
On the remaining window $|j-m_d(t)|\leq R$, the convergence
proved above makes their difference at most $\eta$ after
increasing the time threshold once more. The supremum on
the whole half-tube is therefore at most $2\eta$. Taking
$\eta\downarrow0$ proves the stated uniform convergence.

To prove uniform slow variation, fix $S>0$. Since $m_d'(t)\to c_*$,
for sufficiently large $t$ and $|s|\leq S$,
\[
 |\ell(t+s)-\ell(t)|\leq(c_*+1)S+1.
\]
If slow variation failed, there would be $t_k\to\infty$ and
$s_k\in[-S,S]$ for which $|\zeta_0(t_k+s_k)-\zeta_0(t_k)|$
stays bounded away from zero. Pass to a subsequence with $s_k\to s$,
$\delta_k=m_d(t_k)-\ell(t_k)\to\delta\in[0,1]$,
a fixed integer $q=\ell(t_k+s_k)-\ell(t_k)$, and fixed residue
$r_0=(\ell(t_k),0)\bmod\Gamma$.
Proposition~\ref{prop:profile-limits} and
Theorem~\ref{thm:critical-uniqueness} give a locally $C^1$ limit of
$u(t_k+v,(\ell(t_k),0)+z)$ of the form
$\Phi_*(z_1-c_*v+\xi,z+r_0)$.
The uniform bound on $u_t$ permits evaluation at $s_k\to s$.
By \eqref{eq:md}, $m_d(t_k+s_k)-m_d(t_k)\to c_*s$.
Writing $\delta_k'=m_d(t_k+s_k)-\ell(t_k+s_k)$ gives
$\delta_k'-\delta_k\to c_*s-q$.
Periodicity identifies the inverse profile at the second site with
$\bigl(\Phi_*(\cdot,r_0+qe_1)\bigr)^{-1}$.
Its value at $u(t_k+s_k,\ell(t_k+s_k),0)$ tends to
$q-c_*s+\xi$, because the limiting value lies in $(0,1)$ and the inverse is continuous.
Equation~\eqref{eq:phase-definition} now gives
\[
 \zeta_0(t_k)\to\xi+\delta,\qquad
 \zeta_0(t_k+s_k)\to(q-c_*s+\xi)+(\delta+c_*s-q)=\xi+\delta.
\]
Thus $\zeta_0(t_k+s_k)-\zeta_0(t_k)\to0$, contrary to its choice.
We obtain $\sup_{|s|\leq S}|\zeta_0(t+s)-\zeta_0(t)|\to0$ as
$t\to\infty$, including across jumps of $\ell$.

Finally, let $t_k\to\infty$ and $x_k=(j_k,y_k)$ satisfy
$|j_k-m_d(t_k)|\leq C$ and $|y_k|\leq A\sqrt{t_k}$ for fixed
 {$C,A\geq0$}, with $x_k\bmod\Gamma=r_0$ and
$u(t_k,x_k)\to\theta\in(0,1)$. By
Proposition~\ref{prop:profile-limits} and
Theorem~\ref{thm:critical-uniqueness}, every entire subsequential
limit of $u(t_k+s,x_k+z)$ has the form
$\Phi_*(z_1-c_*s+\xi,z+r_0)$. At $(s,z)=(0,0)$ its value is
$\theta$, so $\Phi_*(\xi,r_0)=\theta$. The strict decrease from
Proposition~\ref{thm:rigidity} determines the unique
$\xi=\xi_{\theta,r_0}$ satisfying
$\Phi_*(\xi_{\theta,r_0},r_0)=\theta$.
Every convergent subsequence therefore has this same limit.
Local compactness from Proposition~\ref{prop:profile-limits} gives
\eqref{eq:residue-limit} for the full sequence.
For a sequence with bounded $y_k$, $j_k\geq0$ and
$u(t_k,j_k,y_k)\to\theta\in(0,1)$, choose
$0<\theta_-<\theta<\theta_+<1$. The upper bound
\eqref{eq:upper-main} excludes $j_k\geq m_d(t_k)+C$, while
Remark~\ref{rem:fixed-transverse-rear}, applied at level $\theta_+$,
excludes $j_k\leq m_d(t_k)-C$. Together, the upper and lower bounds
give $|j_k-m_d(t_k)|\leq C$ for all sufficiently large $k$.
For the more general sequences with $|y_k|\leq A\sqrt{t_k}$ for a
fixed $A>0$, the longitudinal bound $|j_k-m_d(t_k)|\leq C$
remains a hypothesis, as in Proposition~\ref{prop:profile-limits}.

It remains to regularize $\zeta_0$ and prove
\eqref{eq:profile-time-derivative}.
Choose $\varrho\in C_c^\infty((-1,1))$,
$\varrho\geq0$, with
$\int_{\R}\varrho(s)\,\dd s=1$. Define
\begin{equation}\label{eq:smoothed-phase}
 \zeta(t)=\int_{\R}\varrho(s)\zeta_0(t-s)\,\dd s.
\end{equation}
After changing variables,
$\zeta(t)=\int_\R\varrho(t-r)\zeta_0(r)\,\dd r$.
On compact time intervals, the integrand and every time derivative have
common compact support in $r$ and a bounded majorant. Dominated convergence gives
\[
 \zeta^{(k)}(t)=\int_{\R}\varrho^{(k)}(s)\zeta_0(t-s)\,\dd s
 \quad\text{for every integer }k\geq0.
\]
Thus $\zeta$ is smooth on $\R$ and $\|\zeta\|_\infty\leq\|\zeta_0\|_\infty$.  {Restricting $\zeta$ to $[2,\infty)$ gives the bounded phase $\zeta\in C^\infty([2,\infty))$ stated in Theorem~\ref{thm:selection}.}
Uniform slow variation gives
\[
 |\zeta(t)-\zeta_0(t)|
 \leq\sup_{|s|\leq1}|\zeta_0(t-s)-\zeta_0(t)|\to0 \quad\text{as }t\to\infty.
\]
Since $\int_\R\varrho'(s)\,\dd s=0$,
\[
 |\zeta'(t)|\leq\|\varrho'\|_{L^1(\R)}
       \sup_{|s|\leq1}|\zeta_0(t-s)-\zeta_0(t)|\to0
 \quad\text{as }t\to\infty.
\]
Since $0<\Phi_*<1$, \eqref{eq:pulsating-profile-equation} gives
$\|\partial_1\Phi_*\|_\infty\leq(2d+F_0)/c_*$.
Consequently, for every $s\in\R$ and $x\in\Z^d$,
\[
 |\Phi_*(s+\zeta(t),x)-\Phi_*(s+\zeta_0(t),x)|
 \leq\|\partial_1\Phi_*\|_\infty|\zeta(t)-\zeta_0(t)|\to0 \quad\text{as }t\to\infty.
\]
The uniform bound preserves \eqref{eq:profile-selection} and its consequences for prescribed levels.

For $x=(j,y)\in\Z\times\Z^{d-1}$, set
$P(t,x)=\Phi_*(j-m_d(t)+\zeta(t),x)$. At every fixed time,
\eqref{eq:pulsating-profile-equation} gives
$\Delta_{\Z^d}P(t,x)+F(x,P(t,x))
=-c_*\partial_1\Phi_*(j-m_d(t)+\zeta(t),x)$.
Let
\[
 e_{B+1}(t)=\sup_{\substack{j\in\Z,\ j\geq0\\y\in\Z^{d-1},\ |y|\leq B+1}}
                   |u(t,j,y)-P(t,j,y)|.
\]
For $j\geq1$ and $|y|\leq B$, every neighbor $(k,z)$ of $(j,y)$
satisfies $k\geq0$ and $|z|\leq B+1$.
Subtracting the displayed identity for $P$ from \eqref{eq:main} therefore gives
\[
 |u_t(t,j,y)+c_*\partial_1\Phi_*
       (j-m_d(t)+\zeta(t),(j,y))|
 \leq(4d+F_1)e_{B+1}(t).
\]
At $j=0$, the finitely many values $u(t,k,z)$ with
$|(k,z)-(0,y)|_1\leq1$ and $|y|\leq B$ tend to one by
\eqref{eq:heat-comparison} and Lemma~\ref{lem:amplification}.
Equation~\eqref{eq:main} therefore gives $u_t(t,0,y)\to0$ uniformly
for $|y|\leq B$ as $t\to\infty$.
In the equation for $\Phi_*$, all shifted first arguments tend to
$-\infty$ together; its finite neighbor sum and $F(x,1)=0$
then give $\partial_1\Phi_*\to0$ uniformly on the cell.
Differentiating $P$ in time gives
\[
 P_t=(-m_d'(t)+\zeta'(t))\partial_1\Phi_*.
\]
Since $e_{B+1}(t)\to0$, the estimates for $j\geq1$ and $j=0$ prove
\eqref{eq:profile-time-derivative}.
Boundedness of $\partial_1\Phi_*$ and
$m_d'(t)-\zeta'(t)\to c_*$ also give $u_t-P_t\to0$ uniformly on
$j\geq0$, $|y|\leq B$ as $t\to\infty$.

\end{proof}

\begin{proof}[Proof of Corollary~\ref{cor:arrival-times}]
Choose $0<\eta_-<\theta_-<\theta_+<\eta_+<1$ and fix $B\geq0$.
Take $ {\varepsilon_1}>0$ with
$2 {\varepsilon_1}<\min\{\theta_--\eta_-,\eta_+-\theta_+\}$.
For all sufficiently large $t$, \eqref{eq:profile-selection} has error
less than $ {\varepsilon_1}$ on $j\geq0$, $|y|\leq B$.
Thus $u(t,j,y)\in(\theta_-- {\varepsilon_1},\theta_++ {\varepsilon_1})$ implies
$\Phi_*(j-m_d(t)+\zeta(t),(j,y))\in[\eta_-,\eta_+]$.
For each $r\in\mathcal C$, define the compact interval
\[
 I_r=\{s\in\R:\eta_-\leq\Phi_*(s,r)\leq\eta_+\}.
\]
The proof of Theorem~\ref{thm:critical-uniqueness} gives
$\partial_1\Phi_*<0$. Continuity and finiteness of $\mathcal C$ yield
\[
 \nu_0:=\min_{r\in\mathcal C}\min_{s\in I_r}
             \bigl(-c_*\partial_1\Phi_*(s,r)\bigr)>0.
\]
By \eqref{eq:profile-time-derivative}, there is $t_0$ such that
$u_t(t,j,y)\geq\nu_0/2$ for $t\geq t_0$, $j\geq0$, $|y|\leq B$,
whenever $u(t,j,y)\in(\theta_-- {\varepsilon_1},\theta_++ {\varepsilon_1})$.
This proves \eqref{eq:temporal-transversality} and gives the same
positive bound in an open neighborhood of $[\theta_-,\theta_+]$.

Fix $\kappa>0$ and apply \eqref{eq:global-exponential-bound} with
$\varepsilon=\kappa$. Since $|(j,y)|_1\geq j$ for $j\in\Z_{\geq0}$
and $y\in\Z^{d-1}$, it follows that
\[
 \sup_{\substack{t\in[0,t_0]\\y\in\Z^{d-1},\ |y|\leq B}}u(t,j,y)
 \leq C_\kappa\ee^{c_\kappa t_0-\kappa j}\to0
 \text{ as }j\to\infty.
\]
For each fixed site $(j,y)$, \eqref{eq:heat-comparison} gives
$u(1,j,y)>0$. Applying Lemma~\ref{lem:amplification} with $\eta=u(1,j,y)$
and each $\theta\in(0,1)$ shows that
$u(t,j,y)\to1$ as $t\to\infty$. For large $j$, \eqref{eq:global-exponential-bound} puts $u(t,j,y)$ below $\theta_-$ on
$[0,t_0]$. Convergence to one makes it eventually exceed $\theta_+$. Continuity therefore gives a
crossing of each intervening level after $t_0$.
Fix a level $\theta$ in
the stated interval and a first crossing time $T$. The positive
derivative at $T$ gives $ {\varepsilon}>0$ such that
$u(T+h,j,y)>\theta$ for $0<h\leq {\varepsilon}$. If a later equality
occurred, let $T'>T+ {\varepsilon}$ be the first time after $T+ {\varepsilon}$
at which equality holds. Continuity gives
$u(t,j,y)>\theta$ for $T+ {\varepsilon}\leq t<T'$.
The left difference quotients
$(u(T',j,y)-u(T'-h,j,y))/h$ are nonpositive for small $h>0$,
so differentiability gives $u_t(T',j,y)\leq0$.
This contradicts the positive lower bound for the time derivative at the prescribed level.
Thus the crossing time in \eqref{eq:arrival-definition} is unique.
The exponential bound on $[0,t_0]$ is uniform for $|y|\leq B$, and
the derivative bound is uniform for $\theta\in[\theta_-,\theta_+]$.
Hence the threshold for $j$ is uniform in these sites and levels.

Put $b_d=(d+2)/(2\lambda_*)$ and recall
$T_\theta(j,y)=\inf\{t\geq0:u(t,j,y)\geq\theta\}$.
Write $r(j,y)=(j,y)\bmod\Gamma$ and let $\xi_{\theta,r}$ be
the unique number with $\Phi_*(\xi_{\theta,r},r)=\theta$.
Estimate~\eqref{eq:global-exponential-bound}, applied on any fixed
time interval, shows that $T_\theta(j,y)\to\infty$ uniformly for
$y\in\Z^{d-1}$ with $|y|\leq B$ and
$\theta\in[\theta_-,\theta_+]$ as $j\to\infty$.
At $t=T_\theta(j,y)$, $u(T_\theta(j,y),j,y)=\theta$.
Define
\[
 e_B(t)=\sup_{\substack{j\in\Z,\ j\geq0\\y\in\Z^{d-1},\ |y|\leq B}}
 \left|u(t,j,y)-\Phi_*(j-m_d(t)+\zeta(t),(j,y))\right|.
\]
Let $J=[\eta_-,\eta_+]$ be the enlarged level interval
chosen above. For each residue $r$, the inverse
$\Phi_*(\cdot,r)^{-1}$ is uniformly continuous on $J$.
Taking the maximum of its finitely many moduli of continuity
gives a function $\omega( {\varepsilon})\to0$ as $ {\varepsilon}\downarrow0$.
The approximation error at time $T_\theta(j,y)$ is bounded by
$e_B(T_\theta(j,y))$. Hence the profile argument
$j-m_d(T_\theta(j,y))+\zeta(T_\theta(j,y))$ differs
from $\xi_{\theta,r}$ by at most $\omega(e_B(T_\theta(j,y)))$,
once both $\theta$ and
$\Phi_*\bigl(j-m_d(T_\theta(j,y))+\zeta(T_\theta(j,y)),(j,y)\bigr)$
lie in $J$. Uniform divergence of $T_\theta(j,y)$ therefore
gives $\omega(e_B(T_\theta(j,y)))\to0$ uniformly as $j\to\infty$.
Using $m_d(T)=c_*T-b_d\log T$, we obtain
\begin{equation}\label{eq:arrival-phase-identity}
 j-c_*T_\theta(j,y)+b_d\log T_\theta(j,y)+\zeta(T_\theta(j,y))
        =\xi_{\theta,r(j,y)}+o(1),
\end{equation}
uniformly for $|y|\leq B$ and $\theta\in[\theta_-,\theta_+]$
as $j\to\infty$.

Boundedness of $\zeta$ and $\xi_{\theta,r}$ gives
$|j-c_*T_\theta(j,y)+b_d\log T_\theta(j,y)|\leq C$ uniformly.
For all sufficiently large $T=T_\theta(j,y)$,
$b_d\log T+C\leq c_*T/2$, so
\[
 c_*T/2\leq j\leq2c_*T.
\]
Dividing the position estimate by $T$ gives
$j/T=c_*-b_d\log T/T+O(T^{-1})\to c_*$ uniformly as $j\to\infty$.

Thus $\log T_\theta(j,y)=\log j-\log c_*+o(1)$ uniformly.
Substitution into \eqref{eq:arrival-phase-identity} proves
\eqref{eq:arrival-logarithm} with a uniform bounded remainder.

For explicit uniform error bounds, write the arrival expansion as
\[
 T_\theta(j,y)=j/c_*+(b_d/c_*)\log j+E_\theta(j,y),
 \qquad |E_\theta(j,y)|\leq C
\]
for all sufficiently large $j$, all $|y|\leq B$ and all stated
levels. At the same site this bounds the difference of two
arrival times by $2C$. At sites one period apart it bounds the
difference by
$L_1/c_*+(b_d/c_*)\log(1+L_1/j)+2C$.
{ Thus a single constant $M$ bounds the time separation $|T-T'|$ for every pair $(T,T')$ appearing in \eqref{eq:arrival-level-gap}--\eqref{eq:arrival-cell-gap}, and consequently
\[
 |\log T-\log T'|\leq\frac{M}{\min\{T,T'\}},\qquad
 |\zeta(T)-\zeta(T')|
 \leq M\sup_{r\geq\min\{T,T'\}}|\zeta'(r)|.
\]
Since $\min\{T,T'\}\to\infty$ uniformly as $j\to\infty$ and $\zeta'(r)\to0$, both differences tend to zero uniformly. Writing $r=r(j,y)=r(j+L_1,y)$ and subtracting \eqref{eq:arrival-phase-identity} at the two arrival times yields
$c_*\bigl(T_{\theta_2}(j,y)-T_{\theta_1}(j,y)\bigr)=\xi_{\theta_1,r}-\xi_{\theta_2,r}+o(1)$
and $c_*\bigl(T_\theta(j+L_1,y)-T_\theta(j,y)\bigr)=L_1+o(1)$, which establish \eqref{eq:arrival-level-gap} and \eqref{eq:arrival-cell-gap}.}

\end{proof}

\section{A lower-bound criterion for general periodic growth}\label{sec:extension}
Theorem~\ref{thm:axial} gives the logarithmic lower bound when
$a(x)=a_\perp(y)$. For general periodic $a(x)$,
Proposition~\ref{prop:general-lower-criterion} gives the lower bound under
\eqref{eq:criterion-seed-inputs}. We first construct the linear
Dirichlet evolution and prove Proposition~\ref{thm:Dirichlet-quotient};
heat propagation and Lemma~\ref{lem:amplification} then give the criterion.

Fix $\theta\in(0,1)$, $e\in\mathbb S^{d-1}$, and $\rho\geq\sqrt d$,
and define, for $t\geq0$,
\begin{equation}\label{eq:tube-level-definition}
 R_{\theta,e,\rho}(t)=\sup\{x\cdot e:x\in\Z^d,\ x\cdot e\geq0,
       \ |x-(x\cdot e)e|\leq\rho,\ u(t,x)\geq\theta\}.
\end{equation}
Since $u_0\not\equiv0$, \eqref{eq:heat-comparison} gives $u(1,0)>0$.
Lemma~\ref{lem:amplification} then gives $u(t,0)\geq\theta$ for all
sufficiently large $t$. Thus the set in \eqref{eq:tube-level-definition}
is eventually nonempty. By \eqref{eq:global-exponential-bound}, there is
$C_\theta>0$, independent of $t$, such that $u(t,x)<\theta$ whenever
$|x|_1>C_\theta(t+1)$. Hence $R_{\theta,e,\rho}(t)<\infty$ whenever the set is nonempty.
\begin{conjecture}\label{conj:general-periodic}
Under \eqref{eq:periodic-kpp}--\eqref{eq:u0}, for every fixed
$\theta\in(0,1)$, $e\in\mathbb S^{d-1}$ and $\rho\geq\sqrt d$,
\begin{equation}\label{eq:general-target}
 R_{\theta,e,\rho}(t)=w(e)t-
               \frac{d+2}{2p_e\cdot e}\log t+O(1).
\end{equation}
\end{conjecture}
Theorem~\ref{thm:upper} proves the upper bound in
\eqref{eq:general-target}. Theorem~\ref{thm:axial} gives the matching
lower bound when the linear growth rate is independent of the chosen
coordinate and $e$ is its signed axis. For variation along that axis
and for oblique directions, Proposition~\ref{prop:general-lower-criterion}
gives the lower bound under the explicit estimates
\eqref{eq:criterion-seed-inputs}.

Proposition~\ref{thm:Dirichlet-quotient} applies to every periodic reaction
$F$ satisfying \eqref{eq:periodic-kpp}, with $a(x)$ allowed to depend on every coordinate.
Set $a_{\max}=\max_{x\in\mathcal C}a(x)$.
Let $D(t)\subset\Z^d$ be decreasing in $t$. More precisely, assign each
$x\in D(0)$ a deletion time $\tau_x\in(0,\infty]$ and set
$D(t)=\{x\in D(0):t<\tau_x\}$. Values at deletion times are their
post-deletion values. For $0\leq s\leq t$, let
$\mathcal U_D(t,s)f=v(t,\cdot)$ for $f\in\ell^\infty(\Z^d)$, where
$v_t=Lv$ holds on $D(t)$ sitewise for almost every $t>s$, with
$L=\Delta_{\Z^d}+a(x)$,
$v(s,x)=f(x)$ for $x\in D(s)$, and $v(t,x)=0$ for
$x\in\Z^d\setminus D(t)$, including at deletion.  {The propagator} $\mathcal U_D(t,s)$ is
constructed by finite-box exhaustion below.
For $v_0:\Z^d\to[0,\infty)$ satisfying $0\leq v_0\leq u_0$ and
$\supp v_0\subset D(0)$, put
$v(t,\cdot)=\mathcal U_D(t,0)v_0$ and
$J_v(t,\cdot)=\int_0^t\mathcal U_D(t,s)(v(s,\cdot)^2)\,\dd s$.
All time integrals involving $\mathcal U_D$ are interpreted coordinatewise.

\begin{proposition}\label{thm:Dirichlet-quotient}
Assume \eqref{eq:periodic-kpp}--\eqref{eq:u0}, and choose
\begin{equation}\label{eq:quotient-K}
 K\geq\max\left\{1,\max_{x\in\mathcal C}a(x),K_F\right\}.
\end{equation}
For $t\geq0$ and $x\in\Z^d$, define
$w_v(t,x)=v(t,x)^2/(v(t,x)+KJ_v(t,x))$ when
$v(t,x)+KJ_v(t,x)>0$, and $w_v(t,x)=0$ otherwise.
For every $D$ and $v_0$ satisfying the conditions above,
\begin{equation}\label{eq:Dirichlet-quotient-bound}
 u(t,x)\geq\frac{v(t,x)^2}{v(t,x)+KJ_v(t,x)}
       \text{ for }t\geq0,\ x\in\Z^d,
\end{equation}
where the right-hand side equals $w_v(t,x)$ under the convention above.
Moreover, $w_v$ takes values in $[0,1]$ and is an integral subsolution,
including at deletion times, in the sense specified below.
\end{proposition}
By \eqref{eq:Dirichlet-quotient-bound}, the estimates
$v(t,x_t)\geq m>0$ and $J_v(t,x_t)\leq C_*v(t,x_t)^2$ imply
$u(t,x_t)\geq(m^{-1}+KC_*)^{-1}$. Proposition~\ref{prop:general-lower-criterion}
also requires the first inequality in
\eqref{eq:criterion-seed-inputs}, stated below.

\subsection{An exact time change on the lattice}
Fix $e\in\mathbb S^{d-1}$, and abbreviate $p=p_e$, $h=H(p)=p\cdot v_e$ and
$\kappa=\kappa_d$. Choose $T>\kappa/h$ and $b\in\R$. Define
\begin{equation}\label{eq:moving-Dirichlet-domain}
D(t)=\left\{x\in\Z^d:p\cdot x>
       ht-\kappa\log\frac{t+T}{T}-b\right\}\text{ for }t\geq0.
\end{equation}
The threshold has derivative $h-\kappa/(t+T)>0$.
Before making the time change, we construct $\mathcal U_D(t,s)$ by
finite-box exhaustion for every decreasing domain
$D(t)=\{x\in D(0):t<\tau_x\}$ specified above.
In particular, the construction applies to
\eqref{eq:moving-Dirichlet-domain} for each $e\in\mathbb S^{d-1}$.
For $N\in\mathbb N$ set $D_N(t)=D(t)\cap[-N,N]^d$.
For $f\in\ell^\infty(\Z^d)$, set
$v_N(s,x)=f(x)\1_{D_N(s)}(x)$. Between the finitely many deletions
in $[-N,N]^d$, let $v_N$ solve $(v_N)_r=Lv_N$ on $D_N(r)$,
with $v_N(r,x)=0$ for $x\in\Z^d\setminus D_N(r)$.
When $x$ leaves $D_N(r)$, set $v_N(r,x)=0$ and continue the ODE.
Define $\mathcal U_{D_N}(t,s)f=v_N(t,\cdot)$ on $\Z^d$. For $f\in\ell^\infty(\Z^d)$ with $f\geq0$, use
$D_N(t)\subset D_{N+1}(t)$ and Lemma~\ref{lem:fixed-domain-comparison}
between consecutive deletion times. At a common deleted vertex, both values are reset to zero.
At a vertex contained only in $D_{N+1}(t)$, the value for $D_N(t)$ is zero
and the value in the larger box is nonnegative. Iterating over the
finitely many deletion times in the two boxes gives
$\mathcal U_{D_N}(t,s)f\le\mathcal U_{D_{N+1}}(t,s)f$ for every
$f\ge0$. Applying Lemma~\ref{lem:fixed-domain-comparison} to
$\mathcal U_{D_N}(t,s)f$ and $\ee^{(t-s)L}f$, which solve $v_t=Lv$
on $D_N(t)$ and $\Z^d$, respectively, gives
$0\le\mathcal U_{D_N}(t,s)f\le \ee^{(t-s)L}f$ for $f\ge0$. The whole-space bound is uniform in $N$ on every bounded time interval, so the
monotone limit is finite there. We define
$\mathcal U_D(t,s)f=\lim_{N\to\infty}\mathcal U_{D_N}(t,s)f$ for
$f\ge0$. For $f\in\ell^\infty(\Z^d)$, define $\mathcal U_D(t,s)f$ by
$\mathcal U_D(t,s)f=\mathcal U_D(t,s)f_+-\mathcal U_D(t,s)f_-$, where $f_-=(-f)_+$. Both terms are finite on bounded time intervals by the bound
$\mathcal U_{D_N}(t,s)f\leq \ee^{(t-s)L}f$ for $f\geq0$. For $y\in\Z^d$, define
\[
 \mathcal K_{D_N}(t,s;x,y)
 =[\mathcal U_{D_N}(t,s)\1_{\{y\}}](x),\qquad
 \mathcal K_D(t,s;x,y)=\lim_{N\to\infty}\mathcal K_{D_N}(t,s;x,y).
\]
The entries are nonnegative and increase with $N$. Summing first over
finite sets of $y$, then using monotone convergence for $f_{0,+}$ and
$f_{0,-}$, gives
\[
 [\mathcal U_D(t,s)f_0](x)
 =\sum_{y\in\Z^d}\mathcal K_D(t,s;x,y)f_0(y),\qquad
 \sum_{y\in\Z^d}\mathcal K_D(t,s;x,y)\leq \ee^{a_{\max}(t-s)}
\]
for $f_0\in\ell^\infty(\Z^d)$.  {The bound} $\sum_{y\in\Z^d}\mathcal K_D(t,s;x,y)\leq \ee^{a_{\max}(t-s)}$
follows by comparison with $\ee^{a_{\max}(t-s)}$ in each finite box and
passage to the limit $N\to\infty$. Hence
$\sum_{y\in\Z^d}\mathcal K_D(t,s;x,y)|f_0(y)|
\leq \ee^{a_{\max}(t-s)}\|f_0\|_\infty<\infty$.
Each finite-box kernel is Borel measurable in $(t,s)$: on every region
determined by the finitely many deletion times, it is a finite product
of matrix exponentials and the maps $f\mapsto f\1_{D_N(r)}$ at deletion.
The limiting kernel is measurable as a pointwise limit.

For $s<r<t$, finite-dimensional evolution gives
\[
 \mathcal K_{D_N}(t,s;x,y)
 =\sum_{z\in\Z^d}\mathcal K_{D_N}(t,r;x,z)
                         \mathcal K_{D_N}(r,s;z,y).
\]
Every summand is nonnegative and increases with $N$. Monotone convergence
gives this identity for $\mathcal K_D$; summation against $f_{0,+}$ and
$f_{0,-}$ gives the operator composition law below.
For measurable $f:[s,\infty)\times\Z^d\to[0,\infty)$ with
$\sup_{s\leq r\leq S}\|f(r,\cdot)\|_\infty<\infty$ for every $S>s$,
let $z_N$ solve $(z_N)_t=Lz_N+f$ in $D_N(t)$, with
$z_N(s,\cdot)=0$ and $z_N(r,x)=0$ for $x\notin D_N(r)$,
including the resets at deletion. Finite-dimensional variation of constants gives
$z_N(t,\cdot)=\int_s^t\mathcal U_{D_N}(t,r)f(r,\cdot)\,\dd r$.
For $s\leq r\leq t$,
$\|\mathcal U_{D_N}(t,r)f(r,\cdot)\|_\infty
\leq \ee^{a_{\max}(t-r)}\|f(r,\cdot)\|_\infty$.
Since $f$ is locally bounded in $\ell^\infty(\Z^d)$,
$\int_s^t[\mathcal U_{D_N}(t,r)f(r,\cdot)](x)\,\dd r<\infty$.
Monotone convergence in the kernel sum and time integral yields
\[
 \mathcal U_D(t,r)\mathcal U_D(r,s)=\mathcal U_D(t,s),\qquad
 z(t,\cdot)=\int_s^t\mathcal U_D(t,r)f(r,\cdot)\,\dd r,
\]
where $z(t,x)=\lim_{N\to\infty}z_N(t,x)$.

For $f_0\in\ell^\infty(\Z^d)$ with $f_0\geq0$ and $s\leq r\leq S$,
$v_N(r,\cdot)=\mathcal U_{D_N}(r,s)f_0$ satisfies
$\|v_N(r,\cdot)\|_\infty\leq \ee^{a_{\max}(S-s)}\|f_0\|_\infty$.
Before deletion, its coordinate derivatives are bounded by
$(4d+a_{\max})\ee^{a_{\max}(S-s)}\|f_0\|_\infty$.
The common Lipschitz bound and pointwise convergence imply uniform
convergence on compact intervals before a vertex's deletion: use a finite
mesh, convergence at its points, and the Lipschitz bound between them.
Extending coordinates by their left limits gives the same conclusion
up to each finite deletion time. The limit is locally Lipschitz before
deletion and zero at and after deletion. Passing to the limit in the
coordinate integral equations gives $v_t=Lv$ almost everywhere before
deletion and right-continuity at every vertex. Fix $x\in D(s)$. On every open subinterval of $(s,\tau_x)$ containing
no deletion time of a neighbor of $x$, the function $v(\cdot,x)$ is
$C^1$ and satisfies $v_t(t,x)=Lv(t,x)$.
All continuity assertions concern individual lattice coordinates.

Write $d_x=a(x)-2d$. Then $v$ satisfies, for $x\in D(t)$,
\[
 v(t,x)=\ee^{d_x(t-s)}f_0(x)
       +\int_s^t \ee^{d_x(t-r)}\sum_{\substack{z\in\Z^d\\|z|_1=1}}v(r,x+z)\,\dd r.
\]
For $x\notin D(t)$, $v(t,x)=0$. Indeed, if $x\in D(t)$,
then $x\in D(r)$ for every $r\in[s,t]$, so the formula follows
by the scalar integrating factor. Whenever $x+z\notin D(r)$,
$v(r,x+z)=0$. All
integrands are bounded and measurable, because each finite-box
coordinate has this property and pointwise limits preserve it.
{ For any two bounded solutions on $[s,S]$ sharing the same initial datum $f_0$ (and the same source $f$, which cancels upon subtraction), subtracting their integrating-factor representations shows that the measurable supremum $E(t)=\sup_{x\in\Z^d}|v(t,x)-\widetilde v(t,x)|$ satisfies
\[
 E(t)\leq2d\,\ee^{(2d+\max_{x\in\mathcal C}|a(x)|)(S-s)}\int_s^t E(r)\,\dd r,
\]
whence $E(t)\equiv0$ by Gronwall's inequality. Moreover, coordinatewise dominated convergence passes the integrating-factor representation with source $f$ directly to the exhaustion limit.}

At equal times,
$\mathcal K_D(s,s;x,y)=\1_{D(s)}(x)\1_{\{x\}}(y)$.
The function $\mathcal U_D(t,s)f_0$ vanishes outside $D(t)$ and depends
only on $f_0|_{D(s)}$. Thus composition also holds for $r=s$ and $r=t$.
The source integral is zero for $t=s$.

A function $w:[0,\infty)\times\Z^d\to[0,1]$, right-continuous at
each vertex and satisfying $w(t,x)=0$ for $x\in\Z^d\setminus D(t)$,
is a subsolution of
\eqref{eq:main} in the integral sense if, for every $x\in\Z^d$
and $0\leq s<t$,
\begin{equation}\label{eq:integral-subsolution-definition}
 w(t,x)-w(s,x)\leq
 \int_s^t\bigl[\Delta_{\Z^d}w(r,x)+F(x,w(r,x))\bigr]\,\dd r.
\end{equation}
{ Whenever $w(\cdot,x)$ is locally absolutely continuous between the deletion times
of $x$ and its neighbors, satisfies $w_t\le\Delta_{\Z^d}w+F(x,w)$ almost everywhere
on each such interval, and has only nonpositive jumps at deletion times,
integrating over each subinterval and summing across the deletion times yields
\eqref{eq:integral-subsolution-definition}. This integral formulation permits
direct passage from finite boxes to the full lattice in
Proposition~\ref{thm:Dirichlet-quotient}.}

The time substitution below follows the change of time in
\cite[Section~3.2.1]{Shabani}. It leaves the lattice sites and
coefficient phases fixed and makes the boundary linear in the new time.
\begin{lemma}\label{lem:exact-time-change}
Let $v_0:\Z^d\to[0,\infty)$ satisfy $0\leq v_0\leq u_0$ and
$\supp v_0\subset D(0)$, and set
$v(t,\cdot)=\mathcal U_D(t,0)v_0$ for the domain
\eqref{eq:moving-Dirichlet-domain}. Define
\begin{equation}\label{eq:exact-time-map}
 \sigma(t)=t-\frac\kappa h\log\frac{t+T}{T},\qquad
 \mathfrak t=\sigma^{-1},\qquad
 q(\tau,x)=\frac{\ee^{p\cdot x-h\mathfrak t(\tau)}}{\phi_p(x)}
                        v(\mathfrak t(\tau),x).
\end{equation}
Then $\sigma:[0,\infty)\to[0,\infty)$ is a strictly increasing bijection and
$q$ solves, with the same deletion convention,
\begin{equation}\label{eq:time-changed-Dirichlet}
 (1- {\omega}(\tau))q_\tau=Q_pq
 \quad {\text{on }}\{x\in\Z^d:p\cdot x>h\tau-b\},\qquad
  {\omega}(\tau)=\frac\kappa{h(\mathfrak t(\tau)+T)}.
\end{equation}
As $\tau\to\infty$,
$ {\omega}(\tau)=\kappa/(h\tau)+O(\log\tau/\tau^2)$ and
$ {\omega}'(\tau)=O(\tau^{-2})$.
\end{lemma}
\begin{proof}
From \eqref{eq:exact-time-map}, one has $\sigma'(t)=1-\kappa/(h(t+T))>0$, $\sigma(0)=0$ and
$\sigma(t)\to\infty$. Using \eqref{eq:moving-Dirichlet-domain}, the identity
$h\sigma(t)=ht-\kappa\log((t+T)/T)$ changes the domain to
$\{x\in\Z^d:p\cdot x>h\tau-b\}$.
Set $w(t,x)=\ee^{p\cdot x-ht}v(t,x)/\phi_p(x)$; then
$w(t,x)=0$ for $x\in\Z^d\setminus D(t)$, including at deletion. Using
$L=\Delta_{\Z^d}+a(x)$ and $B_p\phi_p=h\phi_p$, direct
substitution into $v_t=Lv$ gives
\[
 \partial_t w(t,x)=
 \sum_{\substack{z\in\Z^d\\|z|_1=1}}\ee^{-p\cdot z}\frac{\phi_p(x+z)}{\phi_p(x)}
             \bigl(w(t,x+z)-w(t,x)\bigr)=Q_pw(t,x).
\]
This identity holds almost everywhere before deletion at each
vertex. It is first obtained between deletions in each finite
box and then passes to the monotone limit of those solutions, since the
sum at a fixed vertex has finitely many terms. Applying the chain
rule to \eqref{eq:exact-time-map} gives
$q_\tau=\mathfrak t'(\tau)Q_pq$; since
$1/\mathfrak t'(\tau)=1- {\omega}(\tau)$, this proves
\eqref{eq:time-changed-Dirichlet} with the same deletion convention.
Here $q(0,x)=\ee^{p\cdot x}v_0(x)/\phi_p(x)$ is bounded, since
$v_0\leq u_0$ and $\supp u_0$ is finite. Define $q_N$ by
\eqref{eq:exact-time-map} with $v$ replaced by
$v_N(t,\cdot)=\mathcal U_{D_N}(t,0)v_0$.
In each finite box, comparison for
$(1- {\omega})(q_N)_\tau=Q_pq_N$ with the constants $0$ and
$\|q(0,\cdot)\|_\infty$ gives
$0\leq q_N(\tau,x)\leq\|q(0,\cdot)\|_\infty$ between deletions.
The zero resets preserve these inequalities. Passing to the exhaustion
limit gives $0\leq q(\tau,x)\leq\|q(0,\cdot)\|_\infty$.
For the asymptotics, \eqref{eq:exact-time-map} gives
$\mathfrak t(\tau)-\tau=(\kappa/h)\log((\mathfrak t(\tau)+T)/T)$.
Set $\eta_0=1-\kappa/(hT)>0$.
The mean-value inequalities for $\sigma$ give
$\eta_0t\leq\sigma(t)\leq t$, hence
$\tau\leq\mathfrak t(\tau)\leq\tau/\eta_0$ and
$1\leq\mathfrak t'(\tau)\leq\eta_0^{-1}$.
For $\tau\geq1$, subtraction of the two fractions yields
\[
 \left| {\omega}(\tau)-\frac{\kappa}{h\tau}\right|
 =\frac{\kappa\,|\mathfrak t(\tau)+T-\tau|}
        {h\tau(\mathfrak t(\tau)+T)}
 \leq C\frac{1+\log(1+\tau)}{\tau^2},\qquad
  {\omega}'(\tau)=-\frac{\kappa\mathfrak t'(\tau)}
                    {h(\mathfrak t(\tau)+T)^2}.
\]
The first estimate uses the inverse identity and
$\mathfrak t(\tau)\leq\tau/\eta_0$ inside its logarithm.
Hence $ {\omega}(\tau)=\kappa/(h\tau)+O(\log\tau/\tau^2)$ and
$ {\omega}'(\tau)=O(\tau^{-2})$ as $\tau\to\infty$, with constants
depending on the fixed $h,\kappa,T$. On compact time intervals the two maps are
bi-Lipschitz, so they send null sets to null sets. The differential equations, which hold almost everywhere,
remain valid under the change
of time. If $\tau_x<\infty$, deletion at $\tau_x$ occurs at
$\sigma(\tau_x)$ in the new time; multiplication by the positive
continuous weight preserves the zero value. If $\tau_x=\infty$,
the vertex is never deleted.

\end{proof}

Define $\mathcal A(\tau)=\ee^{h(\mathfrak t(\tau)-\tau)}$ for $\tau\geq0$.
Solving \eqref{eq:exact-time-map} for $v$ gives
\begin{equation}\label{eq:time-change-amplitude}
 \mathcal A(\tau)=\ee^{h(\mathfrak t(\tau)-\tau)}
       =\left(\frac{\mathfrak t(\tau)+T}{T}\right)^\kappa,
 \qquad
 v(\mathfrak t(\tau),x)=
 \ee^{-p\cdot x+h\tau}\phi_p(x)\mathcal A(\tau)q(\tau,x).
\end{equation}
Differentiation gives $\mathcal A'/\mathcal A=h {\omega}/(1- {\omega})$.
{ Note that the time-change rate $\omega(\tau)$ in \eqref{eq:time-changed-Dirichlet} matches the notation $\omega$ in \cite[(2.13)--(2.17)]{HNRRPeriodic}, and \eqref{eq:time-change-amplitude} records the exact relation between the original Dirichlet evolution $v$ and the time-changed profile $q$.}

Define
\begin{equation}\label{eq:shifted-radial-position}
  {m_e^T}(t)=w(e)t-\frac\kappa{\alpha_e}\log\frac{t+T}{T} {.}
\end{equation}
Since $\alpha_e w(e)=h$,  {we have} $\alpha_e  {m_e^T}(t)=h\sigma(t)$. More precisely, if
$r_\tau=p\cdot x_\tau-h\tau+b\in(0,R]$ and $t=\mathfrak t(\tau)$,
then \eqref{eq:time-change-amplitude} gives
\[
 v(t,x_\tau)=\phi_p(x_\tau)\ee^{b-r_\tau}
                  \left(\frac{t+T}{T}\right)^\kappa q(\tau,x_\tau).
\]
Thus a bound $q(\tau,x_\tau)\geq c\tau^{-d/2-1}$ yields
$v(t,x_\tau)\geq m>0$ for all sufficiently large $t$. Indeed, $\phi_p$ has a positive minimum, $r_\tau$ is bounded, and
$\kappa=d/2+1$ with $\mathfrak t(\tau)\sim\tau$.

For $0\le\sigma<\tau$, define $\mathsf K_p^{ {\omega}}(\sigma,\tau;x,y)$ by
\eqref{eq:time-changed-Dirichlet}, with $q(\tau,x)=0$ when
$p\cdot x\leq h\tau-b$, and the identity
\[
 q(\tau,x)=\sum_{y\in\Z^d}\mathsf K_p^{ {\omega}}(\sigma,\tau;x,y)q(\sigma,y).
\]
Thus $x$ is the evaluation site at time $\tau$, and $y$ is the initial
site at time $\sigma$.
Define
\begin{equation}\label{eq:operational-time}
 \Theta(\sigma,\tau)=\int_\sigma^\tau\frac{\dd r}{1- {\omega}(r)}.
\end{equation}
The inverse-map derivative in \eqref{eq:exact-time-map} gives $1/(1- {\omega}(r))=\mathfrak t'(r)$,
and integration in \eqref{eq:operational-time} gives $\Theta(\sigma,\tau)=\mathfrak t(\tau)-\mathfrak t(\sigma)$.
For the unperturbed equation $ {\omega}\equiv0$,
Lemma~\ref{lem:killed-kernel-upper} gives the bound
\eqref{eq:killed-kernel-upper}, with denominator
$(1+t-s)^{d/2+1}$ and endpoint factors $1+d_{p,b}(t,x)$ and
$1+d_{p,b}(s,y)$ when $d_{p,b}(t,x)>0$ and $d_{p,b}(s,y)>0$ in
\eqref{eq:normal-coordinate}. A kernel estimate sufficient for
$v(t,x_t)\geq m$ in \eqref{eq:criterion-seed-inputs} is
the following:
for each $0<d_-<d_+<\infty$ and $0\leq A<\infty$, there exist
$c>0$ and $T_0>0$ such that, uniformly for $0\leq\sigma<\tau$
and $x,y\in\Z^d$, including all endpoint residues modulo $\Gamma$,
\begin{equation}\label{eq:target-killed-lower}
 \mathsf K_p^{ {\omega}}(\sigma,\tau;x,y)\ge
 c\,\Theta(\sigma,\tau)^{-d/2-1}
\end{equation}
whenever $\Theta(\sigma,\tau)\ge T_0$,
\[
 d_-\le p\cdot x-h\tau+b\le d_+,\qquad
 d_-\le p\cdot y-h\sigma+b\le d_+,
\]
and, writing $v_p=\nabla H(p)=v_e$,
\[
 |y-x+\Theta(\sigma,\tau)v_p|\le
 A\sqrt{\Theta(\sigma,\tau)}.
\]
The displacement condition centers $y-x$ at $-\Theta(\sigma,\tau)v_p$.
With this kernel convention, the $Q_p$-walk starts at $x$ and uses the
domain and rates in \eqref{eq:time-changed-Dirichlet} in reverse
chronological order.
For $\bar p$ and $p_0$ defined in \cite[(32)]{Shabani},
Proposition~3.2 estimates $\bar p$ with $\bar p(0,\cdot)=p_0$.
The estimate \eqref{eq:target-killed-lower} specifies uniformity
in the two endpoint residues and the starting time. Conditioned local
limit estimates for additive functionals of finite-state Markov chains
were developed by Grama, Lauvergnat and Le Page
\cite{GramaLauvergnatLePage}. The evolution here is in continuous time, its displacement belongs
to $\Z^d$, and the bound is uniform in the endpoint residues.
The factor $1- {\omega}(\tau)$ makes the evolution nonautonomous.
We use \eqref{eq:target-killed-lower} as an explicit sufficient
hypothesis for the lower bound on $v$.

\subsection{A positive quotient comparison}
Hamel et al.\ \cite[Section~2.1.2]{HNRRPeriodic} control the nonlinear
reaction loss by multiplying a linear Dirichlet solution by a
time-dependent damping factor. Linear Dirichlet solutions also enter
the lower-bound construction in \cite[Section~4.1]{Shabani}.
Montoro, Punzo and Sciunzi \cite[Theorem~3.1(ii), $q=2$]{MontoroPunzoSciunzi2018}
obtained a related quotient lower bound for semilinear parabolic equations
on fixed continuous domains. The proof below establishes the estimate on
decreasing lattice domains, including the discrete edge terms and deletion
times.
For a nonnegative solution $v$ of $v_t=\Delta_{\Z^d}v+a(x)v$
and a constant $\varepsilon>0$, the bound $F(x,r)\leq a(x)r$
for $0\leq r\leq1$ gives the supersolution inequality
\begin{equation}\label{eq:linear-rescaling-sign}
 (\partial_t-\Delta_{\Z^d})(\varepsilon v)-F(x,\varepsilon v)
       =a(x)\varepsilon v-F(x,\varepsilon v)\geq0
       \text{ for }0\leq\varepsilon v\leq1.
\end{equation}
With $K$ as in \eqref{eq:quotient-K}, we have
$F(x,r)\geq a(x)r-Kr^2$ for $x\in\Z^d$ and $0\leq r\leq1$.
For $v(t,\cdot)=\mathcal U_D(t,0)v_0$ as in
Proposition~\ref{thm:Dirichlet-quotient}, the component formula for $J_v$ is
\begin{equation}\label{eq:quadratic-Dirichlet-potential}
 J_v(t,x)=\int_0^t
    \bigl[\mathcal U_D(t,s)(v(s,\cdot)^2)\bigr](x)\,\dd s.
\end{equation}
Fix $T_0>0$. Using
$\sum_{y\in\Z^d}\mathcal K_D(t,0;x,y)\leq \ee^{a_{\max}t}$, we obtain
\[
 \|v(t,\cdot)\|_{\ell^\infty(\Z^d)}
 \leq\ee^{a_{\max}t}\|v_0\|_{\ell^\infty(\Z^d)}
 \text{ for }0\le t\le T_0 {.}
\]
Set $M_{T_0}=\ee^{a_{\max}T_0}\|v_0\|_\infty$.
For $0\le s\le t\le T_0$ one has
$v(s,y)^2\le M_{T_0}v(s,y)$ at every site. Applying the positive operator
$\mathcal U_D(t,s)$ and using the composition identity
$\mathcal U_D(t,s)v(s,\cdot)=v(t,\cdot)$ gives
$[\mathcal U_D(t,s)(v(s,\cdot)^2)](x)\leq M_{T_0}v(t,x)$. Integrating this estimate in \eqref{eq:quadratic-Dirichlet-potential} over $s$
yields $J_v(t,x)\le M_{T_0}t\,v(t,x)$.
Consequently, $v$ and $J_v$ are finite on every bounded time interval.

\begin{lemma}\label{lem:quotient-edge-identity}
For $0\leq v\leq S$ and $0\leq\widetilde v\leq\widetilde S$,
with $S,\widetilde S>0$, the function $f(v,S)=v^2/S$ satisfies
\begin{equation}\label{eq:quotient-edge-identity}
\begin{split}
 f(\widetilde v,\widetilde S)-f(v,S)
   -2\frac vS(\widetilde v-v)
   +\frac{v^2}{S^2}(\widetilde S-S)
   =\widetilde S
      \left(\frac{\widetilde v}{\widetilde S}-\frac vS\right)^2.
\end{split}
\end{equation}
With $f(0,0)=0$, this function is locally Lipschitz on the closed cone
$0\leq v\leq S$.
\end{lemma}
\begin{proof}
Writing $R=v/S$ and $\widetilde R=\widetilde v/\widetilde S$,
the left side of \eqref{eq:quotient-edge-identity} is
\[
 \widetilde S\widetilde R^2-SR^2
 -2R(\widetilde S\widetilde R-SR)
 +R^2(\widetilde S-S)
 =\widetilde S(\widetilde R-R)^2.
\]
For $S>0$ in the cone, $f_v=2v/S$ and $f_S=-v^2/S^2$, so
$|f_v|\leq2$ and $|f_S|\leq1$. Integration along a segment with positive
second coordinate gives
$|f(v_1,S_1)-f(v_2,S_2)|\leq2|v_1-v_2|+|S_1-S_2|$.
For an endpoint at the origin, use a point at positive distance
along the segment and pass to that endpoint; $0\leq f(v,S)\leq v$
ensures continuity there. Thus the estimate holds on the entire closed cone, including both
zero endpoints.

\end{proof}

\begin{proof}[Proof of Proposition~\ref{thm:Dirichlet-quotient}]
We first compare the quotient with a logistic solution in each finite box,
then pass to the exhaustion limit. For the decreasing domain $D(t)$ in
Proposition~\ref{thm:Dirichlet-quotient}, put
$D_N(t)=D(t)\cap[-N,N]^d$ for $N\in\mathbb N$.
Let $v_N(t,\cdot)=\mathcal U_{D_N}(t,0)v_{0,N}$, with
$v_N(t,x)=0$ for $x\in\Z^d\setminus D_N(t)$, where
$v_{0,N}=v_0\1_{D_N(0)}$, and define $J_N$ by
\eqref{eq:quadratic-Dirichlet-potential}, with $\mathcal U_D$ and
$v$ replaced by $\mathcal U_{D_N}$ and $v_N$. The finite-box evolution and its variation-of-constants formula
give, at sites $x\in D_N(t)$ and between deletion times,
\[
 (\partial_t-L)v_N=0,\qquad
 (\partial_t-L)J_N=v_N^2,\qquad J_N(0,\cdot)=0.
\]
Put $S_N=v_N+KJ_N$ and
$w_N=v_N^2/S_N$, with zero value when $S_N=0$.
Define $R_N(t,x)=v_N(t,x)/S_N(t,x)$ when $S_N(t,x)>0$,
and set $R_N(t,x)=0$ when $S_N(t,x)=0$.
At a point where $S_N(t,x)>0$, apply the chain rule to $w_N=f(v_N,S_N)$,
where $f(v,S)=v^2/S$. The source $Kv_N^2$ in the equation for $S_N$
contributes
\[
 f_S(v_N,S_N)Kv_N^2
 =-K\frac{v_N^4}{S_N^2}=-Kw_N^2.
\]
By $vf_v+Sf_S=f$, the terms $a(x)v_N$ and $a(x)S_N$ contribute
$a(x)(v_Nf_v+S_Nf_S)=a(x)w_N$.
Equation~\eqref{eq:quotient-edge-identity} gives
the square for each neighbor with $S_N(t,x+z)>0$.
If $S_N(t,x+z)=0$, then $v_N(t,x+z)=0$; using $f(0,0)=0$
and $vf_v+Sf_S=f$ shows that this edge contributes zero.
Summing over all neighbors gives
\begin{equation}\label{eq:quotient-nonlinear-residual}
\begin{split}
 &(\partial_t-\Delta_{\Z^d})w_N(t,x)-a(x)w_N(t,x)+Kw_N(t,x)^2\\
 &\qquad=-\sum_{\substack{z\in\Z^d\\|z|_1=1}}S_N(t,x+z)
                  \bigl(R_N(t,x+z)-R_N(t,x)\bigr)^2\leq0.
\end{split}
\end{equation}
The bounded derivatives
and Lipschitz extension of $f$ in
Lemma~\ref{lem:quotient-edge-identity} also control $w_N$
at $(v_N,S_N)=(0,0)$.

Lemma~\ref{lem:quotient-edge-identity} makes $w_N$ locally absolutely
continuous between deletions. A nonnegative absolutely continuous
function has derivative zero almost everywhere on its zero set.
At points where $w_N(t,x)=0$, its time derivative is therefore zero almost everywhere. There, $\Delta_{\Z^d}w_N(t,x)=\sum_{z\in\Z^d,\,|z|_1=1}w_N(t,x+z)\geq0$ and
$a(x)w_N-Kw_N^2=0$. Hence
$(w_N)_t-\Delta_{\Z^d} w_N-a(x)w_N+Kw_N^2\le0$ on the zero set as well. The subsolution inequality from \eqref{eq:quotient-nonlinear-residual}
therefore holds for almost every time between deletions at every site. At a deletion $w_N$ is reset
to zero, a nonpositive jump. Initially $w_N=v_{0,N}\leq u_0$.
Let $z_N$ solve
$(z_N)_t=\Delta_{\Z^d} z_N+a(x)z_N-Kz_N^2$ for $x\in D_N(t)$,
with $z_N(0,\cdot)=v_{0,N}$ and $z_N(t,x)=0$ for
$x\in\Z^d\setminus D_N(t)$, including at deletion.

For every finite horizon $T_0$, comparison for $(v_N)_t=Lv_N$ on
$D_N(t)$, with $v_N(t,x)=0$ for $x\notin D_N(t)$, gives
\[
 \begin{gathered}
 0\leq w_N\leq v_N\leq M_{T_0}
       \quad\text{for }0\leq t\leq T_0,\\
 M_{T_0}=\ee^{a_{\max}T_0}\|v_0\|_\infty,
 \end{gathered}
\]
independently of $N$. Since $K\geq\max_{x\in\mathcal C}a(x)$, comparison with the
constants $0,1$ gives $0\leq z_N\leq1$.
The map $r\mapsto a(x)r-Kr^2$ is Lipschitz, uniformly in $x$, on $[0,\max\{M_{T_0},1\}]$.
Thus \eqref{eq:quotient-nonlinear-residual} and
Lemma~\ref{lem:fixed-domain-comparison} give $w_N\leq z_N$ between
deletions; resetting both to zero preserves the order.
Iteration over the finitely many deletions gives $w_N\leq z_N\leq1$
on $[0,T_0]$. Now \eqref{eq:quadratic-kpp-loss} and \eqref{eq:quotient-K}
give $a(x)r-Kr^2\leq F(x,r)$ on $[0,1]$.
For $x\notin D_N(t)$, $w_N=z_N=0\leq u$, including at deletion.

Lemma~\ref{lem:fixed-domain-comparison} thus gives $z_N\leq u$
between successive deletions. Since $z_N$ is reset to zero and $u$
remains nonnegative at each deletion, this ordering holds for all
times. Combining the comparisons gives $0\leq w_N\leq z_N\leq u\leq1$.

Since $D_N(t)\subset D_{N+1}(t)$ and $v_{0,N}\leq v_{0,N+1}$,
finite-box comparison, including zero resets, gives $v_N\leq v_{N+1}$.
The construction of $\mathcal U_D$ also gives
$\mathcal U_{D_N}(t,s)f\leq\mathcal U_{D_{N+1}}(t,s)f$ for $f\geq0$.
Since $v_N(s,\cdot)^2\leq v_{N+1}(s,\cdot)^2$, positivity and
\eqref{eq:quadratic-Dirichlet-potential} give
\[
 \begin{aligned}
 J_N(t,\cdot)
 &=\int_0^t\mathcal U_{D_N}(t,s)(v_N(s,\cdot)^2)\,\dd s\\
 &\leq\int_0^t\mathcal U_{D_{N+1}}(t,s)(v_{N+1}(s,\cdot)^2)\,\dd s
 =J_{N+1}(t,\cdot).
 \end{aligned}
\] Thus $v_N\uparrow v$ and $J_N\uparrow J_v$ by the
definitions of $\mathcal U_D$ and \eqref{eq:quadratic-Dirichlet-potential}.
The bound $J_N(t,x)\leq M_{T_0}t\,v_N(t,x)$ for
$x\in\Z^d$ and $0\leq t\leq T_0$ keeps the limit finite.
Where $v>0$, convergence of $v_N$ and $J_N$ gives convergence of
both numerator and denominator, with $v+KJ_v>0$. Thus $w_N$ converges to $w_v$ in \eqref{eq:Dirichlet-quotient-bound}. Where $v=0$
the bound $0\leq w_N\leq v_N$ gives convergence to zero.
Thus \eqref{eq:Dirichlet-quotient-bound} holds, including at deletion
times. For the subsolution assertion, integrate the inequality given by \eqref{eq:quotient-nonlinear-residual} in each finite box
between any two times $s<t$, including the nonpositive deletion jumps. Since
$0\leq w_N\leq1$, the integrands are uniformly bounded on $[s,t]$;
the Laplacian at a fixed vertex contains only finitely many neighbors.
Dominated convergence therefore passes this integral inequality to
$w_v$. Before deletion at each vertex, $v$ and
$J_v$ satisfy, almost everywhere in time,
$(\partial_t-L)v=0$ and $(\partial_t-L)J_v=v^2$, with
$v(0,\cdot)=v_0$ and $J_v(0,\cdot)=0$. For
$x\in\Z^d\setminus D(t)$, including at deletion,
$v(t,x)=J_v(t,x)=0$. Their coordinatewise integral
equations give right-continuity at every vertex, with both values
reset to zero at deletion. Set $S=v+KJ_v$. Since $0\leq v\leq S$,
the Lipschitz extension of $f(v,S)=v^2/S$ in
Lemma~\ref{lem:quotient-edge-identity}, with $f(0,0)=0$, makes
$w_v$ right-continuous as well. Moreover, $w_v$ takes
values in $[0,1]$ and is zero at deleted vertices.
Estimates~\eqref{eq:quadratic-kpp-loss} and \eqref{eq:quotient-K}
give $a(x)r-Kr^2\leq F(x,r)$ on this interval. Substitution into
the limiting integral inequality gives
\eqref{eq:integral-subsolution-definition}. Hence $w_v$ is
a subsolution of \eqref{eq:main} in the stated integral sense.
\end{proof}

As an algebraic check, take $D(t)=\Z^d$, $a(x)=a$ and
$v_0(x)=\gamma\geq0$ for all $x\in\Z^d$. If $\gamma=0$, then
$v=J_v=w_v=0$. For $\gamma>0$, $v(t,x)=\gamma \ee^{at}$ and direct integration gives
\[
 J_v(t,x)=\frac{\gamma^2}{a}(\ee^{2at}-\ee^{at}),\qquad
 w_v(t,x)=\frac{\gamma \ee^{at}}{1+(K\gamma/a)(\ee^{at}-1)}
 \quad\text{for }x\in\Z^d.
\]
The square remainder in \eqref{eq:quotient-nonlinear-residual} vanishes,
and differentiation gives $\partial_t w_v=aw_v-Kw_v^2$ with
$w_v(0,x)=\gamma$.
In Proposition~\ref{prop:general-lower-criterion}, $v_0\not\equiv0$,
$\supp v_0$ is finite, and \eqref{eq:criterion-seed-inputs} is required.

Let $E_v:[0,\infty)\to[0,\infty)$ be locally integrable and satisfy
$E_v(s)\geq\sup_{y\in\Z^d}v(s,y)$. Then positivity and the
composition law give
\begin{equation}\label{eq:quotient-damping-comparison}
 J_v(t,x)\leq v(t,x)\int_0^t E_v(s)\,\dd s,
 \qquad
 \frac{v(t,x)^2}{v(t,x)+KJ_v(t,x)}
 \geq\frac{v(t,x)}{1+K\int_0^t E_v(s)\,\dd s}.
\end{equation}
The inequalities include $v(t,x)=0$ under the convention that the
quotient is zero when its denominator vanishes.
Indeed, apply the positive operator $\mathcal U_D(t,s)$ to
$v(s,\cdot)^2\leq E_v(s)v(s,\cdot)$ and use
$\mathcal U_D(t,s)v(s,\cdot)=v(t,\cdot)$ before integrating in $s$. When
$E_v\in L^1(0,\infty)$, the factor
$(1+K\int_0^tE_v(s)\,\dd s)^{-1}$ has a positive lower bound independent of time, as in \cite[Section~2.1.2]{HNRRPeriodic}.

\begin{lemma}\label{lem:Dirichlet-bridge}
For $v(t,x)>0$ and $0\leq s\leq t$, put
\begin{equation}\label{eq:Dirichlet-bridge-marginal}
 \mathfrak b_{t,x}(s,y)=
       \frac{\mathcal K_D(t,s;x,y)v(s,y)}{v(t,x)}.
\end{equation}
For each $s\in[0,t]$, these are probability weights on
$\Z^d$ and, for $0\leq s_1<s_2\leq t$ and $y_1\in\Z^d$, satisfy
\[
 \sum_{\substack{y_2\in\Z^d\\v(s_2,y_2)>0}}
 \mathfrak b_{t,x}(s_2,y_2)\mathfrak b_{s_2,y_2}(s_1,y_1)
 =\mathfrak b_{t,x}(s_1,y_1).
\]
Moreover,
\begin{equation}\label{eq:normalized-potential-identity}
 \frac{J_v(t,x)}{v(t,x)^2}
 =\int_0^t\sum_{y\in\Z^d}\mathfrak b_{t,x}(s,y)
                   \frac{v(s,y)}{v(t,x)}\,\dd s.
\end{equation}
\end{lemma}
\begin{proof}
By the composition law,
$v(t,\cdot)=\mathcal U_D(t,s)v(s,\cdot)$, so the weights in
\eqref{eq:Dirichlet-bridge-marginal} sum to one over $y\in\Z^d$.
For $0\leq s_1<s_2\leq t$ and $v(s_2,y_2)>0$, substitution of
\eqref{eq:Dirichlet-bridge-marginal} gives
\[
 \mathfrak b_{t,x}(s_2,y_2)\mathfrak b_{s_2,y_2}(s_1,y_1)
 =\frac{v(s_1,y_1)}{v(t,x)}
   \mathcal K_D(t,s_2;x,y_2)\mathcal K_D(s_2,s_1;y_2,y_1).
\]
If $v(s_2,y_2)=0$, all terms in the nonnegative sum
\[
 v(s_2,y_2)=\sum_{y\in\Z^d}
                 \mathcal K_D(s_2,s_1;y_2,y)v(s_1,y)
\]
vanish. Hence
$\mathcal K_D(s_2,s_1;y_2,y_1)v(s_1,y_1)=0$ whenever
$v(s_2,y_2)=0$. Such $y_2$ contribute zero to the kernel sum.
Summing over
$y_2\in\Z^d$ and applying the kernel composition law proves the
asserted consistency. Inserting
\eqref{eq:Dirichlet-bridge-marginal} into
\eqref{eq:quadratic-Dirichlet-potential} gives \eqref{eq:normalized-potential-identity}
by Tonelli's theorem. For $t\leq T_0$, the bound
$J_v(t,x)\leq M_{T_0}t\,v(t,x)$, with
$M_{T_0}=\ee^{a_{\max}T_0}\|v_0\|_\infty$, gives
$J_v(t,x)/v(t,x)^2\leq M_{T_0}t/v(t,x)<\infty$ at the fixed
endpoint where $v(t,x)>0$.
\end{proof}

\begin{proposition}\label{prop:general-lower-criterion}
Assume \eqref{eq:periodic-kpp}--\eqref{eq:u0}, fix $e\in\mathbb S^{d-1}$, and use
\eqref{eq:moving-Dirichlet-domain}. Suppose that
$v_0:\Z^d\to[0,\infty)$ satisfies $0\leq v_0\leq u_0$,
$v_0\not\equiv0$, $|\supp v_0|<\infty$ and $\supp v_0\subset D(0)$. Let $v(t,\cdot)=\mathcal U_D(t,0)v_0$.
Assume that there are $m,C_*,R,t_0>0$ such that, for every $t\geq t_0$,
one site $x_t\in D(t)$ satisfies all three inequalities
\begin{equation}\label{eq:criterion-seed-inputs}
 |x_t- {m_e^T}(t)e|\leq R,\qquad v(t,x_t)\geq m,\qquad
 J_v(t,x_t)\leq C_*v(t,x_t)^2.
\end{equation}
Then \eqref{eq:general-target} holds for every $\theta\in(0,1)$ and
$\rho\geq\sqrt d$. Moreover, for every $\theta\in(0,1)$ and $\rho\geq0$,
there are $C>0$ and $t_1\geq0$ such that
\begin{equation}\label{eq:criterion-inner-tube}
 u(t,x)\geq\theta
 \text{ for }t\geq t_1,\ x\in\Z^d,\ |x-re|\leq\rho,
 \ r\in[0, {m_e^T}(t)-C].
\end{equation}
\end{proposition}
\begin{proof}
Proposition~\ref{thm:Dirichlet-quotient} and
\eqref{eq:criterion-seed-inputs} give, for every $s\geq t_0$,
$u(s,x_s)\geq\eta:=\min\{1/2,(m^{-1}+KC_*)^{-1}\}>0$.
Restarting \eqref{eq:heat-comparison} at time $s$ and using
$u(s,x_s)\ge\eta$ together with positivity of the heat semigroup,
we obtain
\[
 u(s+1,x)\geq[\ee^{\Delta_{\Z^d}}u(s,\cdot)](x)
 \ge\eta\pheat_d(1,x-x_s).
\]
For $|x- {m_e^T}(s)e|\le\rho$ and
$|x_s- {m_e^T}(s)e|\le R$, the displacement $x-x_s$ belongs to the finite set
$\{z\in\Z^d:|z|\le R+\rho\}$. Hence
$\eta_\rho:=\eta\min_{z\in\Z^d,\,|z|\leq R+\rho}\pheat_d(1,z)>0$,
so $u(s+1,x)\ge\eta_\rho$.
Applying Lemma~\ref{lem:amplification} with its parameter $\eta$
replaced by $\eta_\rho$ gives $u(t,x)\geq\theta$ for
$t\geq s+1+S_\rho$, where $S_\rho$ is independent of $s$ and $x$.

Since $\alpha_e w(e)=h$ and $T>\kappa/h$, differentiation of
\eqref{eq:shifted-radial-position} gives
\[
 ( {m_e^T})'(t)=w(e)-\frac{\kappa}{\alpha_e(t+T)}\in(0,w(e)],
 \text{ for }t\geq0.
\]
As $ {m_e^T}(t)\to\infty$, for every $r\geq  {m_e^T}(t_0)$ there is a
unique $s\geq t_0$ such that $ {m_e^T}(s)=r$. If
$r\leq  {m_e^T}(t)-C$, where $C>w(e)(1+S_\rho)$, then
$C\leq  {m_e^T}(t)- {m_e^T}(s)\leq w(e)(t-s)$.
Thus $s\leq t-1-S_\rho$. Hence \eqref{eq:criterion-inner-tube} holds for these $r$. The remaining bounded range of $r$ contains only finitely many lattice
vertices in the tube. At time $1$, \eqref{eq:heat-comparison} is strictly
positive at each of these vertices; taking the minimum over this finite set
gives a positive lower bound. A further application of
Lemma~\ref{lem:amplification} gives one time after which all these
vertices exceed $\theta$.

For $\rho\geq\sqrt d$ and all sufficiently large $t$, set
$r_t= {m_e^T}(t)-C-2\sqrt d\geq0$ and round the coordinates of $r_te$
to obtain $z_t\in\Z^d$, with $|z_t-r_te|\leq\sqrt d/2$.
Then $z_t\cdot e\geq r_t-\sqrt d/2\geq0$ for large $t$, and
$|z_t-(z_t\cdot e)e|\leq\sqrt d/2\leq\rho$.
Since $r_t\leq  {m_e^T}(t)-C$, \eqref{eq:criterion-inner-tube} gives
$u(t,z_t)\geq\theta$. Also $z_t\cdot e=r_t+O(\sqrt d)= {m_e^T}(t)-O(1)$, so
$R_{\theta,e,\rho}(t)\geq  {m_e^T}(t)-O(1)$.
Theorem~\ref{thm:upper} gives
$R_{\theta,e,\rho}(t)\leq  {m_e^T}(t)+O(1)$. Finally, \eqref{eq:shifted-radial-position} gives
$ {m_e^T}(t)=w(e)t-(\kappa_d/\alpha_e)\log t+O(1)$ for fixed $T$.
Substituting this asymptotic in the two bounds proves \eqref{eq:general-target}.
\end{proof}

\begingroup \begin{remark}\label{rem:seed-verification}\color{black}
 {The kernel bound below gives a sufficient condition for $v(t,x_t)\geq m$; the bound on $J_v(t,x_t)$ must be verified separately.}
The time change retains the coefficient $1- {\omega}(\tau)$ of $q_\tau$
in \eqref{eq:time-changed-Dirichlet}, as in
\cite[Proposition~3.2]{Shabani}.
Lemma~\ref{lem:killed-kernel-upper} gives \eqref{eq:killed-kernel-upper}
for $\mathsf K_p$, associated with $Q_p$ in that lemma.
Assume \eqref{eq:target-killed-lower} and take $v_0$ satisfying the
hypotheses of Proposition~\ref{prop:general-lower-criterion}.
Choose $y_0\in D(0)$ with $v_0(y_0)>0$, put $\tau=\sigma(t)$, and choose
$\ell\in\Z^d$ with $p\cdot\ell>|b|+|p|\sqrt d/2+1$.
Such an $\ell$ exists because some component of $p$ is nonzero:
take a sufficiently large integer multiple of the corresponding signed coordinate vector.
Round the coordinates of $ {m_e^T}(t)e+\ell$ to obtain
$x_t= {m_e^T}(t)e+\ell+ {\varepsilon_t}\in\Z^d$, with $| {\varepsilon_t}|\leq\sqrt d/2$.
Then $|x_t- {m_e^T}(t)e|\leq|\ell|+\sqrt d/2$, and
$r_t=p\cdot x_t-h\tau+b$ satisfies
\[
 p\cdot\ell+b-|p|\sqrt d/2\leq r_t
 \leq p\cdot\ell+b+|p|\sqrt d/2.
\]
The lower endpoint is positive. Choose $0<d_-<d_+$ so that
$[d_-,d_+]$ contains both this interval and the positive number
$p\cdot y_0+b$.
By \eqref{eq:operational-time} and Lemma~\ref{lem:geometry},
$\Theta(0,\tau)=t$ and $v_p=w(e)e$. Hence
\[
 y_0-x_t+t v_p
 =y_0-\ell- {\varepsilon_t}+
   \frac{\kappa}{\alpha_e}\log\frac{t+T}{T}\,e.
\]
Its norm divided by $\sqrt t$ tends to zero, so for any fixed $A>0$
the displacement condition and $t\geq T_0$ hold for all sufficiently large $t$.
With $q(0,y_0)=\ee^{p\cdot y_0}v_0(y_0)/\phi_p(y_0)>0$,
\eqref{eq:target-killed-lower} gives
$q(\tau,x_t)\geq c\,q(0,y_0)t^{-\kappa}$.
Undoing the weight by \eqref{eq:time-change-amplitude} gives
$v(t,x_t)=\phi_p(x_t)\ee^{b-r_t}((t+T)/T)^\kappa q(\tau,x_t)$.
The positive cell minimum of $\phi_p$ and the bounds on $r_t$
therefore yield $v(t,x_t)\geq m>0$ uniformly for large $t$.
The choice of $\ell$ requires only $p\ne0$, and coordinate rounding gives
$| {\varepsilon_t}|\leq\sqrt d/2$ in every direction. Thus the construction also
covers signed and irrational directions.
At these same $x_t$, \eqref{eq:criterion-seed-inputs} also requires
$J_v(t,x_t)\leq C_*v(t,x_t)^2$ with $C_*$ independent of $t$.
\end{remark}\endgroup

\begingroup \begin{remark}\label{rem:At-integrability}\color{black}
For every $t\geq t_0$, choose a site $x_t\in D(t)$ satisfying
$|x_t- {m_e^T}(t)e|\leq R$ and $v(t,x_t)\geq m$, with fixed $R,m>0$.
For these same sites define
\[
 A_t(s)=\sum_{y\in\Z^d}\mathfrak b_{t,x_t}(s,y)
                       \frac{v(s,y)}{v(t,x_t)},\qquad 0\leq s\leq t.
\]
The function $A_t$ is nonnegative and measurable in $s$ for each fixed $t$;
no measurable choice of $t\mapsto x_t$ is needed.
By \eqref{eq:normalized-potential-identity},
$J_v(t,x_t)/v(t,x_t)^2=\int_0^tA_t(s)\,\dd s$.
Thus the remaining bound in \eqref{eq:criterion-seed-inputs} is equivalent to
$\sup_{t\geq t_0}\int_0^tA_t(s)\,\dd s<\infty$.
The elementary estimate
$J_v(t,x)\leq t\bigl(\sup_{s\in[0,t],\,y\in\Z^d}v(s,y)\bigr)v(t,x)$
follows from positivity and composition, as in \eqref{eq:quotient-damping-comparison}.
For example, the required uniform bound on
$J_v(t,x_t)/v(t,x_t)^2$ follows from
\[
 A_t(s)\leq C\bigl((1+s)^{-1-\varepsilon}
                   +(1+t-s)^{-1-\varepsilon}\bigr)
 \quad\text{for almost every }s\in[0,t]
\]
with fixed $C,\varepsilon>0$ for every $t\geq t_0$. Integration gives
$\int_0^tA_t(s)\,\dd s\leq2C/\varepsilon$, uniformly in $t$.
More generally, let $f_1,f_2\in L^1(0,\infty)$ be fixed and nonnegative.
If $A_t(s)\leq f_1(s)+f_2(t-s)$ for almost every $s\in[0,t]$ and every
$t\geq t_0$, then
$\int_0^tA_t(s)\,\dd s\leq\|f_1\|_{L^1(0,\infty)}+\|f_2\|_{L^1(0,\infty)}$.
The endpoint values $A_t(0)$ and $A_t(t)=1$ leave the integral unchanged,
since individual times have zero Lebesgue measure.

If for some sequence
$t_k\to\infty$ there are measurable sets $E_k\subset[0,t_k]$
with Lebesgue measures $|E_k|\to\infty$, and numbers $c,\delta>0$, such that
\[
 \sum_{\substack{y\in\Z^d\\v(s,y)\geq c\,v(t_k,x_{t_k})}}
       \mathfrak b_{t_k,x_{t_k}}(s,y)\geq\delta
 \quad\text{for almost every }s\in E_k,
\]
then $\int_0^{t_k}A_{t_k}(s)\,\dd s\geq c\delta|E_k|\to\infty$,
and these sites cannot satisfy the uniform bound on
$J_v(t,x_t)/v(t,x_t)^2$ in
\eqref{eq:criterion-seed-inputs}. A fixed
positive lower bound for $A_t$ on a time interval of length
proportional to $t$ is a special case. Independently, if
\[
 \sup_{\substack{x\in D(t)\\|x- {m_e^T}(t)e|\leq R}}v(t,x)\to0
\]
along some sequence of times tending to $\infty$, then no family $x_t\in D(t)$ with $|x_t- {m_e^T}(t)e|\leq R$ can satisfy
$v(t,x_t)\geq m>0$ for all sufficiently large $t$. These conclusions concern the
prescribed $v_0$, $D(t)$ and $(x_t)_{t\geq t_0}$ in
\eqref{eq:criterion-seed-inputs}.

For $0<\varepsilon\leq1$, replacing $v_0$ by $\varepsilon v_0$
preserves $0\leq\varepsilon v_0\leq u_0$. By linearity, $v$ becomes
$\varepsilon v$ and $J_v$ becomes $\varepsilon^2J_v$.
Thus $\varepsilon v(t,x_t)\geq\varepsilon m>0$, whereas
$A_t$ and $J_v/v^2$ are unchanged. In particular, any divergence of
$J_v/v^2$ is preserved.
For the same $v$ and sites $x_s$, the function $E_v$ in
\eqref{eq:quotient-damping-comparison} satisfies
\[
 E_v(s)\geq\sup_{y\in\Z^d}v(s,y)\geq v(s,x_s)\geq m
       \quad\text{for }s\geq t_0.
\]
Consequently, $\int_{t_0}^tE_v(s)\,\dd s\geq m(t-t_0)$, so
$E_v$ cannot belong to $L^1(0,\infty)$. The bound
$v(t,x)\int_0^tE_v(s)\,\dd s$ may exceed $J_v(t,x)$ because
the probability weights can assign small mass to sites where
$v(s,\cdot)$ is largest.

To apply Proposition~\ref{prop:general-lower-criterion}, the estimates in
\eqref{eq:criterion-seed-inputs} must hold at the same sites $x_t$.
Under \eqref{eq:target-killed-lower}, the sites constructed above satisfy
$v(t,x_t)\geq m>0$ for all sufficiently large $t$.
For the same sites, $\sup_{t\geq t_0}\int_0^tA_t(s)\,\dd s<\infty$
is equivalent to $J_v(t,x_t)\leq C_*v(t,x_t)^2$ for all $t\geq t_0$,
with $C_*$ independent of $t$.
Under \eqref{eq:transverse-linearization}, the estimates in
Lemma~\ref{lem:transverse-kernel}, \eqref{eq:product-comparison} and
\eqref{eq:scalar-profile-lower}, together with Lemma~\ref{lem:amplification}, give
$\sup\{j\in\Z:u(t,j,0)\geq\theta\}\geq m_d(t)-C_\theta$
for each $\theta\in(0,1)$ and all sufficiently large $t$, as proved in
Theorem~\ref{thm:axial}.
\end{remark}\endgroup

\subsection{Examples, equilibrium normalization {, and further discussions}}
First assume \eqref{eq:transverse-linearization} and let
$\bar a_\perp=|\mathcal C_\perp|^{-1}
\sum_{y\in\mathcal C_\perp}a_\perp(y)$.
Identify $\zeta\in\R^{\mathcal C_\perp}$ with its periodic extension
to $\Z^{d-1}$. The principal eigenvalue $a_{\mathrm{eff}}$ of the symmetric operator
$\Delta_{\Z^{d-1}}+a_\perp(y)$ in \eqref{eq:transverse-eigenpair}
has the Rayleigh characterization below; the test function
$\zeta\equiv1$ gives the lower bound.
\begin{equation}\label{eq:transverse-Rayleigh}
 a_{\mathrm{eff}}=
 \max_{\substack{\zeta\in\R^{\mathcal C_\perp}\\\zeta\ne0}}
 \frac{\displaystyle\sum_{y\in\mathcal C_\perp}a_\perp(y)\zeta(y)^2
       -\frac12\sum_{y\in\mathcal C_\perp}
          \sum_{\substack{b\in\Z^{d-1}\\|b|_1=1}}
                      (\zeta(y+b)-\zeta(y))^2}
      {\displaystyle\sum_{y\in\mathcal C_\perp}\zeta(y)^2}
 \geq\bar a_\perp.
\end{equation}
The Rayleigh identity follows directly by pairing the directed
edges $(y,b)$ and $(y+b,-b)$:
\[
 \sum_{y\in\mathcal C_\perp}\zeta(y)\Delta\zeta(y)
 =-\frac12\sum_{y\in\mathcal C_\perp}\sum_{\substack{b\in\Z^{d-1}\\|b|_1=1}}(\zeta(y+b)-\zeta(y))^2.
\]
Maximizing the quadratic form over the unit sphere in the finite
cell gives an eigenvector by differentiating along every tangent
direction. Replacing that vector by its absolute value cannot
decrease the quotient; connectedness and the eigenvalue equation
make a nonnegative maximizing vector positive. Perron--Frobenius
therefore identifies its eigenvalue with $a_{\mathrm{eff}}$.

Equality in the inequality $a_{\mathrm{eff}}\geq\bar a_\perp$ in \eqref{eq:transverse-Rayleigh} holds only for a constant $a_\perp$. Indeed, equality would
make the constant test vector a maximizer of the Rayleigh quotient,
hence an eigenvector of the symmetric cell matrix. The eigenvalue equation \eqref{eq:transverse-eigenpair} then forces $a_\perp(y)=a_{\mathrm{eff}}$ everywhere.

For $a>0$, define $c(a)$ and $\lambda(a)$ as in
\eqref{eq:critical-system} by
\[
 c(a)=2\sinh\lambda(a),\qquad
 \lambda(a)c(a)=2\cosh\lambda(a)-2+a.
\]
Implicit differentiation gives
$\lambda'(a)=1/(2\lambda(a)\cosh\lambda(a))>0$ and
$c'(a)=1/\lambda(a)>0$. For nonconstant $a_\perp$, \eqref{eq:transverse-Rayleigh} and its equality
case give $a_{\mathrm{eff}}>\bar a_\perp$. The positive derivatives of
$c$ and $\lambda$ then yield
\[
 c(a_{\mathrm{eff}})>c(\bar a_\perp),\qquad
 \frac{\kappa_d}{\lambda(a_{\mathrm{eff}})}
 <\frac{\kappa_d}{\lambda(\bar a_\perp)}.
\]

Let $d=2$ and $a_0,a_1>0$. Define $a_\perp(y)=a_0$ for
$y\in2\Z$ and $a_\perp(y)=a_1$ for $y\in2\Z+1$. On the two-site transverse cell,
\begin{equation}\label{eq:transverse-two-cell}
 a_{\mathrm{eff}}=-2+\frac{a_0+a_1}{2}
                  +\sqrt{\left(\frac{a_0-a_1}{2}\right)^2+4}.
\end{equation}
If $a_0\ne a_1$, the value in \eqref{eq:transverse-two-cell} is strictly larger than their arithmetic mean.

Indeed the two directed transverse neighbors have the same
opposite residue, so the cell matrix is
$\left(\begin{smallmatrix}a_0-2&2\\2&a_1-2\end{smallmatrix}\right)$.
Its characteristic equation is
$(\xi-a_0+2)(\xi-a_1+2)-4=0$.
The larger root is the value of $a_{\mathrm{eff}}$ in \eqref{eq:transverse-two-cell}; the square root
exceeds $2$ precisely when $a_0\ne a_1$. Theorem~\ref{thm:axial} applies along $e_1$ with $c(a_{\mathrm{eff}})$ and
$\lambda(a_{\mathrm{eff}})$, where $a_{\mathrm{eff}}$ is given by
\eqref{eq:transverse-two-cell}, since the growth rate varies in the transverse coordinate.

For $d\geq2$, let $a(x)=a_0$ when $x_1\in2\Z$ and $a(x)=a_1$
when $x_1\in2\Z+1$. On the cell $\{0,1\}\times\{0\}^{d-1}$,
$B_p$ has off-diagonal entries $\ee^{p_1}+\ee^{-p_1}=2\cosh p_1$ and
diagonal entries $a_i-2+2\sum_{k=2}^d(\cosh p_k-1)$ for $i\in\{0,1\}$.
A coordinate of period one contributes $\ee^{p_k}+\ee^{-p_k}-2$:
the two directed jumps have different displacements even when their residues coincide.
Diagonalizing this matrix gives
\begin{equation}\label{eq:two-cell}
 H(p)=2\sum_{i=2}^d(\cosh p_i-1)-2+\frac{a_0+a_1}{2}
       +\sqrt{\left(\frac{a_0-a_1}{2}\right)^2+4\cosh^2p_1}.
\end{equation}
Equation~\eqref{eq:main}
with this alternating linear growth rate is covered by
Theorem~\ref{thm:upper} in every direction and by
Theorem~\ref{thm:axial} along $e_2,\ldots,e_d$.
If $a_0\ne a_1$, applying Proposition~\ref{prop:general-lower-criterion}
along $e_1$ requires verifying
\eqref{eq:criterion-seed-inputs}. If $a_0=a_1$, Theorem~\ref{thm:axial}
also applies along $e_1$.

To normalize a positive periodic equilibrium to $1$, consider
the evolution equation $u_t=\Delta_{\Z^d}u+G(x,u)$. Assume it has a
positive $\Gamma$-periodic equilibrium $\varpi:\Z^d\to(0,\infty)$,
satisfying $\Delta_{\Z^d}\varpi(x)+G(x,\varpi(x))=0$ for $x\in\Z^d$.
Define
\[
 \widetilde G(x,r)=\frac{G(x,\varpi(x)r)-rG(x,\varpi(x))}{\varpi(x)}
 \quad\text{for }x\in\Z^d,\ 0\leq r\leq1.
\]
Substituting $u(t,x)=\varpi(x)v(t,x)$ in the evolution equation and
using the equilibrium identity gives
\begin{equation}\label{eq:equilibrium-normalization}
 \begin{split}
 v_t(t,x)
 &=\sum_{\substack{z\in\Z^d\\|z|_1=1}}\frac{\varpi(x+z)}{\varpi(x)}
                         \bigl(v(t,x+z)-v(t,x)\bigr)\\
 &\quad+\frac{G(x,\varpi(x)v(t,x))-v(t,x)G(x,\varpi(x))}
                  {\varpi(x)}.
 \end{split}
\end{equation}
 {Each jump rate} $\varpi(x+z)/\varpi(x)$ in \eqref{eq:equilibrium-normalization} is positive
and periodic for  {$z\in\Z^d$} with $|z|_1=1$.
These ratios are all equal to one exactly when $\varpi$ is constant, by connectedness of the nearest-neighbor lattice.

Indeed,
\[
 \frac{\Delta(\varpi v)(x)}{\varpi(x)}
 =\sum_{\substack{z\in\Z^d\\|z|_1=1}}\frac{\varpi(x+z)}{\varpi(x)}(v(x+z)-v(x))
       +\frac{\Delta\varpi(x)}{\varpi(x)}v(x).
\]
The equilibrium equation replaces its last coefficient by
$-G(x,\varpi(x))/\varpi(x)$, yielding
\eqref{eq:equilibrium-normalization}. Since the cell is finite,
$0<\min_{x\in\mathcal C}\varpi(x)\leq\max_{x\in\mathcal C}\varpi(x)<\infty$, so $\varpi(x+z)/\varpi(x)$ is
bounded above and away from zero uniformly for $|z|_1=1$.
To apply \eqref{eq:periodic-kpp} to $\widetilde G$, one must verify
$\Gamma$-periodicity in $x$, $\widetilde G(x,\cdot)\in C^2([0,1])$,
$\widetilde G(x,0)=0$, $\partial_r\widetilde G(x,0)>0$,
$\partial_r\widetilde G(x,1)<0$, and
$0<\widetilde G(x,r)\leq r\partial_r\widetilde G(x,0)$ for
$x\in\Z^d$ and $0<r<1$. The formula gives $\widetilde G(x,1)=0$.

Once these conditions on $\widetilde G$ have been verified, a lower comparison
for \eqref{eq:equilibrium-normalization} requires kernel estimates for the jump rates
$\varpi(x+z)/\varpi(x)$ and bounds on
$r\partial_r\widetilde G(x,0)-\widetilde G(x,r)$ for $0\leq r\leq1$.

\begingroup \begin{remark}\label{rem:time-dependent}\color{black}
Spreading in time-dependent KPP media has been studied in
 {\cite{CaoShenLattice,HuangShen,ShenDiscrete,ShenVariational}}. For \eqref{eq:main}, the tilted cell matrix $B_p$ and its left and
right eigenvectors are independent of time, and the tilted lattice
walk is time homogeneous. For a KPP reaction $F(t,x,u)$ that is almost periodic in time and
$\Gamma$-periodic in $x$, assume that
$a(t,x):=\partial_uF(t,x,0)$ exists, is continuous in $t$, and is
bounded uniformly on $\R\times\Z^d$.
For $f\in\ell^\infty(\Z^d)$, define the tilted operator by
\[
 (\mathcal B_p(t)f)(x)=
 \sum_{\substack{z\in\Z^d\\|z|_1=1}}\bigl(\ee^{-p\cdot z}f(x+z)-f(x)\bigr)
       +a(t,x)f(x).
\]
Its restriction to $\Gamma$-periodic functions is the finite
matrix $A_p(t)$ obtained from \eqref{eq:cell} with $a(x)$
replaced by $a(t,x)$. For $s\in\R$ and
$\xi\in(0,\infty)^{\mathcal C}$, let
$v:[s,\infty)\to\R^{\mathcal C}$ solve
$v_t=A_p(t)v$ with $v(s)=\xi$. In the autonomous
case, the fixed right and left vectors $\phi_p,\phi_p^*$ yield
the identities \eqref{eq:Q}--\eqref{eq:reverse-edge}.
A logarithmic position estimate in the nonautonomous setting requires
critical kernel bounds for $f_t=\mathcal B_p(t)f$ on $\Z^d$, uniform in time, together with estimates of $\partial_uF(t,x,0)r-F(t,x,r)$ for $0\leq r\leq1$.
\end{remark}\endgroup

\section*{Statements and Declarations}

\paragraph{Funding.}
This work was partially supported by the National Natural Science
Foundation of China (12201184 and 12301641) and the Strategic Priority
Research Program of the Chinese Academy of Sciences (XDB0900100).

\paragraph{Competing interests.}
The authors have no relevant financial or non-financial interests to
disclose.

\paragraph{Data availability.}
No data was used for the research described in the article.


\begin{thebibliography}{99}
\bibitem{AlfaroGilettiXiao}
M. Alfaro, T. Giletti and D. Xiao,
The Bramson correction in the Fisher--KPP equation: from delay to advance,
\emph{Math. Ann.} \textbf{392} (2025), 5275--5316.

\bibitem{AW1978}
D. G. Aronson and H. F. Weinberger,
Multidimensional nonlinear diffusion arising in population genetics,
\emph{Adv. Math.} \textbf{30} (1978), 33--76.

\bibitem{BH2002}
H. Berestycki and F. Hamel,
Front propagation in periodic excitable media,
\emph{Comm. Pure Appl. Math.} \textbf{55} (2002), 949--1032.

\bibitem{BHN2005}
H. Berestycki, F. Hamel and N. Nadirashvili,
The speed of propagation for KPP type problems. I: Periodic framework,
\emph{J. Eur. Math. Soc.} \textbf{7} (2005), 173--213.

\bibitem{BFRZ}
C. Besse, G. Faye, J.-M. Roquejoffre and M. Zhang,
The logarithmic Bramson correction for Fisher--KPP equations on the
lattice $\Z$, \emph{Trans. Amer. Math. Soc.} \textbf{376} (2023),
8553--8619.

\bibitem{Billingsley}
P. Billingsley,
\emph{Probability and Measure}, 3rd ed., Wiley, New York, 1995.

\bibitem{Bramson1978}
M. D. Bramson,
Maximal displacement of branching Brownian motion,
\emph{Comm. Pure Appl. Math.} \textbf{31} (1978), 531--581.

\bibitem{Bramson}
M. Bramson,
Convergence of solutions of the Kolmogorov equation to travelling waves,
\emph{Mem. Amer. Math. Soc.} \textbf{44} (1983), no.~285, iv+190 pp.

\bibitem{CaoShenLattice}
F. Cao and W. Shen,
Spreading speeds and transition fronts of lattice KPP equations in time
heterogeneous media, \emph{Discrete Contin. Dyn. Syst.}
\textbf{37} (2017), 4697--4727.

\bibitem{ChenFuGuo2006}
X. Chen, S.-C. Fu and J.-S. Guo,
Uniqueness and asymptotics of traveling waves of monostable dynamics on lattices,
\emph{SIAM J. Math. Anal.} \textbf{38} (2006), 233--258.

\bibitem{ChenGuo2003}
X. Chen and J.-S. Guo,
Uniqueness and existence of traveling waves for discrete quasilinear monostable dynamics,
\emph{Math. Ann.} \textbf{326} (2003), 123--146.

\bibitem{ChengLiLin2015}
C.-P. Cheng, W.-T. Li and G. Lin,
Travelling wave solutions in periodic monostable equations on a two-dimensional spatial lattice,
\emph{IMA J. Appl. Math.} \textbf{80} (2015), 1254--1272.

\bibitem{Doetsch}
G. Doetsch,
\emph{Introduction to the Theory and Application of the Laplace Transformation},
Springer-Verlag, Berlin--Heidelberg, 1974.

\bibitem{EbertVanSaarloos}
U. Ebert and W. van Saarloos,
Front propagation into unstable states: universal algebraic convergence towards uniformly translating pulled fronts,
\emph{Phys. D} \textbf{146} (2000), 1--99.

\bibitem{Esseen1966}
C.-G.  {Esseen},
On the Kolmogorov--Rogozin inequality for the concentration function,
\emph{Z. Wahrscheinlichkeitstheorie Verw. Gebiete}
\textbf{5} (1966), 210--216.

\bibitem{FangYuZhao}
J. Fang, X. Yu and X.-Q. Zhao,
Traveling waves and spreading speeds for time--space periodic monotone systems,
\emph{J. Funct. Anal.} \textbf{272} (2017), 4222--4262.

\bibitem{FangZhao2014}
J. Fang and X.-Q. Zhao,
Traveling waves for monotone semiflows with weak compactness,
\emph{SIAM J. Math. Anal.} \textbf{46} (2014), 3678--3704.

\bibitem{Fisher1937}
R. A. Fisher,
The wave of advance of advantageous genes,
\emph{Ann. Eugenics} \textbf{7} (1937), 355--369.

\bibitem{FreidlinGartner1979}
M. Freidlin and J. G\"artner,
On the propagation of concentration waves in periodic and random media,
\emph{Soviet Math. Dokl.} \textbf{20} (1979), 1282--1286.

\bibitem{Gartner1982}
J. G\"artner,
Location of wave fronts for the multi-dimensional K-P-P equation and Brownian first exit densities,
\emph{Math. Nachr.} \textbf{105} (1982), 317--351.

\bibitem{Giletti}
T. Giletti,
Convergence to pulsating traveling waves with minimal speed in some KPP
heterogeneous problems, \emph{Calc. Var. Partial Differential Equations}
\textbf{51} (2014), 265--289.

\bibitem{GrahamPrecise}
C. Graham,
Precise asymptotics for Fisher--KPP fronts,
\emph{Nonlinearity} \textbf{32} (2019), 1967--1998.

\bibitem{GrahamNonlocal}
C. Graham,
The Bramson correction for integro-differential Fisher--KPP equations,
\emph{Commun. Math. Sci.} \textbf{20} (2022), 563--596.

\bibitem{GramaLauvergnatLePage}
I. Grama, R. Lauvergnat and E. Le Page,
Conditioned local limit theorems for random walks defined on finite Markov chains,
\emph{Probab. Theory Related Fields} \textbf{176} (2020), 669--735.

\bibitem{GuoHamel}
J.-S. Guo and F. Hamel,
Front propagation for discrete periodic monostable equations,
\emph{Math. Ann.} \textbf{335} (2006), 489--525.

\bibitem{GuoWu}
J.-S. Guo and C.-C. Wu,
Uniqueness and stability of traveling waves for periodic monostable
lattice dynamical system, \emph{J. Differential Equations}
\textbf{246} (2009), 3818--3833.

\bibitem{GuoWu2008}
J.-S. Guo and C.-H. Wu,
Existence and uniqueness of traveling waves for a monostable 2-D lattice dynamical system,
\emph{Osaka J. Math.} \textbf{45} (2008), 327--346.

\bibitem{GuoWu2010}
J.-S. Guo and C.-H. Wu,
Front propagation for a two-dimensional periodic monostable lattice dynamical system,
\emph{Discrete Contin. Dyn. Syst.} \textbf{26} (2010), 197--223.

\bibitem{HamelDecay}
F. Hamel,
Qualitative properties of monostable pulsating fronts: exponential decay and monotonicity,
\emph{J. Math. Pures Appl.} \textbf{89} (2008), 355--399.

\bibitem{HNRRShort}
F. Hamel, J. Nolen, J.-M. Roquejoffre and L. Ryzhik,
A short proof of the logarithmic Bramson correction in Fisher--KPP equations,
\emph{Netw. Heterog. Media} \textbf{8} (2013), 275--289.

\bibitem{HNRRPeriodic}
F. Hamel, J. Nolen, J.-M. Roquejoffre and L. Ryzhik,
The logarithmic delay of KPP fronts in a periodic medium,
\emph{J. Eur. Math. Soc.} \textbf{18} (2016), 465--505.

\bibitem{HamelRoques}
F. Hamel and L. Roques,
Uniqueness and stability properties of monostable pulsating fronts,
\emph{J. Eur. Math. Soc.} \textbf{13} (2011), 345--390.

\bibitem{HuAxial}
Y. Hu,
Bramson correction and convergence to the critical wave for Fisher--KPP
equations on $\Z^d$, preprint (2026),
 {\href{https://arxiv.org/abs/2609.10143}{arXiv:2609.10143}.}

\bibitem{HuangShen}
J. Huang and W. Shen,
Speeds of spread and propagation for KPP models in time almost and space
periodic media, \emph{SIAM J. Appl. Dyn. Syst.}
\textbf{8} (2009), 790--821.

\bibitem{Kato}
T. Kato,
\emph{Perturbation Theory for Linear Operators}, Classics in Mathematics,
Springer-Verlag, Berlin, 1995.

\bibitem{KPP1937}
A. N. Kolmogorov, I. G. Petrovskii and N. S. Piskunov,
\'Etude de l'\'equation de la diffusion avec croissance de la quantit\'e de mati\`ere et son application \`a un probl\`eme biologique,
\emph{Bull. Univ. \'Etat Moscou, S\'er. Int., Sect. A: Math. et M\'ecan.} \textbf{1} (1937), no.~6, 1--25.

\bibitem{Kuchment}
P. Kuchment,
\emph{Floquet Theory for Partial Differential Equations}, Operator Theory:
Advances and Applications, vol.~60, Birkh\"auser, Basel, 1993.

\bibitem{LiangWangZhouZhou2024}
X. Liang, H. Wang, Q. Zhou and T. Zhou,
Traveling fronts for Fisher--KPP lattice equations in almost-periodic media,
\emph{Ann. Inst. H. Poincar\'e C Anal. Non Lin\'eaire} \textbf{41} (2024), 1179--1237.

\bibitem{LiangYiZhao}
X. Liang, Y. Yi and X.-Q. Zhao,
Spreading speeds and traveling waves for periodic evolution systems,
\emph{J. Differential Equations} \textbf{231} (2006), 57--77.

\bibitem{LiangZhao2007}
X. Liang and X.-Q. Zhao,
Asymptotic speeds of spread and traveling waves for monotone semiflows with applications,
\emph{Comm. Pure Appl. Math.} \textbf{60} (2007), 1--40.

\bibitem{LiangZhao}
X. Liang and X.-Q. Zhao,
Spreading speeds and traveling waves for abstract monostable evolution systems,
\emph{J. Funct. Anal.} \textbf{259} (2010), 857--903.

\bibitem{LiangZhou2020}
X. Liang and T. Zhou,
Spreading speeds of KPP-type lattice systems in heterogeneous media,
\emph{Commun. Contemp. Math.} \textbf{22} (2020), 1850083.

\bibitem{LubetzkyThornettZeitouni}
E. Lubetzky, C. Thornett and O. Zeitouni,
Maximum of branching Brownian motion in a periodic environment,
\emph{Ann. Inst. Henri Poincar\'e Probab. Stat.}
\textbf{58} (2022), no.~4, 2065--2093.

\bibitem{McKean1975}
H. P. McKean,
Application of Brownian motion to the equation of Kolmogorov--Petrovskii--Piskunov,
\emph{Comm. Pure Appl. Math.} \textbf{28} (1975), 323--331.

\bibitem{MontoroPunzoSciunzi2018}
L. Montoro, F. Punzo and B. Sciunzi,
Pointwise estimates for solutions of semilinear parabolic inequalities with a potential,
\emph{Rend. Lincei Mat. Appl.} \textbf{29} (2018), 255--288.

\bibitem{NeyNummelin}
P. Ney and E. Nummelin,
Markov additive processes I. Eigenvalue properties and limit theorems,
\emph{Ann. Probab.} \textbf{15} (1987), no.~2, 561--592.

\bibitem{NRRSingle}
J. Nolen, J.-M. Roquejoffre and L. Ryzhik,
Convergence to a single wave in the Fisher--KPP equation,
\emph{Chinese Ann. Math. Ser. B} \textbf{38} (2017), 629--646.

\bibitem{NRRRefined}
J. Nolen, J.-M. Roquejoffre and L. Ryzhik,
Refined long-time asymptotics for Fisher--KPP fronts,
\emph{Commun. Contemp. Math.} \textbf{21} (2019), 1850072.

\bibitem{Norris}
J. R. Norris,
\emph{Markov Chains}, Cambridge Series in Statistical and Probabilistic
Mathematics, Cambridge University Press, Cambridge, 1997.

\bibitem{RRR2019}
J.-M. Roquejoffre, L. Rossi and V. Roussier-Michon,
Sharp large time behaviour in $N$-dimensional Fisher--KPP equations,
\emph{Discrete Contin. Dyn. Syst.} \textbf{39} (2019), 7265--7290.

\bibitem{RoquejoffreZhang}
J.-M. Roquejoffre and M. Zhang,
Large time behaviour in the multidimensional nonlocal Fisher--KPP equation,
preprint (2026), \href{https://arxiv.org/abs/2609.32051}{arXiv:2609.32051}.

\bibitem{Rossi2017}
L. Rossi,
The Freidlin--G\"artner formula for general reaction terms,
\emph{Adv. Math.} \textbf{317} (2017), 267--298.

\bibitem{SaumardWellner2014}
A. Saumard and J. A. Wellner,
Log-concavity and strong log-concavity: A review,
\emph{Statist. Surv.} \textbf{8} (2014), 45--114.

\bibitem{Seneta}
E. Seneta,
\emph{Non-negative Matrices and Markov Chains}, 2nd ed., Springer Series in
Statistics, Springer, New York, 2006.

\bibitem{Shabani}
B. Shabani,
Logarithmic Bramson correction for multi-dimensional periodic Fisher--KPP
equations,  {preprint (2019), \href{https://arxiv.org/abs/1910.08178}{arXiv:1910.08178}.}

\bibitem{ShenDiscrete}
W. Shen,
Spreading and generalized propagating speeds of discrete KPP models in
time varying environments, \emph{Front. Math. China}
\textbf{4} (2009), 523--562.

\bibitem{ShenVariational}
W. Shen,
Variational principle for spreading speeds and generalized propagating
speeds in time almost periodic and space periodic KPP models,
\emph{Trans. Amer. Math. Soc.} \textbf{362} (2010), 5125--5168.

\bibitem{Weinberger1982}
H. F. Weinberger,
Long-time behavior of a class of biological models,
\emph{SIAM J. Math. Anal.} \textbf{13} (1982), 353--396.

\bibitem{Weinberger2002}
H. F. Weinberger,
On spreading speeds and traveling waves for growth and migration models in a periodic habitat,
\emph{J. Math. Biol.} \textbf{45} (2002), 511--548.

\bibitem{WuChenHsu}
S.-L. Wu, G.-S. Chen and C.-H. Hsu,
Exact asymptotic behavior of pulsating traveling waves for a periodic
monostable lattice dynamical system, \emph{Proc. Amer. Math. Soc.}
\textbf{149} (2021), 1697--1710.

\bibitem{ZhangSharp}
{ M. Zhang,
Sharp asymptotics for the KPP equation with some front-like initial data,
preprint (2025), \href{https://arxiv.org/abs/2505.10580}{arXiv:2505.10580}.}

\bibitem{ZhaoWuXue2024}
{ H. Zhao, S.-L. Wu and X. Xue,
Monostable pulsating traveling waves in discrete periodic media with delay,
\emph{Nonlinear Anal. Real World Appl.} \textbf{77} (2024), article 104055.}

\end{thebibliography}
\end{document}